\documentclass[12pt,amstex]{amsart}

\usepackage{xr}
\usepackage[shortlabels]{enumitem}
\usepackage[english]{babel}
\usepackage{amscd, amsfonts, amsmath, amssymb, amsthm, epsfig, float, geometry, graphicx, pstricks, pst-node, txfonts, enumitem}

\usepackage[all]{xy}
\usepackage{tikz}
\usepackage{multicol}

\usepackage{hyperref}
\hypersetup{
	colorlinks,
	citecolor=black,
	filecolor=black,
	linkcolor=black,
	urlcolor=black
}

\numberwithin{equation}{section}

\theoremstyle{definition}
\newtheorem{thm}{Theorem}[section]
\newtheorem{lem}[thm]{Lemma}
\newtheorem{defn}[thm]{Definition}
\newtheorem{rem}[thm]{Remark}
\newtheorem{prop}[thm]{Proposition}
\newtheorem{ques}[thm]{Question}
\newtheorem{coro}[thm]{Corollary}
\newtheorem{ex}[thm]{Example}
\newtheorem{conj}[thm]{Conjecture}
\newtheorem{conv}[thm]{Convention}

\def \C {\mathbb{C}}
\def \E {\mathbb{E}}
\def \F {\mathbb{F}}

\def \N {\mathbb{N}}

\def \Q {\mathbb{Q}}
\def \R {\mathbb{R}}
\def \T {\mathbb{T}}
\def \V {\mathbb{F}_{p}^{d}}

\def \Z {\mathbb{Z}}

\def \+ {\hat{+}}
\def \- {\hat{-}}

\def \b {\bold{b}}

\def \d {\delta}
\def \e {\epsilon}

\def \w {\bold{w}}
\def \x {\bold{x}}

\def \cc {{\#}_\text{cc}}

\def \dep {{\#}_\text{dd}}

\def \gfv {\mathcal{J}(M)}

\def \poly {\text{poly}}
\def \pp {\perp_{M}}

\def \rank {\text{rank}}

\def \st {\text{HP}}
\def \sp {\text{span}}

\def \Vk {\mathbb{Z}_{K}^{d}}

\def \vol {\text{vol}}

\def \zp {\mathbb{Z}/p^{\mathbb{N}}}
\def \zpn {\mathbb{Z}/p^{N}}

\title[Spherical higher order Fourier analysis over finite fields II]{Spherical higher order Fourier analysis over finite fields II: additive combinatorics for shifted modules}
\author{Wenbo Sun}
\address[Wenbo Sun]{Department of Mathematics, Virginia Tech, 225 Stanger Street, Blacksburg, VA, 24061, USA}
\email{swenbo@vt.edu}

\thanks{The author was partially supported by the NSF Grant DMS-2247331}


\begin{document}

\maketitle

\begin{abstract}
  This paper is the second part of the series \emph{Spherical higher order Fourier analysis over finite fields}, aiming to develop the higher order Fourier analysis method along spheres over finite fields, and to solve the geometric Ramsey conjecture in the finite field setting.
  
  In this paper, we study additive combinatorial properties for shifted modules, i.e. 
  the structure of sets of the form $E\pm E$, where $E$ is a collection of shifted modules of the polynomial ring $\mathbb{R}[x_{1},\dots,x_{d}]$ and we identify two modules if their difference contains the zero polynomial. We show that under appropriate definitions, the set $E\pm E$ enjoys properties similar to the conventional setting where $E$ is a subset of an abelian group. In particular, among other results, we prove the Balog-Gowers-Szemer\'edi theorem, the Rusza's quasi triangle inequality and a weak form of the Pl\"unnecke-Rusza theorem in the setting of shifted modules. We also show that for a special class of maps $\xi$ from $\mathbb{Z}_{K}^{d}$ to the collection of all shifted modules of   $\mathbb{R}[x_{1},\dots,x_{d}]$, if the set $\xi(\mathbb{Z}_{K}^{d})+\xi(\mathbb{Z}_{K}^{d})$ has large additive energy, then $\xi$ is an almost linear Freiman homomorphism. This result is the crucial additive combinatorial input we need to prove the spherical Gowers inverse theorem in later parts of the series.
  \end{abstract}

\tableofcontents

  \section{Introduction}
\subsection{Additive combinatorics for affine subspaces and shifted modules}
  This paper is the second part of the series \emph{Spherical higher order Fourier analysis over finite fields} \cite{SunA,SunC,SunD}.
The purpose of this paper is to provide the additive combinatorial results we need for the proof of the spherical Gowers inverse theorem in \cite{SunC}.
Let $G$ be a discrete abelian group. A central topic in additive combinatorics is to study the structure of subsets $E$ of $G$ with large additive energy (meaning that the set $\{(a,b,c,d)\in E^{4}\colon a+b=c+d\}$ has cardinality $O(\vert E\vert^{3})$) or with small doubling constant (meaning that $\vert E+E\vert=O(\vert E\vert)$). Such sets are known to have structures related to arithmetic progressions. We refer the readers to \cite{TV06} for more details. 
Based on the existing knowledge in additive combinatorics, it was shown in Lemma F.1 of \cite{GTZ12} that if a function $f\colon \{0,\dots,N-1\}\to \mathbb{T}$ has large additive energy, then $f$ must have an almost linear structure in the sense that $$f(h)=C+\sum_{i=1}^{L}\{\alpha_{i} h\}\beta_{i} \mod \Z$$ for a large set of $h\in\{0,\dots,N-1\}$ for some $L\in\N$ and $C,\alpha_{i},\beta_{i}\in\R$.  This linearization result plays an important role in the proof of the Gowers inverse theorem in \cite{GTZ12}.

 In order to prove the spherical Gowers inverse theorem in \cite{SunC}, we also need a linearization result in the spherical setting.  In our context, the structure for functions with large additive energy is characterized by the following definition (recall that $\tau\colon \F_{p}\to\{0,\dots,p-1\}$ and $\iota\colon \Z\to\F_{p}$ are the natural embeddings):

\begin{defn}[Almost linear function and Freiman homomorphism]\label{2:llfh}
Let $L\in\N_{+}$,  
$G$ be an additive group, $H$ be a subset of $G$, $R$ be a subset of $\R$, and $\xi\colon H\to\R$ be a map.
We say that $\xi$ is an \emph{$R$-almost linear function} of \emph{complexity} at most $L$ if 
there exist  $\alpha_{i}\in \widehat{G}$ (the Pontryagin dual of $G$, see Section \ref{2:s:defn} for the definition) and $\beta_{i}\in R$ for all $1\leq i\leq L$ such that for all $h\in H$,
$$\xi(h):=\sum_{i=1}^{L}\{\alpha_{i}\cdot h\}\beta_{i}.
$$
We say that an $R$-almost linear function $\xi$  is an \emph{$R$-almost linear Freiman homomorphism} if  for all $h_{1},h_{2},h_{3},h_{4}\in H$ with $h_{1}+h_{2}=h_{3}+h_{4}$, we have that $$\xi(h_{1})+\xi(h_{2})\equiv\xi(h_{3})+\xi(h_{4}) \mod \Z.$$

Let $p$ be a prime.
We say that a map $\xi\colon H\to \F_{p}$ is an \emph{almost linear function/Freiman homomorphism} of complexity at most $L$ if there exists a $\Z/p$-almost linear function/Freiman homomorphism $\xi'$ with $\xi'(H)\subseteq\Z/p$ of complexity at most $L$ such that
$\xi(h)=\iota(p\xi'(h))$ for all $h\in H$.
%
%

We say that a map $\xi\colon H\to \R[x_{1},\dots,x_{d}]$ (resp. $\F_{p}[x_{1},\dots,x_{d}]$) is an \emph{$R$-almost linear function/$R$-Freiman homomorphism} (resp. \emph{almost linear function/Freiman homomorphism}) of complexity at most $L$ if $\xi_{m}$ is an $R$-almost linear function/$R$-Freiman homomorphism (resp. almost linear function/Freiman homomorphism) of complexity at most $L$ for all $m=(m_{1},\dots,m_{d})\in\N^{d}$, where $\xi_{m}(h)$ is the coefficient of the monomial $x_{1}^{m_{1}}\dots x_{d}^{m_{d}}$ of $\xi(h)$.
\end{defn}

In order to prove the 3-step spherical Gowers inverse theorem, the following linearization result is needed in \cite{SunC}.\footnote{The linearization result is not necessary for the proof of 2 or less step spherical Gowers inverse theorems.}

\begin{prop}[Linearization result for affine supsaces]\label{2:aadd0}
	Let $d\in\N_{+}$ with $d\gg 1$, $\d>0$ and $p\gg_{\d,d}1$ be a prime. Let  $H\subseteq \V$ with $\vert H\vert>\d p^{d}$, and $\xi_{i}$ be a map from $H$ to $\V$ for $1\leq i\leq 4$.
	Suppose that there exists 
	  $U\subseteq \{(h_{1},h_{2},h_{3},h_{4})\in H^{4}\colon h_{1}+h_{2}=h_{3}+h_{4}\}$
	with $\vert U\vert>\d \vert H\vert^{3}$ such that for all $(h_{1},h_{2},h_{3},h_{4})\in U$, 
	\begin{equation}\label{2:aadd01}
	\xi_{1}(h_{1})+\xi_{2}(h_{2})-\xi_{3}(h_{3})-\xi_{4}(h_{4})\in \sp_{\F_{p}}\{h_{1},h_{2},h_{3}\}.\footnote{Note that $\sp_{\F_{p}}\{h_{1},h_{2},h_{3}\}=\sp_{\F_{p}}\{h_{1},h_{2},h_{3},h_{4}\}$ since $h_{4}=h_{1}+h_{2}-h_{3}$. A similar remark applies throughout the paper.}
	\end{equation}
	Then there exist   $H'\subseteq H$ with $\vert H'\vert\gg_{\d,d} p^{d}$, some $g\in\F_{p}$, and a map $T\colon H'\to \V$ whose projection to each of the $d$ coordinates is an almost linear Freiman homomorphism  of complexity $O_{\d,d}(1)$  
	 such that 
	$$\xi_{1}(h)-T(h)-g\in \sp_{\F_{p}}\{h\} \text{ for all $h\in H'$.}$$
	  
\end{prop}

An alternative way to interpret Proposition \ref{2:aadd0}  is as follows. For affine subspaces $V+c$ and $V'+c'$ of $\V$, we informally write $V+c\sim V'+c'$ if $c-c'\in V+V'$ (see Section \ref{2:s:adbf} for the precise definitions).   
Roughly speaking, Proposition \ref{2:aadd0} claims that if the maps $$\tilde{\xi}_{i}\colon h\mapsto \xi_{i}(h)+\sp_{\F_{p}}\{h\}$$ from $\V$ to at most one dimensional affine subspaces of $\V$ has large additive energy in the sense 
that 
$$\tilde{\xi}_{1}(h_{1})+\tilde{\xi}_{2}(h_{2})\sim\tilde{\xi}_{3}(h_{3})+\tilde{\xi}_{4}(h_{4})$$
for all $(h_{1},h_{2},h_{3},h_{4})\in U$, then $\tilde{\xi}_{i}(h)$  equals to an almost linear Freiman homomorphism plus a constant modulo $\sp_{\F_{p}}\{h\}$.

In order to study the  higher-step spherical Gowers inverse theorem, we need to replace the functions $\xi_{i}$ in Proposition \ref{2:aadd0} by functions taking values in the ring of polynomials.
 Let $R$ be a subset of $\R$, $d\in\N_{+}$ and $s\in\N$.
We use $\st_{R,d}(s)$ to denote the set of homogeneous polynomial in $\R[x_{1},\dots,x_{d}]$ of degree $s$ whose coefficients are taken from $R$. 

\begin{conv}\label{2:csts}
In this paper, the zero polynomial is considered as a degree $s$ polynomial for any $s\in\Z$. In particular, the zero polynomial belongs to $\st_{R,d}(s)$ for any $s\in\N$.
For convenience, if $s<0$, then we use $\st_{R,d}(s)$ to denote the set consisting only the zero polynomial.
\end{conv}

In order to generalize (\ref{2:aadd01}) to $\st_{\R,d}(s)$-valued functions, we also need to extend the concept of subspaces of $\V$ accordingly:

\begin{defn}[$M$-module for real polynomials]\label{2:defmmp0}
		Let $d\in\N_{+}, p$ be a prime and $M\colon\V\to\F_{p}$ be a quadratic form associated with the matrix $A$ (see Section \ref{2:s:defn} for definitions). 
	We say that a subset $I$ of the polynomial ring $\R[x_{1},\dots,x_{d}]$ 
	is an \emph{$M$-module} if there exists a subspace $V$ of $\V$ such that 
	$$f\in I\Leftrightarrow f(n)\in\Z \text{ for all } n\in\Z^{d} \text{ with } (\iota(n)A)\cdot\iota(n)=0 \text{ and } (hA)\cdot \iota(n)=0 \text{ for all } h\in V.\footnote{This is equivalently of saying that $f\in I\Leftrightarrow f(n)\in\Z \text{ for all } n\in \iota^{-1}(V(M_{0})\cap V^{\pp})$, where $M_{0}(n):=(nA)\cdot n$. See Section \ref{2:s:defn} for definitions.}$$ 
	In this case we denote $I$ as $J^{M}_{V}$. For convenience we denote 
	$$J^{M}:=J^{M}_{\{\bold{0}\}} \text{ and } J^{M}_{h_{1},\dots,h_{k}}:=J^{M}_{\sp_{\F_{p}}\{h_{1},\dots,h_{k}\}}.$$ 
	for all $h_{1},\dots,h_{k}\in\V$.
	\end{defn}

For convenience, in the rest of the paper, we denote 
\begin{equation}\label{2:thisisns}
N(s):=(2s+16)(15s+453) \text{ for all } s\in\Z.
\end{equation}
The following is the additive combinatorial property we need in order to prove the general spherical Gowers inverse theorem in \cite{SunC}, which is also the main result of this paper.  

\begin{thm}[Linearization result for shifted modules]\label{2:aadd}
	Let $d,K,r,s\in\N_{+}$, with $d\geq N(s)$, $\d>0$ and $p\gg_{\d,d,r} 1$ be a prime dividing $K$,
	  $M\colon\V\to\F_{p}$ be a non-degenerate quadratic form,  $H\subseteq \Vk$ with $\vert H\vert>\d K^{d}$, and $\xi_{i}$ be a map from $H$ to $\st_{\Z/p^{r},d}(s)$ for $1\leq i\leq 4$.
	Suppose that there exists 
	  $U\subseteq \{(h_{1},h_{2},h_{3},h_{4})\in H^{4}\colon h_{1}+h_{2}=h_{3}+h_{4}\}$
	with $\vert U\vert>\d \vert H\vert^{3}$ such that for all $(h_{1},h_{2},h_{3},h_{4})\in U$, 
	\begin{equation}\nonumber
	\xi_{1}(h_{1})+\xi_{2}(h_{2})-\xi_{3}(h_{3})-\xi_{4}(h_{4})\in J^{M}_{\iota(h_{1}),\iota(h_{2}),\iota(h_{3})}.\footnote{Here $\iota$ is understood as the natural projection from $\Vk$ to $\F_{p}^{d}$.}
	\end{equation}
	Then there exist   $H'\subseteq H$ with $\vert H'\vert\gg_{\d,d,r} K^{d}$, some $g\in\st_{\Z/p^{r},d}(s)$, and a $\Z/p^{r}$-almost linear Freiman homomorphism  $T\colon H'\to \st_{\Z/p^{r},d}(s)$  of complexity $O_{\d,d}(1)$.  
	 such that 
	$$\xi_{1}(h)-T(h)-g\in J^{M}_{\iota(h)} \text{ for all $h\in H'$.}$$
\end{thm}	


We will make essential use of  Theorem \ref{2:aadd} in the Furstenburg-Weiss type argument in \cite{SunC}, which is an intermediate step towards the proof of the spherical Gowers inverse theorem. We refer the readers to Steps 2 and 3 of Section 
1.2 of \cite{SunC} for details.


\begin{rem}\label{2:i2a}
Note that if $s=1,  \Z_{K}^{d}=\F_{p}^{d}$ and $M(n)=n\cdot n$, then there is a natural group homomorphism $\sigma$ between $(\V,+)$ and $(\st_{\Z/p,d}(1),+)$ given by $\sigma(h)(n):=\frac{1}{p}(n\cdot \tau(h))$,  which maps  $\sp_{\F_{p}}\{h\}$ to $J^{M}_{h}\cap \st_{\Z/p,d}(1)$ since all polynomials in $J^{M}_{h}\cap \st_{\Z/p,d}(1)$ are integer scalar multiples of $\sigma(h)$.  
Moreover, for all $v\in\V$ and subspace $V$ of $\V$, we have $v\in V$ if and only if $\sigma(v)\in J^{M}_{V}$. So (\ref{2:aadd01}) is equivalent to
	$$\xi'_{1}(h_{1})+\xi'_{2}(h_{2})-\xi'_{3}(h_{3})-\xi'_{4}(h_{4})\in J^{M}_{h_{1},h_{2},h_{3}},$$
	where $\xi_{i}':=\sigma\circ \xi_{i}$ is a map from $H$ to $\st_{\Z/p,d}(1)$.
Therefore, Proposition \ref{2:aadd0} is a special case of Theorem \ref{2:aadd} when $s=1$.
\end{rem}

While additive combinatorial properties for abelian groups have been studied extensively in literature, very little is known for additive combinatorial properties of affine subspaces or shifted modules. In particular, all the powerful tools in additive combinatoric are unavailable. In order to prove  Theorem \ref{2:aadd}, we need to establish foundational tools in analogous to the abelian group setting from scratch. An interesting feature of this paper is that we use many tools in graph theory to study additive combinatorial properties for shifted modules (especially in Section \ref{2:s:115}).

  Since the proof of Theorem \ref{2:aadd} is very technical, when reading this paper for the first time,   the readers are recommended to focus on the case $s=1$ and $\Z_{K}^{d}=\F_{p}^{d}$ throughout, which captures almost all the main idea of the proof of the general case. In this case $\xi_{i}$ can be viewed as a map from $H$ to $\V$, and $J^{M}_{h_{1},\dots,h_{k}}$ can be interpreted as $\sp_{\F_{p}}\{h_{1},\dots,h_{k}\}$ (if $d\gg_{k} 1$).

\subsection{Organizations and notations}\label{2:s:defn}

In Section \ref{2:s:adbff}, we define a non-transitive relation  for shifted modules, and introduce some terminologies in graph theory which will be used later to prove Theorem \ref{2:aadd}. We also provide an outline of the proof Theorem \ref{2:aadd} with the help of these new languages.
In Section \ref {2:s:a1}, we prove an intersection property for $M$-modules, which is an approach that will be used frequently later in this paper to study the intersections of different $M$-modules sharing a common factor. 
Sections \ref{2:s:ff}--\ref{2:s:a2} are devoted to the proof of Theorem \ref{2:aadd}. We refer the readers to Section \ref{2:s:rdm} for a more details explanation for the roles these sections play in the proof of Theorem \ref{2:aadd} (after introducing all the relevant definitions). In Section \ref{2:s:op}, we collect some open questions. In Appendix \ref{2:AppA}, we recall some definitions and propositions from \cite{SunA} which are used in this paper.

  Below are the notations we use in this paper:

\begin{itemize}
	\item Let $\N,\N_{+},\Z,\Q,\R,\R+,\C$ denote the set of non-negative integers, positive integers, integers, rational numbers, real numbers, positive real numbers, and complex numbers, respectively. Denote $\T:=\R/\Z$. Let $\F_{p}$ denote the finite field with $p$ elements. Let $\Z_{K}$ denote the cyclic group with $K$ elements.
	\item Let $\zp$ denote the set of numbers of the form $a/p^{b}$ for some $a\in\Z$ and $b\in\N$.
		\item Throughout this paper, $d$ is a fixed positive integer, $p$ is a prime number and $K$ is a positive integer dividing $p$.
		\item Throughout this paper, unless otherwise stated, all vectors are assumed to be horizontal vectors.
		\item Throughout this paper, we use
   $\tau\colon\F_{p}\to \{0,\dots,p-1\}$ to denote the natural bijective embedding, and use $\iota$ to denote the map from $\Z$ (or $\Z_{K}$ for any $K$ divisible by $p$) to $\F_{p}$ given by $\iota(n):=\tau^{-1}(n \mod p\Z)$.
	We also use 
	$\tau$ to denote the map from $\F_{p}^{k}$ to $\Z^{k}$ (or $\Z_{K}^{k}$) given by $\tau(x_{1},\dots,x_{k}):=(\tau(x_{1}),\dots,\tau(x_{k}))$,
	and
	$\iota$ to denote the map from $\Z^{k}$ (or $\Z_{K}^{k}$) to $\F_{p}^{k}$ given by $\iota(x_{1},\dots,x_{k}):=(\iota(x_{1}),\dots,$ $\iota(x_{k}))$. 
	When there is no confusion, we will not state the domain and range of $\tau$ and $\iota$ explicitly.
	%
		\item Let $\mathcal{C}$ be a collection of parameters and $A,B,c\in\R$. We write $A\gg_{\mathcal{C}} B$ if $\vert  A\vert\geq K\vert B\vert$ and $A=O_{\mathcal{C}}(B)$ if $\vert A\vert\leq K\vert B\vert$ for some $K>0$ depending only on the parameters in $\mathcal{C}$.
		In the above definitions, we allow the set $\mathcal{C}$ to be empty. In this case $K$ will be a universal constant.
		\item Let $[N]$ denote the set $\{0,\dots,N-1\}$.
    \item For $x\in\R$, let $\lfloor x\rfloor$ denote the largest integer which is not larger than $x$, and $\lceil x\rceil$ denote the smallest integer which is not smaller than $x$. Denote $\{x\}:=x-\lfloor x\rfloor$.
	\item Let $X$ be a finite set and $f\colon X\to\C$ be a function. Denote $\E_{x\in X}f(x):=\frac{1}{\vert X\vert}\sum_{x\in X}f(x)$, the average of $f$ on $X$.
	\item For $F=\Z^{k}$ or $\F_{p}^{k}$, and $x=(x_{1},\dots,x_{k}), y=(y_{1},\dots,y_{k})\in F$, let $x\cdot y\in \Z$ or $\F_{p}$ denote the dot product given by
	$x\cdot y:=x_{1}y_{1}+\dots+x_{k}y_{k}.$
	\item We write affine subspaces of $\V$ as $V+c$, where $V$ is a subspace of $\V$ passing through $\bold{0}$, and $c\in\V$.
	\item For an additive group $G$, we use $\widehat{G}$ to denote the  \emph{Pontryagin dual} of $G$, i.e. the set of all continuous group homomorphism from $G$ to $\R/\Z$. For $g\in G$ and $h\in\widehat{G}$, we use $\{h\cdot g\}$ to denote the representative of $h(g)$ on the fundamental domain $[0,1)$ of $\R/\Z$.
	\item For a polynomial $P\in\poly(\F_{p}^{k}\to\F_{p})$, let $V(P)$ denote the set of $n\in\F_{p}^{k}$ such that $P(n)=0$. For a polynomial $P\in\poly(\Z^{k}\to\R)$, let $V_{p}(P)$ denote the set of $n\in \Z^{k}$ such that $P(n+pm)\in \Z$ for all $m\in\Z^{k}$.
	\item There is a natural correspondence between polynomials taking values in $\F_{p}$ and polynomials  taking values in $\Z/p$.   Let $F\in\poly(\V\to\F_{p}^{d'})$ and $f\in \poly(\Z^{d}\to (\Z/p)^{d'})$ be polynomials of degree at most $s$ for some $s<p$. If  $F=\iota\circ pf\circ\tau$, then we say that $F$ is \emph{induced} by $f$ and $f$ is a \emph{lifting} of $F$.\footnote{As is explained in \cite{SunA},   $\iota\circ pf\circ\tau$ is well defined.}
     We say that $f$ is a  \emph{regular lifting} of $F$ if in addition $f$ has the same degree as $F$ and $f$ has $\{0,\frac{1}{p},\dots,\frac{p-1}{p}\}$-coefficients. 
 \end{itemize}

We also need to recall the notations regarding quadratic forms defined in \cite{SunA}.	
	
\begin{defn}	
 We say that a function $M\colon\V\to\F_{p}$ is a \emph{quadratic form} if 
	$$M(n)=(nA)\cdot n+n\cdot u+v$$
	for some $d\times d$ symmetric matrix $A$ in $\F_{p}$, some $u\in \F_{p}^{d}$ and $v\in \F_{p}$.
We say that $A$ is the matrix \emph{associated to} $M$.
We say that $M$ is \emph{pure} if $u=\bold{0}$.
We say that $M$ is \emph{homogeneous} if $u=\bold{0}$ and $v=0$. We say that $M$ is \emph{non-degenerate} if $M$ is of rank $d$, or equivalently, $\det(A)\neq 0$.
\end{defn}

We use $\rank(M):=\rank(A)$ to denote the \emph{rank} of $M$.
Let $V+c$ be an affine subspace of $\V$ of dimension $r$. There exists a (not necessarily unique) bijective linear transformation $\phi\colon \F_{p}^{r}\to V$.
 We define the \emph{rank} $\rank(M\vert_{V+c})$ of $M$  restricted to $V+c$ as the rank of the quadratic form $M(\phi(\cdot)+c)$. It was proved in \cite{SunA} that   $\rank(M\vert_{V+c})$ is independent of the choice of $\phi$.

For any subspace $V$ of $\V$, let $V^{\pp}$ denote the set of $\{n\in\V\colon (mA)\cdot n=0 \text{ for all } m\in V\}$.



Quadratic forms can also be defined in the $\Z/p$-setting.

\begin{defn}	
   We say that a function $M\colon\Z^{d}\to\Z/p$ is a \emph{quadratic form} if 
	$$M(n)=\frac{1}{p}((nA)\cdot n+n\cdot u+v)$$
	for some $d\times d$ symmetric matrix $A$ in $\Z$, some $u\in \Z^{d}$ and $v\in \Z$.
We say that $A$ is the matrix \emph{associated to} $M$. 
 \end{defn}

By Lemma 
A.1 of \cite{SunA},
any quadratic form $\tilde{M}\colon\Z^{d}\to\Z/p$ associated with the matrix $\tilde{A}$ induces a quadratic form $M:=\iota\circ p\tilde{M}\circ\tau\colon\F_{p}^{d}\to\F_{p}$ associated with the matrix $\iota(\tilde{A})$. Conversely, any quadratic form $M\colon\F_{p}^{d}\to\F_{p}$ associated with the matrix $A$ admits a regular lifting $\tilde{M}\colon\Z^{d}\to\Z/p$, which is a quadratic form associated with the matrix $\tau(A)$.

For a quadratic form $\tilde{M}\colon\Z^{d}\to\Z/p$, we say that $\tilde{M}$ is \emph{pure/homogeneous/$p$-non-degenerate} if the  quadratic form $M:=\iota\circ p\tilde{M}\circ\tau$ induced by $\tilde{M}$ is pure/homogeneous/non-degenerate. 
The \emph{$p$-rank} of $\tilde{M}$, denoted by $\rank_{p}(\tilde{M})$, is defined to be the rank of $M$.

We say that $h_{1},\dots,h_{k}\in\Z^{d}$ are \emph{$p$-linearly independent} if for all $c_{1},\dots,c_{k}\in\Z/p$, $c_{1}h_{1}+\dots+c_{k}h_{k}\in\Z$
 implies that $c_{1},\dots,c_{k}\in\Z$, or equivalently, if $\iota(h_{1}),\dots,\iota(h_{k})$ are linearly independent.




%

\section{The relation graph for additive tuples}\label{2:s:adbff}

\subsection{A non-transitive relation for shifted $M$-modules}\label{2:s:adbf}
	Before proving Theorem \ref{2:aadd}, we need to rephraze the problem using some new terminology to be defined in this section. Throughout this paper, $p$ is a prime, $d$ is the dimension of the vector field over $\F_{p}$ and $\Q$,   $M\colon\V\to\F_{p}$ is a non-degenerate quadratic form, and $\tilde{M}\colon\Z^{d}\to\Z/p$ is the regular lifting of $M$.
For an $M$-module $I$ 
and a polynomial $f\in \Q[x_{1},\dots,x_{d}]$, we call $I+f$ a \emph{shifted $M$-module} of $\Q[x_{1},\dots,x_{d}]$. The space of all shifted $M$-modules of $\Q[x_{1},\dots,x_{d}]$ is denoted as $\gfv$.

We may define an ``additive" operation on $\gfv$ as follows.
For $I+f, I'+f'\in\gfv$, denote $(I+f)\+  (I'+f'):=I+I'+f+f'$ and $(I+f)\-  (I'+f'):=I-I'+f-f'$. Clearly, both $(\gfv,\+ )$ and $(\gfv,\- )$ are semi-groups, but the operations $\+ $ and $\- $ are not the inverse of each other.\footnote{We write the operations as $\+ $ and $\hat{-}$ instead of $+$ and $-$ to emphasize that they are not invertible.} 

We also need to define a relation between elements in $\gfv$ as follows.  
Denote $I+f\sim I'+f'$ if $f-f'\in I+I'$. 
Clearly, if $x\sim y$ and $x'\sim y'$, then $x\+ x'\sim y\+ y'$ and $x\- x'\sim y\- y'$.
It is easy to see that $u\sim u$, and $u\sim v\Rightarrow v\sim u$. However, the next example shows that $\sim$ is not transitive and so is not an equivalence relation.

\begin{ex}\label{2:nott}
	Let $\Z_{K}^{d}=\V$ and $M(n)=n\cdot n$.
   For all $x\in\V$, let $L_{x}\colon\Z^{d}\to\Z/p$ denote the function given by $L_{x}(n)=\frac{1}{p}\tau(x)\cdot n$.	Let $e_{1},\dots,e_{d}$ be the standard unit vectors of $\V$. Then $J^{M}_{e_{1}}+L_{e_{2}}\sim J^{M}_{e_{2}}$ and $J^{M}_{e_{2}}\sim J^{M}_{e_{3}}$. However, by Lemma \ref{2:counting02} and Proposition \ref{2:noloop3}, it is not hard to see that $J^{M}_{e_{1}}+L_{e_{2}}\not\sim J^{M}_{e_{3}}$ if $d\gg 1$ and $p\gg_{d} 1$. So $\sim$ is not transitive.
\end{ex}

The above mentioned operations and relation can be extended naturally to the set $\Vk\times \gfv$.
For $(h,I+f),(h',I'+f')\in\Vk\times \gfv$, denote $(h,I+f)\+ (h',I'+f'):=(h+h',I+I'+f+f')$ and $(h,I+f)\- (h',I'+f'):=(h+h',I-I'+f-f')$. We write $(h,I+f)\sim(h',I'+f')$ if $h=h'$ and $I+f\sim I'+f'$.

\begin{conv}
	Throughout this paper, we  use  $\pi\colon \Vk\times \gfv\to\Vk$ to denote the projection map.	
\end{conv}

In this paper, we are particularly interested in the non-transitive relation $\sim$ on a special subset  $\Gamma^{s}(\Vk,M)$ of $\Vk\times\gfv$, which denotes the set of all elements of the form $(h,J^{M}_{V}+f)$ for some $h\in\Vk$, some $f\in\st_{\zp,d}(s)$, and some $M$-module $J^{M}_{V}$ for some subspace $V$ of $\V$ with $\iota(h)\in V$.   We also need to divide the set $\Gamma^{s}(\Vk,M)$ further  according to the dimension of $V$. 
  For $k\in\N_{+},$
let $\Gamma^{s}_{k}(\Vk,M)$ denote the set of all $(h,J^{M}_{V}+f)\in\Gamma^{s}(\Vk,M), \iota(h)\neq \bold{0}$ with $\dim(V)\leq k$  union the set of all $(h,J^{M}_{V}+f)\in\Gamma^{s}(\Vk,M), \iota(h)=\bold{0}$ with $(V)\leq k-1$.
Equivalently, $\Gamma^{s}_{k}(\Vk,M)$ is the set of all $(h,J^{M}_{V}+f)\in\Gamma^{s}(\Vk,M), h\in\Vk$ where $V$ can be written as the span of $k$ vectors $h_{1},\dots,h_{k}\in\V$ with one of them being $\iota(h)$.
 In particular, $\Gamma^{s}_{1}(\Vk,M)$ consists of all elements of the form  $(h,J^{M}_{\iota(h)}+f), h\in\Vk, f\in \st_{\zp,d}(s)$.

\begin{ex}
Let $\xi_{i}\colon H\to \st_{\zp,d}(s), 1\leq i\leq 4$ be given as in Theorem \ref{2:aadd}. Let $\tilde{\xi}_{i}\colon H\to \Gamma^{s}_{1}(\Vk,M)$ be the map given by $\tilde{\xi}_{i}(h):=(h,\xi_{i}(h)+J^{M}_{\iota(h)})$. Informally speaking, the assumption in Theorem \ref{2:aadd} is equivalent of saying that the ``additive energy" of $\tilde{\xi}_{i}(H)\+ \tilde{\xi}_{i}(H)$ (i.e. the number of quadraples $(h_{1},h_{2},h_{3},h_{4})\in H^{4}$ with $\tilde{\xi}_{i}(h_{1})\+ \tilde{\xi}_{i}(h_{2})\sim \tilde{\xi}_{i}(h_{3})\+ \tilde{\xi}_{i}(h_{4})$) is at least $\d\vert H\vert^{3}$.
\end{ex}

We summarize some basic properties for later uses.
The following lemma is straightforward: 
\begin{lem}\label{2:a+b}
		Let $s\in\N$ and $k,k'\in\N_{+}$.
	If $a\in\Gamma^{s}_{k}(\Vk,M)$ and $b\in\Gamma^{s}_{k'}(\Vk,M)$, then $a\+ b\in\Gamma^{s}_{k+k'}(\Vk,M)$. 
	\end{lem}
\begin{proof}
	Assume that $a=(h,J^{M}_{V}+f)$ and  $b=(h',J^{M}_{V'}+f')$ for some $h,h'\in\Vk$, $f,f'\in\st_{\zp,d}(s)$ and subspaces $V,V'$ of $\V$  with $\iota(h)\in V$ and $\iota(h)'\in V'$. Then $\iota(h)+\iota(h')\in V+V'$.
	
	If $\iota(h)=\bold{0}$, then $\dim(V)\leq k-1$. Since $\dim(V')\leq k'$, we have $\dim(V+V')\leq (k+k')-1$. Similarly, $\dim(V+V')\leq (k+k')-1$ if $\iota(h')=\bold{0}$. In both cases, we have  $a\+ b\in\Gamma^{s}_{k+k'}(\Vk,M)$.
	
	We now assume that $\iota(h),\iota(h')\neq\bold{0}$. 
	Then $\dim(V)\leq k$ and $\dim(V')\leq k'$.
	Since $\dim(V+V')\leq k'+k$, we have that $a\+ b\in\Gamma^{s}_{k+k'}(\Vk,M)$ if $\iota(h+h')\neq\bold{0}$. If $\iota(h+h')=\bold{0}$, then $V\cap V'\neq\{\bold{0}\}$ and so $\dim(V+V')\leq k'+k-1$. So we still have that $a\+ b\in\Gamma^{s}_{k+k'}(\Vk,M)$.
\end{proof}		

Although $\sim$ is not a transitive relation in general, for the following special case it is:  

\begin{lem}\label{2:spsp1}
	Let $s\in\N$, $x_{1},x_{2}\in \Gamma^{s}(\Vk,M)$ and $y\in \Gamma^{s}_{1}(\Vk,M)$.
	 If $x_{1}\sim y$ and $x_{2}\sim y$, then $x_{1}\sim x_{2}$. In particular, $\sim$ is an equivalence relation on $\Gamma^{s}_{1}(\Vk,M)$.
\end{lem}	
\begin{proof}
	Write $x_{i}=(v_{i},J^{M}_{V_{i}}+f_{i})$ for $i=1,2$ and let $u=\pi(y)$. If $x_{1}\sim y$ and $x_{2}\sim y$, then  $v_{1}=u=v_{2}$. Since $\iota(v_{i})\in V_{i}$, we have that  $\iota(u)\in V_{1}\cap V_{2}$. So for $i=1,2$, $x_{i}\equiv y\mod J^{M}_{V_{i}+\sp_{\F_{p}}\{\iota(u)\}}=J^{M}_{V_{i}}$. So $x_{1}\equiv y\equiv x_{2} \mod J^{M}_{V_{1}+V_{2}}$ and thus $x_{1}\sim x_{2}$.
\end{proof}

\subsection{The relation graph of a subset of $\Gamma^{s}(\Vk,M)$}\label{2:s:rgg}

 The level sets (i.e. sets of the form $\{(a,a')\in A\times A\colon a+a'=b\}$ for some $b\in G$) play an important role in the study of conventional additive combinatorics. In our setting, we do not have such a convenience tool. This is because the operation $\+ $ does not have the cancellation property (i.e. $a\+ b\sim a'\+ b\not\Rightarrow a\sim a'$), and the relation $\sim$ is not an equivalent relation. 
 To overcome this difficulty, we use a graph to characterize the relation $\sim$.

\begin{defn}[Relation graph]
Let $X$ be a subset of $\Gamma^{s}(\Vk,M)$. we say that a graph $G=(X,E)$ is the \emph{relation graph} of $X$ if for all $u,v\in X, u\neq v$, $\{u,v\}\in E$ if and only if $u\sim v$.
\end{defn}

\begin{rem}
	 The structure of relation graphs for the conventional additive combinatorics on abelian groups is quite simple. In fact,
    if $\sim$ is an equivalence relation on $X$, then the graph $G$ is a disjoint union of complete subgraphs, each of which represent an equivalence class (or a level set). However, in our setting, the behavior of the relation graph can be very complicated.
\end{rem}

Since we need to frequently partition a relation graph into smaller components, it is helpful to introduce the following notion.

\begin{defn}[Graph partitions]
Let $G=(X,E)$ be a graph and $G_{i}=(X_{i},E_{i}), 1\leq i\leq k$ be subgraphs of $G$. We say that $G_{1},\dots,G_{k}$ is a \emph{disjoint partition} of $G$ if $X_{1},\dots,X_{k}$ are pairwise disjoint, $X=\cup_{i=1}^{k}X_{i}$ and $E=\cup_{i=1}^{k}E_{i}$. 

Similarly,  Let $G=(X,Y,E)$ be a bipartite graph and $G_{i}=(X_{i},Y_{i},E_{i}), 1\leq i\leq k$ be subgraphs of $G$. We say that $G_{1},\dots,G_{k}$ is a \emph{disjoint partition} of $G$ if $X_{1},\dots,X_{k}, Y_{1},\dots,Y_{k}$ are pairwise disjoint, $X=\cup_{i=1}^{k}X_{i}$, $Y=\cup_{i=1}^{k}Y_{i}$ and $E=\cup_{i=1}^{k}E_{i}$. 
\end{defn}

\begin{conv}
	Let $G=(X,E)$ be a graph  and  $X_{1},\dots,X_{k}$ be pairwise disjoint and $X=\cup_{i=1}^{k}X_{i}$. When we say that $G_{i}=(X_{i},E_{i}), 1\leq i\leq k$ is a disjoint partition of $G$, we implicitly indicate that $E_{i}$ is the set of all edges in $E$ with both vertices coming from $X_{i}$. 
	
	Similarly,   Let $G=(X,Y,E)$ be a bipartite graph  and  $X_{1},\dots,X_{k}, Y_{1},\dots,Y_{k}$ be pairwise disjoint, $X=\cup_{i=1}^{k}X_{i}$ and $Y=\cup_{i=1}^{k}Y_{i}$. When we say that $G_{i}=(X_{i},Y_{i},E_{i}), 1\leq i\leq k$ is a disjoint partition of $G$, we implicitly indicate that $E_{i}$ is the set of all edges in $E$ with one vertex from $X_{i}$ and the other from $Y_{i}$. 
\end{conv}

\begin{rem}
	Let $X$ be a subset of $\Gamma^{s}(\Vk,M)$. 
	Note that for $u,v\in X$, we have that $u\sim v$ only if $\pi(u)=\pi(v)$. Therefore, the relation graph $G=(X,E)$ has a disjoint partition induced by the partition $\cup_{h\in\Vk} \pi^{-1}(h)\cap X$ of $X$. So it is helpful to think of $G$ as the union of $K^{d}$ disjoint smaller graphs.  
\end{rem}

Next we introduce two quantities for graphs which play important roles in this paper. 
Recall that in graph theory, the \emph{(vertex) independence number} of a graph is the maximum number of vertices which are pairwise disjoint.  
The first quantity we need is a variation of the concept of independence number:

\begin{defn}[Auxiliary graph and density dependence number]
Let $G=(X,E)$ be a graph and $G'=(X,Y,E')$ be a bipartite graph, we say that $G'$ is an \emph{auxiliary graph} of $G$ if for all $x,x'\in X$ and $y\in Y$, if $(x,y), (x',y)\in E'$, then $\{x,x'\}\in E$. 
For $\e>0$, we say that $G'$ is \emph{$\e$-dense}  if for every $x\in X$, there exist at least $\e\vert Y\vert$ many $y\in Y$ such that $(x,y)\in E'$. 
 
Let $s\in\N$, $h\in \Vk$ and $G=(X,E)$ be the relation graph of a subset $X$ of $\Gamma^{s}(\Vk,M)$  such that $\pi(X)=\{h\}$.
The \emph{density dependence number} of $G$ (or $X$), denoted by $\dep(G)$ (or $\dep(X)$), is the smallest integer $C>0$ such that $G$ admits a $C^{-1}$-dense auxiliary graph.

For the relation graph $G=(X,E)$ of a general subset $X$ of $\Gamma^{s}(\Vk,M)$. We define $\dep(G)$ (or $\dep(X)$) as the maximum of $\dep(\pi^{-1}(h)\cap X), h\in\Vk$.
\end{defn}

By the Pigeonhole Principle, it is not hard to see that the independence number of a graph $G$ is at most $\dep(G)$. It is natural to ask about the converse direction:

\begin{ques}\label{que211}
	For any graph, is the independent number equal to $\dep(G)$? If not, then is the independent number bounded below by a constant depending on $\dep(G)$?
\end{ques}

The second quantity we need is one that describes the optimal way in which a graph can be covered by cliques.
Recall that in graph theory,
a \emph{clique} of a graph $G=(X,E)$ is a subset $S$ of $X$ such that the subgraph of $G$ induced by $S$ is a complete graph. In this paper, we measure the largeness of graph by looking at the minimal number of clique we need to cover all the vertices.

\begin{defn}[Clique cover number]
	Let $s\in\N$, $h\in \Vk$ and $G=(X,E)$ be the relation graph of a subset   $X$ of $\Gamma^{s}(\Vk,M)$ such that $\pi(X)=\{h\}$. 	
	The \emph{clique cover number} of $G$ (or $X$), denoted by $\cc(G)$ (or $\cc(X)$) is the smallest integer $C$ such that  we may partition $X$ into $\sqcup_{i=0}^{C}X_{i}$ such that each $X_{i}$ is a  clique of $G$.

	 For the relation graph $G=(X,E)$ of a general subset  $X$ of $\Gamma^{s}(\Vk,M)$. We define $\cc(G)$ (or $\cc(X)$) as the maximum of $\cc(\pi^{-1}(h)\cap X), h\in\Vk$.
\end{defn}

We summarize some properties on
  the clique cover number and density dependence number of a graph for later uses. The following lemma is straightforward.

\begin{lem}\label{2:basicdn}
	Let $s\in\N$ and $G=(X,E)$ be the relation graph of a non-empty subset $X$ of $\Gamma^{s}(\Vk,M)$. Then
	\begin{enumerate}[(i)]
		\item $\dep(G)$ is finite and $\cc(G)\geq \dep(G)\geq 1$;
		\item for any subgraph $G'$ of $G$, we have that $\cc(G')\leq \cc(G)$ and $\dep(G')\leq \dep(G)$;
		\item if $\pi(X)=\{h\}$ for some $h\in\Vk$, then any subset $X'$ of $X$ contains a clique of size at least $\lceil\frac{\vert X'\vert}{\dep(G)}\rceil$; 
 		\item   $\dep(G)=1\Leftrightarrow \cc(G)=1\Leftrightarrow u\sim v$ for all $h\in\Vk$ and $u,v\in \pi^{-1}(h)\cap X$. 
	\end{enumerate}	
\end{lem}	
\begin{proof}
	We first prove Part (i).
	Assume without loss of generality that $\pi(X)=\{h\}$ for some $h\in\Vk$.
	Clearly, 
	$\dep(G)\geq 1$. Note that we may partition $X$ into $\sqcup_{i=1}^{\cc(G)}X_{i}$ such that each $X_{i}$ is a clique of $G$. 
	Let $G'=(X,\{1,\dots,\cc(G)\},E')$ be the bipartite graph given by $(x,i)\in E'\Leftrightarrow x\in X_{i}$. Since  each $X_{i}$ is a clique of $G$, we have that $G'$ is an auxiliary graph of $G$. On the other hand, the density of this auxiliary graph is $\frac{1}{\cc(G)}$. So we have that $\dep(G)$ is finite and $\cc(G)\geq \dep(G)$.
	
	Part (ii) is straightforward. For Part (iii), let $G'=(X,Y,E')$ be an auxiliary graph of $G$ of density at least $\dep(G)^{-1}$. Then there are at least $\frac{\vert X'\vert\cdot\vert Y\vert}{\dep(G)}$ edges in $E'$ with one of the vertices coming from $X'$. By the Pigeonhole Principle, there exists $y\in Y$ which is connected by  edges in $E'$ to at least $\lceil\frac{\vert X'\vert}{\dep(G)}\rceil$ vertices in $X'$. These vertices induces a clique in $G$ by the definition of auxiliary graphs.

	Part (iv) follows from the definitions and Part (iii).
\end{proof}

 The next lemma allows us to expand an auxiliary graph without changing its density.
 
   \begin{lem}\label{2:duplicate}
   	Let $s\in\N$ and $X$ be a subset of $\Gamma^{s}(\Vk,M)$ with $\pi(X)=\{h\}$ for some $h\in\Vk$, $\e>0$, and $G'=(X,Y,E')$ be an $\e$-dense auxiliary graph of $X$. Then for all $N\in\N_{+}$ divisible by $\vert Y\vert$,   $X$ admits an $\e$-dense auxiliary graph $G''=(X,Y',E'')$ with $\vert Y'\vert=N$.
   \end{lem}	
   \begin{proof}
   	Denote $m:=N/\vert Y\vert$ and let $Y'=\sqcup_{i=1}^{m}Y_{i}$ be $m$ identical copies of $Y$ and let $\iota_{i}\colon Y_{i}\to Y$ be the bijection for $1\leq i\leq m$. For $x\in X$ and $y\in Y'$, set $(x,y)\in E''$ if and only if $y$ belongs to some $Y_{i}$ and $(x,\iota_{i}(y))\in E'$. It is not hard to see that $\vert Y'\vert=N$ and $G''=(X,Y',E'')$  is an  $\e$-dense auxiliary graph of $X$.
   \end{proof}
 
 Let $X$ be a subset of $\Gamma^{s}(\Vk,M)$.
 If we add an extra point $x$ to $X$ such that $x\not\sim y$ for all $y\in X$, then it is clear that the clique cover number of $X$ is increased by 1. In this case, it turns out that the density dependence number of $X$ is also increased by 1.

 \begin{lem}\label{2:lonely}
 Let $s\in\N$ and $X$ be a subset of $\Gamma^{s}(\Vk,M)$ with $\pi(X)=\{h\}$ for some $h\in\Vk$. Let $y\in \Gamma^{s}(\Vk,M)$ with $\pi(y)=h$  be such that $x\not\sim y$ for all $x\in X$. Then $\dep(X\cup\{y\})=\dep(X)+1$.	
 \end{lem}	
 \begin{proof}
 	For convenience denote $C:=\dep(X)$ and $D:=\dep(X\cup\{y\})$. Then
   $X\cup\{y\}$ admits a $D^{-1}$-dense auxiliary graph
 	$G'=(X\cup\{y\},Z,E')$. Let $Z':=\{z\in Z\colon (y,z)\in E'\}$ and $Z'':=Z\backslash Z'$. Then $\vert Z'\vert\geq D^{-1}\vert Z\vert$. Since $y\not\sim x$ for all $x\in X$, we have that $(x,z)\notin E'$ for all $x\in X$ and $z\in Z'$.
 	Consider the bipartite graph $G''=(X,Z'',E'')$  induced by $G'$. 
 	Since $(x,z)\notin E'$ for all $x\in X$ and $z\in Z'$, it is not hard to see that $G''$ is an auxiliary graph of $X$, and that the density of  this auxiliary graph is at least $$\frac{D^{-1}\vert Z\vert}{\vert Z\vert-\vert Z'\vert}\geq \frac{D^{-1}\vert Z\vert}{\vert Z\vert-D^{-1}\vert Z\vert}=(D-1)^{-1}.$$
 	This implies that $C\leq D-1$.
 	
 	Conversely, since $\dep(X)=C$, $X$ admits a $C^{-1}$-dense auxiliary graph $G'=(X,Z,E')$. By Lemma \ref{2:duplicate}, we may assume without loss of generality that $\vert Z\vert$ is divisible by $C$. Let $Z'$ be any set of cardinality $\vert Z\vert/C$ which is disjoint from $Z$, and let $G''=(X\cup\{y\},Z\sqcup Z',E'')$ be the bipartite graph such that for all $x\in X\cup\{y\}$ and $z\in Z\sqcup Z'$, $(x,z)\in E''$ if and only if either (i) $x\in X$, $z\in Z$ and $(x,z)\in E'$; or (ii) $x=y$ and $z\in Z'$. Clearly, $G''$ is an auxiliary graph of $X\cup\{y\}$. On the other hand, the density of $G''$ is 
 	$$\frac{\min\{C^{-1}\vert Z\vert,\vert Z'\vert\}}{\vert Z\vert+\vert Z'\vert}=\frac{C^{-1}\vert Z\vert}{\vert Z\vert+C^{-1}\vert Z\vert}=(C+1)^{-1}.$$
 	This implies that $D\leq C+1$.
 	
 	Combining both directions, we have that $D=C+1$.
  \end{proof}

\subsection{Outline of the proof of Theorem \ref{2:aadd}}\label{2:s:rdm}

We are now ready to provide an overview for the proof of Theorem \ref{2:aadd}:

\textbf{Step 0.} 
Our first goal is to solve the \emph{Freiman equation} (see (\ref{2:gsole})) on a generalized arithmetic progression (see Theorem \ref{2:gsol}).
This can be viewed as a special case of Theorem \ref{2:aadd}   when $U$ equals to $\{(h_{1},h_{2},h_{3},h_{4})\in H^{4}\colon h_{1}+h_{2}=h_{3}+h_{4}\}$, $H$ is a generalized arithmetic progression, and $\xi_{1}=\xi_{2}=\xi_{3}=\xi_{4}$. 
This result will be  used later in Step 4 to be described below.
 While the the Freiman equation for abelian groups (see (\ref{2:gsold})) can be solved easily, the  Freiman equation is the setting of shifted modules is much harder. This step is done in Section \ref{2:s:ff}.

 \textbf{Step 1.} We show that if the conditions of  Theorem \ref{2:aadd} are satisfied, then there exists a large subset $H'$ of $H$ such that   $\pi^{-1}(\bold{0})\cap(8\tilde{\xi}(H')\- 8\tilde{\xi}(H'))$ has a small  density dependence number, where $\tilde{\xi}(h):=(h,J^{M}_{\iota(h)}+\xi_{1}(h))\in\Gamma^{s}_{1}(\Vk,M)$. To achieve this, we need to extend many classical results in combinatorics to shifted modules, including the Balog-Gowers-Szemer\'edi Theorem, the Ruzsa triangle inequality, the Green-Ruzsa inequality, and the Pl\"unnecke-Rusza theorem. This is done in  Section \ref{2:s:115}.
 
 \textbf{Step 2.} We show that if $\pi^{-1}(\bold{0})\cap(8\tilde{\xi}(H')\- 8\tilde{\xi}(H'))$ has a small density dependence number, then we may decompose the set $\pi^{-1}(\bold{0})\cap(8\tilde{\xi}(H')\- 8\tilde{\xi}(H'))$ into two parts. The first is the structure part with a small clique covering number, and the second is the obstacle part which has nontrivial intersections with an object with low dimension. This is done in Section \ref{2:s:116}. Unfortunately we were unable to remove the obstacle part, and thus we need carry this error term throughout the proof, which makes the proof of the remaining steps more technical. This step is done in  Section \ref{2:s:116}.

 \textbf{Step 3.} We show that if $\pi^{-1}(\bold{0})\cap(8\tilde{\xi}(H')\- 8\tilde{\xi}(H'))$ has a decomposition described in Step 2, then for many ``good" subspaces $V$ of $\V$, there exists a large subset $H''$ of $H'$ such that $\xi$ is a Freiman $M$-homomorphism of order 16 on $H''\cap \iota^{-1}(V)$. On the other hand, we show that $8H''- 8H''$ has rich generalized arithmetic progression structures. The proof  is inspired by  the constructions of appropriate Bohr sets introduced in \cite{GT08b}. This step is done in  Section \ref{2:s:117}.

 \textbf{Step 4.} If we could show that $\xi$ is a  Freiman $M$-homomorphism of order 16 on a large subset $H''$ of $H'$ (instead of on $H''\cap \iota^{-1}(V)$), then one can directly apply the solutions to Freiman equations obtained in Step 0  to conclude that $\xi$ is locally linear, and thus Theorem \ref{2:aadd} can be proved easily. Unfortunately the subspaces $V$ in Step 3 can not be taken to be $\V$. To overcome this difficulty, we first solve for $\xi$ on ``good"  subspaces $V$, and then combine this information on different $V$ to conclude that $\xi$  is a  Freiman $M$-homomorphism of order 16 on a large subset $H''$ of $H'$, modulo a module of the form $J^{M}_{T}$ for some subspace $T$ of $\V$. 
  This is done is Section \ref{2:s:a3}.  
 
 \textbf{Step 5.} By intersecting the Freiman $M$-homomorphism structure of $\xi$ obtained in Step 4 for many different $J^{M}_{T}$, we conclude that $\xi$ is a genial  Freiman $M$-homomorphism of order 16. We may then invoke Theorem \ref{2:gsol} to show that $\xi$ is locally linear and thus complete the proof of Theorem \ref{2:aadd}.
 This is done is Section \ref{2:s:a2}.

\section{Intersection property for $M$-modules}\label{2:s:a1}

We begin with the notation of linear independence for subspaces of $\V$:

\begin{defn}[Linearly independent subspaces]
Let $V_{1},\dots,V_{k}$ be subspaces of $\V$.
We say that $V_{1},\dots,V_{k}$ are \emph{linearly independent} if for all $v_{i}\in V_{i}, 1\leq i\leq k$, $\sum_{i=1}^{k}v_{i}=\bold{0}$ implies that $v_{1}=\dots=v_{k}=\bold{0}$.
In this case we also say that   $(V_{1},\dots,V_{k})$ is a \emph{linearly independent tuple} (or a \emph{linearly independent pair} if $k=2$). 

Let $V_{1},\dots,V_{k}$ be subspaces of $\V$ and $v_{1},\dots,v_{k'}\in\V$. 
We say that $v_{1},\dots,v_{k'}$, $V_{1},\dots,V_{k}$ are \emph{linearly independent} if for all $c_{i}\in\F_{p}, 1\leq i\leq k$ and $v'_{i'}\in V_{i'}, 1\leq i'\leq k'$, $\sum_{i=1}^{k}c_{i}v_{i}+\sum_{i'=1}^{k'}v'_{i'}=\bold{0}$ implies that $c_{1}=\dots=c_{k}=0$ and $v'_{1}=\dots=v'_{k'}=\bold{0}$.
In this case we also say that 
 $(v_{1},\dots,v_{k'},V_{1},\dots,V_{k})$ is a \emph{linearly independent tuple} (or a \emph{linearly independent pair} if $k+k'=2$).\footnote{We caution the readers that according to our definition, the zero vector $\bold{0}$ is linearly dependent with every subspace  of $\V$, while the trivial subspace $\{\bold{0}\}$ is linearly independent  with every subspace  of $\V$. This unconventional definition turns out to be convenient for this paper. 
}

Let $V,V_{1},\dots,V_{k}$ be subspaces of $\V$ with $V\subseteq\cap_{i=1}^{k}V_{i}$. We say that $V_{1},\dots,V_{k}$ are \emph{linearly independent modulo $V$} if for all $v_{i}\in V_{i}, 1\leq i\leq k$, $\sum_{i=1}^{k}v_{k}=\bold{0}$ implies that $v_{1},\dots,v_{k}\in V$.
\end{defn}

In order to illustrate the main idea of this section, it is helpful to
consider the following motivative question:

\begin{ques}\label{2:qiiq}
Let $V$ be a subspace  of $\V$ and $V_{1},\dots,V_{N}$ be nontrivial linearly independent subspaces of $\V$. When do we have that $\cap_{i=1}^{N}(V+V_{i})=V$? 
\end{ques}

For example, let $N=2$ and $V,V_{1}$ and $V_{2}$ be the $\F_{p}$-span of $(1,0,\dots,0)$, $(0,1,0,\dots,0)$ and $(1,1,0,\dots,0)$ respectively. It is clear that $(V+V_{1})\cap (V+V_{2})=\sp_{\F_{p}}\{e_{1},e_{2}\}\neq V=\sp_{\F_{p}}\{e_{1}\}$. However, if we add an addition subspace $V_{3}=\sp_{\F_{p}}\{h\}$ with $V_{1},V_{2},V_{3}$ being linearly independent, then this will force $h$ not to belong to $\sp_{\F_{p}}\{e_{1},e_{2}\}$, which implies that $\cap_{i=1}^{3}(V+V_{i})=V$. Motivated by this example, it is natural to expect that $\cap_{i=1}^{N}(V+V_{i})=V$ whenever $N$ is sufficiently large. Indeed, we have 
 
 \begin{prop}[Intersection property for subspaces of $\V$]\label{2:gr-1}
 Let $V$ be a subspace  of $\V$ and $V_{1},\dots,V_{N}$ be linearly independent subspaces of $\V$.  If $N\geq \dim(V)+2$, then $\cap_{i=1}^{N}(V+V_{i})=V$.
 \end{prop}

 Informally, one can think of the unions of $V+V_{i}$ as a ``sunflower" with $V$ being its ``core" and $V_{i}$ being its ``petals". The philosophy behind Proposition \ref{2:gr-1} is that if we take the intersection of sufficiently many linearly independent ``petals", then eventually the intersection will decrease to $V$. Therefore, we call Proposition \ref{2:gr-1} the \emph{intersection property} for subspaces.

 To prove Proposition \ref{2:gr-1}, it suffices to show the following lemma (and then set $s=2$):

\begin{lem}\label{2:w22s}
	Let $m,s\in\N$, $d,N\in\N_{+}$, $p$ be a prime, and $V,V_{1},\dots,V_{N}$ be  subspaces of $\V$ such that $\dim(V)=m$. 
	 Suppose that $V_{1},\dots,V_{N}$ are linearly independent and $N\geq m+s$, then there exist $\{i_{1},\dots,i_{s}\}\subseteq \{1,\dots,N\}$ such that $V,V_{i_{1}},\dots,V_{i_{s}}$
 are linearly independent.
\end{lem}	
\begin{proof}
There is nothing to prove when $s=0$.
	Suppose that the conclusion holds for some $s\in\N$ when $N\geq m+s$. We prove that the conclusion holds for $s+1$ when $N\geq m+s+1$.

	By induction hypothesis, if  $N\geq m+s$, then we may assume without loss of generality that $V,V_{1},\dots,V_{s}$ are linearly independent. Suppose that for all $1\leq i\leq N$, $V,V_{1},\dots,V_{s},V_{i}$ are not linearly independent. Then there exists a nontrivial $v_{i}\in V_{i}\cap (V+V_{1}+\dots+V_{s})$ for all $s+1\leq i\leq N$. Since $\dim(V)=m$, if $N-s\geq m+1$, then
	there exist $c_{s+1},\dots,c_{N}\in\F_{p}$ not all equal to 0 such that
	$\sum_{i=s+1}^{N}c_{i}v_{i}\in  V_{1}+\dots+V_{s}$. Since $V_{1},\dots,V_{N}$ are linearly independent, we have that   $c_{i}=0$ for all $s+1\leq i\leq N$, a contradiction. Therefore, if $N\geq m+s+1$, then there must be some $1\leq j\leq N$  such that $V,V_{1},\dots,V_{s},V_{j}$ are linearly independent.
\end{proof}

\begin{rem}
	We remark that the lower bound $N\geq m+s$  in Lemma \ref{2:w22s} is optimal. To see this, let $V=\sp_{\F_{p}}\{e_{1},\dots,e_{m}\}$ and $V_{i}=\sp_{\F_{p}}\{e_{i}\}$ for $1\leq i\leq m+s-1$, where $e_{j}$ is the $j$-th standard unit vector. Then $V_{1},\dots,V_{m+s-1}$ are linearly independent, but for any $s$ of $V_{1},\dots,V_{m+s-1}$, since at least one of them comes from $V_{1},\dots,V_{m}$, it has nontrivial intersection with $V$.
\end{rem}

Following the philosophy of Proposition \ref{2:gr-1}, it is natural to ask for the intersection property for shifted modules.
We have the following, which is the main result of this section.

\begin{prop}[Intersection property for weakly independent $M$-modules]\label{2:gri}
	Let $m,s\in\N$, $d,N,r\in\N_{+}$, $p\gg_{d} 1$ be a prime, $M\colon\V\to \F_{p}$ be a non-degenerate quadratic form, and $V,V_{1},\dots,V_{N}$ be subspaces of $\V$ with $V_{1},\dots,V_{N}$ being  linearly independent such that   $\dim(V)=m$ and $\dim(V_{i})\leq r, 1\leq i\leq N$. Suppose that $N\geq s+m+1$ and that
  either $\rank(M\vert_{V^{\perp_{M}}})\geq 2N(r-1)+7$ or $d\geq 2m+2N(r-1)+7$.\footnote{See Appendix \ref{2:AppA1} for definitions.}	
	Then for all $f\in\st_{\zp,d}(s)$, we have that
	$$f\in\cap_{i=1}^{N}J^{M}_{V+V_{i}}\Leftrightarrow f\in J^{M}_{V}.$$
\end{prop}

 It is an interesting question to ask what is the best lower bound of $d$ for Proposition \ref{2:gri}.  although we do not pursuit this  in the paper.

By Lemma \ref{2:w22s}, to prove Proposition \ref{2:gri}, it suffices to show  follows  the following weaker version of it:


 \begin{prop}[Intersection property for strongly independent $M$-modules]\label{2:gr0}  
	Let $m,s\in\N$, $d,N,r\in\N_{+}$, $p\gg_{d} 1$ be a prime, $M\colon\V\to \F_{p}$ be a non-degenerate quadratic form, and $V,V_{1},\dots,V_{N}$ be linearly independent subspaces of $\V$ such that   $\dim(V)=m$ and $\dim(V_{i})\leq r, 1\leq i\leq N$.	Suppose that $N\geq s+1$ and that
  either $\rank(M\vert_{V^{\perp_{M}}})\geq 2N(r-1)+7$ or $d\geq 2m+2N(r-1)+7$.
	Then for all $f\in\st_{\zp,d}(s)$, we have that
	$$f\in\cap_{i=1}^{N}J^{M}_{V+V_{i}}\Leftrightarrow f\in J^{M}_{V}.$$
\end{prop}

We need some preparation before proving Proposition \ref{2:gr0}. The next lemma says that if $h,h_{1},\dots,h_{k}\in\Vk$  are $p$-linearly independent, then the degree 1 polynomial $L_{h}(n):=\frac{(n\tau(A))\cdot h}{p}$ is ``irreducible"  in $J^{M}_{\iota(h_{1}),\dots,\iota(h_{k})}$.

\begin{lem}\label{2:killLq}
 Let $d\in\N_{+}$, $k,s\in\N$, $p\gg_{d,k,s} 1$ be a prime,	
  $M\colon\V\to\F_{p}$ be a non-degenerate quadratic form associated with the matrix $A$, and $h,h_{1},\dots,h_{k}\in\Z^{d}$ be $p$-linearly independent vectors. 
  Let 
  $f\colon \Z^{d}\to\Q$ be a $\Z/p^{\N}$-coefficient polynomial of degree at most $s$
    such that   $pL_{h}f\in J^{M}_{\iota(h_{1}),\dots,\iota(h_{k})}$, where $L_{h}(n):=\frac{(n\tau(A))\cdot h}{p}$. If 
  $\rank(M\vert_{\sp_{\F_{p}}\{\iota(h),\iota(h_{1}),\dots,\iota(h_{k})\}^{\pp}})\geq 3$ or $d\geq 2k+5$, then we have that $f\in J^{M}_{\iota(h_{1}),\dots,\iota(h_{k})}$.  
\end{lem}	
\begin{proof}
  
  We use the $p$-expansion trick introduced in \cite{SunA}.
  With may assume without loss of generality that $f$ is homogeneous, and that $\deg(f)=s$.
  Let $\tilde{M}\colon \Z^{d}\to\Z/p$ be the quadratic form given by $\tilde{M}(n):=\frac{1}{p}((nA)\cdot n)$. 
  Denote $L_{h_{j}}(n):=\frac{(n\tau(A))\cdot h}{p}$ for all $1\leq j\leq k$.
 We say that $g$ is \emph{good} if 
  $$g=\sum_{i:=(i_{0},\dots,i_{k})\in\N^{k+1}, 2i_{0}+i_{1}+\dots+i_{k}\leq s}R_{i}\tilde{M}^{i_{0}}L_{h_{1}}^{i_{1}}\dots L_{h_{k}}^{i_{k}}$$
  for some   integer coefficient polynomials $R_{i}$ with $\deg(R_{i})=s-(2i_{0}+i_{1}+\dots+i_{k})$. 
  %
  %
  %
 Let $t$ be the smallest positive integer such that $f$ can be written as 
  $$f=\sum_{i=0}^{t}\frac{f_{i}}{p^{i}}$$
  for some good polynomial $f_{i}$ of degree $s$. Clearly such $t$ exists since all the coefficients of $f$ are in $\Z/p^{\N}$. 
  Our goal is to show $t=0$.
  
Suppose on the countrary that $t>0$.
Since $f_{t}$ is good, we may write 
  $$f_{t}=\sum_{i:=(i_{0},\dots,i_{k})\in\N^{k+1}, 2i_{0}+i_{1}+\dots+i_{k}\leq s}R_{i}\tilde{M}^{i_{0}}L_{h_{1}}^{i_{1}}\dots L_{h_{k}}^{i_{k}}$$
  for some   integer coefficient polynomials  $R_{i}$ with $\deg(R_{i})=s-(2i_{0}+i_{1}+\dots+i_{k})$. 
  Since $p^{t}L_{h}f\equiv L_{h}f_{t} \mod \Z$, we have that $L_{h}(n)f_{t}(n)\in\Z$ whenever $M(n), L_{h_{i}}(n)\in\Z$ for all $1\leq i\leq k$.
   Let $V$ be the span of $\iota(h_{1}),\dots,\iota(h_{k})$. Then for any $n\in V_{p}(\tilde{M})\cap \iota^{-1}(V^{\pp})$ and $m\in\Z^{d}$, we have that $pL_{h}(n+pm)f_{t}(n+pm)\in\Z$. Note that $pL_{h}(n+pm)\equiv (n\tau(A))\cdot h \mod \Z$ and
  \begin{equation}\label{2:woei}
   \begin{split}
       &\quad\frac{1}{p}f_{t}(n+pm)\equiv  
       \\&\frac{1}{p}\sum_{i:=(i_{0},\dots,i_{k})\in\N^{k+1}, 2i_{0}+i_{1}+\dots+i_{k}\leq s}R_{i}(n)(M(n)+2(n\tau(A))\cdot m)^{i_{0}}\prod_{j=1}^{k}(L_{h_{j}}(n)+(h_{j}\tau(A))\cdot m)^{i_{j}} \mod \Z
       \end{split}
  \end{equation}
   if $n\in V_{p}(\tilde{M})\cap \iota^{-1}(V^{\pp})$. 
  If we assume in addition that $(n\tau(A))\cdot h\notin p\Z$ and that $n,h_{1},\dots,h_{k}$ are $p$-linearly independent, then the map $m\mapsto ((n\tau(A))\cdot m, (h_{1}\tau(A))\cdot m,\dots, (h_{k}\tau(A))\cdot m) \mod p\Z^{k+1}$ is a surjection from $\Z^{d}$ to $[p]^{k+1}$.  So if we further assume that $(n\tau(A))\cdot h\notin p\Z$, then it follows from (\ref{2:woei}) and the assumption that that 
    \begin{equation}\label{2:woei2}
   \begin{split}
            \frac{1}{p}\sum_{i:=(i_{0},\dots,i_{k})\in\N^{k+1}, 2i_{0}+i_{1}+\dots+i_{k}\leq s}R_{i}(n)(M(n)+x_{0})^{i_{0}}\prod_{j=1}^{k}(L_{h_{j}}(n)+x_{j})^{i_{j}} \in\Z
       \end{split}
  \end{equation}
 for all $x_{0},\dots,x_{k}\in\Z^{d}$. Since  (\ref{2:woei2}) is a polynomial in $x_{0},\dots,x_{k}$ with $\Z/p$-coefficients (where $n$ is fixed), by Lemma \ref{2:ns}, all the coefficients of this polynomial must be in $\Z$. Therefore, we must have that $R_{i}(n), R(n)\in p\Z$. So in conclusion, we have that $R_{i}(n), R(n)\in p\Z$ for all $n\in V_{p}(M)\cap \iota^{-1}(V^{\pp})$ with $(n\tau(A))\cdot h\notin p\Z$ and with $n,h_{1},\dots,h_{k}$ being $p$-linearly independent.

 We claim that  $R_{i}(n), R(n)\in p\Z$ for all $n\in V_{p}(\tilde{M})\cap \iota^{-1}(V^{\pp})$.  
 Let $V'$ be the set of $n\in V^{\pp}$ with $(nA)\cdot \iota(h)=0$.
	Since 
	$\rank(M'\vert_{\sp_{\F_{p}}\{\iota(h),\iota(h_{1}),\dots,\iota(h_{k})\}^{\pp}})\geq 3$ by Lemma \ref{2:iissoo},  
	it follows from Lemma \ref{2:counting01} that $\vert V^{\pp}\vert\gg_{d,k,s}p^{d-k-1}$ and $\vert V'\vert=O_{d,k,s}(p^{d-k-2})$. 
	Let $V''$ be the set of $n\in V^{\pp}$ with $n,\iota(h_{1}),\dots,\iota(h_{k})$ being linearly dependent. If we can show that $\vert V''\vert=O_{d,k,s}(p^{d-k-2})$, then the claim follows from  Proposition \ref{2:noloop3} (translated into the $\Z^{d}$-setting).
	
	We now estimate $\vert V''\vert$. Let $U$ be the space of $\V$ spanned by $\iota(h_{1}),\dots,\iota(h_{k})$. Then for all $n\in V^{\pp}$, we have that $n\in U^{\perp_{M}}$. If  $n,\iota(h_{1}),\dots,\iota(h_{k})$ are linearly dependent, then $n\in U$. 	So by Lemma \ref{2:iissoo}, $\vert V''\vert\leq \vert U\cap U^{\perp_{M}}\vert\leq p^{d-k-\rank(M\vert_{U^{\perp_{M}}})}\leq p^{d-k-3}$.   This completes the proof of the claim.

	\

 By the claim, we have that  $\frac{(s!)^{d}}{p}f_{t}$ is a good polynomial. Since $f_{t}$ is a good polynomial, we have that all the coefficients of $\frac{1}{p}f_{t}$ belong to $\frac{1}{(s!)^{d}}\Z\cap \frac{1}{p^{s+1}}\Z=\Z$. Therefore $\frac{1}{p}f_{t}$ is also 
 good polynomial. We may then absorb the term $f_{t}/p^{t}$ by $f_{t-1}/p^{t-1}$ and get a contradiction to the minimality of $t$. 
 


 Therefore we have that $t=0$ and so $f$ itself is a good polynomial, which implies that $f\in J^{M}_{\iota(h_{1}),\dots,\iota(h_{k})}$.
\end{proof}

We need some definitions before proving Proposition \ref{2:gr0}.

Let $A$ be the matrix associated with $M$ and 
  $\tilde{M}\colon \Z^{d}\to\Z/p$ be the quadratic form given by $\tilde{M}(n):=\frac{1}{p}((n\tau(A))\cdot n)$. 
For $h\in\V$, let $L_{h}\colon \Z^{d}\to\Z/p$ be the linear map given by $L_{h}(n):=\frac{1}{p}((\tau(h)\tau(A))\cdot n)$
Let $h_{1},\dots,h_{m}$ be a basis  of $V$, and for each $1\leq i\leq N$, let $h_{i,1},\dots,h_{i,r_{i}}$ be a  basis  of $V_{i}$. 
For $1\leq N'\leq N$, let $$\mathcal{A}_{N'}:=\Bigl\{\tilde{M}, L_{h_{1}},\dots,L_{h_{m}}, p^{N'-1}\prod_{i=1}^{N'}L_{h_{i,j_{i}}}\colon 1\leq j_{i}\leq r_{i}\Bigr\}$$
and enumerate $\mathcal{A}_{N'}$ as $\{F_{N',1},\dots,F_{N',m_{N'}}\}$. 
For $z\in\N$,
let $\mathcal{B}_{N',z}$ be the set of polynomials of the form
$$\sum_{i=(i_{1},\dots,i_{m_{N'}})\in\N^{m_{N'}}\colon w_{N'}(i)\leq s, \vert i\vert\leq z}R_{i}\prod_{t=1}^{m_{N'}}F_{N',t}^{i_{t}}$$
for some $R_{i}\in\st_{\Z,d}(s-w_{N'}(i))$, where $w_{N'}(i):=\sum_{t=1}^{m_{N'}}\deg(F_{N',t})i_{t}$. 

 The key to the proof of Proposition \ref{2:gr0} is the following Hilbert Nullstellensatz type result for the intersection of $M$-deals.

\begin{prop}\label{2:w3s}
	Let the notations be the same as in Proposition \ref{2:gr0} and as above.
	Then for all $1\leq N'\leq N$ and $f\in\st_{\Z/p^{z},d}(s)$ with $f\in\cap_{i=1}^{N'}J_{V+V_{i}}^{M}$, we have that $f\in\mathcal{B}_{N',J}$.
	In particular, 
	we may write $f=f_{1}+f_{2}$ for some $f_{1}\in  \st_{\Z/p^{z},d}(s)\cap J^{M}_{V}$ and $f_{2}\in  \st_{\Z/p^{z},d}(s)\cap \cap_{i=1}^{N'} J^{M}_{V_{i}}$.
\end{prop}	
\begin{proof}
    Throughout the proof we assume that $p\gg_{d,s} 1$.
By Lemma \ref{2:iissoo} (iii), it suffices to consider the case when $\rank(M\vert_{V^{\perp_{M}}})\geq 2N(r-1)+7$.  
	By Lemma \ref{2:cbn}, $\rank(M\vert_{(V+V_{1})^{\pp}})\geq 3$.

	Since
	$$\mathcal{A}_{1}:=\Bigl\{\tilde{M}, L_{h_{1}},\dots,L_{h_{m}}, L_{h_{1,1}},\dots,L_{h_{1,r_{1}}}\Bigr\},$$
	by  Proposition \ref{2:basicpp12}, Proposition \ref{2:w3s} holds for $N'=1$.	
%
	Suppose we have shown that the conclusion holds for some $1\leq N'\leq N-1$. 	We now show that the conclusion holds for $N'+1$. 
	Pick any  $f\in \st_{\Z/p^{z},d}(s)\cap\cap_{i=1}^{N'+1}J^{M}_{V+V_{i}}$. 
	Since the conclusion holds for $N'$,  
	we may write
	\begin{equation}\label{2:grii3}
	f=\sum_{i=(i_{1},\dots,i_{m_{N'}})\in\N^{m_{N'}}\colon w_{N'}(i)\leq s, \vert i\vert\leq z}f_{i}\prod_{t=1}^{m_{N'}}F_{N',t}^{i_{t}}
	\end{equation}
	for some $f_{i}\in\st_{\Z,d}(s-w_{N'}(i))$.
	%
		By Lemma \ref{2:cbn}, $\rank(M\vert_{(V+V_{N+1})^{\pp}})\geq 3$.  
	Since $f\in J^{M}_{V+V_{N+1}}$, by Proposition \ref{2:basicpp12}, we may also write 
	\begin{equation}\label{2:grii39}
	f=\sum_{j=(j_{0},\dots,j_{r_{N'+1}})\in\N^{r_{N'+1}+1}\colon 2j_{0}+j_{1}+\dots+j_{r_{N'+1}}\leq s, \vert j\vert\leq z}g_{j} \tilde{M}^{j_{0}}L_{h_{N'+1,1}}^{j_{1}}\dots L_{h_{N'+1,r_{N'+1}}}^{j_{r_{N'+1}}}
	\end{equation}
	for some $g_{j}\in\st_{\Z,d}(s-(2j_{0}+j_{1}+\dots+j_{r_{N'+1}}))$. 
	For all $w\geq 1$, denote 
	$$G_{w}:=\sum_{j=(j_{0},\dots,j_{r_{N'+1}})\in\N^{r_{N'+1}+1}\colon 2j_{0}+j_{1}+\dots+j_{w}\leq s, j_{w+1}=\dots=j_{r_{N'}+1}=0, \vert j\vert\leq z}g_{j} \tilde{M}^{j_{0}}\prod_{t=1}^{w}L_{h_{N'+1,t}}^{j_{t}}.$$
	It follows from (\ref{2:grii39}) that we may write $f$ as 
	\begin{equation}\label{2:grii4}
	f=\sum_{i=(i_{1},\dots,i_{m_{N'+1}})\in\N^{m_{N'+1}}\colon w_{N'+1}(i)\leq s, \vert i\vert\leq z}g'_{i}\prod_{t=1}^{m_{N'+1}}F_{N'+1,t}^{i_{t}}+G_{r_{N'+1}}
		\end{equation}
	for some $g'_{i}\in\st_{\Z,d}(s-w_{N'+1}(i))$ (in fact we can simply take $g'_{i}=0$).  In other words, 
	\begin{equation}\label{2:grii47}
	f-G_{r_{N'+1}}\in\mathcal{B}_{N'+1,z}.
		\end{equation}
	 
	Note that for every $F\in\mathcal{A}_{N'+1}$ there exist an integer coefficient homogeneous polynomial $R$ and some $F'\in\mathcal{A}_{N'}$ such that $F=RF'$. So combining (\ref{2:grii3}) and (\ref{2:grii4}), we have that
		\begin{equation}\label{2:grii2}
	\sum_{i=(i_{1},\dots,i_{m_{N'}})\in\N^{m_{N'}}\colon w_{N'}(i)\leq s, \vert i\vert\leq z}g''_{i}\prod_{t=1}^{m_{N'}}F_{N',t}^{i_{t}}+G_{r_{N'+1}}=0
	\end{equation}
for some  $g''_{i}\in\st_{\Z,d}(s-w_{N'}(i))$.
%
Let $V'=\sp_{\F_{p}}\{h_{N'+1,1},\dots,h_{N'+1,r_{N'+1}-1}\}.$
	Then it follows from (\ref{2:grii2}) that
		\begin{equation}\label{2:grii25}
0\equiv G_{r_{N'+1}}\equiv G_{r_{N'+1}}-G_{r_{N'+1}-1}  \mod \cap_{i=1}^{N'}J^{M}_{V+V'+V_{i}}.
		\end{equation}
Note that we may write $$G_{r_{N'+1}}-G_{r_{N'+1}-1}=L_{h_{N'+1},r_{N'+1}}G'$$
 for some $G'\in\st_{\Z/p^{z-1},d}(s-1)$. 
 %
	Since 
	$V, V_{1},\dots,V_{N'+1}$ are linearly independent, 
	\begin{equation}\label{2:grii5}
	\text{the vectors $h_{\ell},h_{i,j}, 1\leq \ell\leq m, 1\leq i\leq N'+1, 1\leq j\leq r_{i}$ are linearly independent.}
	\end{equation}	
Since  $\rank(M\vert_{V^{\pp}})\geq 2(N+1)(r-1)+7\geq 4r+3$ and $\dim(V'+V_{i}+\sp_{\F_{p}}\{h_{N'+1,r_{N'+1}}\})\leq 2r$ for all $1\leq i\leq N'$,
	by  Lemma \ref{2:cbn}, we have that
	 $$\rank(M\vert_{(V+V'+V_{i}+\sp_{\F_{p}}\{h_{N'+1,r_{N'+1}}\})^{\pp}})\geq 3$$ for all $1\leq i\leq N'$.
	Therefore, it follows from  (\ref{2:grii25}) and Lemma \ref{2:killLq} that  
	$$\frac{1}{p}G'\in \cap_{i=1}^{N'}J^{M}_{V+V'+V_{i}}.$$
	
%
	Since $\rank(M\vert_{V^{\pp}})\geq 2(N+1)(r-1)+7$, we have that $\rank(M\vert_{(V+V')^{\pp}})\geq 2N(r-1)+7$ by  Lemma \ref{2:cbn}.
	By  (\ref{2:grii5}), $V+V',V_{1},\dots,V_{N}$ are linearly independent. Applying the induction hypothesis for $V+V', V_{1},\dots,V_{N'}$, 	
	we may write $\frac{1}{p}G'$ as
	\begin{equation}\label{2:grii6}
	\frac{1}{p}G'=\sum_{i=(i_{1},\dots,i_{m_{\ast}})\in\N^{m_{\ast}}\colon w_{\ast}(i)\leq s, \vert i\vert\leq z}G'_{i}\prod_{t=1}^{m_{\ast}}\tilde{F}_{t}^{i_{t}}
	\end{equation}
for some  $G'_{i}\in\st_{\Z,d}(s-w_{\ast}(i))$, where $\{\tilde{F}_{1},\dots,\tilde{F}_{m_{\ast}}\}$ is an enumeration of the set 
$$\tilde{\mathcal{A}}:=\Bigl\{\tilde{M}, L_{h_{1}},\dots,L_{h_{m}}, L_{h_{N'+1},1},\dots, L_{h_{N'+1},r_{N'+1}-1}, p^{N'-1}\prod_{i=1}^{N'}L_{h_{i,j_{i}}}\colon 1\leq j_{1},\dots,j_{t}\leq r_{i}\Bigr\}$$
and $w_{\ast}(i):=\sum_{t=1}^{m_{\ast}}\deg(\tilde{F}_{t})i_{t}$.
So it follows from (\ref{2:grii6}) that 
\begin{equation}\nonumber
	G_{r_{N'+1}}-G_{r_{N'+1}-1}=\sum_{i=(i_{1},\dots,i_{m_{\ast}})\in\N^{m_{\ast}}\colon w_{\ast}(i)\leq s, \vert i\vert\leq z}G'_{i}(pL_{h_{N'+1},r_{N'+1}})\prod_{t=1}^{m_{\ast}}\tilde{F}_{t}^{i_{t}}.
	\end{equation}
Note that for  every $F\in\tilde{\mathcal{A}}$, there exist an integer coefficient homogeneous polynomial $R$ and some $F'\in\mathcal{A}_{N+1}$ such that $F'R=(pL_{h_{N'+1},r_{N'+1}})F$.
Therefore, we have that $G_{r_{N'+1}}-G_{r_{N'+1}-1}\in\mathcal{B}_{N'+1,z}$. It then follows from (\ref{2:grii47}) that 
\begin{equation}\label{2:grii42}
	f-G_{r_{N'+1}-1}\in\mathcal{B}_{N'+1,z}.
	\end{equation}
	Comparing (\ref{2:grii3}) and (\ref{2:grii42}), we  have that
		\begin{equation}\label{2:grii22}
	\sum_{i=(i_{1},\dots,i_{m_{N'}})\in\N^{m_{N'}}\colon w_{N'}(i)\leq s, \vert i\vert\leq z}g'''_{i}\prod_{t=1}^{m_{N'}}F_{N',t}^{i_{t}}+G_{r_{N'+1}-1}=0
	\end{equation}
	for some $g'''_{i}\in\st_{\Z,d}(s-w_{N'}(i))$.
	The differences between (\ref{2:grii47}), (\ref{2:grii2}) and (\ref{2:grii42}), (\ref{2:grii22}) is that we replaced $G_{r_{N'+1}}$ by $G_{r_{N'+1}-1}$, which removes all the terms in $G_{r_{N'+1}}$ having $L_{h_{N'+1},r_{N'+1}}$ as a factor.
	So we may continue this argument to replace $G_{r_{N'+1}-1}$ by $G_{r_{N'+1}-2}$ and eventually show that $f\in\mathcal{B}_{N'+1,z}$.
	
	The ``in particular" part can be proved directly by unpacking the definition of $\mathcal{B}_{N',z}$ and we omit the details.
\end{proof}

We are now ready to prove Proposition \ref{2:gr0}.
\begin{proof}[Proof of Proposition \ref{2:gr0}]
	The $\Leftarrow$ direction is straightforward, and so we only need to prove the	 $\Rightarrow$ direction.  
	
	Recall that  $h_{1},\dots,h_{m}$ be a basis of $V$.  
	Since $f\in\cap_{i=1}^{s+1}J^{M}_{V+V_{i}}$, 
by Proposition \ref{2:w3s}, if 
	$p\gg_{d} 1$, then
	$f\in \mathcal{B}_{N,z}$ for some $z\in\N$ and so we may write $f$ as 
	$$\sum_{i=(i_{1},\dots,i_{m_{N}})\in\N^{m_{N}}\colon w_{N}(i)\leq s}R_{i}\prod_{i=1}^{m_{N}}\tilde{M}^{i_{0}}F_{N,i}^{i_{j}}$$
for some $R_{i}\in\st_{\Z,d}(s-w_{N}(i))$.
	Since $f\in\st_{\zp,d}(s)$ and since $N\geq s+1$, 
	we may assume that $R_{i}=0$ for all $F_{N,i}$ whose degree is greater than $N$. Then the only remaining terms $F_{N,i}$ are $\tilde{M}, L_{h_{1}},\dots,L_{h_{m}}$. This implies that 
	  $f\in J^{M}_{V}$
and we are done.
\end{proof}

We conclude this section with a density version of Proposition \ref{2:gri}:

\begin{prop}[Intersection property for large collection of $M$-modules]\label{2:grm}
Let $m,k,s\in\N$, $d,d',K\in\N_{+}$, $\d>0$,  $p\gg_{\d,d,k} 1$ be a prime dividing $K$, $M\colon\V\to \F_{p}$ be a non-degenerate quadratic form, $V,U_{1},\dots,U_{k}$ be subspaces of $\V$ such that all of $U_{1},\dots,U_{k}$ have the same dimension and that $(U_{i},V)$ is a linearly independent pair for all $1\leq i\leq k$, and $P\subseteq \Vk$ be a subset with $\vert P\vert>\d K^{d'}$. For convenience denote $U_{0}=\{\bold{0}\}$.
	Let 
	$f\in\st_{\zp,d}(s)$ be such that 
	$f\in J^{M}_{\sp_{\F_{p}}\{\iota(h_{1}),\dots,\iota(h_{m})\}+V}$ for all $h_{1},\dots,h_{m}\in P$ with $(h_{1},\dots,h_{m},U_{i},V), 0\leq i\leq k$ being linearly independent tuples.\footnote{We remark that the introduction to the notation $U_{0}$ is redundant if $k\geq 1$,  as  the linear independence of $(x,y,U_{1},V)$ implies that of $(x,y,V)$ and $(x,y,U_{0},V)$. We introduce  the notation $U_{0}$ in order to combine the results for the case $k=0$ and the case $k\geq 1$.}
	If   $d'\geq \max\{\dim(V)+m+s,\dim(U_{k})+\dim(V)+m,2\dim(V)+2m+5\}$, then we have that $f\in J^{M}_{V}.$
\end{prop}
\begin{proof}
	Our strategy is to use the largeness of the set $P$ to iteratively find subspaces of $\V$ satisfying the requirement of Proposition \ref{2:gri}.
	Let $P_{1}$ denote the set of $h_{1}\in P$ such that $(h_{1},U_{i},V), 0\leq i\leq k$ are linearly independent tuples.
	For $2\leq j\leq m$ and $h_{1},\dots,h_{j-1}\in P$ such that $(h_{1},\dots,h_{j-1},U_{i},V), 0\leq i\leq k$ are linearly independent tuples, let 
	$P_{j}(h_{1},\dots,h_{j-1})$ denote the set of $h_{j}\in P$ such that $(h_{1},\dots,h_{j},U_{i},V), 0\leq i\leq k$ are linearly independent tuples.
   By assumption, we have that
   \begin{equation}\label{2:mmmj}
   f\in \cap_{h_{1}\in P_{1}}\cap_{h_{2}\in P_{2}(h_{1})}\dots\cap_{h_{m}\in P_{m}(h_{1},\dots,h_{m-1})}J^{M}_{\sp_{\F_{p}}\{\iota(h_{1}),\dots,\iota(h_{m})\}+V}.
   \end{equation}
	For $1\leq j\leq m$, we say that \emph{Property-$j$} holds if 
	\begin{equation}\nonumber
	f\in \cap_{h_{1}\in P_{1}}\cap_{h_{2}\in P_{2}(h_{1})}\dots\cap_{h_{j}\in P_{j}(h_{1},\dots,h_{j-1})}J^{M}_{\sp_{\F_{p}}\{\iota(h_{1}),\dots,\iota(h_{j-1})\}+V}.
	\end{equation}
	We say that  \emph{Property-0} holds if $f\in J^{M}_{V}$.
	By (\ref{2:mmmj}), we have that Property-$m$ holds. Suppose that we have shown that Property-$(j+1)$ holds for some $0\leq j\leq m-1$. Then 
		\begin{equation}\nonumber
		f\in \cap_{h_{1}\in P_{1}}\cap_{h_{2}\in P_{2}(h_{1})}\dots\cap_{h_{j+1}\in P_{j+1}(h_{1},\dots,h_{j})}J^{M}_{\sp_{\F_{p}}\{\iota(h_{1}),\dots,\iota(h_{j+1})\}+V}.
		\end{equation}
		Fix any $h_{1},\dots,h_{j}\in P$ such that $(h_{1},\dots,h_{j},U_{i},V), 0\leq i\leq k$ are linearly independent tuples.
		Since $d'\geq \max\{\dim(V)+m+s,\dim(U_{k})+\dim(V)+m\}$, it is not hard to find $p$-linearly independent $x_{1},\dots,x_{\dim(V)+j+s+1}\in P$ such that $(h_{1},\dots,h_{j},x_{j'},U_{i},V), 0\leq i\leq k, 1\leq j'\leq \dim(V)+j+s+1$ are linearly independent tuples.
		Since Property-$(j+1)$ holds, we have that
			\begin{equation}\nonumber
		f\in \cap_{i=1}^{\dim(V)+j+s+1}J^{M}_{\sp_{\F_{p}}\{\iota(h_{1}),\dots,\iota(h_{m}),\iota(x_{i})\}+V}.
		\end{equation}
		Since $d'\geq  2\dim(V)+2m+5$,
		applying Proposition \ref{2:gri} for $r=1$ and $m=\dim(V)+j$, we have that $f\in J^{M}_{\sp_{\F_{p}}\{\iota(h_{1}),\dots,\iota(h_{j})\}+V}$. Since $h_{1},\dots,h_{j}$ are arbitrary, this implies that Property-$j$ holds.
		
		Inductively, we conclude that Property-0 holds, meaning that $f\in J^{M}_{V}$.
\end{proof}

\section{Freiman equations for $M$-modules}\label{2:s:ff}

\subsection{Solutions to Freiman equations}\label{2:s:ffh}

In additive combinatorics, subsets of a discrete abelian group with small doubling constant has a strong connection with  generalized arithmetic progressions and locally linear functions, which we recall as follows:

\begin{defn}[Generalized arithmetic progression]
	Let $G$ be an abelian group.
	We say that $P\subseteq G$ is a \emph{generalized arithmetic progression} in $G$ if $P$ can be written as 
	\begin{equation}\label{2:thisisP}
P=\{a+\ell_{1}v_{1}+\dots+\ell_{D}v_{D}\colon \ell_{i}\in\mathbb{Z}\cap(-L_{i},L_{i}) \text{ for all } 1\leq i\leq D\}
\end{equation}
	for some $D\in\N_{+}$, $a,v_{1},\dots,v_{D}\in G$ and $L_{1},\dots,L_{D}>0$. We abbreviate (\ref{2:thisisP}) as $P=a+(-L,L)\cdot v$, where $L:=(L_{1},\dots,L_{D})$ and $v:=(v_{1},\dots,v_{D})$. We say that $a$ is the \emph{base point}, $D$ the \emph{rank}, $v_{1},\dots,v_{D}$ the \emph{generators}, and $L_{1},\dots,L_{D}$ the \emph{lengths} of $P$. If all sums in $P$ are district, 
then we say that $P$ is \emph{proper}. If $a=id_{G}$, then we say that $P$ is \emph{homogeneous}.
\end{defn}	

\begin{defn}[Locally linear function] 
Let $G,G'$ be abelian groups
and
$P\subseteq G$ be a proper and homogeneous generalized arithmetic progression given by (\ref{2:thisisP}). 
We say that a map $\xi$ from $P$ to $G'$ is \emph{locally linear} if there exist $f_{0},f_{1},\dots,f_{D}$ in $G'$ such that
$$\xi(\ell_{1}v_{1}+\dots+\ell_{D}v_{D})=f_{0}+\ell_{1}f_{1}+\dots+\ell_{D}f_{D}$$ 
for all $\ell_{i}\in\mathbb{Z}\cap (-L_{i}, L_{i}), 1\leq i\leq D$.
\end{defn}

In this section, we focus on determine the solutions to  Freiman  equations. For abelian groups, the \emph{Freiman  equation} is of the form
\begin{equation}\label{2:gsold}
F(h_{1})+F(h_{2})\equiv F(h_{3})+F(h_{4}) \mod \Z \text{ for all } h_{1},h_{2},h_{3},h_{4}\in P \text{ with } h_{1}+h_{2}=h_{3}+h_{4},
\end{equation}
 where $P\subseteq G$ is a proper and homogeneous generalized arithmetic progression in $G$ and $F\colon G\to \R$ is a function. Assume that $P$ is given by (\ref{2:thisisP}). Then it is not hard to see that the value of $F$ on $P$ is determined by $F(\bold{0}), F(v_{i}), 1\leq i\leq D$, and thus $F$ is a locally linear function.


 In this paper, we need to understand the solutions to the Freiman  equation for shifted modules.
 Let $M\colon\V\to\F_{p}$ be a  quadratic form, $n\in\N_{+}$ and $H$ be a subset of $\Vk$. 
 %
 We say that a map  $\xi\colon H\to \st_{\zp,d}(s)$ 
 is a \emph{Freiman $M$-homomorphism of order $2^{n}$} on $H$ if for all $h_{1},\dots,h_{2^{n}}\in H$ with $h_{1}+\dots+h_{2^{n-1}}=h_{2^{n-1}+1}+\dots+h_{2^{n}}$, we have that 
  \begin{equation}\label{2:gsolew}
\xi(h_{1})+\dots+\xi(h_{2^{n-1}})\equiv\xi(h_{2^{n-1}+1})+\dots+\xi(h_{2^{n}}) \mod J^{M}_{\iota(h_{1}),\dots,\iota(h_{2^{n}-1})}.
\end{equation}
 In particular, $\xi$ is a Freiman $M$-homomorphism of order 4 (on $H$) if
 \begin{equation}\label{2:gsole}
\xi(h_{1})+\xi(h_{2})\equiv\xi(h_{3})+\xi(h_{4}) \mod J^{M}_{\iota(h_{1}),\iota(h_{2}),\iota(h_{3})}
\end{equation}
 for all $h_{1},h_{2},h_{3},h_{4}\in H$ with $h_{1}+h_{2}=h_{3}+h_{4}$. 
 \footnote{We remark that Freiman $M$-homomorphisms of order 4 is not the same as Freiman homomorphisms introduced in Definition \ref{2:llfh}, since we are allowed to modulo $J^{M}_{\iota(h_{1}),\dots,\iota(h_{2^{n}-1})}$ for  Freiman $M$-homomorphisms.}
 

Clearly, every  locally linear  map is a Freiman $M$-homomorphism of order 4. 
The main result of this section shows that the converse also holds:
 any Freiman $M$-homomorphism of order 4 is also  locally linear  when restricted to an appropriate sub generalized arithmetic progression if the dimension $d$ is large.
For $0<c\leq 1$, let $P(c)$ denote  the  generalized arithmetic progression
$$P(c)=\{\ell_{1}v_{1}+\dots+\ell_{D}v_{D}\colon \ell_{i}\in\mathbb{Z}\cap (-cL_{i}, cL_{i}), 1\leq i\leq D\}\subseteq \Vk.$$  
Since $$\frac{\vert\Z\cap (-cL,cL)\vert}{\vert \Z\cap (-L,L)\vert}\geq \frac{\max\{2cL-1,1\}}{2L+1}\geq \frac{2cL-1}{2L+1}\cdot\bold{1}_{L>1/c}+\frac{1}{2L+1}\cdot\bold{1}_{L\leq 1/c}\geq\frac{c}{c+2}$$ for all $L>0$ and $0<c<1$. We have that
\begin{equation}\label{2:crescale}
\vert P(c)\vert\geq (\frac{c}{c+2})^{D}\vert P\vert.\footnote{This estimate is rather coarse, but is good enough for our uses.}
\end{equation}

The task of this section is to solve the Freiman equations for $M$-modules:

\begin{thm}[Solution to Freiman equations I]\label{2:gsol2}
	Let $d,D,K\in\N_{+}$, $r,s\in\N$ with $d\geq \max\{s+10,11\}$, \footnote{With a little more effort, one can release the restriction to $d\geq s+10$ by studying the case $s=0$ more carefully. We leave the details to interested readers.}
	$\d>0$, $p\gg_{\d,d,D,r}1$ be a prime dividing $K$, and $M\colon\V\to\F_{p}$ be a non-degenerate quadratic form. Let $P\subseteq \Z_{K}^{d}$ be a proper and homogeneous generalized arithmetic progression of dimension $D$ with $\vert P\vert>\d K^{d}$. 
	If $\xi\colon P\to \st_{\Z/p^{r},d}(s)$ is a Freiman $M$-homomorphism of order 4, then there exist a  constant $0<c\leq 1$ 	depending only on $r,s$
	and a  locally linear map $T\colon P(c)\to \st_{\Z/p^{r},d}(s)$ such that 
	$\xi(h)\equiv T(h) \mod J^{M}_{\iota(h)}$ for all $h\in P(c)$.\footnote{It is natural to ask that whether the conclusion Theorem \ref{2:gsol2} holds for all $h\in P$ instead of for $h\in P(c)$. Such as an improvement seems promising by applying another bootstrap argument on Theorem \ref{2:gsol2}, but in this paper we do not need such a general result.}
\end{thm}

\begin{rem}
	Note that  Theorem \ref{2:gsol} fails if the dimension $d$ is small. In particular,   Theorem \ref{2:gsol} fails when $s=1$ and $\dim(P)=d=2$. Indeed, let $\xi\colon\F_{p}^{2}\to \F_{p}^{2}$ be given by $\xi(x,y)=xf(yx^{-1})$ for all $x,y\in\F_{p}^{2},x \neq 0$ and  $\xi(0,y)=0$, where $f\colon \F_{p}\to \F^{2}_{p}$	is arbitrary. Then it is not hard to see that (\ref{2:gsole}) holds for all $h_{1},\dots,h_{4}\in \F_{p}^{2}$ (in fact, if $d=2$, then (\ref{2:gsole}) holds trivially when $h_{1},\dots,h_{4}$ spans $\F_{p}^{2}$, and holds by the construction of $\xi$ when the span of $h_{1},\dots,h_{4}$ is of dimension 0 or 1). On the other hand, in general, $\xi$ does not coincides with a locally linear map mod $J^{M}_{\iota(h)}$. So it is an interesting question to ask what the sharp lower bound of $d$  is for the conclusion of Theorem \ref{2:gsol}  to hold.
\end{rem}	

\subsection{$M$-modules in $\F_{p}[x_{1},\dots,x_{d}]$}

To prove Theorem \ref{2:gsol2}, our first step is to study the case $r=1$. Since any $\Z/p$-coefficient  polynomial $f$ in $\R[x_{1},\dots,x_{d}]$ can be identified naturally as a  polynomial in $\F_{p}[x_{1},\dots,x_{d}]$ given by $\iota\circ(pf)\circ\tau$, it is convenient to extend the concept of $M$-modules to polynomials in $\F_{p}[x_{1},\dots,x_{d}]$.
Let $\st_{d}(s)$ denote the set of homogeneous polynomial in $\F_{p}[x_{1},\dots,x_{d}]$ of degree $s$. 
We adopt a convention similar to Convention \ref{2:csts}.

\begin{defn}[$M$-module for $\F_{p}$ polynomials]\label{2:defmmp}
	Let $d\in\N_{+}, p$ be a prime and $M\colon\V\to\F_{p}$ be a quadratic form associated with the matrix $A$ (see Section \ref{2:s:defn} for definitions). 
	We say that a subset $I$ of the polynomial ring $\F_{p}[x_{1},\dots,x_{d}]$ 
	is an \emph{$M$-module} if there exists a subspace $V$ of $\V$ such that 
	$$f\in I\Leftrightarrow f(n)=0 \text{ for all } n\in\V \text{ with } M(n)=0 \text{ and } (hA)\cdot n=0 \text{ for all } h\in V$$
	(or equivalently, $f\in I\Leftrightarrow f(n)=0 \text{ for all } n\in V(M)\cap V^{\pp}$). 
	In this case we denote $I$ as $J^{M}_{V}$. For convenience we denote 
	$$J^{M}:=J^{M}_{\{\bold{0}\}} \text{ and } J^{M}_{h_{1},\dots,h_{k}}:=J^{M}_{\sp_{\F_{p}}\{h_{1},\dots,h_{k}\}}.$$ 
	for all $h_{1},\dots,h_{k}\in\V$.
	%
%
\end{defn}

\begin{rem}
To simplify notations, we use the same  notation $J^{M}_{V}$ in Definitions \ref{2:defmmp0} and \ref{2:defmmp} to denote $M$-modules over $\F_{p}$ and over $\R$. This will not cause confusions as the meaning of $J^{M}_{V}$ will be clear from the context.
\end{rem}

 Let $M\colon\V\to\F_{p}$ be a  quadratic form, $n\in\N_{+}$ and $H$ be a subset of $\Vk$. 
 %
  We say that a map  $\xi\colon H\to \st_{d}(s)$ 
 is a \emph{Freiman $M$-homomorphism of order $2^{n}$} on $H$  if (\ref{2:gsolew}) holds (in the sense of Definition \ref{2:defmmp}) for all $h_{1},\dots,h_{2^{n}}\in H$ with $h_{1}+\dots+h_{2^{n-1}}=h_{2^{n-1}+1}+\dots+h_{2^{n}}$.
 A key step in proving Theorem \ref{2:gsol2} is to solve the Freiman equations for $M$-modules in $\F_{p}[x_{1},\dots,x_{d}]$.

\begin{thm}[Solution to Freiman equations II]\label{2:gsol}
	Let $d,D,K\in\N_{+}$, $s\in\N$ with $d\geq \max\{s+10,11\}$, 
	$\d>0$, $p\gg_{\d,d,D}1$ be a prime dividing $K$, and $M\colon\V\to\F_{p}$ be a non-degenerate quadratic form. Let $P\subseteq \Vk$ be a proper and homogeneous generalized arithmetic progression of dimension $D$ with $\vert P\vert>\d K^{d}$. 
	If $\xi\colon P\to \st_{d}(s)$ is a Freiman $M$-homomorphism of order 4, then there exist a  constant $0<c\leq 1$ 	depending only on $s$
	and a  locally linear map $T\colon P(c)\to \st_{d}(s)$ such that 
	$\xi(h)\equiv T(h) \mod J^{M}_{\iota(h)}$ for all $h\in P(c)$.
\end{thm}

Since one can identify $\st_{\Z/p,d}(s)$ with $\st_{d}(s)$, it is not hard to see that  Theorem \ref{2:gsol} is the special case of Theorem \ref{2:gsol2} for $r=1$. We will first prove Theorem \ref{2:gsol} in
 Sections \ref{2:s442} to \ref{2:sp4.3}. Then in Section  \ref{2:s446}, we prove  Theorem \ref{2:gsol2} by using Theorem \ref{2:gsol} and the $p$-expansion trick introduced in \cite{SunA}.

 Before proving Theorem \ref{2:gsol}, we summarize some useful properties for $M$-modules in $\F_{p}[x_{1},\dots,x_{d}]$. 
In the rest of this section, for $x,n\in\V$, denote $L_{n}(x):=(nA)\cdot x$, where $A$ is the matrix associated to $M$. 
 If $n\in\Vk$, then we denote $L_{n}:=L_{\iota(n)}$ for convenience. By identifying $\st_{\Z/p,d}(s)$ with $\st_{d}(s)$, we immediate deduce the following results:


 \begin{lem}[A special case of Lemma \ref{2:killLq}]\label{2:killL}
 Let $d\in\N_{+}$, $k,s\in\N$, $p\gg_{d,k,s} 1$ be a prime,	$M\colon\V\to \F_{p}$ be a non-degenerate quadratic form associated with the matrix $A$, and $h,h_{1},\dots,h_{k}\in\V$ be linearly independent vectors. Let $f\colon\V\to \F_{p}$ be a polynomial of degree $s$ such that $L_{h}f\in J^{M}_{h_{1},\dots,h_{k}}$. If 
 $\rank(M\vert_{\sp_{\F_{p}}\{h,h_{1},\dots,h_{k}\}^{\pp}})\geq 3$ or $d\geq 2k+5$, then we have that $f\in J^{M}_{h_{1},\dots,h_{k}}$.
\end{lem}	

\begin{prop}[A special case of Proposition \ref{2:grm}]\label{2:grmq}
Let $m,k,s\in\N$, $d,d',K\in\N_{+}$, $\d>0$,  $p\gg_{\d,d,k} 1$ be a prime dividing $K$, $M\colon\V\to \F_{p}$ be a non-degenerate quadratic form, $V,U_{1},\dots,U_{k}$ be subspaces of $\V$ such that all of $U_{1},\dots,U_{k}$ have the same dimension and that $(U_{i},V)$ is a linearly independent pair for all $1\leq i\leq k$, and $P\subseteq \Vk$ be a subset with $\vert P\vert>\d K^{d'}$. For convenience denote $U_{0}=\{\bold{0}\}$.
	Let 
	$f\in\st_{d}(s)$ be such that 
	$f\in J^{M}_{\sp_{\F_{p}}\{\iota(h_{1}),\dots,\iota(h_{m})\}+V}$ for all $h_{1},\dots,h_{m}\in P$ with $(h_{1},\dots,h_{m},U_{i},V), 0\leq i\leq k$ being linearly independent tuples. 	If   $d'\geq \max\{\dim(V)+m+s,\dim(U_{k})+\dim(V)+m,2\dim(V)+2m+5\}$, then we have that $f\in J^{M}_{V}.$
\end{prop}
 

 \begin{prop}[A special case of Proposition \ref{2:noloop3}]\label{2:noloop39}
 	Let  $d\in\N_{+}$, $k,s\in\N$, $p\gg_{d,k,s} 1$ be a prime number, 
 	 $M\colon\V\to\F_{p}$ be a non-degenerate quadratic form, and $V$ be a subspace of $\V$ of dimension $k$ with a basis   $h_{1},\dots,h_{k}\in\V$.  	Suppose that $\rank(M\vert_{V^{\perp_{M}}})\geq 3$. Then for all $f\in \st_{d}(s)$, we have that $f\in J^{M}_{V}$ if and only if
		$$f=R_{0}M_{0}+\sum_{i=1}^{k}R_{i}L_{h_{i}}$$
		for some $R_{0}\in\st_{d}(s-2)$ and $R_{i}\in\st_{d}(s-1)$,
		where $M_{0}(n):=(nA)\cdot n$.
		In particular, the conclusion of this proposition holds if $d\geq k+\dim(V\cap V^{\pp})+3$, or if $d\geq 2k+3$.
 \end{prop}
 
\subsection{Solutions to some fundamental equations}\label{2:s442}

The proof of Theorem \ref{2:gsol} is rather intricate.
Unlike in the abelian setting, for Theorem \ref{2:gsol}, the values of $\xi$  at $\bold{0},v_{i}, 1\leq i\leq D$ are far from being enough to determine the entire function $\xi$.
In order to prove Theorem \ref{2:gsol}, we begin with solving some (relatively) simple equations. Then we gradually increase the complexity of the equations until we reach  (\ref{2:gsole}).

We start with the following trivial equation:
\begin{equation}\label{2:sbeq1}
f(x)\equiv f(y) \mod \Z \text{ for all $x,y\in G$},
\end{equation}
where $f$ is a map from $G$ to $\R$ for some group $G$.   It is obvious that the solution to (\ref{2:sbeq1}) is $f\equiv c\mod \Z$ for some $c\in\R$.
In the setting of shifted modules, one can study a similar equation:
\begin{equation}\label{2:sbeq2}
F(x)\equiv F(y)\mod J^{M}_{\iota(x),\iota(y)} \text{ for all $x,y\in \Vk$},
\end{equation}
 where $F$ is a map from $\Vk$ to $\st_{d}(s)$. Comparing (\ref{2:sbeq1}) with (\ref{2:sbeq2}), it is natural to conjecture that the  solution to (\ref{2:sbeq1}) is $F(x)\equiv f\mod J^{M}_{\iota(x)}$ for some $f\in \st_{d}(s)$. This is indeed the case at least when $\iota(x)\neq \bold{0}$.
 
 \begin{prop}\label{2:coco01}
	Let $d,K\in\N_{+}$, $s\in\Z$,  and $p\gg_{d,s} 1$  be a prime dividing $K$.
	Let   $M\colon\V\to\F_{p}$ be a non-degenerate quadratic form and $F\colon \Vk\to\st_{d}(s)$ be a map such that
	$$F(x)\equiv F(y)\mod J^{M}_{\iota(x),\iota(y)}$$ for all $x,y\in \Vk$. If  $d\gg_{s} 1$, then there exists $G\in\st_{d}(s)$ such that $$F(x)\equiv G\mod J^{M}_{\iota(x)}$$ for all $x\in \Vk$ with $\iota(x)\neq \bold{0}$.
\end{prop}

Proposition \ref{2:coco01} is a special case of the following result:

\begin{prop}\label{2:coco1}
	Let $d,K\in\N_{+}$, $k\in\N, s\in\Z$, $\d>0$, $p\gg_{\d,d,k} 1$ be a prime dividing $K$,   $H$ be a subset of $\Vk$ with $\vert H\vert>\delta K^{d}$ and $M\colon\V\to\F_{p}$ be a non-degenerate quadratic form.
	Let $k\in\N$ and $U_{1},\dots,U_{k},V\subseteq\V$ be (possibly trivial) subspaces of $\V$ with $U_{1},\dots,U_{k}$ having the same dimension such that $(U_{i},V)$ is a linearly independent pair for all $1\leq i\leq k$. 
	For convenience denote $U_{0}=\{\bold{0}\}$.
	Let $F\colon \Vk\to\st_{d}(s)$ be a map such that
	\begin{equation}\label{2:coco1e1}
	F(x)\equiv F(y)\mod J^{M}_{\sp_{\F_{p}}\{\iota(x),\iota(y)\}+V}
	\end{equation}
	 for all $x,y\in H$ with $(\iota(x),\iota(y),U_{i},V)$ being a linearly independent tuple for all $0\leq i\leq k$. If $d\geq s+\dim(U_{k})+2\dim(V)+8$, then there exists $G\in\st_{d}(s)$ such that 
	\begin{equation}\label{2:coco1e2}
	F(x)\equiv G\mod J^{M}_{\sp_{\F_{p}}\{\iota(x)\}+V}
	\end{equation}
	  for all $x\in H$ with $(\iota(x),U_{i},V)$ being  a linearly independent tuple for all $0\leq i\leq k$.
	\end{prop}

 Note that Proposition  \ref{2:coco1} improves Proposition \ref{2:coco01} in the following senses. Firstly, we work on a large subset $H$ of $\Vk$  instead of working on the entire $\Vk$. Secondly,
  we allow to  quotient  the map $F$   by a module $J^{M}_{V}$. Finally, we allow (\ref{2:coco1e1}) to fail on a set of tuples of $(x,y)$ with small cardinality  (at the cost that the expression (\ref{2:coco1e2}) will also fail on a set with small cardinality).   Proposition \ref{2:coco1} is a key property that will be used frequently in the rest of this section.

\begin{proof}[Proof of Proposition \ref{2:coco1}]
By Convention \ref{2:csts}, there is nothing to prove when $s<0$. So we assume that $s\in\N$.
Throughout the proof we assume that $p\gg_{\d,d,k} 1$.
	If $s=0$, then $F(x)$ is a constant polynomial for all $x\in\Vk$.
	 Since $J^{M}_{\sp_{\F_{p}}\{\iota(x),\iota(y)\}+V}$ is generated by nontrivial homogeneous polynomials and $d\geq \dim(V)+4$, we have that  $1\notin J^{M}_{\sp_{\F_{p}}\{\iota(x),\iota(y)\}+V}$ and thus it follows from (\ref{2:coco1e1}) that $F(x)=F(y)$ for all $x,y\in H$ with $(\iota(x),\iota(y),U_{i},V), 0\leq i\leq k$ being linearly independent tuples.
	 On the other hand, if $d\geq\dim(U_{k})+\dim(V)+2$, then for any $x,y\in H$ with $(\iota(x),U_{i},V),(\iota(y),U_{i},V), 0\leq i\leq k$ being linearly independent tuples, there exists $z\in H$ such that $(\iota(x),\iota(z),U_{i},V), (\iota(y),\iota(z),U_{i},V), 0\leq i\leq k$ are all  linearly independent tuples (in fact, the number of $z$ such that this condition fails is at most $O_{k}(K^{\dim(U_{1})+\dim(V)+1})$). So we have that $F(x)=F(z)=F(y)$ and thus $F(x)$ is independent of $x$. This proves the case $s=0$.

	\
	
	Suppose now the conclusion holds for $s-1$ for some $s\geq 1$. We prove that the conclusion holds for $s$. 	For all $x,y\in H$ with $(\iota(x),\iota(y),U_{i},V), 0\leq i\leq k$ being linearly independent tuple, since $d\geq 2(\dim(V)+2)+3$,
	by (\ref{2:coco1e1}) and Proposition \ref{2:noloop39}, 
	there exist   $A_{x,y},B_{x,y}\in \st_{d}(s-1)$ such that
	$$F(x)-F(y)\equiv L_{x}A_{x,y}-L_{y}B_{x,y}\mod J^{M}_{V}.$$
	Using the identity $F(x)-F(z)=(F(x)-F(y))+(F(y)-F(z))$, we deduce that
	$$L_{x}(A_{x,y}-A_{x,z})+L_{y}(A_{y,z}-B_{x,y})+L_{z}(B_{x,z}-B_{y,z})\equiv 0\mod J^{M}_{V}$$
	for all  $x,y,z\in H$ with $(\iota(x),\iota(y),\iota(z),U_{i},V), 0\leq i\leq k$ being linearly independent tuples.
	This implies that 
	$L_{x}(A_{x,y}-A_{x,z})\in J^{M}_{\sp_{\F_{p}}\{\iota(y),\iota(z)\}+V}$.  	
	By Lemma \ref{2:killL}, if $d\geq 2\dim(V)+9$, then
	\begin{equation}\label{2:136aa1}
	\text{$A_{x,y}\equiv A_{x,z}  \mod J^{M}_{\sp_{\F_{p}}\{\iota(y),\iota(z)\}+V}$}
	\end{equation}
	if $(\iota(x),\iota(y),\iota(z),U_{i},V), 0\leq i\leq k$ are linearly independent tuples.
	Similarly, 
	\begin{equation}\label{2:136aa2}
	\text{$B_{x,z}\equiv B_{y,z}\mod J^{M}_{\sp_{\F_{p}}\{\iota(x),\iota(y)\}+V}$}
	\end{equation}
	and 
	\begin{equation}\label{2:136aa3}
	\text{$A_{y,z}\equiv B_{x,y}\mod J^{M}_{\sp_{\F_{p}}\{\iota(x),\iota(z)\}+V}$}
	\end{equation}
 if $(\iota(x),\iota(y),\iota(z),U_{i},V), 0\leq i\leq k$ are linearly independent tuples.
	
	By (\ref{2:136aa1}) and the induction hypothesis, since $d\geq 
	s+\dim(U_{k})+2\dim(V)+8=(s-1)+(\dim(U_{k})+1)+2\dim(V)+8$, 
	for all $x\in H$ with  $(\iota(x),U_{i},V), 0\leq i\leq k$ being linearly independent tuples,
	there exists  $G_{x}\in  \st_{d}(s-1)$ such that $A_{x,y}\equiv G_{x}\mod J^{M}_{\sp_{\F_{p}}\{\iota(y)\}+V}$ for all $y\in H$ with $(\iota(x),\iota(y),U_{i},V), 1\leq i\leq k$ being linearly independent tuples. Similarly, by (\ref{2:136aa2}), there exists  $G'_{y}\in  \st_{d}(s-1)$ such that $B_{x,y}\equiv G'_{y}\mod J^{M}_{\sp_{\F_{p}}\{\iota(h)\}+V}$ for all $x,y\in H$ with $(\iota(x),\iota(y),U_{i},V), 0\leq i\leq k$ being linearly independent tuples. 
	
	Fix $y\in H$ with $(\iota(y),U_{i},V), 0\leq i\leq k$ being linearly independent tuples. By (\ref{2:136aa3}),
	for any $x,z\in H$ with $(\iota(x),\iota(y),\iota(z),U_{i},V), 0\leq i\leq k$ being linearly independent tuples,  we have that
	$G_{y}-G'_{y}\equiv A_{y,z}-B_{x,y}\equiv 0 \mod J^{M}_{\sp_{\F_{p}}\{\iota(x),\iota(z)\}+V}$.
	By Lemma \ref{2:grmq}, if $d\geq\max\{\dim(V)+s+2,\dim(U_{k})+\dim(V)+3,2\dim(V)+9\}$, then $G_{y}\equiv G'_{y}\mod J^{M}_{V}.$  
	
	Now we have that 
	$$F(x)-F(y)\equiv L_{x}A_{x,y}-L_{y}B_{x,y}\equiv L_{x}G_{x}-L_{y}G'_{y}\equiv L_{x}G_{x}-L_{y}G_{y} \mod J^{M}_{V}$$
	and so 
	$$F'(x)\equiv  F'(y) \mod J^{M}_{V}$$
	for all $x,y\in H$ with $(\iota(x),\iota(y),U_{i},V), 0\leq i\leq k$ being linearly independent tuples, where $F'(x):=F(x)-L_{x}G_{x}$.
	So for all $x,x'\in H$ with $(\iota(x),U_{i},V), (\iota(x'),U_{i},V), 0\leq i\leq k$ being linearly independent tuples, if $d\geq 2+\dim(U_{k})+\dim(V)$, then there exists $y\in H$  such that $(\iota(x),\iota(y),U_{i},V), (\iota(x'),\iota(y),U_{i},V), 0\leq i\leq k$ are linearly independent tuples (to see this, note that $\vert H\vert>\d K^{d}$ and that the number of $y\in\Vk$ such that this condition fails is at most $O_{k}(K^{\dim(U_{k})+\dim(V)+1})$). So
	$$F'(x)\equiv F'(y)\equiv F'(x') \mod J^{M}_{V}.$$
	In other words, there exists $G\in\st_{d}(s)$ such that for all  $x\in H$ with $(\iota(x),U_{i},V), 0\leq i\leq k$ being linearly independent tuples, 
	$F(x)\equiv L_{x}G_{x}+G \mod J^{M}_{V}$ and thus $F(x)-G\in J^{M}_{\sp_{\F_{p}}\{\iota(x)\}+V}$.
\end{proof}	

 As an immediate corollary of Proposition \ref{2:coco1}, we have

\begin{coro}\label{2:coco1c}
	Let $d,K\in\N_{+}$, $k\in\N,s\in\Z$, $\d>0$, $p\gg_{\d,d,k} 1$ be a prime dividing $K$,   $H$ be a subset of $\Vk$ with $\vert H\vert>\delta K^{d}$ and $M\colon\V\to\F_{p}$ be a non-degenerate quadratic form.
	Let $k\in\N$ and $U_{1},\dots,U_{k},V\subseteq\V$ be (possibly trivial) subspaces of $\V$  with $U_{1},\dots,U_{k}$ having the same dimension such that $(U_{i},V)$ is a linearly independent pair for all $1\leq i\leq k$. 
	For convenience denote $U_{0}=\{\bold{0}\}$.
	Let $F,F'\colon \Vk\to\st_{d}(s)$ be  maps such that
	$$F(x)\equiv F'(y)\mod J^{M}_{\sp_{\F_{p}}\{\iota(x),\iota(y)\}+V}$$ for all $x,y\in H$ with $(\iota(x),\iota(y),U_{i},V)$ being a linearly independent tuple for all $0\leq i\leq k$. If $d\geq s+\dim(U_{k})+2\dim(V)+11$, then there exists $G\in\st_{d}(s)$ such that $$F(x)\equiv F'(x)\equiv G\mod J^{M}_{\sp_{\F_{p}}\{\iota(x)\}+V}$$ for all $x\in H$ with $(\iota(x),U_{i},V)$ being  a linearly independent tuple for all $0\leq i\leq k$.
	\end{coro}	
\begin{proof}
By Convention \ref{2:csts}, there is nothing to prove when $s<0$. So we assume that $s\in\N$.
Throughout the proof we assume that $p\gg_{\d,d,k} 1$.
Fix any $x,y\in H$ with $(\iota(x),U_{i},V),(\iota(y),$ $U_{i},V)$ being  linearly independent tuples for all $0\leq i\leq k$. For any $z\in H$ with $(\iota(x),\iota(y),\iota(z),U_{i},V)$ being  linearly independent tuples for all $0\leq i\leq k$, by assumption 
\begin{equation}\label{2:loli}
F(x)\equiv F'(z)\equiv F(y)\mod J^{M}_{\sp_{\F_{p}}\{\iota(x),\iota(y),\iota(z)\}+V}.
\end{equation}  
By Proposition \ref{2:grmq}, since $d\geq \max\{\dim(V)+s+3,\dim(U_{k})+\dim(V)+3, 2\dim(V)+11\}$, we have that  
$$F(x)\equiv F(y)\mod J^{M}_{\sp_{\F_{p}}\{\iota(x),\iota(y)\}+V}$$
 for all  $x,y\in H$ with $(\iota(x),U_{i},V),(\iota(y),U_{i},V)$ being  linearly independent tuples for all $0\leq i\leq k$. By Proposition \ref{2:coco1}, there exists $G\in\st_{d}(s)$ such that $F(x)\equiv G\mod J^{M}_{\sp_{\F_{p}}\{\iota(x)\}+V}$ for all  $x\in H$ with $(\iota(x),U_{i},V),$ being  linearly independent tuples for all $0\leq i\leq k$. Similarly, there exists $G'\in\st_{d}(s)$ such that $F'(y)\equiv G'\mod J^{M}_{\sp_{\F_{p}}\{\iota(y)\}+V}$ for all  $y\in H$ with $(\iota(y),U_{i},V),$ being  linearly independent tuples for all $0\leq i\leq k$. Fix any $x,y\in H$ with $(\iota(x),\iota(y),U_{i},V)$ being  linearly independent tuples for all $0\leq i\leq k$, we have that 
 $$G\equiv F(x)\equiv F'(y)\equiv G' \mod J^{M}_{\sp_{\F_{p}}\{\iota(x),\iota(y)\}+V}.$$
 By Proposition \ref{2:grmq}, since $d\geq\max\{\dim(V)+s+2,\dim(U_{k})+\dim(V)+2,2\dim(V)+9\}$, we have that $G\equiv G' \mod J^{M}_{V}$ and we are done.
\end{proof}

\begin{rem}
	It is worth summarizing the idea of the proof of Corollary \ref{2:coco1c} briefly.
	Ideally, we wish to arrive at the conclusion that
	\begin{equation}\nonumber
	F(x)\equiv  F(y)\mod J^{M}_{\sp_{\F_{p}}\{\iota(x),\iota(y)\}+V}.
	\end{equation}
	 directly from the hypothesis in the statement. However, due to a loss of information, we were only able to deduce the weaker property (\ref{2:loli}) in the first place,  where we need to modulo an additional module $J^{M}_{z}$. In order to overcome this difficulty, we apply (\ref{2:loli}) for many different $z$, and then intersect all the modules $J^{M}_{{\sp_{\F_{p}}\{\iota(x),\iota(y),\iota(z)\}+V}}$ to recover the lost information. For convenience we refer to this method as the \emph{intersection approach}, which   will be frequently used in this paper.
\end{rem}

We now move from (\ref{2:coco1e1}) to the following slightly different equation:
\begin{equation}\label{2:coco1e3}
	F(x)\equiv F(y)\mod J^{M}_{\iota(x-y)},
	\end{equation}
	in which case we can obtain a better description of the solutions than (\ref{2:coco1e1}). In order to state this result, we need to introduce the following notion:

\begin{defn}[Super polynomials]
Let $d,k\in\N_{+}$ and $s\in\N$.
We say that a map $F\colon\st_{d}(1)^{k}\to\st_{d}(s)$ is a \emph{super polynomial} if for all $i\in\N^{k}$ with $\vert i\vert\leq s$, there exists
$C_{i}\in\st_{d}(s-\vert i\vert)$ such that 
$$F(f_{1},\dots,f_{k})=\sum_{i=(i_{1},\dots,i_{k})\in \N^{k}, \vert i\vert\leq s}C_{i}f_{1}^{i_{1}}\cdot\dots\cdot f_{k}^{i_{k}}$$ for all $f_{1},\dots,f_{k}\in \st_{d}(1)$. 
We say that $F$ is \emph{symmetric} if $F(f_{1},\dots,f_{k})=F(f_{\sigma(1)},\dots,f_{\sigma(k)})$ for all permutation $\sigma\colon\{1,\dots,k\}\to\{1,\dots,k\}$.
The \emph{degree} of $F$ (denoted as $\deg(F)$) is the largest integer $t\in\N$ such that $C_{i}\neq 0$ for some  $i\in\N^{k}$ with $\vert i\vert=t$ (it is clear that $\deg(F)\leq s$). For convenience we define the constant zero function to have degree $s$ for any $s\in\Z$.
We say that $F$ is \emph{homogeneous} if $C_{i}\neq 0$ for all $i\in\N^{k}$ with $\vert i\vert\neq \deg(F)$.
\end{defn}

Since $L_{x}-L_{y}=L_{x-y}\in J^{M}_{x-y}$,
it is not hard to see that $F(x)=G(L_{x})$ is a solution to (\ref{2:coco1e3}) for any super polynomial $G\colon\st_{d}(1)\to\st_{d}(s)$. The next proposition shows that this is the only possibility.

 \begin{prop}\label{2:cocon1}
	Let $d,K\in\N_{+}$, $k\in\N, s\in\Z$,  $\delta>0$ and $p\gg_{\d,d,k} 1$  be a prime dividing $K$.
	Let $H$  be a subset of $\Vk$ with $\vert H\vert>\d K^{d}$ and $M\colon\V\to\F_{p}$ be a non-degenerate quadratic form. Let $U_{1},\dots,U_{k}$ be (possibly trivial) subspaces of $\V$ of same dimensions. For convenience denote $U_{0}=\{\bold{0}\}$. Let $F\colon \Vk\to\st_{d}(s)$ be a map such that
	$$F(x)\equiv F(y)\mod J^{M}_{\iota(x-y)}$$ for all $x,y\in H$ with $(\iota(x),\iota(y),U_{i})$ being  linearly independent tuples for all $0\leq i\leq k$. If  $d\geq s+\dim(U_{k})+8$, then there exists a super polynomial $G\colon\st_{d}(1)\to\st_{d}(s)$ such that $$F(x)\equiv G(L_{x})\mod J^{M}$$ for all $x\in H$ with $(\iota(x),U_{i})$ being  linearly independent tuples for all $0\leq i\leq k$. 
\end{prop}	
\begin{proof} 
By Convention \ref{2:csts}, there is nothing to prove when $s<0$. So we assume that $s\in\N$.
Throughout the proof we assume that $p\gg_{\d,d,k} 1$.
    By Proposition \ref{2:coco1},  since $d\geq s+\dim(U_{k})+8$, there exists $G_{0}\in\st_{d}(s)$ such that $F(x)\equiv G_{0} \mod J^{M}_{\iota(x)}$ for all $x\in H$ with $(\iota(x),U_{i})$ being  linearly independent tuples for all $0\leq i\leq k$. By Proposition \ref{2:noloop39}, since $d\geq 5$, we may write 
$$F(x)\equiv G_{0}+L_{x}F_{1}(x)\mod J^{M}$$
 for   all $x\in H$ with $(\iota(x),U_{i})$ being   linearly independent tuples for all $0\leq i\leq k$ for some $F_{1}(x)\in\st_{d}(s-1)$.
 
 Suppose we have shown that for some $k\in\N$, there
 exist $G_{i}\in\st_{d}(s-i), 0\leq i\leq k$ and some map $F_{k+1}\colon\Vk\to\st_{d}(s-k-1)$ such that
 $$F(x)\equiv L_{x}^{k+1}F_{k+1}(x)+\sum_{i=0}^{k}L_{x}^{i}G_{i}\mod J^{M}$$
  for   all $x\in H$ with $(\iota(x),U_{i})$ being linearly independent tuples for all $0\leq i\leq k$. We show that the same conclusion holds for $k+1$.
Since $F(x)\equiv F(y)\mod J^{M}_{\iota(x-y)}$ for all $x,y\in H$ with $(\iota(x),\iota(y),U_{i})$ being  linearly independent tuples for all $0\leq i\leq k$, we have that 
 $$L_{x}^{k+1}F_{k+1}(x)+\sum_{i=0}^{k}L_{x}^{i}G_{i}\equiv L_{y}^{k+1}F_{k+1}(y)+\sum_{i=0}^{k}L_{y}^{i}G_{i}\mod J^{M}_{\iota(x-y)}.$$
 Since $L_{x}-L_{y}=L_{x-y}$, this implies that 
 $L_{x}^{k+1}(F_{k+1}(x)-F_{k+1}(y)) \in J^{M}_{\iota(x-y)}$. Since $d\geq 7$, by repeatedly using Lemma \ref{2:killL}, we have that $F_{k+1}(x)-F_{k+1}(y) \in J^{M}_{\iota(x-y)}\subseteq J^{M}_{\iota(x),\iota(y)}$ for all $x,y\in H$ with $(\iota(x),\iota(y),U_{i})$ being  linearly independent tuples for all $0\leq i\leq k$. By 
    Proposition \ref{2:coco1},   there exists $G_{k+1}\in\st_{d}(s-k-1)$ such that $F_{k+1}(x)\equiv G_{k+1} \mod J^{M}_{\iota(x)}$ for   all  $x\in H$ with $(\iota(x),U_{i})$ being  linearly independent tuples for all $0\leq i\leq k$.
    By Proposition \ref{2:noloop39}, we may write 
$$F_{k+1}(x)\equiv G_{k+1}+L_{x}F_{k+2}(x)\mod J^{M}$$
 for     $x\in H$ with $(\iota(x),U_{i})$ being   linearly independent tuples for all $0\leq i\leq k$ for some $F_{k+2}(x)\in\st_{d}(s-k-2)$.
     Then
    $$F(x)\equiv L_{x}^{k+1}F_{k+1}(x)+\sum_{i=0}^{k}L_{x}^{i}G_{i}\equiv L_{x}^{k+2}F_{k+2}(x)+\sum_{i=0}^{k+1}L_{x}^{i}G_{i}\mod J^{M}$$
  for   all $x\in H$ with $(\iota(x),U_{i})$ being   linearly independent tuples for all $0\leq i\leq k$.
  
  So by induction, 
  there
 exist $G_{i}\in\st_{d}(s-i), 0\leq i\leq s$ and some map $F_{s+1}\colon\Vk\to\st_{d}(-1)$ (meaning that $F_{s+1}\equiv 0$ by Convention \ref{2:csts}) such that
 $$F(x)\equiv L_{x}^{s+1}F_{s+1}(x)+\sum_{i=0}^{s}L_{x}^{i}G_{i}\equiv  \sum_{i=0}^{s}L_{x}^{i}G_{i}\mod J^{M}$$
  for   all $x\in H$ with $(\iota(x),U_{i})$ being  linearly independent tuples for all $0\leq i\leq k$. This completes the proof by the definition of super polynomials.
  \end{proof}

\subsection{Properties for super polynomials}

Before moving  from equation (\ref{2:coco1e3}) to more sophisticated equations, we prove some properties for super polynomials for later uses.

By Lemma \ref{2:ns}, if a polynomial $P\in\poly(\V\to\F_{p})$ takes value zero on a large subset of $\V$, then $P\equiv 0$. A similar result holds for super polynomials:

 \begin{prop}\label{2:cocozero}
  Let $d,k,K\in\N_{+}$, $s\in\Z$, $\d>0$  and $p\gg_{\d,d} 1$  be a prime  dividing $K$.
	Let $U$ be a subspace of $\V$, $H$ be a subset of $\Vk$ with $\vert H\vert>\d K^{d}$,   $M\colon\V\to\F_{p}$ be a non-degenerate quadratic form, and $F\colon\st_{d}(1)^{k}\to\st_{d}(s)$ be a super polynomial given by $$F(f_{1},\dots,f_{k})=\sum_{i=(i_{1},\dots,i_{k})\in \N^{k}, \vert i\vert\leq s}C_{i}f_{1}^{i_{1}}\cdot\ldots\cdot f_{k}^{i_{k}}.$$
	 Suppose that $F(L_{x_{1}},\dots,L_{x_{k}})\equiv 0 \mod J^{M}$ for all $x_{1},\dots,x_{k}\in H$ with $(\iota(x_{1}),\dots,\iota(x_{k}),U)$ being a linearly independent tuple. If $d\geq \max\{s+1,\dim(U)+k,7\}$, then
  $C_{i}\equiv 0 \mod J^{M}$ for all $i\in \N^{k}, \vert i\vert\leq s$.
  \end{prop}
\begin{proof}
By Convention \ref{2:csts}, there is nothing to prove when $s<0$. So we assume that $s\in\N$.
Throughout the proof we assume that $p\gg_{\d,d} 1$.
It suffices to prove Proposition \ref{2:cocozero} under the assumption that the conclusion holds for $k-1$ (when $k=1$, then we prove Proposition \ref{2:cocozero} without any induction hypothesis).

We may rewrite $F$ as 
$$F(f_{1},\dots,f_{k})=\sum_{0\leq i\leq s}F_{i}(f_{1},\dots,f_{k-1}) f_{k}^{i}$$
 for some super polynomial $F_{i}\colon\st_{d}(1)^{k-1}\to\st_{d}(s-i)$ (where $F_{i}$ is regarded as a constant if $k=1$).
For any  $x_{1},\dots,x_{k}\in H$ with $(\iota(x_{1}),\dots,\iota(x_{k}),U)$ being a linearly independent tuple, since $F(L_{x_{1}},\dots,L_{x_{k}})\equiv 0 \mod J^{M}$, we have that $F_{0}(L_{x_{1}},\dots,L_{x_{k-1}})\equiv 0 \mod J^{M}_{\iota(x_{k})}$. 
Note that the set of $x_{k}\in H$ with  $(\iota(x_{1}),\dots,\iota(x_{k}),U)$ not being a linearly independent tuple is of cardinality $O(K^{\dim(U)+k-1})$.
By Proposition \ref{2:grmq} (setting $m=1$, $V=\{\bold{0}\}$, and $U_{1}=\sp_{\F_{p}}\{\iota(x_{1}),\dots,\iota(x_{k-1})\}+U$), since $d\geq \max\{s+1,\dim(U)+k,7\}$,  we have that $F_{0}(L_{x_{1}},\dots,L_{x_{k-1}})\equiv 0 \mod J^{M}$ for any  $x_{1},\dots,x_{k-1}\in H$ with $(\iota(x_{1}),\dots,\iota(x_{k-1}),U)$ being a linearly independent tuple.
By induction hypothesis, we have that  $F_{0}$ takes values in $J^{M}$. So for any  $x_{1},\dots,x_{k}\in H$ with $(\iota(x_{1}),\dots,\iota(x_{k}),U)$ being a linearly independent tuple, the fact that $F(L_{x_{1}},\dots,L_{x_{k}})\equiv 0 \mod J^{M}$ implies that $L_{x_{k}}F'(L_{x_{1}},\dots,L_{x_{k}})\in J^{M}$, where 
 $$F'(f_{1},\dots,f_{k})=\sum_{1\leq i\leq s}F_{i}(f_{1},\dots,f_{k-1}) f_{k}^{i-1}.$$
By Lemma \ref{2:killL}, since $d\geq 5$, we have that $F'(L_{x_{1}},\dots,L_{x_{k}})\in J^{M}$. We may then repeat the previous discussion to inductively show that all of $F_{0},\dots,F_{s}$ takes values in $J^{M}$.  The conclusion now follows from the induction hypothesis.
\end{proof}

The following technical proposition will be used in the next subsection.

  \begin{prop}\label{2:cocoprr}
  Let $d,K\in\N_{+}$, $s\in\Z$,  $\delta>0$ and $p\gg_{\d,d} 1$  be a prime  dividing $K$.
	Let  $H$ be a subset of $\Vk$ with $\vert H\vert>\d K^{d}$, $M\colon\V\to\F_{p}$ be a non-degenerate quadratic form, and $G_{x}\colon \st_{d}(1)\to\st_{d}(s)$ be a super polynomial for all $x\in H$ with $\iota(x)\neq\bold{0}$ such that
	\begin{equation}\label{2:gxyz4}
\begin{split}
G_{x}(L_{y})\equiv G_{y}(L_{x}) \mod J^{M}
\end{split}
\end{equation}
 for all $p$-linearly independent $x,y\in H$. If $d\geq s+11$, then there exists a symmetric super polynomial $Q\colon \st_{d}(1)\times\st_{d}(1)\to\st_{d}(s)$  such that 
 $$G_{x}(L_{y})\equiv Q(L_{x},L_{y})  \mod J^{M}$$
for all $p$-linearly independent $x,y\in H$.
  \end{prop}

  Before proving Proposition \ref{2:cocoprr}, we need a  lemma:
  
  \begin{lem}\label{2:sbppp}
 Let $d\in\N_{+}$, $k,s\in\N$ with $k\leq s$,  and $p\gg_{d,s} 1$  be a prime.
	Let   $M\colon\V\to\F_{p}$ be a non-degenerate quadratic form and $x,y\in\V$ be linearly independent. 
 Suppose that 
 \begin{equation}\label{2:splll}
 \sum_{i=0}^{k}C_{i}L_{x}^{i}L_{y}^{k-i}\equiv PL_{x}^{k+1}+QL_{y}^{k+1} \mod J^{M}
 \end{equation}
 for some $C_{i}\in\st_{d}(s-k)$ and some $P,Q\in\st_{d}(s-k-1)$. If $d\geq 7$, then we have that $C_{i}\in J^{M}_{x,y}$ for all $0\leq i\leq k$.
 \end{lem} 
 \begin{proof}
Throughout the proof we assume that $p\gg_{d,s} 1$.
 If $k=0$, then then it follows from (\ref{2:splll}) that $C_{0}
\equiv PL_{x}+QL_{y} \mod J^{M}$ and thus $C_{0}\in J^{M}_{x,y}$. If $k=1$, then it follows from (\ref{2:splll}) that 
  \begin{equation}\nonumber
  C_{0}L_{y}+C_{1}L_{x}\equiv PL_{x}^{2}+QL_{y}^{2} \mod J^{M}.
 \end{equation}
 It follows  that $C_{0}L_{y}\equiv QL_{y}^{2} \mod J^{M}_{x}$. By Lemma \ref{2:killL}, since $d\geq 7$, we have that $C_{0}\equiv QL_{y} \mod J^{M}_{x}$ and thus $C_{0}\in J^{M}_{x,y}$.  Similarly, we have $C_{1}\in J^{M}_{x,y}$. So the conclusion holds.

 Now assume that the conclusion holds for $k-2$ for some $2\leq k\leq s$. We prove that the conclusion holds for $k$.
 It follows from (\ref{2:splll}) that $C_{0}L_{y}^{k}\equiv QL_{y}^{k+1} \mod J^{M}_{x}$.
 Since $d\geq 7$, by repeatedly using Lemma \ref{2:killL},
   we have that $C_{0}\equiv QL_{y} \mod J^{M}_{x}$ and thus $C_{0}\in J^{M}_{x,y}$. By Proposition \ref{2:noloop39}, since $d\geq 7$, there exist $P',Q'\in \st_{d}(s-k-1)$ such that $C_{0}\equiv L_{x}P'+L_{y}Q' \mod J^{M}$. Similarly, there exist $P'',Q''\in \st_{d}(s-k-1)$ such that $C_{k}\equiv L_{x}P''+L_{y}Q'' \mod J^{M}$. Substituting the expressions of $C_{0}$ and $C_{k}$ back to (\ref{2:splll}), we have that 
  \begin{equation}\label{2:splll2}
 L_{x}L_{y}(\sum_{i=0}^{k-2}C_{i+1}L_{x}^{i}L_{y}^{k-2-i})\equiv (P-P'')L_{x}^{k+1}+(Q-Q')L_{y}^{k+1}-P'L_{x}L_{y}^{k}-Q''L_{x}^{k}L_{y} \mod J^{M}.
 \end{equation}
 In particular, we have that $(P-P'')L_{x}^{k+1}\in J^{M}_{y}$.  By  Lemma \ref{2:killL},  we have that $P-P''\mod J^{M}_{y}$. By Proposition \ref{2:noloop39},  there exists $P'''\in \st_{d}(s-k-2)$ such that $P-P''\equiv L_{y}P'''\mod J^{M}$. Similarly, there exists $Q'''\in \st_{d}(s-k-2)$ such that $Q-Q'\equiv L_{x}Q'''\mod J^{M}$. Substituting these back to (\ref{2:splll2}), we have that 
  \begin{equation}\nonumber
 L_{x}L_{y}(\sum_{i=0}^{k-2}C_{i+1}L_{x}^{i}L_{y}^{k-2-i})\equiv L_{x}L_{y}(P'''L_{x}^{k}+Q'''L_{y}^{k}-P'L_{y}^{k-1}-Q''L_{x}^{k-1}) \mod J^{M}.
 \end{equation}
 By repeatedly using Lemma \ref{2:killL}, we have that 
  \begin{equation}\nonumber
 \sum_{i=0}^{k-2}C_{i+1}L_{x}^{i}L_{y}^{k-2-i}\equiv (P'''L_{x}-Q'')L_{x}^{k-1}+(Q'''L_{y}-P')L_{y}^{k-1} \mod J^{M}.
 \end{equation}
 By induction hypothesis, we have that $C_{i}\in J^{M}_{x,y}$ for all $1\leq i\leq k-1$ and thus for all $0\leq i\leq k$. This completes the proof.
  \end{proof}

  \begin{proof}[Proof of Proposition \ref{2:cocoprr}]
  By Convention \ref{2:csts}, there is nothing to prove when $s<0$. So we assume that $s\in\N$.
Throughout the proof we assume that $p\gg_{\d,d} 1$.
    For $0\leq k\leq s+1$, we say that \emph{Property $k$} holds if there exist  a symmetric super polynomial $Q_{k}\colon \st_{d}(1)\times\st_{d}(1)\to\st_{d}(s)$, for each $x\in H$ with $\iota(x)\neq \bold{0}$ a homogeneous super polynomial $Q_{k,x}\colon \st_{d}(1)\times\st_{d}(1)\to\st_{d}(s)$ of degree $k$  and a super polynomial $G_{k,x}\colon \st_{d}(1)\to\st_{d}(s-k-1)$    such that 
 \begin{equation}\label{2:gxyz5}
\begin{split}
G_{x}(L_{y})\equiv Q_{k}(L_{x},L_{y})+Q_{k,x}(L_{x},L_{y})+L_{y}^{k+1}G_{k,x}(L_{y}) \mod J^{M}
\end{split}
\end{equation}
  for all $p$-linearly independent $x,y\in H$, where when $k=s+1$, (\ref{2:gxyz5}) is understood as 
  $$G_{x}(L_{y})\equiv Q_{s+1}(L_{x},L_{y})  \mod J^{M}.$$

 Assume that 
 $G_{x}(f)=\sum_{i=0}^{s}C_{i,x}f^{i}$ for some $C_{x,i}\in\st_{d}(s-i)$ for all $x\in H$ with $\iota(x)\neq \bold{0}$.
 We have that Property 0 holds by setting $Q_{0}\equiv 0$, $Q_{0,x}(f,g)\equiv C_{i,0}$, and $G_{0,x}(f):=\sum_{i=1}^{s}C_{i,x}f^{i-1}$.
 Now assume that Property $k$ holds for some $0\leq k\leq s$. We show that Property $k+1$ holds.
 By (\ref{2:gxyz4}) and (\ref{2:gxyz5}), we have that 
  \begin{equation}\label{2:gxyz6}
\begin{split}
Q_{k,x}(L_{x},L_{y})+L_{y}^{k+1}G_{k,x}(L_{y})\equiv  Q_{k,y}(L_{y},L_{x})+L_{x}^{k+1}G_{k,y}(L_{x})\mod J^{M}
\end{split}
\end{equation}
 for all  $p$-linearly independent $x,y\in H$.
 Assume that 
  \begin{equation}\label{2:gxyz5d2}
\begin{split}
Q_{k,x}(f,g)=\sum_{i=0}^{k}A_{x,i}f^{i}g^{k-i}
\end{split}
\end{equation}
   for some $A_{x,i}\in\st_{d}(s-k)$. It follows from (\ref{2:gxyz6}) that 
  \begin{equation}\label{2:gxyz7}
\begin{split}
\sum_{i=0}^{k}(A_{x,i}-A_{y,k-i})L_{x}^{i}L_{y}^{k-i}\equiv  L_{x}^{k+1}G_{k,y}(L_{x})-L_{y}^{k+1}G_{k,x}(L_{y})\mod J^{M}
\end{split}
\end{equation}  
 for all  $p$-linearly independent $x,y\in H$.
By Lemma \ref{2:sbppp}, since $d\geq 7$, we have that 
$A_{x,i}\equiv A_{y,k-i} \mod J^{M}_{\iota(x),\iota(y)}$ for all $0\leq i\leq k$ and $p$-linearly independent $x,y\in H$. By Corollary \ref{2:coco1c}, since  $d\geq s-k+11$,  for all $0\leq i\leq k$, there exists $A_{i}\in\st_{d}(s-k)$ such that 
$$A_{x,i}\equiv A_{i}\equiv A_{x,k-i}\equiv A_{k-i} \mod J^{M}_{\iota(x)}$$
for all $x\in H$ with $\iota(x)\neq \bold{0}$.  
We may assume without loss of generality that $A_{i}=A_{k-i}$.
By Proposition \ref{2:noloop39}, since $d\geq 5$, we may write 
 \begin{equation}\label{2:gxyz5d3}
\begin{split}
A_{x,i}\equiv A_{i}+L_{x}A'_{x,i} \mod J^{M}
\end{split}
\end{equation}
  for some $A'_{x,i}\in\st_{d}(s-k-1)$. 
Note that we may write $G_{k,x}$ as
 \begin{equation}\label{2:gxyz5d}
\begin{split}
G_{k,x}(f)=P_{x}+fG_{k+1,x}(f)
\end{split}
\end{equation}
for some $P_{x}\in\st_{d}(s-k-1)$ and super polynomial $G_{k+1,x}\colon\st_{d}(1)\to\st_{d}(s-k-2)$. By (\ref{2:gxyz5}), (\ref{2:gxyz5d2}), (\ref{2:gxyz5d3}) and (\ref{2:gxyz5d}), we have that 
\begin{equation}\nonumber
\begin{split}
G_{x}(L_{y})\equiv Q_{k+1}(L_{x},L_{y})+Q_{k+1,x}(L_{x},L_{y})+L_{y}^{k+2}G_{k+1,x}(L_{y}) \mod J^{M}
\end{split}
\end{equation}
for all $p$-linearly independent $x,y\in H$,
where
$$Q_{k+1}(f,g):=Q_{k}(f,g)+\sum_{i=0}^{k}A_{i}f^{i}g^{k-i}$$
and
$$Q_{k+1,x}(f,g):=P_{x}g^{k+1}+\sum_{i=0}^{k}A'_{x,i}f^{i+1}g^{k-i}.$$
Since $Q_{k}$ is symmetric and $A_{i}=A_{k-i}$, we have that $Q_{k+1}$ is also symmetric. On the other hand, it is clear that $Q_{k+1,x}$ is a homogeneous super polynomial of degree $k+1$. So Property $k+1$ holds.

By induction, we have that Property $s+1$ holds, meaning that there exist a symmetric super polynomial $Q_{s+1}\colon \st_{d}(1)\times\st_{d}(1)\to\st_{d}(s)$ such that 
$$G_{x}(L_{y})\equiv Q_{s+1}(L_{x},L_{y})  \mod J^{M}$$
 for all  $p$-linearly independent $x,y\in H$.
We are done.
  \end{proof}

 \subsection{Solutions to some special equations}
 
Next, we move forward from equation (\ref{2:sbeq1})  to the following more advanced cocycle equation:
\begin{equation}\label{2:sbeq3}
C(x,y)+C(y,z)\equiv C(x,z) \mod \Z \text{ for all $x,y,z\in G$},
\end{equation}
where $C$ is a map from $G\times G$ to $\R$ and $G$ is an abelian group. Setting $z$ to be an arbitrary element in $G$, it is obvious that the solution to (\ref{2:sbeq3}) is $C(x,y)\equiv F(x)-F(y)\mod \Z$ for some map $F\colon G\to \R$ for all $x,y\in G$.
In the setting of shifted modules, one can study a similar equation:
\begin{equation}\label{2:sbeq4}
C(x,y)+C(y,z)\equiv C(x,z)\mod J^{M}_{\iota(x-y),\iota(y-z)} \text{ for all $x,y,z\in \Vk$},
\end{equation}
 where $C$ is a map from $\Vk$ to $\st_{d}(s)$. Comparing (\ref{2:sbeq3}) with (\ref{2:sbeq4}), it is natural to conjecture that the  solution to (\ref{2:sbeq4}) is $C(x,y)\equiv F(x)-F(y)\mod J^{M}_{\iota(x-y)}$ for some $F\colon \Vk\to\st_{d}(s)$. This is indeed the case at least when $x,y$ are $p$-linearly independent.

 \begin{prop}\label{2:coco02}
	Let $d,K\in\N_{+},s\in\Z$, $p\gg_{d,s} 1$ be a prime dividing $K$,    $M\colon\V\to\F_{p}$ be a non-degenerate quadratic form and $C\colon (\Vk)^{2}\to\st_{d}(s)$ be a map such that (\ref{2:sbeq4}) holds. If $d\gg_{s} 1$,
	then there exists a map $F\colon \Vk\to\st_{d}(s)$ such that $C(x,y)\equiv F(x)-F(y)\mod J^{M}_{\iota(x-y)}$ for all $p$-linearly independent $x,y\in \Vk$. 
	\end{prop}	

Again, instead of proving
Proposition \ref{2:coco02}, we prove the following more general result:

\begin{prop}\label{2:coco4}
	Let $d,K\in\N_{+},s\in\Z$, $\delta>0$, and $p\gg_{\d,d} 1$ be a prime dividing $K$. Let $H$ be a subset of $\Vk$ with $\vert H\vert>\delta K^{d}$, $U$ be a (possibly trivial) subspace of $\V$, and
	$M\colon\V\to\F_{p}$ be a non-degenerate quadratic form.
	 For all  $x,y\in H$ with $(\iota(x),\iota(y),U)$ being a linearly independent tuple,
 let $C(x,y)$ be an element in $\st_{d}(s)$. Suppose that for all   $x,y,z\in H$ with  $(\iota(x),\iota(y),\iota(z),U)$ being a linearly independent tuple, 
we have
	\begin{equation}\label{2:138ppd}
	C(x,y)+C(y,z)\equiv C(x,z) \mod J^{M}_{\iota(x-y),\iota(y-z)}.
	\end{equation}
	If  
$d\geq s+\dim(U)+9$,  then there exist  $\phi(x)\in \st_{d}(s)$ for all $x\in H$ with $(\iota(x),U)$ being a linearly independent tuple
 such that 
	\begin{equation}\nonumber
	C(x,y)\equiv \phi(x)-\phi(y) \mod J^{M}_{\iota(x-y)}
	\end{equation}
	for all $x,y\in H$ with  $(\iota(x),\iota(y),U)$ being a linearly independent tuple.
\end{prop}	
\begin{proof}
By Convention \ref{2:csts}, there is nothing to prove when $s<0$. So we assume that $s\in\N$.
Throughout the proof we assume that $p\gg_{\d,d} 1$.
	Suppose that either $s=0$, or $s\geq 1$
  and we have shown the conclusion holds for $s-1$. We prove that the conclusion holds for $s$.
	
	Applying Proposition \ref{2:noloop39} to $J^{M}_{\iota(x-y),\iota(y-z)}$, we have that if $d\geq 7$, then for all $x,y,b\in H$ with  $(\iota(x),\iota(y),\iota(b),U)$ being a linearly independent tuple, there exists  $P_{x,y,b}\in\st_{d}(s-1)$   such that 
	\begin{equation}\label{2:139pp}
	C(x,y)\equiv(C(x,b)-C(y,b))+L_{b-x}P_{x,y,b} \mod J^{M}_{\iota(x-y)}.
	\end{equation}
	For all $x,y,z,b\in H$ with  $(\iota(x),\iota(y),\iota(z),\iota(b),U)$ being a linearly independent tuple, 
	since $C(x,y)+C(y,z)\equiv C(x,z) \mod J^{M}_{\iota(x-y),\iota(y-z)}$, 
	it follows from (\ref{2:139pp}) that 
	$$L_{b-x}(P_{x,y,b}+P_{y,z,b}-P_{x,z,b})\equiv L_{b-x}P_{x,y,b}+L_{b-y}P_{y,z,b}-L_{b-x}P_{x,z,b}\equiv 0 \mod J^{M}_{\iota(x-y),\iota(y-z)}.$$
	Since $(\iota(x),\iota(y),\iota(z),\iota(b),U)$ is a linearly independent tuple, by Lemma \ref{2:killL}, if $d\geq 9$, 
	then
	\begin{equation}\label{2:138pp}
	P_{x,y,b}+P_{y,z,b}-P_{x,z,b}\equiv 0 \mod J^{M}_{\iota(x-y),\iota(y-z)}.
	\end{equation}
 We arrive at an equation similar to (\ref{2:138ppd}), except that $P_{x,y,b}$ are elements in $\st_{d}(s-1)$ instead of $\st_{d}(s)$. Therefore, we may use the induction hypothesis to obtain a description for $P_{x,y,b}$.

	\textbf{Claim 1.} For any $b\in H$ with $(\iota(b),U)$ being a linearly independent tuple, there exist $\psi_{b}(x)\in \st_{d}(s)$ for all $x\in H$ with $(\iota(x),\iota(b),U)$ being a linearly independent tuple such that 
     $$C(x,y)\equiv\psi_{b}(x)-\psi_{b}(y) \mod J^{M}_{\iota(x-y)}$$
	for all $x,y\in H$ with $(\iota(x),\iota(y),\iota(b),U)$  being a linearly independent tuple.
	
		If $s=0$, then $P_{x,y,b}\equiv 0$ by Convention \ref{2:csts}. Setting $\psi_{b}(x):=C(x,b)$ for all $x\in H$ with $(\iota(x),\iota(b),U)$ being a linearly independent tuple, it follows from (\ref{2:139pp}) that   $C(x,y)=\psi_{b}(x)-\psi_{b}(y) \mod J^{M}_{\iota(x-y)}$ whenever $(\iota(x),\iota(y),\iota(b),U)$  is a linearly independent tuple.
	
	If $s\geq 1$, then
since $d\geq s+\dim(U)+9=(s-1)+(\dim(U)+1)+9$, we may apply
  the induction hypothesis to (\ref{2:138pp}) and  $P_{x,y,b}$ to conclude that  for all $b\in H$ with $(\iota(b),U)$ being a linearly independent tuple, there exist   $\psi_{b}(x)\in\st_{d}(s-1)$ for all $x\in H$ with $(\iota(x),\iota(b),U)$ being a linearly independent tuple such that
	$$P_{x,y,b}\equiv\psi_{b}(x)-\psi_{b}(y) \mod J^{M}_{\iota(x-y)}$$
	for all $x,y\in H$ with $(\iota(x),\iota(y),\iota(b),U)$ being a linearly independent tuple. Setting $\psi_{b}(x):=C(x,b)+L_{b-x}\psi_{b}(x)$, it follows from (\ref{2:139pp}) that
	\begin{equation}\nonumber
	C(x,y)\equiv(C(x,b)-C(y,b))+L_{b-x}P_{x,y,b}\equiv\psi_{b}(x)-\psi_{b}(y) \mod J^{M}_{\iota(x-y)}
	\end{equation}
	for all $x,y\in H$ with $(\iota(x),\iota(y),\iota(b),U)$ being a linearly independent tuple. We have thus proved Claim 1 in both cases.
	
	\
	
	In order to complete the proof of Proposition \ref{2:coco4}, we need to remove the restriction that $x,y,b$ are $p$-linearly independent in the claim. Our strategy is to apply the claim for many linearly independent $b$ and to obtain many different maps $\psi_{b}$. Then by studying the interactions of these $\psi_{b}$, we extend them to a universal map $\phi$ such that  $C(x,y)\equiv\phi(x)-\phi(y) \mod J^{M}_{\iota(x-y)}$ holds for all $x,y\in H$ with $(\iota(x),\iota(y),U)$ being a linearly independent tuple.

	By Claim 1, for all $x,y,b,b'\in H$ with $(\iota(x),\iota(y),\iota(b),U)$ and $(\iota(x),\iota(y),\iota(b'),U)$ being linearly independent tuples, we have that
	\begin{equation}\nonumber
	(\psi_{b}-\psi_{b'})(x)-(\psi_{b}-\psi_{b'})(y)\equiv C(x,y)-C(x,y)=0 \mod J^{M}_{\iota(x-y)}.
	\end{equation}
	By Proposition \ref{2:cocon1}, if $d\geq s+\dim(U)+9$, then for all $b,b'\in H$ with $(\iota(b),U)$  and $(\iota(b'),U)$  being linearly independent tuples, there exists   a super polynomial $G_{b,b'}\colon\st_{d}(1)\to\st_{d}(s)$ such that 
	$$(\psi_{b}-\psi_{b'})(x)\equiv G_{b,b'}(L_{x}) \mod J^{M}$$ 
	for all $x\in H$ with $(\iota(x),\iota(b),U)$ and $(\iota(x),\iota(b'),U)$  being linearly independent tuples.
	So for all $x,b,b',b''\in H$ with $(\iota(x),\iota(b),U)$, $(\iota(x),\iota(b'),U)$ and $(\iota(x),\iota(b''),U)$ being linearly independent tuples, we have that
	$$G_{b,b'}(L_{x})-G_{b,b''}(L_{x})\equiv(\psi_{b}-\psi_{b'})(x)-(\psi_{b}-\psi_{b''})(x)\equiv G_{b'',b'}(L_{x}) \mod J^{M}.$$

	Let
	$B$ be a set of $d-\dim(U)$ vectors in $H$ such that $\iota(B)$ completes $U$ into a basis of $\V$ (such $B$ exists since $\vert H\vert>\d K^{d}$). 
	For $b\in B$, let $U_{b}$ be the set of $x\in H$ with $(\iota(x),\iota(b),U)$ being a linearly independent tuple. Fix some $b_{0}\in B$. For $b\in B$ and $x\in U_{b}$, define $\phi_{b_{0}}(x):=\psi_{b_{0}}(x)$ and $\phi_{b}(x):=\psi_{b}(x)-G_{b,b_{0}}(L_{x})$ if $b\neq b_{0}$.
	
	\textbf{Claim 2.} that $\phi_{b}(x)\equiv\phi_{b'}(x) \mod J^{M}$ for all $b,b'\in B$ and $x\in U_{b}\cap U_{b'}$.
	
	Indeed, we may assume without loss of generality that $b\neq b'$ and $b\neq b_{0}$. If $b'=b_{0}$, then
	$$\phi_{b}(x)-\phi_{b'}(x)\equiv\psi_{b}(x)-G_{b,b_{0}}(L_{x})-\psi_{b_{0}}(x)\equiv 0 \mod J^{M}.$$
	If $b'\neq b_{0}$, then
	\begin{equation}\nonumber
	\begin{split}
	&\quad \phi_{b}(x)-\phi_{b'}(x)\equiv(\psi_{b}(x)-G_{b,b_{0}}(L_{x}))-(\psi_{b'}(x)-G_{b',b_{0}}(L_{x}))
	\\&\equiv G_{b,b'}(L_{x})-G_{b,b_{0}}(L_{x})+G_{b',b_{0}}(L_{x})\equiv 0 \mod J^{M}.
	\end{split}
	\end{equation}
	This proves Claim 2.
	
	\
	
	By Claim 2, there exists  $\phi(x)\in\st_{d}(s)$ for all $x\in\cup_{b\in B}U_{b}$ such that 
	$$\phi(x)\equiv\phi_{b}(x)\mod J^{M}$$
	for all $b\in B$ with $x\in U_{b}$. Note that for all $x,y\in H$ with  $(\iota(x),\iota(y),U)$ being a linearly independent tuple, since $d\geq \dim(U)+3$,  there exists $b\in B$ with $(\iota(x),\iota(y),\iota(b),U)$ being a linearly independent tuple. So $x,y\in U_{b}$ and thus
	$$\phi(x)-\phi(y)\equiv\phi_{b}(x)-\phi_{b}(y)\equiv\psi_{b_{0}}(x)-\psi_{b_{0}}(y)\equiv C(x,y)\mod J^{M}_{\iota(x-y)}$$
	if $b=b_{0}$, and
	$$\phi(x)-\phi(y)\equiv \phi_{b}(x)-\phi_{b}(y)\equiv(\psi_{b}(x)-G_{b,b_{0}}(L_{x}))-(\psi_{b}(y)-G_{b,b_{0}}(L_{y}))\equiv C(x,y)\mod J^{M}_{\iota(x-y)}$$
	if $b\neq b_{0}$, where we used the fact that $G_{b,b_{0}}(L_{x})\equiv G_{b,b_{0}}(L_{y}) \mod J^{M}_{\iota(x-y)}$ since $G_{b,b_{0}}$ is a super polynomial. 
	Finally, since $\cup_{b\in B}U_{b}$ consists of all $x\in H$ with $(\iota(x),U)$ being a linearly independent tuple, we are done.
\end{proof}

  We need to explore the solutions to one more equation before proving Theorem \ref{2:gsol}.

  \begin{prop}\label{2:cocon2}
	Let $d,D,K\in\N_{+}$, $s\in\Z$,  $\delta>0$ and $p\gg_{\d,d,D} 1$  be a prime dividing $K$.
	Let 
	\begin{equation}\nonumber
P=\{\ell_{1}v_{1}+\dots+\ell_{D}v_{D}\colon \ell_{i}\in\mathbb{Z}\cap(-L_{i},L_{i}) \text{ for all } 1\leq i\leq D\}\subseteq \Vk
\end{equation}
	  be a  proper and  homogeneous generalized arithmetic progression of dimension $D$ with $\vert P\vert>\d K^{d}$,  $M\colon\V\to\F_{p}$ be a non-degenerate quadratic form and $F\colon P\to\st_{d}(s)$ be a map such that
\begin{equation}\label{2:ffxyr0}
F(x)+F(y)\equiv F(x+y) \mod J^{M}_{\iota(x)}\cap J^{M}_{\iota(y)}
\end{equation}
	 for all $p$-linearly independent $x,y\in P(\frac{1}{2})$. If  $d\geq s+9$, then there exist a super polynomial $G\colon\st_{d}(1)\to\st_{d}(s-2)$   and a locally linear map $T\colon P(\frac{1}{100})\to \st_{d}(s)$ such 
 that $$F(x)\equiv L_{x}^{2}G(L_{x})+T(x)\mod J^{M}$$ for all $x\in P(\frac{1}{100})$ with $\iota(x)\neq \bold{0}$. 
\end{prop}	
\begin{proof}
By Convention \ref{2:csts}, there is nothing to prove when $s<0$. So we assume that $s\in\N$.
Throughout the proof we assume that $p\gg_{\d,d,D} 1$, and we implicitly use  (\ref{2:crescale}) for the estimate the cardinality of sets of the form $P(c)$.
We start with a reformulation of (\ref{2:ffxyr0}).
Fix any $p$-linearly independent $x,y\in P(\frac{1}{2})$. Since $F(x)+F(y)\equiv F(x+y) \mod J^{M}_{\iota(x)}$ and $d\geq 5$, by Proposition \ref{2:noloop39},  we have that 
$F(x)+F(y)\equiv F(x+y)+L_{x}R'(x,y)\mod J^{M}$ for some $R'(x,y)\in\st_{d}(s-1)$. Since $F(x)+F(y)\equiv F(x+y) \mod J^{M}_{\iota(y)}$, we have that  $L_{x}R'(x,y)\equiv 0\mod J^{M}_{\iota(y)}$. By Lemma \ref{2:killL}, since $d\geq 7$, we have that $R'(x,y)\equiv 0\mod J^{M}_{\iota(y)}$. So by Proposition \ref{2:noloop39}, we may write 
$R'(x,y)\equiv  L_{y}R(x,y)\mod J^{M}$ for some $R(x,y)\in\st_{d}(s-2)$. In conclusion, we have that 
\begin{equation}\label{2:ffxyr}
F(x)+F(y)\equiv F(x+y)+L_{x}L_{y}R(x,y) \mod J^{M}
\end{equation}
for some $R(x,y)\in\st_{d}(s-2)$ for all $p$-linearly independent $x,y\in P(\frac{1}{2})$.

 We first assume that $s\geq 2$.
By symmetry, (\ref{2:ffxyr}) implies that $L_{x}L_{y}(R(x,y)-R(y,x))\in J^{M}$ for all $p$-linearly independent $x,y\in P(\frac{1}{2})$. By Lemma \ref{2:killL}, since $d\geq 5$, this implies that 
\begin{equation}\label{2:gxyz0}
\begin{split}
R(x,y)\equiv R(y,x)\mod J^{M}
\end{split}
\end{equation}
for all $p$-linearly independent $x,y\in P(\frac{1}{2})$.

Fix any  $p$-linearly independent $x,y,z\in P(\frac{1}{4})$. Then (\ref{2:ffxyr}) implies that
\begin{equation}\label{2:gxyz}
\begin{split}
&\quad F(x)+F(y)+F(z)-F(x+y+z)
\\&=(F(y)+F(z)-F(y+z))+(F(x)+F(y+z)-F(x+y+z))
\\&\equiv L_{y}L_{z}R(y,z)+L_{x}L_{y+z}R(x,y+z)\mod J^{M}.
\end{split}
\end{equation}
By symmetry, we also have that 
\begin{equation}\label{2:gxyz2}
\begin{split}
F(x)+F(y)+F(z)-F(x+y+z)\equiv L_{x}L_{z}R(x,z)+L_{y}L_{x+z}R(y,x+z) \mod J^{M}.
\end{split}
\end{equation}
Comparing (\ref{2:gxyz}) and (\ref{2:gxyz2}), we have that  
\begin{equation}\label{2:gxyz3}
\begin{split}
&\quad L_{y}L_{z}(R(y,z)-R(y,x+z))+L_{x}L_{y}(R(x,y+z)-R(y,x+z))
\\&+L_{x}L_{z}(R(x,y+z)-R(x,z))\equiv 0 \mod J^{M}.
\end{split}
\end{equation}
In particular, $L_{y}L_{z}(R(y,z)-R(y,x+z))\in J^{M}_{\iota(x)}$. By Lemma \ref{2:killL}, since $d\geq 7$, we have that $R(y,z)\equiv R(y,x+z)\mod J^{M}_{\iota(x)}$. In other words, for all $p$-linearly independent $x,y,z\in P(\frac{1}{8})$, we have that $R(y,x)\equiv R(y,z)\mod J^{M}_{\iota(x-z)}$. By Proposition \ref{2:cocon1}, since $d\geq s+7$, for all $y\in P(\frac{1}{8})$ with $\iota(y)\neq \bold{0}$, there exists a super polynomial $G_{y}\colon\st_{d}(1)\to\st_{d}(s-2)$ such that 
$R(y,x)\equiv G_{y}(L_{x}) \mod J^{M}$
 for all $x\in P(\frac{1}{8})$ with $x,y$ being $p$-linearly independent.
By (\ref{2:gxyz0}), we have that 
\begin{equation}\nonumber
\begin{split}
G_{x}(L_{y})\equiv G_{y}(L_{x}) \mod J^{M}
\end{split}
\end{equation}
 for all $p$-linearly independent $x,y\in P(\frac{1}{8})$. 
 By Proposition \ref{2:cocoprr},   since $d\geq s+9$,
 there exists a  symmetric super polynomial $Q\colon \st_{d}(1)\times\st_{d}(1)\to\st_{d}(s-2)$ such that
$$R(x,y)=G_{x}(L_{y})\equiv Q(L_{x},L_{y})  \mod J^{M}$$
for all $p$-linearly independent $x,y\in P(\frac{1}{8})$.
We may then write $Q(f,g)=\sum_{k=0}^{s-2}Q_{k}(f,g)$ for some symmetric  homogeneous super polynomials $Q_{k}\colon \st_{d}(1)\times\st_{d}(1)\to\st_{d}(s-2)$ of degree $k$. Write
$$Q_{k}(f,g)=\sum_{i=0}^{k}C_{k,i}f^{i}g^{k-i}$$
for some $C_{k,i}\in\st_{d}(s-2-k)$. For convenience write $C_{k,-1}=0$. Then we may rewrite (\ref{2:gxyz3}) as 
\begin{equation}\label{2:gxyz33}
\begin{split}
\sum_{k=0}^{s-2}\sum_{0\leq a,b,c\leq k, a+b+c=k+2}L_{x}^{a}L_{y}^{b}L_{z}^{c}W(a,b,c)\equiv 0 \mod J^{M},
\end{split}
\end{equation}
where 
\begin{equation}\nonumber
\begin{split}
&\quad W(a,b,c)
:=-C_{k,b-1}\binom{k-(b-1)}{a}\bold{1}_{a\neq 0}+C_{k,a-1}\binom{k-(a-1)}{b-1}
\\&\qquad-C_{k,b-1}\binom{k-(b-1)}{a-1}+C_{k,a-1}\binom{k-(a-1)}{b}\bold{1}_{b\neq 0}.
\end{split}
\end{equation}
Note that $W(a,b,c)=0$ if $a=0$ or $b=0$. If $a,b\neq 0$, then
$$W(a,b,c)=C_{k,a-1}\binom{k+2-a}{b}-C_{k,b-1}\binom{k+2-b}{a}.$$
So by (\ref{2:gxyz33}) and Proposition \ref{2:cocozero}, since
$d\geq \max\{s+1,7\}$, 
 we have that
$$W(a+1,b+1,c)=C_{k,a}\binom{k+1-a}{b+1}-C_{k,b}\binom{k+1-b}{a+1}\equiv 0 \mod J^{M}$$
for all $a,b,c\in\N$ with $a+b+c=k$. From this it is not hard to see that there exists $C_{k}\in \st_{d}(s-2-k)$ such that 
$C_{k,a}\equiv\binom{k+2}{a+1}C_{k} \mod J^{M}$ for all $0\leq a\leq k$. Therefore,
\begin{equation}\nonumber
\begin{split}
&\quad L_{x}L_{y}R(x,y)\equiv L_{x}L_{y}\sum_{k=0}^{s-2}\sum_{i=0}^{k}\binom{k+2}{i+1}C_{k}L_{x}^{i}L_{y}^{k-i}
\\&=\sum_{k=0}^{s-2}\sum_{i=1}^{k+1}\binom{k+2}{i}C_{k}L_{x}^{i}L_{y}^{k+2-i}=\sum_{k=0}^{s-2}C_{k}(L_{x+y}^{k+2}-L_{x}^{k+2}-L_{y}^{k+2}) \mod J^{M}.
\end{split}
\end{equation}
Let $Z$ denote the set of $x\in\Vk$ with $\iota(x)=\bold{0}$.
Then for all $x\in P(\frac{1}{8})\backslash Z$, writing $F'(x):=\sum_{k=0}^{s-2}C_{k}L_{x}^{k+2}$, $G(f):=\sum_{k=0}^{s-2}C_{k}f^{k}$, and $T(x):=F(x)-F'(x)$, we have that $F'(x)=L_{x}^{2}G(L_{x})\in J^{M}$ and
\begin{equation}\label{2:finalle}
T(x)+T(y)\equiv T(x+y) \mod J^{M}
\end{equation}
for all $p$-linearly independent $x,y\in P(\frac{1}{8})$.  

Now if $s<2$, then $R(x,y)\equiv 0$. Setting $G\equiv 0$ and $T:=F$, by (\ref{2:ffxyr}) ,  (\ref{2:finalle})   still holds for all  $p$-linearly independent $x,y\in P(\frac{1}{8})$.

Now for all $x\in P(\frac{1}{32})\cap Z$, define
$T_{y}(x):=T(x+y)-T(y)$ for any $y\in P(\frac{1}{16})\backslash Z$ (which exists since $\vert Z\vert\leq (K/p)^{d}$). We claim that $T:=T_{y} \mod J^{M}$ is independent of the choices of $y$. Indeed, for any $y,y'\in P(\frac{1}{16})\backslash Z$, if $y$ and $y'$ are $p$-linearly independent, then it follows from (\ref{2:finalle}) that $$T(x+y)+T(y')\equiv T(x+y+y')\equiv T(x+y')+T(y) \mod J^{M}$$
and so $T_{y}\equiv T_{y'} \mod J^{M}$. If $y$ and $y'$ are not $p$-linearly independent, then there exists $z\in P(\frac{1}{16})\backslash Z$ with $(y,z)$ and $(y',z)$ being $p$-linearly independent pairs. Then 
$$T_{y}\equiv T_{z}\equiv T_{y'} \mod J^{M}.$$
So $T \mod J^{M}$ is well defined.

 \textbf{Claim.} We have that (\ref{2:finalle}) holds for all (not necessarily $p$-linearly independent) $x,y\in P(\frac{1}{100})$.
 
We first show that (\ref{2:finalle}) holds for all $x,y\in P(\frac{1}{100})\backslash Z$ with $x+y\notin Z$.
In fact, since $d\geq 2$, there exists $z\in P(\frac{1}{16})$ such that $(x,z),(x+y,z)$ and $(y,x+z)$ are $p$-linearly independent pairs. So it follows from (\ref{2:finalle}) that 
$$(T(x)+T(z))+(T(y)+T(x+z))\equiv T(x+z)+T(x+y+z)\equiv T(x+z)+(T(x+y)+T(z)) \mod J^{M},$$
which implies that (\ref{2:finalle}) holds.

We next show that (\ref{2:finalle}) holds for all $x,y\in P(\frac{1}{100})\backslash Z$ with $x+y\in Z$.
Pick any $z\in P(\frac{1}{100})$ such that $(x,z)$ and $(y,x+z)$ are $p$-linearly independent pairs. So it follows from (\ref{2:finalle}) that 
$$T(x)+T(y)\equiv  T(x+z)+T(y)-T(z)\equiv  T(x+y+z)-T(z)\equiv  T(x+y)\mod J^{M}.$$

We finally show that (\ref{2:finalle}) holds for all $x,y\in P(\frac{1}{100})$ with at least one of $x,y$ in $Z$.
If both $x$ and $y$ belong to $Z$, then pick any $p$-linearly independent $z,w\in P(\frac{1}{100})$. So it follows from (\ref{2:finalle}) and the previous case that
$$T(x)+T(y)\equiv T(x+z)-T(z)+T(y+w)-T(w)\equiv T(x+z-w)+T(y+w-z)\equiv T(x+y)\mod J^{M}.$$
If $x\in Z$ and $y\notin Z$, then pick any $z\in P(\frac{1}{100})$ with $(x+z,y), (z,x+y)$ being $p$-linearly independent pairs. It follows from (\ref{2:finalle})  that
$$T(x)+T(y)\equiv T(x+z)-T(z)+T(y)\equiv T(x+y+z)-T(z)\equiv T(x+y)\mod J^{M}.$$
This completes the proof of the claim.

\

Since $\st_{d}(s)/(J^{M}\cap \st_{d}(s))$ is an $\F_{p}$-vector space, it is not hard to see from the claim that $T \mod J^{M}$ is determined by its values on the generators of $P$. So there exist $f_{1},\dots,f_{D}\in \st_{d}(s)$ such that 
$$T(\ell_{1}v_{1}+\dots+\ell_{D}v_{D})\equiv\ell_{1}f_{1}+\dots+\ell_{D}f_{D}\mod J^{M}$$
for all $\ell_{i}\in\mathbb{Z}\cap(-L_{i}/100,L_{i}/100)$.
So $T$ is equivalent mod $J^{M}$ to a genial locally linear map on $P(\frac{1}{100})$ and we are done.
\end{proof}

\subsection{Proof of Theorem \ref{2:gsol}}\label{2:sp4.3}

We are now ready to prove Theorem \ref{2:gsol}.
Throughout the proof we assume that $p\gg_{\d,d,D} 1$, and we implicitly use  (\ref{2:crescale}) for the estimate the cardinality of sets of the form $P(c)$.
It is not hard to see that 
\begin{equation}\label{2:inpp}
P(c)+P(c')\subseteq P(c+c') \text{  for all $0<c,c'<1$  with $c+c'\leq 1$.}
\end{equation}

In order to prove Theorem \ref{2:gsol}, we first solve $\xi$ for $h\in P(c)$ with $\iota(h)\neq \bold{0}$, and then remove the additional condition $\iota(h)\neq \bold{0}$. To this end, we start by showing the following:

\begin{prop}\label{2:gsol00} 
	If $\xi\colon P\to \st_{d}(s)$ is an \emph{almost Freiman $M$-homomorphism of order 4} in the sense that 
	\begin{equation}\nonumber
\xi(h_{1})+\xi(h_{2})\equiv\xi(h_{3})+\xi(h_{4}) \mod J^{M}_{\iota(h_{1}),\iota(h_{2}),\iota(h_{3})}
\end{equation}
 for all $h_{1},h_{2},h_{3},h_{4}\in P$ with $h_{1}+h_{2}=h_{3}+h_{4}$ and with $h_{1},h_{2},h_{3}$ being $p$-linearly independent,  there exist a  constant $0<c\leq 1$ 	depending only on $s$
	and a  locally linear map $T\colon P(c)\to \st_{d}(s)$ such that 
	$\xi(h)\equiv T(h) \mod J^{M}_{\iota(h)}$ for all $h\in P(c)$ with $\iota(h)\neq \bold{0}$.
\end{prop}

\textbf{Step 1: a preliminary description for $\xi(x)-\xi(y)$.} 
We first show that $\xi(x)-\xi(y) \mod J^{M}_{\iota(x),\iota(y)}$ depends only on the difference $x-y$.
For $t\in P(\frac{1}{2})$, denote $\xi_{t}(x):=\xi(x+t)-\xi(x)$. By assumption, we have that $$\xi_{t}(x)=\xi_{t}(y) \mod J^{M}_{\iota(x),\iota(y),\iota(t)}$$ for all $p$-linearly independent $x,y,t\in P(\frac{1}{2})$ (which ensures that $x+t,y+t\in P$ by (\ref{2:inpp})). 
By Proposition \ref{2:coco1}, since $d\geq s+10$, there exists  $G_{t}\in\st_{d}(s)$ such that 
$$\xi_{t}(x)\equiv G_{t} \mod J^{M}_{\iota(x),\iota(t)}$$
for all  $p$-linearly independent $x,t\in P(O(1)^{-1})$. 
In other words,
$$\xi(x)-\xi(y)\equiv G_{x-y} \mod J^{M}_{\iota(x),\iota(y)}$$
for all  $p$-linearly independent $x,y\in P(O(1)^{-1})$ (which enforces $x-y\in P(O(1)^{-1})$ by (\ref{2:inpp})).

Applying Proposition \ref{2:noloop39} to $J^{M}_{\iota(x),\iota(y)}$,
 since
$d\geq 7$, then we may write
\begin{equation}\label{2:hhll101}
\xi(x)-\xi(y)\equiv L_{x}A_{x,y}-L_{y}B_{x,y}+G_{x-y} \mod J^{M}
\end{equation}
 for some   $A_{x,y},B_{x,y}\in\st_{d}(s-1)$ for all $p$-linearly independent $x,y\in P(O(1)^{-1})$.

\textbf{Step 2: solutions to $A_{x,y}$ and $B_{x,y}$.}
Our next step is to use a cocycle approach to solve $A_{x,y}$ and $B_{x,y}$ from (\ref{2:hhll101}).
Using (\ref{2:hhll101}) and the identity $\xi(x)-\xi(z)=(\xi(x)-\xi(y))+(\xi(y)-\xi(z))$, we have that 
\begin{equation}\label{2:hhll10}
L_{x}(A_{x,y}-A_{x,z})+L_{y}(A_{y,z}-B_{x,y})+L_{z}(B_{x,z}-B_{y,z})+(G_{x-y}+G_{y-z}-G_{x-z})\equiv 0 \mod J^{M}
\end{equation}
for all $p$-linearly independent $x,y,z\in P(O(1)^{-1})$. For convenience denote $$\Delta_{t}A_{x,y}:=A_{x+t,y+t}-A_{x,y} \text{ and } \Delta_{t}B_{x,y}:=B_{x+t,y+t}-B_{x,y}.$$ Replacing $x,y,z$ by $x+t,y+t,z+t$ in (\ref{2:hhll10}) to cancel the term $G_{x-y}+G_{y-z}-G_{x-z}$, we deduce that
\begin{equation}\label{2:hhll11}
\begin{split}
&\quad L_{x}(\Delta_{t}A_{x,y}-\Delta_{t}A_{x,z})+L_{y}(\Delta_{t}A_{y,z}-\Delta_{t}B_{x,y})+L_{z}(\Delta_{t}B_{x,z}-\Delta_{t}B_{y,z})
\\&\equiv-L_{t}((A_{x+t,y+t}+A_{y+t,z+t}-A_{x+t,z+t})-(B_{x+t,y+t}+B_{y+t,z+t}-B_{x+t,z+t})) \mod J^{M}
\end{split}
\end{equation}
for all $p$-linearly independent $x,y,z,t\in P(O(1)^{-1})$. In particular, $L_{x}(\Delta_{t}A_{x,y}-\Delta_{t}A_{x,z})\in J^{M}_{\iota(y),\iota(z),\iota(t)}$. Since $x,y,z,t$ are $p$-linearly independent  and $d\geq 11$, 
by Lemma \ref{2:killL},  
we have that 
$$\Delta_{t}A_{x,y}\equiv\Delta_{t}A_{x,z}\mod J^{M}_{\iota(y),\iota(z),\iota(t)}.$$
By Proposition \ref{2:coco1}, since $d\geq s+10$, there exists  $C_{x,t}\in\st_{d}(s-1)$ such that 
\begin{equation}\label{2:136bb}
\Delta_{t}A_{x,y}\equiv C_{x,t}\mod J^{M}_{\iota(y),\iota(t)}
\end{equation}
for all $p$-linearly independent $x,y,t\in P(O(1)^{-1})$.

We solve $C_{x,t}$ before solving $A_{x,y}$.
Since $\Delta_{t}A_{x,y}+\Delta_{s}A_{x+t,y+t}=\Delta_{t+s}A_{x,y}$ for all $x,y,t,s\in P(O(1)^{-1})$, writing $C'_{x,y}:=C_{x,y-x}$, $a=x, b=x+t, c=x+s+t$, it follows from (\ref{2:136bb}) that
\begin{equation}\nonumber
C'_{a,b}+C'_{b,c}\equiv C'_{a,c} \mod J^{M}_{\iota(a-b),\iota(b-c),\iota(y)}
\end{equation}
for all $p$-linearly independent  $a,b,c,y\in P(O(1)^{-1})$.  By Proposition \ref{2:grmq} (setting $k=1,\dim(U_{1})=1,\dim(V)=2$ and $m=1$), since
$d\geq\max\{s+2,11\}$, we have
\begin{equation}\nonumber
C'_{a,b}+C'_{b,c}\equiv C'_{a,c} \mod J^{M}_{\iota(a-b),\iota(b-v)}
\end{equation}
for all $p$-linearly independent $a,b,c\in P(O(1)^{-1})$.
By Proposition \ref{2:coco4}, since $d\geq s+8$,  there exist $\phi(x)\in\st_{d}(s-1)$ for all $x\in P(O(1)^{-1})$ such that
$$C'_{x,y}\equiv \phi(y)-\phi(x) \mod J^{M}_{\iota(x-y)}$$
for all $p$-linearly independent $x,y\in P(O(1)^{-1})$. 

We now use the solution of $C'_{x,y}$ to solve $A_{x,y}$. By (\ref{2:136bb}), 
  $$A_{x+t,y+t}-A_{x,y}=\Delta_{t}A_{x,y}\equiv C_{x,t}=C'_{x,x+t}\equiv \phi(x+t)-\phi(x) \mod J^{M}_{\iota(y),\iota(t)}$$ for all $p$-linearly independent $x,y,t\in P(O(1)^{-1})$. 
  Writing $y'=y+t$, $z=x-y$ and $A'_{z,y}:=A_{y+z,y}-\phi(y+z)$, we have that
$$A'_{z,y}\equiv A'_{z,y'}\mod  J^{M}_{\iota(y),\iota(y')}$$
for all $p$-linearly independent  $y,y',z\in P(O(1)^{-1})$.
By Proposition \ref{2:coco1}, since $d\geq s+8$,  there exist   $f(z)\in\st_{d}(s-1)$ for all $z\in P(O(1)^{-1})$ such that
$$A'_{z,y}\equiv f(z)\mod J^{M}_{\iota(y)}$$
for all $p$-linearly independent $y,z\in P(O(1)^{-1})$. So
\begin{equation}\label{2:hhlla1}
A_{x,y}\equiv\phi(x)+f(x-y) \mod J^{M}_{\iota(y)}
\end{equation}
for all $p$-linearly independent $x,y\in P(O(1)^{-1})$. 

We may solve $B_{x,y}$ in a similar way and conclude that there exist   $\psi(y)\in\st_{d}(s-1)$ for all $y\in P(O(1)^{-1})$ and $g(z)\in\st_{d}(s-1)$ for all $z\in P(O(1)^{-1})$ such that
\begin{equation}\label{2:hhlla2}
B_{x,y}\equiv\psi(y)+g(x-y) \mod J^{M}_{\iota(x)}
\end{equation}
for all $p$-linearly independent $x,y\in P(O(1)^{-1})$.  
  
  \textbf{Step 3: a description for $\phi-\psi$.}
We next determine the connection between $\phi$ and $\psi$.
Applying Lemma \ref{2:killL} to the term $L_{y}(\Delta_{t}A_{y,z}-\Delta_{t}B_{x,y})$ in (\ref{2:hhll11}), we have that
$$\Delta_{t}A_{y,z}\equiv\Delta_{t}B_{x,y} \mod J^{M}_{\iota(x),\iota(z),\iota(t)}$$
for all $p$-linearly independent $x,y,z,t\in P(O(1)^{-1})$. This implies that 
$$\phi(y)-\psi(y)\equiv\phi(y+t)-\psi(y+t) \mod J^{M}_{\iota(x),\iota(z),\iota(t)}$$
for all $p$-linearly independent $x,y,z,t\in P(O(1)^{-1})$, which implies that 
$$\phi(y)-\psi(y)\equiv\phi(y')-\psi(y') \mod J^{M}_{\iota(x),\iota(z),\iota(y-y')}$$
for all $p$-linearly independent $x,y,y',z\in P(O(1)^{-1})$. By
 Proposition \ref{2:grmq} (setting $k=1,\dim(U_{1})=1,\dim(V)=1$ and $m=2$), 
since $d\geq \max\{s+2,11\}$, we have that
$$\phi(y)-\psi(y)\equiv\phi(y')-\psi(y') \mod J^{M}_{\iota(y-y')}$$
for all $p$-linearly independent $y,y'\in P(O(1)^{-1})$. 

 By Proposition \ref{2:cocon1},  since $d\geq s+7$,  there exists a super polynomial $Q\colon\st_{d}(1)\to\st_{d}(s-1)$ such that $\phi(y)-\psi(y)\equiv Q(L_{y})\mod J^{M}$ for all $y\in P(O(1)^{-1})$ with $\iota(y)\neq \bold{0}$. 
  
  We now substitute the expression of $\phi(y)-\psi(y)$ to the previous equations.
  Note that 
  $$\psi(y)+g(x-y)-\phi(y)=Q(L_{y})+g(x-y)\equiv Q(L_{x-y})+g(x-y) \mod J^{M}_{\iota(x)}.$$
  So
replacing $g(x-y)$ by $Q(L_{x-y})+g(x-y)$ if necessary, we may assume without loss of generality that $Q\equiv 0$ and $\phi(x)=\psi(x)$ for all $x\in P(O(1)^{-1})$. By Proposition \ref{2:noloop39}, (\ref{2:hhll101}), (\ref{2:hhlla1}) and (\ref{2:hhlla2}),
\begin{equation}\nonumber
\begin{split}
&\quad \xi(x)-\xi(y)\equiv L_{x}A_{x,y}-L_{y}B_{x,y}+G_{x-y}
\\&\equiv L_{x}\phi(x)-L_{y}\phi(y)+L_{x}f(x-y)-L_{y}g(x-y)+G_{x-y}+L_{x}L_{y}R_{x,y} \mod J^{M}
\end{split}
\end{equation}
for some $R_{x,y}\in\st_{d}(s-2)$
for all $p$-linearly independent $x,y\in P(O(1)^{-1})$. Writing 
\begin{equation}\label{2:xi1st}
\begin{split}
\xi'(x):=\xi(x)-L_{x}\phi(x),
\end{split}
\end{equation}
  $F=f-g$ and $G'_{z}:=G_{z}+L_{z}f(z)$, we have that 
\begin{equation}\label{2:hhlla3}
\begin{split}
\xi'(x)-\xi'(y)\equiv L_{y}F(x-y)+G'_{x-y}+L_{x}L_{y}R_{x,y} \mod J^{M}
\end{split}
\end{equation}
for all $p$-linearly independent $x,y\in P(O(1)^{-1})$. 

\textbf{Step 4: the elimination of $G'_{x-y}$.}
Next we deal with the term  $G'_{x-y}$ by using Proposition \ref{2:cocon2}. Note that (\ref{2:hhlla3}) implies that
$$\xi'(x)-\xi'(y)\equiv G'_{x-y} \mod J^{M}_{\iota(y)}$$
for all $p$-linearly independent $x,y\in P(O(1)^{-1})$. 
Using the identity $\xi'(x)-\xi'(z)=(\xi'(x)-\xi'(y))-(\xi'(y)-\xi'(z))$, we have that 
$$G'_{x-y}+G'_{y-z}\equiv G'_{x-z} \mod J^{M}_{\iota(y),\iota(z)}$$
for all $p$-linearly independent $x,y,z\in P(O(1)^{-1})$.  Writing $a=x-y$  and $b=y-z$, we have that
$$G'_{a}+G'_{b}\equiv G'_{a+b} \mod J^{M}_{\iota(b),\iota(z)}$$
for all $p$-linearly independent $z,a,b\in P(O(1)^{-1})$.   By Proposition \ref{2:grmq} (setting $k=m=\dim(U_{1})=1, \dim(V)=1$), since $d\geq \max\{s+2,9\}$, we have that
$$G'_{a}+G'_{b}\equiv G'_{a+b} \mod J^{M}_{\iota(b)}.$$
Similarly, $G'_{a}+G'_{b}\equiv G'_{a+b} \mod J^{M}_{\iota(a)}$ and so 
$$G'_{a}+G'_{b}\equiv G'_{a+b} \mod J^{M}_{\iota(a)}\cap J^{M}_{\iota(b)}$$
for all $p$-linearly independent $a,b\in P(O(1)^{-1})$. By Proposition \ref{2:cocon2}, since $d\geq s+9$, there exist  a super polynomial $R\colon\st_{d}(1)\to \st_{d}(s)$ with $R(0)=0$ and a locally linear map $T\colon P(O(1)^{-1})\to \st_{d}(s)$ such that 
\begin{equation}\label{2:hhlla44h}
\begin{split}
G'_{a}\equiv R(L_{a})+T(a) \mod J^{M}
\end{split}
\end{equation}
for all
$a\in P(O(1)^{-1})\backslash Z$, where $Z$ is the set of $a\in\Vk$ such that $\iota(a)=\bold{0}$. 
Since $R$ is a super polynomial with $R(0)=0$,
it is not hard to see that
 \begin{equation}\label{2:hhlla44}
\begin{split}
R(L_{a})+R(L_{b})\equiv R(L_{a+b})+L_{a}L_{b}R'_{a,b} \mod J^{M}
\end{split}
\end{equation}
 for some $R'_{a,b}\in\st_{d}(s-2)$ for all $a,b\in\Vk$.

 We now substitute the expression of $G'_{x-y}$ to the previous equations. 
By (\ref{2:hhlla3}), (\ref{2:hhlla44h}) and (\ref{2:hhlla44}), writing 
\begin{equation}\label{2:xi2nd}
\begin{split}
\xi''(x):=\xi'(x)-T(x)-R(L_{x}),
\end{split}
\end{equation}
  we have that 
\begin{equation}\label{2:fdsfjwoeif}
\begin{split}
\xi''(x)-\xi''(y)\equiv L_{y}F(x-y)+L_{x}^{2}R_{1,x,y}+L_{x}L_{y}R_{2,x,y}+L^{2}_{y}R_{3,x,y} \mod J^{M}
\end{split}
\end{equation}
for some $R_{1,x,y},R_{2,x,y},R_{3,x,y}\in\st_{d}(s-2)$
for all $p$-linearly independent $x,y\in P(O(1)^{-1})$. 

\textbf{Step 5: completion of the proof.}
It follows from (\ref{2:fdsfjwoeif}) that
  $\xi''(x)\equiv \xi''(y) \mod J^{M}_{\iota(x),\iota(y)}$ for all $p$-linearly independent $x,y\in P(O(1)^{-1})$. 
Since $d\geq s+8$, by Proposition \ref{2:coco1}, there exists $C\in \st_{d}(s)$ such that 
\begin{equation}\label{2:xi4th}
\begin{split}
\xi''(x)\equiv C \mod J^{M}
\end{split}
\end{equation}
 for all $x\in P(O(1)^{-1})$ with $\iota(x)\neq \bold{0}$. 
Combining (\ref{2:xi1st}), (\ref{2:xi2nd}) and (\ref{2:xi4th}), we have that
$$\xi(x)\equiv T(x)+C+R(L_{x})+L_{x}\phi(x) \mod J^{M}$$
 and thus 
 \begin{equation}\label{2:xi4th2}
\begin{split}
\xi(x)\equiv T(x)+C \mod J^{M}_{\iota(x)}
\end{split}
\end{equation}
for all   $x\in P(O(1)^{-1})$ with $\iota(x)\neq \bold{0}$ (since $R(0)=0$).
This completes the proof of Proposition \ref{2:gsol00}.

We now continue to prove Theorem \ref{2:gsol} based on the construction given by (\ref{2:xi4th2}).
Note that for all $x\in P(O(1)^{-1})$ with $\iota(x)=\bold{0}$, and for any $p$-linearly independent $y,z\in P(O(1)^{-1})$, we have
\begin{equation}\nonumber
\begin{split}
&\quad \xi(x)+T(y+z)+C\equiv\xi(x)+\xi(y+z)
\\&\equiv \xi(x+y)+\xi(z)\equiv T(x+y)+T(z)+2C \mod J^{M}_{\iota(y),\iota(z)},
\end{split}
\end{equation}
which implies that 
$$\xi(x)\equiv T(x)+C \mod J^{M}_{\iota(y),\iota(z)}$$
since $T$ is locally linear.  
%
By Proposition \ref{2:grmq}, since
$d\geq \max\{s+2,9\}$,
we have 
$$\xi(x)\equiv T(x)+C \mod J^{M}$$
and thus $\xi(x)\equiv T(x)+C \mod J^{M}_{\iota(x)}$ for all $x\in P(O(1)^{-1})$. 
This completes the proof of Theorem \ref{2:gsol}.

 \subsection{Lifting Theorem \ref{2:gsol} by the $p$-expansion trick}\label{2:s446}
 
In this section, we use Theorem \ref{2:gsol} and  the $p$-expansion trick  to deduce Theorem \ref{2:gsol2}.



         If $r=0$, then $\xi(P)\subseteq \st_{\Z,d}(s)\subseteq J^{M}$ and so we can simply take $T\equiv 0$. Now suppose that the conclusion of Theorem \ref{2:gsol2} holds  if $\xi(P)\subseteq  \st_{\Z/p^{r},d}(s)$ for some $r\in\N$, we show that the same conclusion holds if $\xi(P)\subseteq  \st_{\Z/p^{r+1},d}(s)$.  
         

          Note that if $\xi$ is a Freiman $M$-homomorphism of order 4, then so is $p\xi$.
          So by induction hypothesis, there exists a locally linear map $T\colon P(c)\to \st_{\Z/p^{r},d}(s)$ such that $$p\xi(h)\equiv T(h) \mod J^{M}_{\iota(h)}$$ for all $h\in P(c)$.
          Let $A$ be the matrix associated with $M$,
          and  $\tilde{M}\colon\Z^{d}\to\Z/p$ be the quadratic form given by $\tilde{M}(n):=\frac{1}{p}((n\tau(A))\cdot n)$. Denote $\tilde{L}_{h}(n):=\frac{1}{p}(n\tau(A))\cdot h$ for all $n,h\in\Z^{d}$.
Since $d\geq 5$, by Proposition \ref{2:basicpp12}, for all $h\in P(c)$, we may write 
 $$Q_{0}(p\xi(h)-T(h))=\sum_{i=(i_{0},i_{1})\in\N^{2}\colon 2i_{0}+i_{1}\leq s, i_{0}+i_{1}\leq r}R_{h,i}\tilde{M}^{i_{0}}\tilde{L}_{h}^{i_{1}}$$
	for some  $Q_{0}\in\N_{+}, Q_{0}\leq O_{d}(1)$ and some $R_{h,i}\in\st_{\Z,d}(s-(2i_{0}+i_{1}))$, where we slight abuse the notation to treat $h$ as an element in $\Z^{d}$ in the expression $\tilde{L}_{h}$.
	So
	\begin{equation}\label{2:pwouf}
	   \begin{split}
	   Q_{0}\xi(h)=\frac{T(h)}{p}+\sum_{i=(i_{0},i_{1})\in\N^{2}\colon 2i_{0}+i_{1}\leq s, i_{0}+i_{1}\leq r}R'_{h,i}\tilde{M}^{i_{0}}\tilde{L}_{h}^{i_{1}},
	   \end{split}
	\end{equation}	
	where $R'_{h,i}:=\frac{R_{h,i}}{p}$. Denote 
	$$V(h_{1},\dots,h_{k}):=\sp_{\F_{p}}\{\iota(h_{1})\dots,\iota(h_{k})\}^{\pp}$$
	for all $h_{1},\dots,h_{k}\in\Vk$.
	Fix any $h_{1},\dots,h_{4}\in P(c)$ with $h_{1}+h_{2}=h_{3}+h_{4}$. Since $\xi$ is a Freiman $M$-homomorphism of order 4, for all $n\in V_{p}(\tilde{M})\cap \iota^{-1}(V(h_{1},h_{2},h_{3}))$  and $m\in\Z^{d}$, since $n+pm\in V_{p}(\tilde{M})\cap \iota^{-1}(V(h_{1},h_{2},h_{3}))$, we have that 
	\begin{equation}\label{2:pwouf2}
		   \begin{split}
	   &\quad 0\equiv Q_{0}(\xi(h_{1})+\xi(h_{2})-\xi(h_{3})-\xi(h_{4}))(n+pm)
	   \\&\equiv\sum_{j=1}^{4}c_{j}\sum_{i=(i_{0},i_{1})\in\N^{2}\colon 2i_{0}+i_{1}\leq s}(R'_{h_{j},i}\tilde{M}^{i_{0}}\tilde{L}_{h_{j}}^{i_{1}})(n+pm) \mod \Z,
	   \end{split}
	\end{equation}
	where $c_{1}=c_{2}:=1$ and $c_{3}=c_{4}:=-1$.  Note that
	$$\tilde{M}(n+pm)\equiv \tilde{M}(n)+2(n\tau(A))\cdot m \mod p\Z$$
	and $$\tilde{L}_{h_{j}}(n+pm)\equiv \tilde{L}_{h_{j}}(n)+(h_{j}\tau(A))\cdot m \mod p\Z$$
	for all $1\leq j\leq 4$.
	Since  $n\in V_{p}(\tilde{M})\cap \iota^{-1}(V(h_{1},h_{2},h_{3}))=V_{p}(\tilde{M})\cap \iota^{-1}(V(h_{1},h_{2},h_{3},h_{4}))$  and since $R_{h_{j},i}$ has integer coefficients, it follows from (\ref{2:pwouf2}) that 
		\begin{equation}\label{2:pwouf3}
		   \begin{split}
	    \sum_{j=1}^{4}c_{j}\sum_{i=(i_{0},i_{1})\in\N^{2}\colon 2i_{0}+i_{1}\leq s}R'_{h_{j},i}(n)(\tilde{M}(n)+2(n\tau(A))\cdot m)^{i_{0}}(\tilde{L}_{h_{j}}(n)+(h_{j}\tau(A))\cdot m)^{i_{1}} \equiv 0\mod \Z
	   \end{split}
	\end{equation}
	for all $m\in\Z^{d}$.
	Since $\tilde{M}$ is non-degenerate,
	if $n,h_{1},h_{2},h_{3}$ are $p$-linearly independent, then the map 
	$$m\mapsto (2(n\tau(A))\cdot m,(h_{1}\tau(A))\cdot m,\dots,(h_{4}\tau(A))\cdot m) \mod p\Z^{5}$$
	 is an injection from $[p]^{d}$ to $\{(x_{0},\dots,x_{4})\in \F_{p}^{5}\colon x_{1}+x_{2}=x_{3}+x_{4}\}$.
	Since the left hand side of (\ref{2:pwouf3}) is a $\Z/p$-coefficient polynomial in $m$, we have that 
	\begin{equation}\nonumber
		   \begin{split}
	    \sum_{j=1}^{4}c_{j}\sum_{i=(i_{0},i_{1})\in\N^{2}\colon 2i_{0}+i_{1}\leq s}R'_{h_{j},i}(n)x_{0}^{i_{0}}x_{j}^{i_{1}} \equiv 0\mod \Z
	   \end{split}
	\end{equation}
	for all $x_{0},\dots,x_{4}\in \Z$ with $x_{1}+x_{2}\equiv x_{3}+x_{4} \mod p\Z$. 
	
	We claim that
	for all   $n\in V_{p}(\tilde{M})\cap \iota^{-1}(V(h_{1},h_{2},h_{3}))$ with $n,h_{1},h_{2},h_{3}$ being $p$-linearly independent,  we have that 
		\begin{enumerate}[(i)]
	    \item $R'_{h_{1},(i_{0},i_{1})}(n), \dots, R'_{h_{4},(i_{0},i_{1})}(n)\in \Z$ for all $i_{0}\in\N$ and $i_{1}\geq 2$;
	    \item $R'_{h_{1},(i_{0},1)}(n)\equiv \dots\equiv R'_{h_{4},(i_{0},1)}(n) \mod \Z$ for all $i_{0}\in\N$;
	    \item $R'_{h_{1},(i_{0},0)}(n)+R'_{h_{2},(i_{0},0)}(n)\equiv R'_{h_{3},(i_{0},0)}(n)+R'_{h_{4},(i_{0},0)}(n) \mod \Z$  for all $i_{0}\in\N$.
	\end{enumerate}

	It is convenient to prove the claim in the $\V$ setting. It suffices to show that for any $F\in\poly(\F_{p}^{5}\to\F_{p})$ of degree at most $s$ given by 
	$$F(x_{0},\dots,x_{4}):=\sum_{j=1}^{4} \sum_{i=(i_{0},i_{1})\in\N^{2}\colon 2i_{0}+i_{1}\leq s}c_{i,j}x_{0}^{i_{0}}x_{j}^{i_{1}}$$
	for some $c_{i,j}\in\F_{p}$, if $F(x_{0},\dots,x_{4})=0$ whenever $x_{1}+x_{2}=x_{3}+x_{4}$, we have that 
 \begin{itemize}
	    \item $c_{i,1}=\dots=c_{i,4}=0$ for all $i_{0}\in\N$ and $i_{1}\geq 2$;
	    \item $c_{i,1}=c_{i,2}=-c_{i,3}=-c_{i,4}$ for all $i_{0}\in\N$;
	    \item $c_{i,1}+\dots+c_{i,4}=0.$  for all $i_{0}\in\N$.
	\end{itemize}
	
	Indeed, separating the variable $x_{0}$, we see that for any $i_{0}\in\N$, we have that 
	$$F_{i_{0}}(x_{1},\dots,x_{4}):=\sum_{j=1}^{4} \sum_{i_{1}\in\N^{2}\colon 2i_{0}+i_{1}\leq s}c_{i,j}x_{j}^{i_{1}}$$
	is equal to zero whenever $x_{1}+x_{2}=x_{3}+x_{4}$. So $F_{i_{0}}$ is divisible by $x_{1}+x_{2}-x_{3}-x_{4}$. Comparing the terms of different degrees, we have that 
	$$F_{i_{0},i_{1}}(x_{1},\dots,x_{4}):=\sum_{j=1}^{4} c_{i,j}x_{j}^{i_{1}}$$
	is divisible by $x_{1}+x_{2}-x_{3}-x_{4}$ for all $i_{1}\in\N$. This is equivalent of saying that 
	$$(x_{1},x_{2},x_{3})\mapsto c_{i,1}x_{1}^{i_{1}}+c_{i,2}x_{2}^{i_{1}}+c_{i,3}x_{3}^{i_{1}}+c_{i,4}(x_{1}+x_{2}-x_{3})^{i_{1}}$$
	is the trivial polynomial, which is already impossible if $i_{1}\geq 2$ unless $c_{i,1}=\dots=c_{i,4}=0$. If $i_{1}=1$, then this is only possible if $c_{i,1}=c_{i,2}=-c_{i,3}=-c_{i,4}$. If $i_{1}=0$, then this is only possible if $c_{i,1}+\dots+c_{i,4}=0.$ This completes the proof of the claim.
	
	\

	Note that if $h_{1},h_{2},h_{3}$ are  $p$-linearly independent, then the number of $n\in[p]^{d}$ with $n,h_{1},h_{2},h_{3}$ being $p$-linearly dependent is $p^{3}$. On the other hand, the cardinality of the set $V_{p}(\tilde{M})\cap \iota^{-1}(V(h_{1},h_{2},h_{3}))\cap [p]^{d}$ is $O(p^{d-4})$ by Lemma \ref{2:counting02} (translated into the $\Z$-setting). Since $d\geq 8$, it follows from Proposition \ref{2:noloop3} that conditions (i)--(iii) hold for all $n\in V_{p}(\tilde{M})\cap \iota^{-1}(V(h_{1},h_{2},h_{3}))$ for all  $p$-linearly independent $h_{1},h_{2},h_{3}\in P(c)$. 
	
	Note that if $i_{1}\geq 2$, then Condition (i) implies that 
		$R'_{h_{1},(i_{0},i_{1})}\in J^{M}_{\iota(h_{1}),\iota(h_{2}),\iota(h_{3})}$ for all $p$-linearly independent $h_{1},h_{2},h_{3}\in P(c)$. Since $d\geq \max\{s+3,11\}$,  it follows from Proposition \ref{2:grm} (setting $k=0, m=2, \dim(V)=1$) that
	$R'_{h_{1},(i_{0},i_{1})}\in J^{M}_{\iota(h_{1})}$ 
	for all   $h_{1}\in P(c)$  with $\iota(h_{1})\neq \bold{0}$. 	
	Condition (ii) implies that 
	$$R'_{h_{1},(i_{0},1)}\equiv R'_{h_{2},(i_{0},1)} \mod J^{M}_{\iota(h_{1}),\iota(h_{2}),\iota(h_{3})}$$ 
	for all  $p$-linearly independent $h_{1},h_{2},h_{3}\in P(c)$. 
	Since $d\geq \max\{s+3,11\}$, it follows from Proposition \ref{2:grm} (setting $k=0, m=1, \dim(V)=2$) that
	$$R'_{h_{1},(i_{0},1)}\equiv R'_{h_{2},(i_{0},1)} \mod J^{M}_{\iota(h_{1}),\iota(h_{2})}$$ 
	for all  $p$-linearly independent $h_{1},h_{2}\in P(c)$. 	
	So identifying $R'_{h_{j},i}$ as elements in $\st_{d}(s)$, it follows from Proposition \ref{2:coco1} (since $d\geq s+8$)  that 
	  there exists $G_{i_{0}}\in \st_{\Z/p,d}(s)$ such that  
	$$R'_{h,(i_{0},1)}\equiv G_{i_{0}} \mod J^{M}_{\iota(h)}$$
for all $h\in P(c)$ with $\iota(h)\neq \bold{0}$.
Finally, by identifying $R'_{h_{j},i}$ as elements in $\st_{d}(s)$,
by condition (iii), and by  using Proposition \ref{2:gsol00}, we have that there exist $0<c':=c'(c,s)<1$ (which eventually depends only on $r$ and $s$)  and a
locally linear map $T_{i_{0}}\colon P(c')\to \st_{\Z/p,d}(s)$ such that 
$$R'_{h,(i_{0},0)}\equiv T_{i_{0}}(h) \mod J^{M}_{\iota(h)}$$
for all $h\in P(c')$ with $\iota(h)\neq \bold{0}$. 
It then follows from (\ref{2:pwouf}) that for all $h\in P(c')$ with $\iota(h)\neq \bold{0}$ and $n\in V_{p}(\tilde{M})\cap \iota^{-1}(V(h))$,
\begin{equation}\nonumber
	   \begin{split}
	  &\quad \xi(h)(n)\equiv Q^{\ast}Q_{0}\xi(h)(n)
	  \\&\equiv Q^{\ast}\Bigl(\frac{T(h)(n)}{p}+\sum_{i=0}^{\min\{\lfloor \frac{s}{2}\rfloor,r\}} \tilde{M}(n)^{i}T_{i}(h)(n)+\sum_{i=0}^{\min\{\lfloor \frac{s-1}{2}\rfloor,r-1\}}\tilde{M}(n)^{i}G_{i}(n)\tilde{L}_{h}(n)\Bigr) \mod\Z,
	   \end{split}
	\end{equation}
where $Q^{\ast}$ is any integer with $Q^{\ast}Q_{0}\equiv 1 \mod p^{r}\Z$.
    In other words, 
    $$\xi(h)\equiv T'(h):=Q^{\ast}\Bigl(\frac{T(h)}{p}+\sum_{i=0}^{\min\{\lfloor \frac{s}{2}\rfloor,r\}} \tilde{M}^{i}T_{i}(h)+\sum_{i=0}^{\min\{\lfloor \frac{s-1}{2}\rfloor,r-1\}}\tilde{M}^{i}G_{i}\tilde{L}_{h}\Bigr) \mod J^{M}_{\iota(h)}$$
    for all $h\in P(c')$ with $\iota(h)\neq \bold{0}$, where $T'$ clearly is a locally linear map from $P(c')$ to $\st_{\Z/p^{r+1},d}(s)$.
    
    Finally,
    similar to  the argument in Step 5 of Section \ref{2:sp4.3}, we have that 
$\xi(h)\equiv T'(h) \mod J^{M}_{\iota(h)}$ for all $h\in P(c')$.
This completes the proof of Theorem \ref{2:gsol2}. 
%
 
 \
  
  It is an interesting question to ask that in Theorem \ref{2:gsol2}, whether one can remove the dependence on $r$ for the constant $c$. However, we do not need such an improvement in this series.
  
  We conclude this section with a generalization of Proposition \ref{2:coco1}, which will be used in \cite{SunC}.

  \begin{prop}\label{2:coco1p}
Let $d,K\in\N_{+}$, $r,s\in\N$, $\d>0$, $p\gg_{\d,d,r} 1$ be a prime dividing $K$,   $H$ be a subset of $\Vk$ with $\vert H\vert>\delta K^{d}$, $\bold{0}\notin\iota(H)$ and $M\colon\V\to\F_{p}$ be a non-degenerate quadratic form.
	Let $F\colon H\to\st_{\Z/p^{r},d}(s)$ be a map such that
	\begin{equation}\nonumber
	F(x)\equiv F(y)\mod J^{M}_{\iota(x),\iota(y)}
	\end{equation}
	 for all $x,y\in H$.  If  $d\geq s+8$, then there exists $G\in\st_{\Z/p^{r},d}(s)$ such that 
	\begin{equation}\nonumber
	F(x)\equiv G\mod J^{M}_{\iota(x)}
	\end{equation}
	  for all $x\in H$. 
\end{prop}

The approach to deduce Proposition \ref{2:coco1p} from  Proposition \ref{2:coco1} is very similar to the approach we use to deduce  Theorem \ref{2:gsol2} from Proposition \ref{2:gsol00}. We leave the details to the interested readers.

 \section{Structures for sets with small density dependence numbers}\label{2:s:115}

In this section, we extend some important results in additive combinatorics for shifted modules.
 
 \subsection{Two notions of equivalence classes}\label{2:s:44}
 
Although $\sim$ is not an equivalence relation $\Vk\times \gfv$, it is still possible to talk about equivalence classes in the following way: 

\begin{defn}[Equivalence classes]
	We say that a subset $X$ of $\Gamma^{s}(\Vk,M)$ is an \emph{equivalence class} (or a \emph{weak equivalence class}) if $x\sim x'$ for all $x,x'\in X$. Clearly, $\pi(x)$ is a constant for all $x\in X$, and so we denote this constant by $\pi(X)$ when there is no confusion.
	We say that an equivalence class $X\subseteq \Gamma^{s}(\Vk,M)$ is a \emph{strong equivalence class} if there exists $u\in\Gamma^{s}_{1}(\Vk,M)$ such that $u\sim x$ for all $x\in X$.
\end{defn}

 \begin{ex}\label{2:exws1}
Assume that $\Vk=\V$ and $M(n)=n\cdot n$. For $x\in\V$, let $L_{x}\in\st_{\zp,d}(1)$ be the map given by $L_{x}(n):=\frac{1}{p}(\tau(x)\cdot n)$. Let $e_{1},\dots,e_{d}$ be the standard unit vectors of $\V$. 
Let $x_{1}=(\bold{0},J^{M}_{e_{1}})$, $x_{2}=(\bold{0},J^{M}_{e_{2}})$ and $x_{3}=(\bold{0},J^{M}_{e_{1}+e_{2}}+L_{e_{1}})$. It is not hard to check that $\{x_{1},x_{2},x_{3}\}$ is a weak equivalence class in $\Gamma^{1}_{2}(\Vk,M)$. On the other hand, suppose that there exists $x\in \Gamma^{1}_{1}(\Vk,M)$ such that $x\sim x_{i}$ for all $1\leq i\leq 3$. Since $x\in \Gamma^{1}_{1}(\Vk,M)$, we may write $x=(\bold{0},J^{M}+L_{y})$ for some  $y\in\V$. Since $x\sim x_{i}$ for all $1\leq i\leq 3$, we have that $y,y$ and $y-e_{1}$ are scalar multiples of $e_{1},e_{2}$ and $e_{1}+e_{2}$ respectively, which is not possible. So $\{x_{1},x_{2},x_{3}\}$ is not a strong equivalence class. 
\end{ex}

\begin{ex}\label{2:exws2}
	Let the notations be the same as Example \ref{2:exws1}. Let $x_{4}=(\bold{0},J^{M}_{e_{1}-e_{2}})$. It is not hard to see that $\{x_{1},x_{2},x_{4}\}$ is a strong equivalence class since $(\bold{0},J^{M})\sim x_{i}$ for all $i=1,2,4$. If we add $x_{3}$ into this set, then it is not hard to see that $\{x_{1},x_{2},x_{3},x_{4}\}$ is a weak  equivalence class. However, by Example \ref{2:exws1}, $\{x_{1},x_{2},x_{3},x_{4}\}$ is no longer a strong equivalence class. 
\end{ex}

In conclusion, Example \ref{2:exws1} shows that there exists a weak equivalence class $X$ which is not a strong equivalence class. Example \ref{2:exws2} shows that there is a strong equivalent class $X$ and some $u\in\Gamma^{s}(\Vk,M)$ such that $x\sim u$ for all $x\in X$, but the set $X\cup\{u\}$ is no longer a strong equivalent class. In this paper, such equivalence classes $X$ are considered as abnormal. In fact, the following proposition shows that ``generic" equivalence classes do not have the properties mentioned above.

\begin{prop}\label{2:gwts} 
	Let $d,k,k',K\in\N_{+}, k\geq 2$, $s\in\N$ and $p\gg_{d} 1$ be a prime dividing $K$. Let $M\colon\V\to\F_{p}$ be a non-degenerate quadratic form. Let $X\subseteq \Gamma^{s}_{k}(\Vk,M)$ be a finite set with $\pi(X)=\{h\}$ for some $h\in\Vk$. 
	Denote $t=2$ if $\iota(h)\neq \bold{0}$ and $t=0$ otherwise. Then
	\begin{enumerate}[(i)]
		\item if  $d\geq 2k+2(s+1)(k-2)+5+t$ and $X$ is a weak equivalence class, then either $X$ is a strong equivalence class or 
		there exists a subspace $Y$ of $\V$ of dimension at most $(s+k-1)(k-1)$  such that $V\cap Y\neq \{\bold{0}\}$ for all $(h,J^{M}_{V}+f)\in X$.
		\item if  $d\geq 2k'+2(s+1)(k-2)+5+t$ and $X$ is a strong equivalence class, then either  there exists a subspace $Y$ of $\V$ of dimension at most $(s+k'-1)(k-1)$ such that $V\cap Y\neq \{\bold{0}\}$ for all $(h,J^{M}_{V}+f)\in X$, or for all $u\in\Gamma^{s}_{k'}(\Vk,M)$ such that $u\sim a$ for all $a\in X$, we have that $X\cup\{u\}$ is also a strong equivalence class.
			\end{enumerate}	
%
\end{prop}

Before proving Proposition \ref{2:gwts}, we need the following lemma from linear algebra:

\begin{lem}\label{2:gweakdic}
	Let $d,k,m\in\N_{+}, k'\in\N$, $p$ be a prime, $U$ be a subspace of $\V$ of dimension $k'$ (possibly the trivial subspace), and $\mathcal{U}$ be a collection of subspaces of $\V$ of dimension at most $k$ containing $U$.
	Then  either there 
	exist $V_{i}\in \mathcal{U}, 1\leq i\leq m$ such that $V_{1},\dots,V_{m}$ are linearly independent modulo $U$, or there exists a subspace $W$ of $\V$ of dimension at most $k'+(m-1)(k-k')$ containing $U$ such that $W\cap V\neq U$ for all $V\in\mathcal{U}$.
\end{lem}	
\begin{proof}
	Pick any $V_{1} \in\mathcal{U}$.
	Suppose that we have chosen $V_{i}\in \mathcal{U}, 1\leq i\leq r$ for some $1\leq r \leq m-1$ such that $V_{1},\dots,V_{r}$ are linearly independent modulo $U$.
	If for all $V\in \mathcal{U}$, $V\cap (V_{1}\dots+V_{r})\neq U$, then we are done since $\dim(V_{1}\dots+V_{r})=k'+r(k-k')\leq k'+(m-1)(k-k')$. If not, then we may find $V_{r+1}\in \mathcal{U}$ such that $V_{1},\dots,V_{r+1}$ are linearly independent modulo $U$. 
	
	In conclusion, either we are done or there 
	exist $V_{i}\in \mathcal{U}, 1\leq i\leq m$ such that $V_{1},\dots,V_{m}$ are linearly independent modulo $U$.
\end{proof}	

\begin{proof}[Proof of Proposition \ref{2:gwts}] 
Since $X$ is finite, there exists $N\in\N$ such that $f\in\st_{\Z/p^{N},d}(s)$ for all $(h,J_{V}+f)\in X$.
	We first consider the case $\iota(h)=\bold{0}$. For convenience we assume that $h=\bold{0}$.
	We start with Part (i). Denote $K:=s+k$.
	By Lemma \ref{2:gweakdic}, either
	there exists a subspace $Y$ of $\V$ of dimension at most $(K-1)(k-1)$ such that $V\cap Y\neq \{\bold{0}\}$ for all $(\bold{0},J^{M}_{V}+f)\in X$, or there 
	exist $u_{i}=(\bold{0},J^{M}_{V_{i}}+f_{i})\in X, 1\leq i\leq K$ such that $V_{1},\dots,V_{K}$ are linearly independent.
	If it is the former case then we are done. We now assume that it is the latter case.

	Obviously $(\bold{0},J^{M}+f_{1})\in\Gamma^{s}_{1}(\Vk,M)$ is equivalent to $u_{1}$.
	Suppose that for some $1\leq n\leq s$, there exists $u\in\Gamma^{s}_{1}(\Vk,M)$ such that $u\sim u_{i}$ for all $1\leq i\leq n$. We show that  there exists $u'\in\Gamma^{s}_{1}(\Vk,M)$ such that $u'\sim u_{i}$ for all $1\leq i\leq n+1$.
	Assume that $u=(\bold{0},J^{M}+f)$.
	Then for all $1\leq i\leq m$, since  $X$ is an equivalence class, we have that $u\sim u_{i}\sim u_{n+1}$. So $f\equiv f_{i}\equiv f_{n+1}\mod J^{M}_{V_{n+1}+V_{i}}$
	and thus 
	$f\equiv f_{n+1}\mod\cap_{i=1}^{n}J^{M}_{V_{n+1}+V_{i}}.$
	Since $V_{1},\dots,V_{n+1}$ are linearly independent, 
	and $\pi(X)=\{\bold{0}\}$, the dimensions of  $V_{1},\dots,V_{n+1}$ are at most $k-1$. 
	Since $d\geq 2k+2s(k-2)+5$,  by Lemma \ref{2:w3s}   (setting $V=V_{n+1}$, $m=k-1$, $r=k-1$, and $N=n$), 
	we may  
	assume that $f-f_{n+1}=g+g'$ for some $g\in \st_{\Z/p^{N},d}(s)\cap J^{M}_{V_{n+1}}$ and $g'\in \st_{\Z/p^{N},d}(s)\cap \cap_{i=1}^{n}J^{M}_{V'_{i}}$, and set $u':=(\bold{0},J^{M}+(f-g'))$.
	Then $(f-g')-f_{n+1}=g\in J^{M}_{V_{n+1}}$. On the other hand, for all $1\leq i\leq n$,
	$(f-g')-f_{i}=(f-f_{i})-g'\in J^{M}_{V_{i}}$. This means that $u'\sim u_{i}$ for all $1\leq i\leq n+1$.
	
	Inductively, we may find some $u_{0}\in \Gamma^{s}_{1}(\Vk,M)$ such that $u_{0}\sim u_{i}$ for all $1\leq i\leq s+1$.

	Let $s+2\leq i\leq K$. Since $u_{0}\sim u_{j}\sim u_{i}$ for $1\leq j\leq s+1$, we have that 
	$f_{0}\equiv f_{j}\equiv f_{i} \mod J^{M}_{V_{j}+V_{i}}$ and so $f_{0}\equiv f_{i} \mod \cap_{j=1}^{s+1}J^{M}_{V_{j}+V_{i}}$. Since $V_{1},\dots,V_{s+1}, V_{i}$ are linearly independent and of dimensions at most $k-1$, and since $d\geq 2k+2(s+1)(k-2)+5$,
	by Proposition \ref{2:gr0}    (setting $m=k-1$ and $r=k-1$), we have that $f_{0}\equiv f_{i} \mod J^{M}_{V_{i}}$ and so $u_{0}\sim u_{i}$.
	In other words, $u_{0}\sim u_{i}$ for all $1\leq i\leq K$.
	
	Now pick any $u'=(\bold{0},J^{M}_{V'}+f')\in X$. Since $u_{0}\sim u_{i}\sim u'$ for all $1\leq i\leq K$, we have that $f_{0}\equiv f_{i}\equiv f' \mod J^{M}_{V_{i}+V'}$. So $f_{0}\equiv f' \mod \cap_{i=1}^{K}J^{M}_{V_{i}+V'}$. Since $V_{1},\dots,V_{K}$ are linearly independent and of dimensions at most $k-1$, and since $d\geq 2k+2(s+1)(k-2)+5$, by Proposition \ref{2:gri} (setting $m=k-1$ and $r=k-1$),   we have that $f_{0}\equiv f' \mod  J^{M}_{V}$.  So $u_{0}\sim u'$ and thus $X$ is a strong equivalence class.

	We now prove Part (ii). 	
	Denote $K:=s+k'$.
	By Lemma \ref{2:gweakdic}, either
	there exists a subspace $Y$ of $\V$ of dimension at most $(K-1)(k-1)$ such that $V\cap Y\neq \{\bold{0}\}$ for all $(\bold{0},J^{M}_{V}+f)\in X$, or there 
	exist $u_{i}=(\bold{0},J^{M}_{V_{i}}+f_{i})\in X, 1\leq i\leq K$ such that $V_{1},\dots,V_{K}$ are linearly independent.
	If it is the former case then we are done. We now assume that  it is the latter case.
	
	Let $u_{0}=(\bold{0},J^{M}+f_{0})\in\Gamma^{s}_{1}(\Vk,M)$ be such that $u_{0}\sim a$ for all $a\in X$. It suffices to show that for all $u=(\bold{0},J^{M}_{V}+f)\in \Gamma^{s}_{k'}(\Vk,M)$, if $u\sim a$ for all $a\in X$, then $u_{0}\sim u$. Since $u_{0}\sim u_{i}\sim u$ for all $1\leq i\leq K$, we have that $f_{0}\equiv  f_{i}\equiv f \mod J^{M}_{V_{i}+V}$ and $\pi(u_{0})=\pi(u)=\bold{0}$. So $V$ is of dimension at most $k'-1$ and
	$f_{0}\equiv f \mod \cap_{i=1}^{K}J^{M}_{V_{i}+V}$. Since $V_{1},\dots,V_{K}$ are linearly independent and of dimension at most $k-1$, and since $d\geq 2k'+2(s+1)(k-2)+5$,  by Proposition \ref{2:gri} (setting $m=k'-1$ and $r=k-1$), we have that $f_{0}\equiv f \mod  J^{M}_{V}$.  So $u_{0}\sim u$ and we are done.
	
	\
	
	We next consider the case $\iota(h)\neq \bold{0}$. The proof is very similar to the case $\iota(h)=\bold{0}$, except that the dimensions of relevant subspaces of $\V$ appearing in the proof need to be changed slightly. We provide the details of the proof for completeness.
	
		We start with Part (i). Denote $K:=s+k+1$.
		By Lemma \ref{2:gweakdic}, either
		there exists a subspace $Y$ of $\V$ of dimension at most $(K-1)(k-1)+1$ containing $h$ such that $V\cap W\neq \sp_{\F_{p}}\{\iota(h)\}$ for all $(h,J^{M}_{V}+f)\in X$, or there 
		exist $u_{i}=(h,J^{M}_{V_{i}}+f_{i})\in X, 1\leq i\leq K$ such that $V_{1},\dots,V_{K}$ are linearly independent modulo $\sp_{\F_{p}}\{\iota(h)\}$.
		If it is the former case then we are done. We now assume that  it is the latter case.

		Obviously $(h,J^{M}_{h}+f_{1})\in\Gamma^{s}_{1}(\Vk,M)$ is equivalent to $u_{1}$.
		Suppose that for some $1\leq n\leq s$, there exists $u\in\Gamma^{s}_{1}(\Vk,M)$ such that $u\sim u_{i}$ for all $1\leq i\leq n$. We show that  there exists $u'\in\Gamma^{s}_{1}(\Vk,M)$ such that $u'\sim u_{i}$ for all $1\leq i\leq n+1$.
		Assume that $u=(h,J^{M}_{h}+f)$.
		Then for all $1\leq i\leq m$, since  $X$ is an equivalence class, we have that $u\sim u_{i}\sim u_{n+1}$. So $f\equiv f_{i}\equiv f_{n+1} \mod J^{M}_{V_{n+1}+V_{i}}$
		and thus 
		$f\equiv f_{n+1}\mod\cap_{i=1}^{n}J^{M}_{V_{n+1}+V_{i}}.$
		Since $V_{1},\dots,V_{n+1}$ are linearly independent modulo $\sp_{\F_{p}}\{h\}$, 
		we may write $V_{n+1}+V_{i}=V_{n+1}+V'_{i}, 1\leq i\leq n$ for some  subspaces $V'_{1},\dots,V'_{n}$ of $\V$
		of dimensions at most $k-1$ such that $V'_{1},\dots,V'_{n},V_{n+1}$ are linearly independent. Since $d\geq 2k+2s(k-2)+7$,
		by Lemma \ref{2:w3s} (setting $V=V_{n+1}$, $m=k$, $r=k-1$  and $N=n$), 
		we may
	assume  that $f-f_{n+1}=g+g'$ for some $g\in \st_{\Z/p^{N},d}(s)\cap J^{M}_{V_{n+1}}$ and $g'\in \st_{\Z/p^{N},d}(s)\cap \cap_{i=1}^{n}J^{M}_{V'_{i}}$. Set $u':=(h,J^{M}_{h}+(f-g'))$.
		Then $(f-g')-f_{n+1}=g\in J^{M}_{V_{n+1}}$. On the other hand, for all $1\leq i\leq n$,
		$(f-g')-f_{i}=(f-f_{i})-g'\in J^{M}_{V_{i}}$. This implies that $u'\sim u_{i}$ for all $1\leq i\leq n+1$.
		
		Inductively, may find some $u_{0}\in \Gamma^{s}_{1}(\Vk,M)$ such that $u_{0}\sim u_{i}$ for all $1\leq i\leq s+1$.

		Let $s+2\leq i\leq K$. Since $u_{0}\sim u_{j}\sim u_{i}$ for $1\leq j\leq s+1$, we have that 
		$f_{0}\equiv f_{j}\equiv f_{i} \mod J^{M}_{V_{j}+V_{i}}$ and so $f_{0}\equiv f_{i} \mod \cap_{j=1}^{s+1}J^{M}_{V_{j}+V_{i}}$. Since $V_{1},\dots,V_{s+1}, V_{i}$ are linearly independent modulo $\sp_{\F_{p}}\{\iota(h)\}$, we may write $V_{j}+V_{i}=V'_{j}+V_{i}, 1\leq j\leq s+1$ for some  subspaces $V'_{1},\dots,V'_{s+1}$ of $\V$
		of dimensions at most $k-1$ such that $V'_{1},\dots,V'_{s+1},V_{i}$ are linearly independent. Since $d\geq 2k+2(s+1)(k-2)+7$,
		by Proposition \ref{2:gr0} (setting $m=k$  and $r=k-1$),  we have that  $f_{0}\equiv f_{i} \mod J^{M}_{V_{i}}$ and so $u_{0}\sim u_{i}$.
		In other words, $u_{0}\sim u_{i}$ for all $1\leq i\leq K$.
		
		Now pick any $u'=(h,J^{M}_{V'}+f')\in X$. Since $u_{0}\sim u_{i}\sim u'$ for all $1\leq i\leq K$, we have that $f_{0}\equiv f_{i}\equiv f' \mod J^{M}_{V_{i}+V'}$. So $f_{0}\equiv f' \mod \cap_{i=1}^{K}J^{M}_{V_{i}+V'}$. Since $V_{1},\dots,V_{K}$ are linearly independent modulo $\sp_{\F_{p}}\{\iota(h)\}$, we may write $V_{i}+V'=V'_{i}+V', 1\leq i\leq K$ for some  subspaces $V'_{1},\dots,V'_{K}$ of $\V$
		of dimensions at most $k-1$ such that $V'_{1},\dots,V'_{K}$ are linearly independent. Since $d\geq 2k+2(s+1)(k-2)+7$, by Proposition \ref{2:gri} (setting $m=k$ and $r=k-1$), we have that $f_{0}\equiv f' \mod  J^{M}_{V}$.  So $u_{0}\sim u'$ and thus $X$ is a strong equivalence class.

		We now prove Part (ii). 	
		Denote $K:=s+k'+1$.
		By Lemma \ref{2:gweakdic}, either
		there exists a subspace $Y$ of $\V$ of dimension at most $(K-1)(k-1)+1$ containing $h$ such that $V\cap W\neq \sp_{\F_{p}}\{\iota(h)\}$ for all $(h,J^{M}_{V}+f)\in X$, or there 
		exist $u_{i}=(h,J^{M}_{V_{i}}+f_{i})\in X, 1\leq i\leq K$ such that $V_{1},\dots,V_{K}$ are linearly independent modulo $\sp_{\F_{p}}\{\iota(h)\}$.
		If it is the former case then we are done. We now assume that  it is the latter case.
		
		Let $u_{0}=(h,J^{M}_{\iota(h)}+f_{0})\in\Gamma^{s}_{1}(\Vk,M)$ be such that $u_{0}\sim a$ for all $a\in X$. It suffices to show that for all $u=(h,J^{M}_{V}+f)\in \Gamma^{s}_{k'}(\Vk,M)$, if $u\sim a$ for all $a\in X$, then $u_{0}\sim u$. Since $u_{0}\sim u_{i}\sim u$ for all $1\leq i\leq K$, we have that $f_{0}\equiv f_{i}\equiv f \mod J^{M}_{V_{i}+V}$. So
		$f_{0}\equiv f \mod \cap_{i=1}^{K}J^{M}_{V_{i}+V}$. Since $V_{1},\dots,V_{K}$ are linearly independent modulo $\sp_{\F_{p}}\{\iota(h)\}$,
		we may write $V_{i}+V=V'_{i}+V, 1\leq i\leq K$ for some subspace $V'_{i}$ of $\V$ of dimension at most $k-1$ such that $V'_{1},\dots,V'_{K}$ are linearly independent. Since $d\geq 2k'+2(s+1)(k-2)+7$,
		by Proposition \ref{2:gri} (setting $m=k'$ and $r=k-1$), we have that $f_{0}\equiv f \mod  J^{M}_{V}$.  So $u_{0}\sim u$ and we are done.
\end{proof}

\subsection{The $M$-energy graph}

Let $d,K\in\N_{+}$, $s\in\N$, $p$ be a prime, $M\colon\V\to\F_{p}$ be a quadratic form dividing $K$,  $H$ be a subset of $\Vk$, and
$\xi_{1},\xi_{2},\xi_{3},\xi_{4}\colon H\to\st_{\zp,d}(s)$ be maps.
For $h\in H$, denote $\tilde{\xi}_{i}(h):=(h, J^{M}_{\iota(h)}+\xi_{i}(h))\in\Gamma_{1}^{s}(\Vk,M)$. Consider the directed graph $(H^{2},E)$ where $((h_{1},h_{2}),(h_{3},h_{4}))\in E$ if and only if $\tilde{\xi}_{1}(h_{1})\- \tilde{\xi}_{2}(h_{2})\sim \tilde{\xi}_{3}(h_{3})\- \tilde{\xi}_{4}(h_{4})$. We call $(H^{2},E)$ the \emph{$M$-energy graph} of $(\xi_{1},\xi_{2},\xi_{3},\xi_{4})$, and $\vert E\vert$ the \emph{$M$-energy} of $(\xi_{1},\xi_{2},\xi_{3},\xi_{4})$.
For convenience we define the \emph{$M$-energy graph/$M$-energy} of  $\xi_{1}$ to be the 
$M$-energy graph/$M$-energy of $(\xi_{1},\xi_{1},\xi_{1},\xi_{1})$ (in which case the energy graph is undirected).

In this section, we prove the following Cauchy-Schwartz inequality for $M$-energy:

\begin{thm}[Cauchy-Schwartz inequality]\label{2:gcs}
	Let $d,K\in\N_{+}$, $s\in\N$, $d\geq \max\{s+4,13\}$, $\d>0$, and $p\gg_{\d,d} 1$ be a prime dividing $K$. Let $M\colon\V\to\F_{p}$ be a non-degenerate quadratic form, $H\subseteq \Vk$ and $\xi_{1},\xi_{2},\xi_{3},\xi_{4}\colon H\to\st_{\zp,d}(s)$. If the $M$-energy of $(\xi_{1},\xi_{2},\xi_{3},\xi_{4})$ is at least $\d K^{3d}$, then the $M$-energy   of $(\xi_{1},\xi_{2},\xi_{1},\xi_{2})$ is $\gg_{\d,d} K^{3d}$.
\end{thm}

We start with the following technical proposition.

\begin{prop}\label{2:glargepart}
	Let $d,K\in\N_{+}$, $s\in\N$ with $d\geq \max\{s+4,13\}$, 	$\d>0$ and $p\gg_{\d,d} 1$ be a prime dividing $K$. Let  $M\colon\V\to\F_{p}$ be a non-degenerate quadratic form, $H\subseteq \Vk$, and $\xi_{1},\xi_{2},\xi_{3},\xi_{4}\colon H\to\st_{\zp,d}(s)$. If the $M$-energy of $(\xi_{1},\xi_{2},\xi_{3},\xi_{4})$ is at least $\d K^{3d}$, then there exist $H'\subseteq H$ with $\vert H'\vert\geq \d K^{d}/2$ and a set $C_{h}\subseteq H'\times H'$ for all $h\in H'$ such that for all $h\in H'$
	\begin{itemize}
		\item $\vert C_{h}\vert\gg_{\d,s} K^{d}$;
		\item $h_{1}-h_{2}=h$ for all $(h_{1},h_{2})\in C_{h}$;
		\item $\tilde{\xi}_{1}(h_{1})\- \tilde{\xi}_{2}(h_{2})\sim \tilde{\xi}_{1}(h'_{1})\- \tilde{\xi}_{2}(h'_{2})$ for all $(h_{1},h_{2}),(h'_{1},h'_{2})\in C_{h}$.
	\end{itemize}	
	Similarly, there exist $H''\subseteq H$ with $\vert H''\vert\geq \d K^{d}/2$ and a set $C'_{h}\subseteq H''\times H''$ for all $h\in H''$ such that for all $h\in H''$
	\begin{itemize}
		\item $\vert C'_{h}\vert\gg_{\d,s} K^{d}$;
		\item $h_{1}+h_{4}=h$ for all $(h_{1},h_{4})\in C'_{h}$;
		\item $\tilde{\xi}_{1}(h_{1})\+ \tilde{\xi}_{4}(h_{4})\sim \tilde{\xi}_{1}(h'_{1})\+ \tilde{\xi}_{4}(h'_{4})$ for all $(h_{1},h_{4}),(h'_{1},h'_{4})\in C'_{h}$.
	\end{itemize}	
\end{prop}	
\begin{proof}
	Let $G:=(H^{2},E)$ be the $M$-energy graph of $(\xi_{1},\xi_{2},\xi_{3},\xi_{4})$.
	For $h\in\Vk$, let $U_{h}$ be the set of all vertices $(h_{1},h_{2})\in H^{2}$ such that $h_{1}+h_{2}=h$. 
	Clearly, $E$ does not contain any edge connecting vertices from two different $U_{h}$. So $G$ has a disjoint partition $G_{h}:=(U_{h},E_{h}), h\in \Vk$. Since $\sum_{h\in\Vk}\vert E_{h}\vert=\vert E\vert\geq \d K^{3d}$, and $\vert E_{h}\vert\leq K^{2d}$ for all $h\in\Vk$, we have that there exists $H'\subseteq \Vk$ with $\vert H'\vert\geq \d K^{d}/2$ such that $\vert E_{h}\vert\geq \d K^{2d}/2$ for all $h\in H'$. 
	
	Fix $h\in H'$ and consider the graph $G_{h}=(U_{h},E_{h})$. Enumerate $U_{h}:=\{n_{1},\dots, n_{\vert U_{h}\vert}\}$, and let $a_{i}$ denote the number of edges in $G_{h}$ with one of the vertices being $n_{i}$. 
	We say that $w=(v_{1},\dots,v_{s+4})\in U_{h}^{s+4}$ is \emph{good} if writing $v_{i}=(u_{i,1},u_{i,2})=(u_{i,1},u_{i,1}-h)$ for $1\leq i\leq s+4$, 
	we have that 
	$u_{i,1}, 1\leq i\leq s+4$ are $p$-linearly independent.    
	The number of tuples $(v_{1},\dots,v_{s+4})\in U_{h}^{s+4}$ is at most
	$\vert U_{h}\vert^{s+4}$. On the other hand, the number of tuples $(v,v_{1},\dots,v_{s+4})\in U_{h}^{s+5}$ such that $(v,v_{1}), \dots,(v,v_{s+4})\in E_{h}$ is at least
	$$\sum_{i=1}^{\vert U_{h}\vert}a^{s+4}_{i}\geq\frac{(\sum_{i=1}^{\vert U_{h}\vert}a_{i})^{s+4}}{\vert U_{h}\vert^{s+3}}=\frac{\vert E_{h}\vert^{s+4}}{\vert U_{h}\vert^{s+3}}\gg_{\d,s} K^{(s+5)d}.$$
	For each $w=(v_{1},\dots,v_{s+4})\in U_{h}^{s+4}$, let $W_{w}$ denote the set $v\in U_{h}$ such that $(v,v_{1}),\dots,(v,$ $v_{s+4})\in E_{h}$. Since $\vert W_{w}\vert\leq K^{d}$ for all $w$, the number of $w\in U_{h}^{s+4}$ such that $\vert W_{w}\vert\gg_{\d,s} K^{d}$ is  $\gg_{\d,s}K^{(s+4)d}$. Since $p\gg_{\d,d}1$,  we may choose such a $w$ which is also good (in fact, it is not hard to see from the fact $u_{i,1}-u_{i,2}=h$ and Lemma \ref{2:iiddpp}  that the number $w\in U^{s+4}_{h}$ which is not good is at most $O_{\d,d}(K^{(s+4)d}/p^{d-(s+3)})$ since $d\geq s+4$). 	
	Fix such a choice of 
	$w=(v_{1},\dots,v_{s+4})$, with $v_{i}=(u_{i,1},u_{i,2})$ for $1\leq i\leq s+4$. 
	Pick $(h_{1},h_{2}),(h'_{1},h'_{2})\in W_{w}$. Since $((h_{1},h_{2}),v_{i}),((h'_{1},h'_{2}),v_{i})\in E$, we have that 
	\begin{equation}\label{2:gccc333}
	\xi_{1}(h_{1})-\xi_{2}(h_{2})\equiv\xi_{3}(u_{i,1})- \xi_{4}(u_{i,2})\equiv\xi_{1}(h'_{1})-\xi_{2}(h'_{2}) \mod J^{M}_{\iota(h_{1}),\iota(h_{2}),\iota(h'_{1}),\iota(h'_{2}),\iota(u_{i,1}),\iota(u_{i,2})}
	\end{equation}
	for $1\leq i\leq s+4$.
	Therefore, 
	\begin{equation}\label{2:gccc33}
	\xi_{1}(h_{1})-\xi_{2}(h_{2})\equiv\xi_{1}(h'_{1})-\xi_{2}(h'_{2}) \mod \bigcap_{i=1}^{s+4}J^{M}_{\iota(h_{1}),\iota(h_{2}),\iota(h'_{1}),\iota(h'_{2}),\iota(u_{i,1}),\iota(u_{i,2})}.
	\end{equation}
	By the definition of $U_{h}$, we have that $\bigcap_{i=1}^{s+4}J^{M}_{\iota(h_{1}),\iota(h_{2}),\iota(h'_{1}),\iota(h'_{2}),\iota(u_{i,1}),\iota(u_{i,2})}=\bigcap_{i=1}^{s+4}J^{M}_{\iota(h_{1}),\iota(h'_{1}),\iota(u_{i,1}),\iota(h)}$ and $J^{M}_{\iota(h_{1}),\iota(h'_{1}),\iota(h)}=J^{M}_{\iota(h_{1}),\iota(h_{2}),\iota(h'_{1}),\iota(h'_{2})}.$
	Since $u_{i,1}, 1\leq i\leq s+4$ are $p$-linearly independent and $d\geq 13$, it follows from (\ref{2:gccc33}) and Proposition \ref{2:gri} (setting $m=3$ and $r=1$) that   
	\begin{equation}\label{2:gccc334}
	\xi_{1}(h_{1})-\xi_{2}(h_{2})\equiv\xi_{1}(h'_{1})-\xi_{2}(h'_{2}) \mod J^{M}_{\iota(h_{1}),\iota(h_{2}),\iota(h'_{1}),\iota(h'_{2})}.
	\end{equation}
	Since $h_{1}-h_{2}=h$ for all $(h_{1},h_{2})\in W_{w}$, we are done by setting $C_{h}=W_{w}$ for all $h\in H'$.	
	
	\
	
	The existence of $H''$ and $C'_{h}$ can be proved in a similar way, and we leave the details to the interested readers.
\end{proof}

\begin{rem}
We briefly explain the reason why we consider the set $U_{h}^{s+4}$ in the proof of Proposition \ref{2:glargepart}. If we used  a similar argument with $w=v_{1}\in U_{h}$ instead of $w=(v_{1},\dots,v_{s+4})\in U_{h}^{s+4}$, we would  arrive at (\ref{2:gccc333}) with $i=1$, which is insufficient to deduce  (\ref{2:gccc334}). This is  because there is a loss of information in (\ref{2:gccc333}) where we need to modulo an additional module $J^{M}_{\iota(u_{1,1}),\iota(u_{1,2})}$. In order to overcome this difficulty, we use the intersection approach to find $s+4$ different $v_{i}$  in general positions, and then intersect all the modules $J^{M}_{\iota(h_{1}),\iota(h_{2}),\iota(h'_{1}),\iota(h'_{2}),\iota(u_{i,1}),\iota(u_{i,2})}$ to recover the lost information (the number $s+4$ is determined by the intersection properties for $M$-modules proved in Section \ref{2:s:a1}). Similar approaches will be frequently used in this paper.
\end{rem}

We are now ready to prove Theorem \ref{2:gcs}.

\begin{proof}[Proof of Theorem \ref{2:gcs}]
	By Proposition \ref{2:glargepart},  there exist $H'\subseteq H$ with $\vert H'\vert\geq \d K^{d}/2$ and a set $C_{h}\subseteq H'\times H'$ for all $h\in H'$ such that for all $h\in H'$,
	\begin{itemize}
		\item $\vert C_{h}\vert\gg_{\d,s} K^{d}$;
		\item $h_{1}-h_{2}=h$ for all $(h_{1},h_{2})\in C_{h}$;
		\item $\tilde{\xi}_{1}(h_{1})\- \tilde{\xi}_{2}(h_{2})\sim \tilde{\xi}_{1}(h'_{1})\- \tilde{\xi}_{2}(h'_{2})$ for all $(h_{1},h_{2}),(h'_{1},h'_{2})\in C_{h}$.
	\end{itemize}	
	This means that  the $M$-energy  of $(\xi_{1},\xi_{2},\xi_{1},\xi_{2})$ is at least
	$\sum_{h\in H'}\vert C_{h}\vert^{2}\gg_{\d,d} K^{3d}.$
\end{proof}	

\subsection{The Balog-Gowers-Szemer\'edi theorem}

Let $G$ be a discrete abelian group and $H$ be a subset of $G$. The Balog-Gowers-Szemer\'edi Theorem says that if  $H$ has a large additive energy, then there exists a large subset $H'$ of $H$ with a small doubling constant $\frac{\vert H'+H'\vert}{\vert H'\vert}$ (see for example \cite{TV06}). We extend this result to shifted modules in this section.

For $A,B\subseteq \Gamma^{s}_{k}(\Vk,M)$, we use the notation $A\+ B$ to denote the set of elements of the form $a\+ b$ for some $a\in A$ and $b\in B$, and we define $A\- B$ in a similar way. We refer the readers to Section \ref{2:s:rgg} to recall the definitions we use in graph theory.

\begin{thm}[Balog-Gowers-Szemer\'edi Theorem for $\Gamma^{s}(\Vk,M)$]\label{2:gbig1}
	Let $d,K\in\N_{+}$, $s\in\N$, $d\geq 2s+15$, $\d>0$ and 
	 $p\gg_{\d,d} 1$ be a prime dividing $K$. Let $M\colon\V\to\F_{p}$ be a non-degenerate quadratic form, $H\subseteq \Vk$ and $\xi\colon H\to\st_{\zp,d}(s)$. If the $M$-energy  of $\xi$ is at least $\d K^{3d}$, then
	  there exists $H'\subseteq H$ with $\vert H'\vert\gg_{\d,d}K^{d}$ such that $\dep(\tilde{\xi}(H')\+  \tilde{\xi}(H'))=O_{\d,d}(1)$.
\end{thm}

Note that in our setting, given a subset $X$ of $\Gamma^{s}(\Vk,M)$, there is no such a concept as the ``doubling constant" of $X$. So in Theorem \ref{2:gbig1}, we use the density dependence number of the relation graph of $X$ to replace the doubling constant in the conventional additive combinatorics. It is helpful for the readers to informally interpret the conclusion $\dep(\tilde{\xi}(H')\+  \tilde{\xi}(H'))=O_{\d,d}(1)$ as $\tilde{\xi}(H')$ having a ``doubling constant" $O_{\d,d}(1)$.

\begin{rem}\label{2:dpcc}
Intuitively, compared with the density dependence number, the clique covering number is a more natural replacement for the cardinality of subsets of $\Gamma^{s}(\Vk,M)$. The reason we use the density dependence number in Theorem \ref{2:gbig1} (as well as in other results in Section \ref{2:s:115}) is because it is easier to handle.
In Section \ref{2:s:116}, we will explain how to pass from the density dependence number to the clique covering number.
\end{rem}

 The rest of this subsection is devoted to the proof of Theorem \ref{2:gbig1}. 
 Our method to prove Theorem \ref{2:gbig1} is inspired by the work of Ballog and Szemer\'edi \cite{BS94}.
The approach in \cite{BS94} relies heavily on the analysis of the cardinality of the level set, i.e. a set of the form $\{(h_{1},h_{2})\in G\times G\colon h_{1}+h_{2}=c\}$ for some abelian group $G$ and $c\in G$. In our setting,
 for any  $c\in\Gamma^{s}(\Vk,M)$, we may also define the level set at $c$ as  $S(c):=\{(h_{1},h_{2})\in H^{2}\colon  \tilde{\xi}(h_{1})\+  \tilde{\xi}(h_{2})\sim c\}$.  We have the following basic properties:  

\begin{lem}\label{2:gbsg2}
	Let $a,a',a''\in\Vk$ and $c,c'\subseteq\Gamma^{s}(\Vk,M)$. 
	
	(i) If $(a,a'),(a,a'')\in S(c)$, then $a'=a''$.  
	
	(ii) If $(a,a')\in S(c)\cap S(c')$, then $\pi(c)=\pi(c')$.
\end{lem}	
\begin{proof}
	For  Part (i), if $(a,a'),(a,a'')\in S(c)$, then $a+a'=a+a''=\pi(c)$, and so $a'=a''$. 
	For Part (ii), we have $a+a'=\pi(c)=\pi(c')$.  
\end{proof}

The first step is to show that there exist many $c\in \Gamma^{s}_{1}(\Vk,M)$ such that the level set $S(c)$ has a large cardinality.

\begin{prop}\label{2:gbsg1}
	Suppose that $d\geq \max\{s+3,13\}$ and $p\gg_{\d,d} 1$.
	There exist  $H'\subseteq H$ with $\vert H'\vert\geq \d K^{d}/2$ and for each $h\in H'$ some $c_{h}\in \Gamma^{s}_{1}(\Vk,M)$ with $\pi(c_{h})=h$ such that $S(c_{h})$ is  of cardinality $\gg_{\d,s}K^{d}$ (in particular $S(c_{h}), h\in H'$ are disjoint).	
\end{prop}	
\begin{proof}
	By Proposition \ref{2:glargepart},  there exist $H'\subseteq H$ with $\vert H'\vert\geq \d K^{d}/2$ and a set $C'_{h}\subseteq H_{1}\times H_{1}$ for all $h\in H'$ such that for all $h\in H'$,
	\begin{itemize}
		\item $\vert C'_{h}\vert\gg_{\d,s}K^{d}$;
		\item $h_{1}+h_{2}=h$ for all $(h_{1},h_{2})\in C'_{h}$;
		\item $\tilde{\xi}(h_{1})\+ \tilde{\xi}(h_{2})\sim \tilde{\xi}(h'_{1})\+ \tilde{\xi}(h'_{2})$ for all $(h_{1},h_{2}),(h'_{1},h'_{2})\in C'_{h}$.
	\end{itemize}
	
	Let $C_{h}:=\{\tilde{\xi}(h_{1})\+ \tilde{\xi}(h_{2})\colon (h_{1},h_{2})\in C'_{h}\}$.
	Then $C_{h}$ is an equivalence class.
	By Lemma \ref{2:a+b}, $C_{h}\subseteq \Gamma^{s}_{2}(\Vk,M)$.
	 Since $\vert C'_{h}\vert\gg_{\d,s}K^{d}$,
	 for any $h\in H'$ and subspace $Y$ of $\V$ of dimension $s+2$ and containing $\iota(h)$, since $d\geq s+3$, 
	 it is impossible to have $\sp_{\F_{p}}\{\iota(h_{1}),\iota(h_{2})\}\cap Y\neq \sp_{\F_{p}}\{\iota(h)\}$ for all $(h_{1},h_{2})\in C'_{h}$. 
	By Proposition \ref{2:gwts}, since $d\geq 11$, we have that $C_{h}$ is a strong equivalence class. So there exist $c_{h}\in \Gamma^{s}_{1}(\Vk,M)$ such that $\tilde{\xi}(h_{1})\+ \tilde{\xi}(h_{2})\sim c_{h}$ for all $h\in H'$ and $(h_{1},h_{2})\in C'_{h}$. So we have that  $\vert S(c_{h})\vert\geq \vert C'_{h}\vert\gg_{\d,s}  K^{d}$, Since $\pi(c_{h})=h$, we are done. 
	\end{proof}

Let $H'$ and $c_{h},h\in H'$ satisfy the requirement of Proposition \ref{2:gbsg1} and set $\mathcal{X}:=\{c_{h}\colon h\in H'\}$. 
The second step is to show that there exists a large subset $\mathcal{X'}$ of $\mathcal{X}$ such that the relation graph of $\mathcal{X'}\+ \mathcal{X'}$ has a  small density dependence number.
Consider the graph $G_{1}=(\Vk,E_{1})$, where $E_{1}=\cup_{c\in\mathcal{X}}S(c)$. 
We say that an edge in $E_{1}$ has \emph{color} $c$ if it belongs to $S(c)$. By Proposition \ref{2:gbsg1},
\begin{equation}\label{2:gbs5}
\vert E_{1}\vert=\sum_{c\in\mathcal{X}}\vert S(c)\vert\gg_{\d,s}K^{2d}.
\end{equation} 
For $W_{1},W_{2}\subseteq \Vk$, let $E(W_{1},W_{2})$ denote the set of edges in $E_{1}$ joining a vertex in $W_{1}$ and a vertex in $W_{2}$.  For $W\subseteq \Vk$ and $c\in \mathcal{X}$, let $W(c)$ denote the set of all $x\in W$ such that $(x,y)\in S(c)$ for some $y\in \Vk$.
Following the method similar to pages 265--266 of \cite{BS94}, there exist $W_{1},W_{2}\subseteq \Vk$ and $\mathcal{X}'\subseteq \mathcal{X}$ such that 
$\vert \mathcal{X}'\vert\gg_{\d,d,s}K^{d}$, and that for all $c_{1}, c_{2}\in \mathcal{X}'$, $\vert E(W_{1}(c_{1}),W_{2}(c_{2}))\vert\gg_{\d,s}K^{2d}$. Moreover, for all $c\in\mathcal{X}'$, $E(W_{1}(c_{1}),W_{2}(c_{2}))$ contains at least one edge of color $c$.

Consider the bipartite graph $G_{2}=(X_{2},Y_{2},E_{2})$, where $X_{2}=\mathcal{X}'\times \mathcal{X}'$, $Y_{2}=\mathcal{X}'\times \Vk\times \Vk$, and $\{(c_{1},c_{2}), (c,h_{1},h_{2})\}\in E_{2}$ if 
there exists  $(a_{1},a_{2})\in E(W_{1}(c_{1}),W_{2}(c_{2}))$ such that $(a_{1},h_{1})\in S(c_{1})$, $(a_{2},h_{2})\in S(c_{2})$ and $(a_{1},a_{2})\in S(c)$. For $c_{1},c_{2}\in \mathcal{X}'$, let $\Xi(c_{1},c_{2})$ denote the set of all edges in $E_{2}$ with $(c_{1},c_{2})$ as one of the vertices.

\begin{lem}\label{2:gbsg4}
	We have that $\vert \Xi(c_{1},c_{2})\vert=\vert E(W_{1}(c_{1}),W_{2}(c_{2}))\vert\gg_{\d,s}K^{2d}$. 
\end{lem}	
\begin{proof}
	Fix $c_{1},c_{2}\in \mathcal{X}'$ and let $(a_{1},a_{2})\in E(W_{1}(c_{1}),W_{2}(c_{2}))$. By definition and Lemma \ref{2:gbsg2}, there exist unique $h_{1}, h_{2}\in\Vk$ such that $(a_{1},h_{1})\in S(c_{1})$ and $(a_{2},h_{2})\in S(c_{2})$. By the definition of $E_{1}$, we have that $(a_{1},a_{2})\in S(c)$ for some unique $c\in\mathcal{X}'$. Therefore, $\lambda\colon (a_{1},a_{2})\to ((c_{1},c_{2}),(c,h_{1},h_{2}))$ is a well defined map from $E(W_{1}(c_{1}),W_{2}(c_{2}))$ to $\Xi(c_{1},c_{2})$. 
	
	It suffices to show that $\lambda$ is a bijection. Suppose that both $\lambda(a_{1},a_{2})$ and $\lambda(a'_{1},a'_{2})$ equal to $((c_{1},c_{2}),(c,h_{1},h_{2}))$. Then $(a_{1},h_{1}),(a'_{1},h_{1})\in S(c)$. By Lemma \ref{2:gbsg2}, $a_{1}=a'_{1}$. Similarly, $a_{2}=a'_{2}$. So $\lambda$ is injective. On the other hand, by the definition of $E_{2}$, $\lambda$ is subjective. This finishes the proof.
\end{proof}

For $h\in \Vk$, denote $$X_{2,h}:=\{c_{1}\+ c_{2}\colon (c_{1},c_{2})\in \mathcal{X}'\times \mathcal{X}', \pi(c_{1})+\pi(c_{2})=h\}$$ and $$Y_{2,h}:=\{(c,h_{1},h_{2})\in \mathcal{X}'\times \Vk\times \Vk\colon \pi(c)+h_{1}+h_{2}=h\}.$$ It is easy to see that there is no edge in $E_{2}$ connecting $X_{2,h}$ and $Y_{2,h'}$ if $h\neq h'$. Therefore, $G_{2}$ has a disjoint partition $G_{2,h}=(X_{2,h},Y_{2,h},E_{2,h}), h\in\Vk$.

Let $G'_{2,h}=(X_{2,h},E'_{2,h})$ be the graph given by $(c_{1}\+ c_{2},c'_{1}\+ c'_{2})\in E'_{2,h}$ if $c_{1}\+  c_{2}\sim c'_{1}\+  c'_{2}$.
The goal of the second step is equivalent of showing that $G'_{2,h}$ has density dependence number $O_{\d,s}(1)$.
We use the intersection method.
 We say that $$((D_{1},h_{1,1},h_{1,2}),\dots,(D_{s+4},h_{s+4,1},h_{s+4,2}))\in Y^{s+4}_{2,h}$$ is \emph{good} if $h_{i,1},h_{i,2}, 1\leq i\leq s+4$ are $p$-linearly independent. 
Let $Z_{2,h}$ denote the set of all good elements in $Y^{s+4}_{2,h}$. 
Consider the  bipartite graph $G''_{2,h}=(X_{2,h}, Z_{2,h}, E''_{2,h})$  such that $(t,(r_{1},\dots,r_{s+4}))\in E''_{2,h}$ if $(t,r_{1}),\dots,(t,r_{s+4})\in E_{2,h}$, where $r_{i}\in Y_{2,h}$.
 
\begin{prop}\label{2:gcf1}
	Let $G'_{2,h}$ and $G''_{2,h}$ be defined as above. If $p\gg_{\d,d} 1$ and $d\geq 2s+15$, then  $G''_{2,h}$ is an $O_{\d,s}(1)^{-1}$-dense auxiliary graph of  $G'_{2,h}$. Therefore, $\dep(G'_{2,h})=O_{\d,s}(1)$.
\end{prop}	
\begin{proof}
	We first show that $G''_{2,h}$ is an  auxiliary graph of  $G'_{2,h}$.
	Let $t_{1}=c_{1}\+ c_{2}$, $t_{2}=c'_{1}\+ c'_{2}$ and $$w=(r_{1}=(D_{1},h_{1,1},h_{1,2}),\dots,r_{s+4}=(D_{s+4},h_{s+4,1},h_{s+4,2}))\in Z_{2,h}$$ be such that $(t_{1},w),(t_{2},w)\in E''_{2,h}$. Our goal is to show that $(t_{1},t_{2})\in E'_{2,h}$. Assume that $c_{j}=(\pi(c_{j}),J^{M}_{\iota(\pi(c_{j}))}+f_{j})$ and $c'_{j}=(\pi(c_{j}),J^{M}_{\iota(\pi(c_{j}))}+f_{j'})$ for $j=1,2$.
	For each $1\leq i\leq s+4$,
	since $((c_{1},c_{2}),r_{i})\in E_{2,h}$, by  definition, there exist $a_{1},a_{2}\in\Vk$ such that $(a_{1}, h_{i,1})\in S(c_{1}), (a_{2}, h_{i,2})\in S(c_{2})$ and $(a_{1}, a_{2})\in S(D_{i})$. 
So
	\begin{equation}\nonumber
	\begin{split}
	&\quad f_{1}+ f_{2}\equiv(\xi(a_{1})+\xi(h_{i,1}))+(\xi(a_{2})+\xi(h_{i,2})) 
	\mod J^{M}_{\iota(\pi(c_{1})),\iota(\pi(c_{2})),\iota(a_{1}),\iota(a_{2}),\iota(h_{i,1}),\iota(h_{i,2})}.
	\end{split}
	\end{equation}
	Since $\pi(c_{1})=a_{1}+h_{i,1}$, $\pi(c_{2})=a_{2}+h_{i,2}$ and $\pi(c_{1})+\pi(c_{2})=h$,  we have that 
		\begin{equation}\label{2:gsf2}
	\begin{split}
 f_{1}+ f_{2}\equiv \xi(a_{1})+\xi(a_{2})+\xi(h_{i,1})+\xi(h_{i,2})
	\mod J^{M}_{\iota(h),\iota(\pi(c_{1})),\iota(h_{i,1}),\iota(h_{i,2})}.
	\end{split}
	\end{equation}
	Since  $((c'_{1},c'_{2}),r_{i})$ belong to $E_{2,h}$, a symmetric argument shows that 
	\begin{equation}\label{2:gsf3}
	\begin{split}
	 f'_{1}+f'_{2}\equiv
	\xi(a'_{1})+\xi(a'_{2})+\xi(h_{i,1})+\xi(h_{i,2}) \mod J^{M}_{\iota(h),\iota(\pi(c'_{1})),\iota(h_{i,1}),\iota(h_{i,2})}
	\end{split}
	\end{equation}
	for some $(a'_{1},a'_{2})\in S(D_{i})$.
	Since $(a_{1},a_{2}),(a'_{1},a'_{2})\in S(D_{i})$, we have that 
		\begin{equation}\label{2:gsf4}
		\begin{split}
		\xi(a_{1})+\xi(a_{2})\equiv D_{i}\equiv \xi(a'_{1})+\xi(a'_{2}) \mod J^{M}_{\iota(a_{1}),\iota(a_{2}),\iota(a'_{1}),\iota(a'_{2})}.
		\end{split}
		\end{equation}	
	Combining (\ref{2:gsf2}), (\ref{2:gsf3}) and (\ref{2:gsf4}), we have that 
	\begin{equation}\nonumber
	\begin{split}
 f_{1}+f_{2}\equiv f'_{1}+f'_{2} \mod J^{M}_{\iota(h),\iota(\pi(c_{1})),\iota(\pi(c'_{1})),\iota(h_{i,1}),\iota(h_{i,2})}
	\end{split}
	\end{equation}
	for all $1\leq i\leq s+4$.
	Therefore,
	\begin{equation}\label{2:gsf5}
	\begin{split}
 f_{1}+f_{2}\equiv f'_{1}+f'_{2} \mod \cap_{i=1}^{s+4}J^{M}_{\iota(h),\iota(\pi(c_{1})),\iota(\pi(c'_{1})),\iota(h_{i,1}),\iota(h_{i,2})}.
	\end{split}
	\end{equation}

 Since $w$ is good, by Proposition \ref{2:gri} (setting $m=3$ and $r=2$), if $p\gg_{d} 1$ and $d\geq 2s+15$, then (\ref{2:gsf5}) implies that
	\begin{equation}\nonumber
	\begin{split}  
	f_{1}+f_{2}\equiv f'_{1}+f'_{2} \mod J^{M}_{\iota(h),\iota(\pi(c_{1})),\iota(\pi(c'_{1}))}.
	\end{split}
	\end{equation}
	So   $c_{1}\+ c_{2}\sim c'_{1}\+ c'_{2}$. In other words, $(t_{1},t_{2})\in E'_{2,h}$.

	\
	
	We next show that $G''_{2,h}$ is  $O_{\d,s}(1)^{-1}$-dense.  Let $t\in X_{2,h}$. By Lemma \ref{2:gbsg4}, the number of tuples $w=(r_{1},\dots,r_{s+4})\in Y^{s+4}_{2,h}$ such that $(t,r_{1}),\dots,(t,r_{s+4})\in E_{2,h}$ is $\gg_{\d,s}K^{(2s+8)d}$. Since $\vert Z_{2,h}\vert\leq K^{(2s+8)d}$ and there are at most $(2s+8)K^{(2s+8)d}/p^{d-2s-7}$ many $(r_{1},\dots,r_{s+4})\in Y^{s+4}_{2,h}$ which are not good by Lemma \ref{2:iiddpp},  we have that $G''_{2,h}$ is of density $\gg_{\d,s} 1$ provided that $p\gg_{\d,d} 1$ and $d\geq 2s+8$.	
\end{proof}

    We are now ready to complete the proof of Theorem \ref{2:gbig1}.
	Consider the graph $G_{3}=(H,E_{3})$, where $(h,h')\in E_{3}$ if and only if $(h,h')\in S(c)$ for some $c\in\mathcal{X}'$. By Proposition \ref{2:gbsg1}, $\vert E_{3}\vert\gg_{\d,s}K^{2d}$.
	Again we use the intersection method.
	Let $a_{h}$ denote the number of edges in $G_{3}$ with one of the vertices being $h$. The number of tuples $(h,h_{1},\dots,h_{s+4})\in (\Vk)^{s+5}$ such that $(h,h_{1}), \dots,(h,h_{s+4})\in E_{3}$ is at least
	$$\sum_{h\in\Vk}a^{s+4}_{h}\geq\frac{(\sum_{h\in\Vk}a_{h})^{s+4}}{K^{(s+3)d}}=\frac{(2\vert E_{3}\vert)^{s+4}}{K^{(s+3)d}}\gg_{\d,s}K^{(s+5)d}.$$
	For each $w=(h_{1},\dots,h_{s+4})\in(\Vk)^{s+4}$, let $W_{w}$ denote the set $h\in \Vk$ such that $(h,h_{1}),\dots,$ $(h,h_{s+4})\in E_{3}$. Since $\vert W_{w}\vert\leq K^{d}$ for all $w$, the number of $w\in (\Vk)^{s+4}$ such that $\vert W_{w}\vert\gg_{\d,s}K^{d}$ is $\gg_{\d,s}K^{(s+4)d}$.
	
	We say that $w=(h_{1},\dots,h_{s+4})\in(\Vk)^{s+4}$ is \emph{good} if $h_{1},\dots,h_{s+4}$ are $p$-linearly independent.
	Since the number of $w\in (\Vk)^{s+4}$ which is not good is at most 
	$(s+4)K^{(s+4)d}/p^{d-s-3}$
	 by Lemma \ref{2:iiddpp},
	if $p\gg_{\d,d}1$ and $d\geq s+4$, then we may choose a good $w\in (\Vk)^{s+4}$ such that $\vert W_{w}\vert\gg_{\d,s}K^{d}$. 	 
	
	Fix such a $w=(h_{1},\dots,h_{s+4})$ and let $H'=W_{w}$. Then $\vert H'\vert\gg_{\d,s}K^{d}$. It remains to show that $\dep(\tilde{\xi}(H')\+  \tilde{\xi}(H'))=O_{\d,s}(1)$. Let $G_{4}=(\tilde{\xi}(H')\+  \tilde{\xi}(H'),E_{4})$ be the relation graph of $\tilde{\xi}(H')\+  \tilde{\xi}(H')$. Clear, we have a partition
	 $G_{4}=\cup_{h\in\V}(V_{4,h},E_{4,h})$, where $V_{4,h}:=\pi^{-1}(h)\cap(\tilde{\xi}(H')\+  \tilde{\xi}(H'))$. 
	
		For all $x_{1},x_{2}\in H'$ with $x_{1}+x_{2}=h$, by definition, for $1\leq i\leq s+4$, $(x_{1},h_{i})\in S(c_{i,1})$ and $(x_{2},h_{i})\in S(c_{i,2})$ for some $c_{i,1},c_{i,2}\in \mathcal{X}'$.
		Then $\tilde{\xi}(x_{1})\+  \tilde{\xi}(x_{2})\+ \tilde{\xi}(h_{i})\+ \tilde{\xi}(h_{i})\sim c_{i,1}\+ c_{i,2}$. Since $\pi(c_{i,1})+\pi(c_{i,2})=x_{1}+x_{2}-2h_{i}=h-2h_{i}$, we have that $c_{i,1}\+ c_{i,2}\in X_{2,h-2h_{i}}$.

	Let $G'_{4,h}=(V_{4,h},Z_{2,h-2h_{1}}\times\dots\times Z_{2,h-2h_{s+4}},E'_{4,h})$ be the bipartite graph such that $(u,(z_{1},\dots,$ $z_{s+4}))\in E'_{4,h}$ if and only if there exist $x_{1},x_{2}\in H'$ with $x_{1}+x_{2}=h$, $c_{i,1},c_{i,2}\in\mathcal{X}', 1\leq i\leq s+4$ such that $u=\tilde{\xi}(x_{1})\+ \tilde{\xi}(x_{2})$, $(x_{1},h_{i})\in S(c_{i,1})$, $(x_{2},h_{i})\in S(c_{i,2})$, $(c_{i,1}\+ c_{i,2},z_{i})\in E''_{2,h-2h_{i}}$ for all $1\leq i\leq s+4$. For convenience for say that $c_{i,1},c_{i,2},1\leq i\leq s+4$ are the \emph{parameters} associated to $(u,(z_{1},\dots,z_{s+4}))$.	
	Note that for all $x_{1},x_{2}\in H'$ with $x_{1}+x_{2}=h$, by the construction of $H'$, 
	$(x_{1},h_{i})\in S(c_{i,1})$, $(x_{2},h_{i})\in S(c_{i,2})$ for some $c_{i,1},c_{i,2}\in\mathcal{X}'$ for all $1\leq i\leq s+4$.
	On the other hand, the number of choices for $z_{i}$ with $(c_{i,1}\+ c_{i,2},z_{i})\in E''_{2,h-2h_{i}}$ is $\gg_{\d,s} \vert Z_{2,h-2h_{i}}\vert$ since $G''_{2,h}=(X_{2,h},Z_{2,h},E''_{2,h})$ is an $O_{\d,s}(1)^{-1}$-dense auxiliary graph of $G'_{2,h}$ by Proposition \ref{2:gcf1}.
	This means that $G'_{4,h}$ is an $O_{\d,s}(1)^{-1}$-dense graph.

	So in order to show that $\dep(\tilde{\xi}(H')\+ \tilde{\xi}(H'))=O_{\d,s}(1)$, it suffices to show that $G'_{4,h}$ is an auxiliary graph of $(V_{4,h},E_{4,h})$ for all $h\in\Vk$. Let  $x_{1},x_{2},x'_{1},x'_{2}\in H'$ with $x_{1}+x_{2}=x'_{1}+x'_{2}=h$ and  $ \bold{z}:=(z_{1},\dots,z_{s+4}), z_{i}\in Z_{2,h-2h_{i}}, 1\leq i\leq s+4$ be such that $(\tilde{\xi}(x_{1})\+ \tilde{\xi}(x_{2}), \bold{z}), (\tilde{\xi}(x_{1})\+ \tilde{\xi}(x_{2}), \bold{z})\in E'_{4,h}$.
Let  $c_{i,1},c_{i,2},1\leq i\leq s+4$ be the parameters associated to $(\tilde{\xi}(x_{1})\+ \tilde{\xi}(x_{2}),\bold{z})$, and $c'_{i,1},c'_{i,2},1\leq i\leq s+4$ be the parameters associated to $(\tilde{\xi}(x'_{1})\+ \tilde{\xi}(x'_{2}),\bold{z})$.
	For all $1\leq i\leq s+4$ and $j=1,2$, since $\tilde{\xi}(x_{j})\+ \tilde{\xi}(h_{i})\sim c_{i,j}$ and $c_{i,j}\in\Gamma^{s}_{1}(\Vk,M)$, we have that
	$$\tilde{\xi}(x_{1})\+ \tilde{\xi}(x_{2})\+ \tilde{\xi}(h_{i})\+ \tilde{\xi}(h_{i})\sim c_{i,1}\+ c_{i,2}$$
	and 
	$$\tilde{\xi}(x'_{1})\+ \tilde{\xi}(x'_{2})+\tilde{\xi}(h_{i})\+ \tilde{\xi}(h_{i})\sim c'_{i,1}\+ c'_{i,2}.$$
	Since $(c_{i,1}\+ c_{i,2},z_{i}),(c'_{i,1}\+ c'_{i,2},z_{i})$ both belong to $E''_{2,h-2h_{i}}$,  by Proposition \ref{2:gcf1}, we have that $c_{i,1}\+ c_{i,2}\sim c'_{i,1}\+ c'_{i,2}$ and so 
		$$\tilde{\xi}(x_{1})\+  \tilde{\xi}(x_{2})\+ \tilde{\xi}(h_{i})\+ \tilde{\xi}(h_{i})\sim\tilde{\xi}(x'_{1})\+  \tilde{\xi}(x'_{2})\+ \tilde{\xi}(h_{i})\+ \tilde{\xi}(h_{i}).$$
		So
		$$\xi(x_{1})+  \xi(x_{2})\equiv\xi(x'_{1})+\xi(x'_{2})\mod \cap_{i=1}^{s+4} J^{M}_{\iota(x_{1}),\iota(x_{2}),\iota(x'_{1}),\iota(x'_{2}),\iota(h_{i})}.$$
		Since $w=(h_{1},\dots,h_{s+4})$ is good and $\dim(\sp_{\F_{p}}\{\iota(x_{1}),\iota(x_{2}),\iota(x'_{1}),\iota(x'_{2})\})\leq 3$ (note that $x_{1}+x_{2}=x'_{1}+x'_{2}=\pi(v_{1})-2h_{1}$), by Proposition \ref{2:gri} (setting $m=3$ and $r=1$), if $d\geq 13$, then
		$$\xi(x_{1})+  \xi(x_{2})\equiv\xi(x'_{1})+\xi(x'_{2})\mod   J^{M}_{\iota(x_{1}),\iota(x_{2}),\iota(x'_{1}),\iota(x'_{2})}.$$
		So by Lemma \ref{2:spsp1},
		$\tilde{\xi}(x_{1})\+  \tilde{\xi}(x_{2})\sim \tilde{\xi}(x'_{1})\+  \tilde{\xi}(x'_{2})$
	and thus $G'_{4,h}$ is an auxiliary graph of $V_{4,h}$. This completes the proof of Theorem \ref{2:gbig1}.

\subsection{Ruzsa's triangle inequality}

Ruzsa's triangle inequality is an important tool in additive combinatorics, which says that for all subsets $A,B,C$ of an abelian group $G$, we have that
$\vert A-C\vert\leq \frac{\vert A-B\vert\cdot\vert B-C\vert}{\vert B\vert},$
or equivalently, $d(A,C)\leq d(A,B)+d(B,C)$ where the ``distance function" is defined as  $d(A,B):=\log \frac{\vert A-B\vert}{\vert A\vert^{1/2}\vert B\vert^{1/2}}$ (see for example Lemma 2.6 of \cite{TV06}). It is natural to ask if Ruzsa's triangle inequality also applies to shifted modules. In our setting, the ``distance function" $d(A,B)$ is replaced by the density dependence number of the set $A\- B$.

\begin{prop}[Ruzsa's quasi triangle inequality for $\Gamma^{s}(\Vk,M)$]\label{2:grt}
	Let $d,D,k,K\in\N_{+}$, $s\in\N$, $\d>0$, $p\gg_{\d,d,D} 1$ be a prime dividing $K$, and $M\colon\V\to\F_{p}$ be a non-degenerate quadratic form. Let $A,C\subseteq \Gamma^{s}_{k}(\Vk,M)$ and $B\subseteq \Gamma^{s}_{1}(\Vk,M)$ with $\vert B\vert\geq \d K^{d}$ such that $\pi\colon B\to \Vk$ is an injection. Suppose that  $\dep(A\- B)$, $\dep(B\- C)\leq D$. If $d\geq \max\{4k+s,8k+5\}$, then we have $\dep(A\- C)=O_{\d,D,k,s}(1)$.
\end{prop}

Note that Proposition \ref{2:grt} does not imply an inequality of the form $\dep(A\- B)+\dep(B\- C)$ $\geq\dep(A\- C)$, but it provides us with a quasi triangle inequality in the sense that the smallness of  $\dep(A\- B)$ and $\dep(B\- C)$ implies the  smallness of $\dep(A\- C)$. 

We remark that the set $B$ in Proposition \ref{2:grt} is very restrictive compared with the sets $A$ and $C$. It seems that our method can be adapted to prove Proposition \ref{2:grt} with less restrictions imposed on the set $B$. However, we do not pursuit this improvement in this paper since  Proposition \ref{2:grt} is good enough for our purposes.

The original Ruzsa triangle inequality is a direct consequence of the injectivity  of the map $(A-C)\times B\mapsto (A-B)\times (B-C), (a-c,b)\mapsto (a-b,b-c)$. However, this method does not apply to Proposition \ref{2:grt} and so we need to use a different method.

\begin{proof}[Proof of Proposition \ref{2:grt}]
	For convenience denote $X_{h}:=\pi(h)^{-1}\cap X$ for all $h\in\Vk$ and $X\subseteq \Gamma^{s}(\Vk,M)$. Since $\dep(A\- B)$, $\dep(B\- C)\leq D$, the relation graph of each $(A\- B)_{h}$ has a $D^{-1}$-dense auxiliary graph $G_{A,h}=((A\- B)_{h},Y_{A,h},E_{A,h})$, and the relation graph of each $(B\- C)_{h}$ has a $D^{-1}$-dense auxiliary graph $G_{C,h}=((B\- C)_{h},Y_{C,h},E_{C,h})$.
By Lemma \ref{2:duplicate}, we may assume without loss of generality that all of $Y_{A,h}, Y_{C,h}, h\in\Vk$ have the same cardinality, which we denote by $L$.
	
	We need to use the intersection method.
		We say that $(b_{1},\dots,b_{4k+s})\in B^{4k+s}$ is \emph{good} if $\pi(b_{1}),\dots,\pi(b_{4k+s})$ $p$-are linearly independent. Consider the bipartite graph $$G'_{h}=((A\- C)_{h},(\sqcup_{h'\in\Vk}Y_{A,h'}\times Y_{C,h-h'})^{4k+s},E'_{h})$$    defined in the following way:
		 for any $u\in (A\- C)_{h}$, $h_{i}\in\Vk$,
		$y_{i}=(y_{i,1},y_{i,2})\in Y_{A,h_{i}}\times Y_{C,h-h_{i}}, 1\leq i\leq 4k+s$,
		set $(u,(y_{1},\dots,y_{4k+s}))\in E'_{h}$ if there exist $a\in A, c\in C$, and some good $(b_{1},\dots,b_{4k+s})\in B^{4k+s}$ such that $u=a\- c$ (which enforces $\pi(a)-\pi(c)=h$) and that $(a\- b_{i},y_{i,1})\in E_{A,h_{i}}$ and $(b_{i}\- c,y_{i,2})\in E_{C,h-h_{i}}$ for all $1\leq i\leq 4k+s$ (which enforces $\pi(b_{i})=\pi(a)-h_{i}$).
	For convenience we say that $a,c,b_{1},\dots,b_{4k+s}$ are the \emph{parameters} associated to the edge $(u,(y_{1},\dots,y_{4k+s}))$ (note that the parameters of an edge is not necessarily unique).

	Note that if $d\geq 4k+s$ and $p\gg_{\d,d} 1$, then by Lemma \ref{2:iiddpp}, the number of good $(b_{1},\dots,b_{4k+s})\in B^{4k+s}$	 is at least $\vert B\vert^{4k+s}/2\geq \d^{4k+s}K^{(4k+s)d}/2$. 
	On the other hand, for each $u=a\- c$ in $(A\- C)_{h}$ and good $(b_{1},\dots,b_{4k+s})\in B^{4k+s}$, 	writing $h_{i}=\pi(a)-\pi(b_{i})$, there exist at least $D^{-2}\vert Y_{A,h_{i}}\vert\cdot \vert Y_{C,h-h_{i}}\vert$ many $(y_{i,1},y_{i,2})\in Y_{A,h_{i}}\times Y_{C,h-h_{i}}$ such that  $(a\- b_{i},y_{i,1})\in E_{A,h_{i}}$ and $(b_{i}\- c,y_{i,2})\in E_{C,h-h_{i}}$ for all $1\leq i\leq 4k+s$. Since $\pi\colon B\to\Vk$ is an injection, the tuples $(h_{1},\dots,h_{4k+s})$ associated with different $(b_{1},\dots,b_{4k+s})$ are different. This means that every $u\in (A\- C)_{h}$ is the vertices of at least $(\d^{4k+s}K^{(4k+s)d}/2)\cdot(D^{-2}L^{2})^{4k+s}$ edges in $E'_{h}$. Since the set $(\sqcup_{h'\in\V}Y_{A,h'}\times Y_{C,h-h'})^{4k+s}$ is of cardinality $(K^{d}L^{2})^{4k+s}$, we have that $G'_{h}$ is $O_{\d,D,k,s}(1)^{-1}$-dense.
		To show that  $\dep(A\- C)=O_{\d,D,k,s}(1)$, it suffices to show that $G'_{h}$ is an auxiliary graph of the relation graph of  $(A\- C)_{h}$ for all $h\in\Vk$.

		Let $u,u'\in (A\- C)_{h}$,  $h_{i}\in\Vk$, and
		$y_{i}=(y_{i,1},y_{i,2})\in Y_{A,h_{i}}\times Y_{C,h-h_{i}}, 1\leq i\leq 4k+s$  be such that $(u,(y_{1},\dots,y_{4k+s}))$, $(u',(y_{1},\dots,y_{4k+s}))\in E'_{h}$, with parameters $a,c,b_{1},\dots,b_{4k+s}$ and $a',c',b'_{1},\dots,b'_{4k+s}$ respectively.
			For $y=a,c,a'$ and $c'$, assume that $y=(\pi(y),J^{M}_{V_{y}}+f_{y})$. 
		Since $(a\- b_{i},y_{i,1}),(a'\- b'_{i},y_{i,1})\in E_{A,h_{i}}$ and  $(b_{i}\- c,y_{i,1}),(b'_{i}\- c',y'_{i,1})\in E_{C,h-h_{i}}$, we have that 
			$$f_{a}-f_{c}\equiv (f_{a}-f_{b_{i}})+(f_{b_{i}}-f_{c})\equiv(f_{a'}-f_{b'_{i}})+(f_{b'_{i}}-f_{c'})\equiv f_{a'}-f_{c'} \mod
			J^{M}_{V+\sp_{\F_{p}}\{\iota(\pi(b_{i})),\iota(\pi(b'_{i}))\}},$$
			where $V:=V_{a}+V_{c}+V_{a'}+V_{c'}$.
			Since $\pi(a)-\pi(b_{i})=\pi(a')-\pi(b'_{i})=h_{i}$, we have that $\iota(\pi(b'_{i}))\in V_{a}+V_{a'}\sp_{\F_{p}}\{\iota(\pi(b_{i}))\}$. So 	
			$$f_{a}-f_{c}\equiv f_{a'}-f_{c'} \mod
			J^{M}_{V+\sp_{\F_{p}}\{\iota(\pi(b_{i}))\}}.$$
			Ranging $i$ over $1,\dots, 4k+s$, we have that 
			\begin{equation}\label{2:gree1}
			f_{a}-f_{c}\equiv f_{a'}-f_{c'} \mod \cap_{i=1}^{4k+s}J^{M}_{V+\sp_{\F_{p}}\{\iota(\pi(b_{i}))\}}.
			\end{equation}
			
			Since $h\in (V)$, we have that
			$\dim(V)\leq 4k-1$. Since $\pi(b_{1}),\dots,\pi(b_{4k+1})$ are $p$-linearly independent, by Proposition \ref{2:gri} (setting $m=4k-1$ and $r=1$), since $d\geq 8k+5$ and $p\gg_{d} 1$, 
			we deduce from (\ref{2:gree1})  that
			$a\- c\sim a'\- c'$. So $G'_{h}$ is an auxiliary graph of the relation graph of  $(A\- C)_{h}$ for all $h\in\Vk$ and we are done.
\end{proof}

\subsection{Pl\"unnecke-Rusza theorem}

Let $A$ be a subset of an abelian group $G$. The Pl\"unnecke-Rusza theorem says that for any fixed $k,\ell\in\N_{+}$, the smallness of $\vert A+A\vert/\vert A\vert$ implies the smallness of $\vert kA-\ell A\vert/\vert A\vert$
(see for example Corollary 6.29 \cite{TV06}).
In our setting,
for $A\in\Gamma^{s}(\Vk,M)$ and $k,\ell\in\N$ with $(k,\ell)\neq (0,0)$, we use the notation $kA\- \ell A$ to denote the set $$(A\+ A\+ \dots\+ A)\- (A\+ A\+ \dots\+ A),$$ where there are $k$ copies of $A$ in the first bracket and $\ell$ copies of $A$ in the second bracket. 
It is natural to ask whether the smallness of $\dep(A\+ A)$ implies the smallness of $\dep(kA\- \ell A)$. 

\begin{conj}[Pl\"unnecke-Rusza theorem for $\Gamma^{s}(\Vk,M)$]\label{2:00ga3}
	Let $d,D,K\in\N_{+}$, $k,\ell,s\in \N$ with $(k,\ell)\neq (0,0)$, $\d>0$, and $p\gg_{\d,d,D} 1$ be a prime dividing $K$. Let $M\colon\V\to\F_{p}$ be a non-degenerate quadratic form. Let $A\in\Gamma^{s}_{1}(\Vk,M)$ be such that $\vert A\vert\geq \d K^{d}$ and that the map $\pi\colon A\to\Vk$ is injective. Suppose that $\dep(A\+ A)\leq D$. If $d\gg_{k,\ell,s} 1$, then   $\dep(kA\- \ell A)= O_{\d,d,D}(1)$.
\end{conj}	

Thanks to Rusza's quasi triangle inequality Proposition \ref{2:grt}, to prove Conjecture \ref{2:00ga3}, it suffices to show that  the smallness of $\dep(A\+ A)$ implies the smallness of $\dep(A\- A\+ A)$ (which is known as the Green-Ruzsa inequality):

\begin{conj}[Green–Ruzsa inequality  for $\Gamma^{s}(\Vk,M)$]\label{2:00ga31}
	Let $d,D,K\in\N_{+}$, $s\in\N$, $\d>0$ and $p\gg_{\d,d,D} 1$ be a prime dividing $K$. Let $M\colon\V\to\F_{p}$ be a non-degenerate quadratic form.
	Let  $A\in\Gamma^{s}_{1}(\Vk,M)$ be such that $\vert A\vert\geq \d K^{d}$ and that the map $\pi\colon A\to\Vk$ is injective. Suppose that $\dep(A\+ A)\leq D$. If $d\gg_{s} 1$, then  $\dep(A\- A\+ A)=O_{\d,d,D}(1)$.
\end{conj}

Conjecture \ref{2:00ga3} can be proved by a simple induction argument combining the Rusza's quasi triangle inequality (Proposition \ref{2:grt}) and Conjecture \ref{2:00ga31}  (see the proof of Proposition \ref{2:ga3} for details).

Unfortunately, we were unable to prove Conjecture \ref{2:00ga31}. However, we have the following weaker version of Conjecture \ref{2:00ga31}, which turns out to be good enough for the purposes of this paper.

\begin{prop}[Weak Green–Ruzsa inequality  for $\Gamma^{s}(\Vk,M)$]\label{2:ga31}
		Let $d,D,K\in\N_{+}$, $s\in\N$, $\d>0$ and $p\gg_{\d,d,D} 1$ be a prime dividing $K$. Let $M\colon\V\to\F_{p}$ be a non-degenerate quadratic form.
		Let  $A\in\Gamma^{s}_{1}(\Vk,M)$ be such that $\vert A\vert\geq \d K^{d}$ and that the map $\pi\colon A\to\Vk$ is injective. Suppose that $\dep(A\+ A)\leq D$. If $d\geq \max\{s+7,19\}$ , then there exists $A'\subseteq A$ with $\vert A'\vert\geq \vert A\vert/2$ such that  $\dep(A'\- A'\+ A')=O_{\d,D,s}(1)$.
\end{prop}

By Propositions \ref{2:grt} and \ref{2:ga31}, we have the following weak version of Pl\"unnecke-Rusza theorem:

\begin{prop}[Weak Pl\"unnecke-Rusza theorem for $\Gamma^{s}(\Vk,M)$]\label{2:ga3}
		Let $d,D,K\in\N_{+}$, $k,\ell,s\in \N$ with $(k,\ell)\neq (0,0)$, $\d>0$, and $p\gg_{\d,d,D} 1$ be a prime dividing $K$. Let $M\colon\V\to\F_{p}$ be a non-degenerate quadratic form. Let $A\in\Gamma^{s}_{1}(\Vk,M)$ be such that $\vert A\vert\geq \d K^{d}$ and that the map $\pi\colon A\to\Vk$ is injective. Suppose that $\dep(A\+ A)\leq D$. If $$d\geq \max\{s+7,19,4k+s,8k+5,4\ell+s,8\ell+5\},$$
		 then there exists $A'\subseteq A$ with $\vert A'\vert\geq \d K^{d}/2$ such that $\dep(kA'\- \ell A)\leq O_{\d,D,k,\ell,s}(1)$.
\end{prop}

\begin{proof}[Proof of Proposition \ref{2:ga3} assuming Proposition \ref{2:ga31}]
	By Proposition \ref{2:ga31}, there exists $A'\subseteq A$ with $\vert A'\vert\geq \d p^{d}/2$ such that  $\dep(A'\- A'\+ A')=O_{\d,D,s}(1)$.
	Since $3A'=(A'\+ A')\- (\- A')$, and both $(A'\+ A')\- A'=A'\- A'\+ A'$ and $A'\- (\- A')=A'\+ A'$ has density dependence number $O_{\d,D,s}(1)$,
	by Proposition \ref{2:grt}, $3A'$ has density dependence number $O_{\d,D,s}(1)$.
	
	It is obvious that $\dep(A')=1$.
	Suppose that we have shown  $\dep(iA')=O_{\d,D,i,s}(1)$ for some $i\in \N_{+}$. Since $(i+1)A'=(i-1)A'\- (\- 2A')$, and both $(i-1)A'\- (\- A')=iA'$ and $2A'\- (\- A')=3A'$ have density dependence number $O_{\d,D,i,s}(1)$ by induction hypothesis, by Proposition \ref{2:grt}, we have $\dep((i+1)A')=O_{\d,D,i,s}(1)$. So we have that $\dep(kA')=O_{\d,D,k,s}(1)$ for all $k\in\N$.
	
	Finally, for all $k,\ell\in \N$, since $kA'\- (\- A')=(k+1)A'$ and $(\- A')\- \ell A'=\- (\ell+1)A'$ have  density dependence numbers $O_{\d,d,D,k,s}(1)$ and $O_{\d,D,\ell,s}(1)$ respectively, by Proposition \ref{2:grt}, we have $\dep(kA'\- \ell A')=O_{\d,D,k,\ell,s}(1)$.
\end{proof}

The rest of the section is devoted to the proof of Proposition \ref{2:ga31}.

For every $B\subseteq A$,
we construct a subset $X_{B}\subseteq B$ as follows. We start with $X_{B}$ being the empty set. If there exists some $b\in B$ such that the number of $a\in B$ for which $a\+ b\sim u$ for some $u\in X_{B}\+ B$ is at most $\vert B\vert/2$, 
then we add one such $b$ in $X_{B}$ and repeat this process. If such $b$ does not exist, we then terminate the construction. The construction of the set $X_{B}$ is inspired by the proof of Lemma 2.17 of \cite{TV06}. However, as we shall see later, the set $X_{B}$ in our setting is much more difficult to handle compared with Lemma 2.17 of \cite{TV06}. 

Suppose that we added $b\in B$ into $X_{B}$ at some step. Let $Q$ denote the set of $a\in B$ such that $a\+ b\not\sim u$ for any $u\in X_{B}\+ B$ (here $X_{B}$ denotes the set before adding $b$ into it).
Then by definition $\vert Q\vert\geq \vert B\vert/2$.  
Since  the map $\pi\colon A\to\Vk$ is injective, we have
$$(\pi^{-1}(\pi(a\+ b))\cap((X_{B}\cup\{b\})\+ B))\backslash(\pi^{-1}(\pi(a\+ b))\cap(X_{B}\+ B))=\{a\+ b\}$$ for all
$a\in Q$.
By Lemma \ref{2:lonely}, we have $$\dep(\pi^{-1}(\pi(a\+ b))\cap((X_{B}\cup\{b\})\+ B))\geq \dep(\pi^{-1}(\pi(a\+ b))\cap(X_{B}\+ B))+1$$
for all
$a\in Q$.
  Since the map $\pi\colon A\to\Vk$ is injective and $\vert Q\vert\geq \vert B\vert$/2, we have that
  $$\sum_{h\in\Vk}\dep(\pi^{-1}(h)\cap((X_{B}\cup\{b\})\+ B))-\sum_{h\in\Vk}\dep(\pi^{-1}(h)\cap(X_{B}\+ B))\geq \vert B\vert/2.$$
  In other words, each time we add some $b\in B$ into the set $X_{B}$, the quantity $$\sum_{h\in\Vk}\dep(\pi^{-1}(h)\cap(X_{B}\+ B))$$ is increased by at least $\vert B\vert/2$. Since $\sum_{h\in\Vk}\dep(\pi^{-1}(h)\cap(A\+ A))\leq DK^{d}$, we have that 
  \begin{equation}\label{2:upbdfb}
  \vert X_{B}\vert\leq 2DK^{d}/\vert B\vert.
  \end{equation}
Finally by the construction of $X_{B}$, for all $b\in B$, the set $$Q_{B}(b):=\{a\in B\colon a\+ b\sim c \text{ for some } c\in  X_{B}\+ B\}$$ is of cardinality at least $\vert B\vert/2$. 

For every $B\subseteq A$,
let $G_{B}=(B\times A^{2},X_{B}\times (A\+ A)^{2},E_{B})$ be the bipartite graph such that $((b_{1},b_{2},b_{3}),(x,c,c'))\in E_{B}$ if   $$\text{there exist $a, a'\in A$ such that $a\+ b_{1}\sim a'\+ x$, $b_{2}\+ a=c$ and $b_{3}\+ a'=c'$.}$$ 
We say that $x\in X_{B}$ is a \emph{major} of some $(b_{1},b_{2},b_{3})\in B\times A^{2}$ in the graph $G_{B}$ if the number of edges in $E_{B}$ of the form $((b_{1},b_{2},b_{3}),(x,c,c'))$  is at least $\vert B\vert/2\vert X_{B}\vert$.

For $(x,c,c')\in X_{B}\times (A\+ A)^{2}$, denote $\tilde{\pi}(x,c,c'):=\pi(c)$.

\begin{lem}\label{2:gcp2}
	Let $B\subseteq A$.
	\begin{enumerate}[(i)]
\item For every $\bold{b}\in B\times A^{2}$,
	there exists a subset $R_{\bold{b}}$ of $X_{B}\times(A\+ A)^{2}$ of cardinality at least $\vert B\vert/2$ such that $(\bold{b},r)\in E_{B}$ for all $r\in R_{\bold{b}}$. Moreover, we may require that $\tilde{\pi}(r)\neq \tilde{\pi}(r')$ for all distinct $r,r'\in R_{\bold{b}}$.

	\item For every $\bold{b}\in B\times A^{2}$, there exists  $x\in X_{B}$ which is a major of $\bold{b}$ in the graph $G_{B}$.
	
\item If $((b_{1},b_{2},b_{3}),(x,c,c'))\in E_{B}$, then  writing $z=(\pi(z),J^{M}_{V_{z}}+f_{z})$ for $z=x,b_{1},b_{2},b_{3},c,$ and $c'$, we have that $V_{z}=\sp_{\F_{p}}\{\iota(\pi(z))\}$ for $z=x,b_{1},b_{2},b_{3}$, $V_{c}=\sp_{p}\{\iota(\pi(c)),$ $\iota(\pi(b_{2}))\}$, $V_{c'}=\sp_{p}\{\iota(\pi(c')),\iota(\pi(b_{3}))\}$ and
$$f_{b_{1}}-f_{b_{2}}+f_{b_{3}}\equiv f_{x}-f_{c}+f_{c'} \mod J^{M}_{\iota(\pi(b_{1})),\iota(\pi(b_{2})),\iota(\pi(b_{3})),\iota(\pi(c)),\iota(\pi(x))}$$
(in particular $b_{1}\- b_{2}\+ b_{3}\sim x\- c\+ c'$).
	\end{enumerate}
\end{lem}	
\begin{proof}
	We start with Part (i). Fix $\bold{b}=(b_{1},b_{2},b_{3})\in B\times A^{2}$ and let $a\in Q_{B}(b_{1})$.  By definition, there exists $c_{0}=a'\+ x\in B\+ X_{B}$ for some $a'\in B$ and $x\in X_{B}$ such that $a\+ b_{1}\sim c_{0}$. Set  $c=b_{2}\+ a$ and $c'=b_{3}\+ a'$. In this way, we may define a map $f_{\bold{b}}\colon Q_{B}(b_{1})\to X_{B}\times(A\+ A)^{2}$ given by $f_{\bold{b}}(a)=(x,b_{2}\+ a,b_{3}\+ a')$.
	Let $R_{\bold{b}}:=f_{\bold{b}}(Q_{B}(b_{1}))$.
	
	Since $\vert Q_{B}(b_{1})\vert\geq \vert B\vert/2$, to show that $R_{\bold{b}}$ is of cardinality at least $\vert B\vert/2$, it suffices to show that $f_{\bold{b}}$ is injective. Suppose that $f_{\bold{b}}(a)=f_{\bold{b}}(\tilde{a})=(x,c,c')$. Then $\pi(a)=\pi(c)-\pi(b_{2})=\pi(\tilde{a})$. Since the map $\pi\colon A\to\V$ is injective, we have that $a=\tilde{a}$. So  $f_{\bold{b}}$ is injective.
	
	On the other hand, for each $(x,c,c')\in R_{\bold{b}}$,  $\pi(c)=\pi(a)+\pi(b_{2})$ is uniquely determined by the choice of $a$. So $\pi(c)\neq \pi(\tilde{c})$ for all distinct $(x,c,c'),(\tilde{x},\tilde{c},\tilde{c}')\in R_{\bold{b}}$.
	This completes the proof of Part (i).
	Part (ii) Follows from Part (i) and the Pigeonhole Principle.

	We now prove Part (iii). Write also $z=(\pi(z),J^{M}_{V_{z}}+f_{z})$ for $z=a,a'$. For   $z=a,a',x,b_{1},b_{2},b_{3}$, since $z\in\Gamma^{s}_{1}(\Vk,M)$, we have $V_{z}=\sp_{\F_{p}}\{\iota(\pi(z))\}$.
	Since $b_{2}\+ a=c$ and $b_{3}\+ a'=c'$, 
	 we have that 
	  $$(\pi(c),J^{M}_{V_{c}}+f_{c})=(\pi(a)+\pi(b_{2}),J^{M}_{\iota(\pi(a)),\iota(\pi(b_{2}))}+f_{a}+f_{b_{2}})$$
	  and $$(\pi(c'),J^{M}_{V_{c'}}+f_{c'})=(\pi(a')+\pi(b_{3}),J^{M}_{\iota(\pi(a')),\iota(\pi(b_{3}))}+f_{a'}+f_{b_{3}}).$$
	  In particular, $\pi(c)=\pi(a)+\pi(b_{2})$ and $\pi(c')=\pi(a')+\pi(b_{3})$.
	  So $J^{M}_{\iota(\pi(a)),\iota(\pi(b_{2}))}=J^{M}_{\iota(\pi(c)),\iota(\pi(b_{2}))}$ and $J^{M}_{\iota(\pi(a')),\iota(\pi(b_{3}))}=J^{M}_{\iota(\pi(c')),\iota(\pi(b_{3}))}$.
	 	  
	  Since $a\+ b_{1}\sim a'\+ x$, we have that $\pi(b_{1})+\pi(a)=\pi(a')+\pi(x)$.
	So $\pi(b_{1}\- b_{2}\+ b_{3})=\pi(x\- c\+ c')$
	and $$J^{M}_{\sp_{\F_{p}}\{\iota(\pi(x))\}+V_{c}+V_{c'}}=J^{M}_{\iota(\pi(x)),\iota(\pi(b_{2})),\iota(\pi(a)),\iota(\pi(b_{3})),\iota(\pi(a'))}=J^{M}_{\iota(\pi(b_{1})),\iota(\pi(b_{2})),\iota(\pi(b_{3})),\iota(\pi(c)),\iota(\pi(x))}.$$
	Finally,  $a\+ b_{1}\sim a'\+ x$ implies that $f_{b_{1}}+f_{a}\equiv f_{a'}+f_{x} \mod J^{M}_{\iota(\pi(a)),\iota(\pi(b_{1})),\iota(\pi(a')),\iota(\pi(b_{x}))}$. Therefore,
	$f_{b_{1}}-f_{b_{2}}+f_{b_{3}}\equiv f_{x}-f_{c}+f_{c'} \mod J^{M}_{\iota(\pi(a)),\iota(\pi(b_{1})),\iota(\pi(a')),\iota(\pi(b_{x}))}$.
	 We are done since $J^{M}_{\iota(\pi(a)),\iota(\pi(b_{1})),\iota(\pi(a')),\iota(\pi(b_{x}))}\subseteq J^{M}_{\iota(\pi(b_{1})),\iota(\pi(b_{2})),\iota(\pi(b_{3})),\iota(\pi(c)),\iota(\pi(x))}=J^{M}_{\sp_{\F_{p}}\{\iota(\pi(b_{1})),\iota(\pi(b_{2})),\iota(\pi(b_{3})),\iota(\pi(x))\}+V_{c}+V_{c'}}$.
\end{proof}

Let $B\subseteq A$. Clearly, the graph $G_{B}$ has a disjoint partition $G_{B,h}=(U_{B,h}, V_{B,h}, E_{B,h}), h\in \Vk$, where $$U_{B,h}=\{(b_{1},b_{2},b_{3})\in B\times A^{2}\colon \pi(b_{1})-\pi(b_{2})+\pi(b_{3})=h\}$$ and $$V_{B,h}=\{(x,c,c')\in X_{B}\times (A\+ A)^{2}\colon \pi(x)-\pi(c)+\pi(c')=h\}.$$ For each $x\in X_{B}$, let $U_{B,h,x}$ denote the set of all $(b_{1},b_{2},b_{3})\in U_{B,h}$  having $x$ as a major (by Lemma \ref{2:gcp2} Part (ii), we have that $U_{B,h}=\cup_{x\in X_{B}}U_{B,h,x}$), and let $V_{B,h,x}$ be the set of elements in $V_{B,h}$ whose first entry is $x$.
Let $$\hat{U}_{B,h,x}:=\{b_{1}\- b_{2}\+ b_{3}\+ (\bold{0},J^{M}_{\pi(x)})\colon (b_{1},b_{2},b_{3})\in U_{B,h,x}\}.$$

\begin{prop}\label{2:rhob}
	Denote $\rho:=\vert B\vert/p^{d}$.
	If $p\gg_{d,D,s,\rho} 1$ and $d\geq\max\{s+7,19\}$, then  $\dep(\hat{U}_{B,h,x})=O_{D,s,\rho}(1)$ for all $h\in\Vk$.
\end{prop}	
\begin{proof}
	Fix $h\in\Vk$ and
	 denote $(A\+ A)_{h'}:=\pi^{-1}(h')\cap (A\+ A)$ for all $h'\in\Vk$.
	 Since $\dep(A\+ A)\leq D$, the relation graph of $(A\+ A)_{h'}$ admits a $D^{-1}$-dense auxiliary graph $G'_{h'}=((A\+ A)_{h'},Y_{h'},E'_{h'})$.
	 By Lemma \ref{2:duplicate}, we may assume without loss of generality that all of $Y_{h'}, h'\in\Vk$ have the same cardinality, which we denote by $L$.
	 
	 We need to use the intersection method.
	  Consider the bipartite graph $$G_{B,h,x}=(\hat{U}_{B,h,x},(\sqcup_{h'\in\Vk} Y_{h'}\times Y_{\pi(x)-h-h'})^{s+7},E_{B,h,x})$$ defined in the following way. For all $u\in \hat{U}_{B,h,x}$, $h_{i}\in\Vk$, $y_{i}=(y_{i,1},y_{i,2})\in Y_{h_{i}}\times Y_{\pi(x)-h-h_{i}}, 1\leq i\leq s+7$,
	   $(u,(y_{1},\dots,y_{s+7}))\in E_{B,h,x}$ if and only if there exist $(b_{1},b_{2},b_{3})\in U_{B,h,x}$, $c_{i},c'_{i}\in A\+ A, 1\leq i\leq s+7$ such that the following holds:
	 \begin{itemize}
	 	\item $u=b_{1}\- b_{2}\+ b_{3}\+ (\bold{0},J^{M}_{\iota(\pi(x))})$;
	 	\item $\pi(c_{1}),\dots,\pi(c_{s+7})$ are $p$-linearly independent;
	 	\item $((b_{1},b_{2},b_{3}),(x,c_{i},c'_{i}))\in E_{B,h}$ for all $1\leq i\leq s+7$;
	 	\item $(c_{i},y_{i,1})\in Y_{h_{i}}$, $(c_{i},y_{i,2})\in Y_{\pi(x)-h-h_{i}}$ for all $1\leq i\leq s+7$ (this enforces $c_{i}\in (A\+ A)_{h_{i}}$ and $c'_{i}\in (A\+ A)_{\pi(x)-h-h_{i}}$).
	 \end{itemize}	
	For convenience we say that $b_{1},b_{2},b_{3},c_{i},c'_{i}, 1\leq i\leq s+7$ are the \emph{parameters} associated to the edge $(u,(y_{1},\dots,y_{s+7}))$ (note that the parameters of an edge is not necessarily unique).
	
	Fix $u=b_{1}\- b_{2}\+ b_{3}\+ (\bold{0},J^{M}_{\pi(x)})$ for some $(b_{1},b_{2},b_{3})\in U_{B,h,x}$. 
	Let $W$ be the set of
	$\w=((x,c_{1},c'_{1}),\dots,(x,c_{s+7},c'_{s+7}))\in V_{B,h,x}^{s+7}$ such that $((b_{1},b_{2},b_{3}),(x,c_{i},c'_{i}))\in E_{B,h}$ for all $1\leq i\leq s+7$. 
	By Lemma \ref{2:gcp2} (ii), 
 $\vert W\vert\geq (\vert B\vert/2\vert X_{B}\vert)^{s+7}\gg_{D,s,\rho} K^{(s+7)d}$. By Lemma \ref{2:gcp2} (i), the tuple $(\pi(c_{1}),\dots,\pi(c_{s+7}))$ is different for all $\w\in W$. By Lemma \ref{2:iiddpp}, if  $p\gg_{d,D,s,\rho} 1$ and $d\geq s+7$, then the set $W'$ of $\w\in W$ such that  $\pi(c_{1}),\dots,\pi(c_{s+7})$ are $p$-linearly independent is of cardinality   $\gg_{D,s,\rho}K^{(s+7)d}$.  
 On the other hand, for every $\w=((x,c_{1},c'_{1}),\dots,(x,c_{s+7},c'_{s+7}))\in W'$,
 the number of $(y_{i,1},y_{i,2})\in Y_{\pi(c_{i})}\times Y_{\pi(x)-h-\pi(c_{i})}$ with $(c_{i},y_{i,1})\in Y_{h_{i}}$, $(c_{i},y_{i,2})\in Y_{\pi(x)-h-h_{i}}$ is at least $D^{-2}L^{2}$ for all $1\leq i\leq s+7$. So there are  $\gg_{D,s,\rho}p^{(s+7)d}\cdot (D^{-2}K^{2})^{s+7}$ many edges in $G_{B,h,x}$ with $u$ as a vertex. Since $\vert (\sqcup_{h'\in\Vk} Y_{h'}\times Y_{\pi(x)-h-h'})^{s+7}\vert=(L^{2}p^{d})^{s+7}$,
	we have that $G_{B,h,x}$ is $O_{D,s,\rho}(1)^{-1}$-dense.
	To show that  $\dep(\hat{U}_{B,h,x})=O_{D,s,\rho}(1)$, it suffices to show that $G_{B,h,x}$ is an auxiliary graph of the relation graph of  $\hat{U}_{B,h,x}$.

Let $u,u'\in \hat{U}_{B,h,x}$, $h_{i}\in\Vk$ and $y_{i}=(y_{i,1},y_{i,2})\in Y_{h_{i}}\times Y_{\pi(x)-h-h_{i}}, 1\leq i\leq s+7$ be such that $(u,(y_{1},\dots,y_{s+7}))$ and $(u',(y_{1},\dots,y_{s+7}))$ are edges in $E_{B,h,x}$ with parameters $b_{1},b_{2},b_{3},c_{i},c'_{i}, 1\leq i\leq s+7$ and $b'_{1},b'_{2},b'_{3},\tilde{c}_{i},\tilde{c}'_{i}, 1\leq i\leq s+7$ respectively. Clearly, $\pi(c_{i})=\pi(\tilde{c}_{i})=h_{i}$.
For convenience write  $z=(\pi(z),J^{M}_{V_{z}}+f_{z})$ for $z=x,b_{j},b'_{j},c_{i},c'_{i},\tilde{c}_{i}$ and $\tilde{c}_{i'}$. 
Since $((b_{1},b_{2},b_{3}),(x,c_{i},c'_{i}))\in E_{B,h}$,	by Lemma \ref{2:gcp2} (iii),
$$f_{c'_{i}}- f_{c_{i}}+f_{x}\equiv f_{b_{1}}- f_{b_{2}}+ f_{b_{3}} \mod J^{M}_{\iota(\pi(b_{1})),\iota(\pi(b_{2})),\iota(\pi(b_{3})),\iota(\pi(x)),\iota(h_{i})}.$$
	Similarly,
	$$f_{\tilde{c}'_{i}}- f_{\tilde{c}_{i}}+f_{x}\equiv f_{b'_{1}}- f_{b'_{2}}+ f_{b'_{3}} \mod J^{M}_{\iota(\pi(b'_{1})),\iota(\pi(b'_{2})),\iota(\pi(b'_{3})),\iota(\pi(x)),\iota(h_{i})}.$$ 
	Since  $(c_{i},y_{i,1}),(\tilde{c}_{i},y_{i,1})\in Y_{h_{i}}$, $(c'_{i},y_{i,2}),(\tilde{c}'_{i},y_{i,2})\in Y_{\pi(x)-h-h_{i}}$, we have that 
	$$c'_{i}\- c_{i}\sim \tilde{c}'_{i}\- \tilde{c}_{i}.$$
	Combining the above equations and the description of $V_{z}$ in Lemma \ref{2:gcp2} (iii) for $z=c_{i},c'_{i},\tilde{c}_{i}$ and $\tilde{c}_{i'}$, we have that 
	$$f_{b_{1}}- f_{b_{2}}+ f_{b_{3}}\equiv f_{b'_{1}}- f_{b'_{2}}+ f_{b'_{3}}\mod \cap_{i=1}^{s+7}J^{M}_{\iota(\pi(b_{1})),\iota(\pi(b_{2})),\iota(\pi(b_{3})),\iota(\pi(b'_{1})),\iota(\pi(b'_{2})),\iota(\pi(b'_{3})),\iota(\pi(x)),\iota(h_{i})}.$$	
	 Since $\sp_{\V}\{\iota(\pi(b_{i})),\iota(\pi(b'_{i})),\iota(\pi(x))\colon 1\leq i\leq 3\}$ is of dimension at most 6 (note that $\pi(b_{1})-\pi(b_{2})+\pi(b_{3})=h=\pi(b'_{1})-\pi(b'_{2})+\pi(b'_{3})$), $h_{i}=\pi(c_{i}), 1\leq i\leq s+7$  are $p$-linearly independent, and $d\geq 19$, by Proposition \ref{2:gri} (setting $m=6$ and $r=1$),  we have that
	 $$f_{b_{1}}- f_{b_{2}}+ f_{b_{3}}\equiv f_{b'_{1}}- f_{b'_{2}}+ f_{b'_{3}} \mod J^{M}_{\iota(\pi(b_{1})),\iota(\pi(b_{2})),\iota(\pi(b_{3})),\iota(\pi(b'_{1})),\iota(\pi(b'_{2})),\iota(\pi(b'_{3})),\iota(\pi(x))},$$
	 or equivalently,
	 $$b_{1}\- b_{2}\+ b_{3}\+ (\bold{0}, J^{M}_{\iota(\pi(x))})\sim b'_{1}\- b'_{2}\+ b'_{3}\+ (\bold{0}, J^{M}_{\iota(\pi(x))}).$$
	 This means that $G_{B,h,x}$ is an auxiliary graph of the relation graph of  $\hat{U}_{B,h,x}$ and we are done.
\end{proof}

We are now ready to complete the proof of Proposition \ref{2:ga31}.
It follows from Proposition \ref{2:rhob} that for each $x\in X_{A}$, the set $\hat{U}_{A,h,x}$ has a small density dependence number. However, this is insufficient to conclude that $A\- A\+ A$ has a small density dependence number because of the loss of information caused by modulo $J^{M}_{\iota(\pi(x))}$ in the set $\hat{U}_{A,h,x}$. To overcome this difficulty, we show that  $\hat{U}_{B,h,x_{i}}$ has a small density dependence number for some large subset $B\subseteq A$ and for many $x_{i}\in\Vk$ in general positions, and then we use the intersection method to  recover the loss of information.
Unfortunately, in the constructions we have to pass to a subset $B$ of $A$ (although the density of $B$ can be arbitrarily close to $A$). This is the reason why we have to pass to a subset of $A$ in the statement of Proposition \ref{2:ga31}.

	We now construct the series of special vectors $x_{i_{1},\dots,i_{r}}\in \Vk, 1\leq r\leq s+6, i_{j}\in A$ we need as follows. 
	For $x\in A$, we say that $x$ is \emph{admissible} if $x\in X_{A}$ and $\iota(\pi(x))\neq \bold{0}$. For $2\leq r\leq s+6$, we say that a sequence $(x_{1},\dots,x_{r})\in A^{r}$ is \emph{admissible} if $(x_{1},\dots,x_{r-1})$ is admissible and $x_{r}\in X_{A\backslash\pi^{-1}(\sp_{\F_{p}}\{\iota(\pi(x_{1})),\dots,\iota(\pi(x_{r-1}))\})}$. We call $r$ the \emph{length} of the tuple.
	Since $d\geq s+6$ and $$\vert X_{A\backslash\pi^{-1}(\sp_{\F_{p}}\{\iota(\pi(x_{1})),\dots,\iota(\pi(x_{r-1}))\})}\vert\leq 2DK^{d}/(\vert A\vert-K^{s+5})\leq 2DK^{d}/(\d K^{d}-K^{s+5})$$
	for all $1\leq r\leq s+6$ by (\ref{2:upbdfb}),	
	we have that there are in total $O_{\d,D,s}(1)$ many admissible tuples of length $s+6$. Since $X_{B}$ is non-empty for all non-empty $B$, there is at least one admissible tuple of length $s+6$. 
	Let $$A':=A\backslash \pi^{-1}(\cup_{(x_{1},\dots,x_{s+6}) \text{ is admissible}}\sp_{\F_{p}}\{\iota(\pi(x_{1})),\dots,\iota(\pi(x_{s+6}))\}).$$	
	By (\ref{2:upbdfb}), there are  $O_{\d,D,s}(1)$ many admissible tuples of length $s+6$. So we have that $\vert A'\vert\geq \vert A\vert/2$ if $p\gg_{\d,d,D} 1$. It then suffices to show that $A'\- A'\+ A'$ has density dependence number $O_{\d,D,s}(1)$. 
	
	Fix  $h\in\V$ and fix any $u\in U_{A',h}$. 
	Since $u\in U_{A',h}\subseteq  U_{A,h}$, by Lemma \ref{2:gcp2} (ii), there exists $x_{1}\in X_{A'}$ which is a major of $u$. In other words, $u\in U_{A',h,x_{1}}$. Suppose that we have found an admissible tuple $(x_{1},\dots,x_{r})\in A^{r}$ for some $1\leq r\leq s+5$ such that $u\in U_{A',h,x_{i}}$ for all $1\leq i\leq r$. 
	Denote $B:=A\backslash\pi^{-1}(\sp_{\F_{p}}\{\iota(\pi(x_{1})),\dots,\iota(\pi(x_{r}))\})$. Since $A'\subseteq B$, we have
	$u\in U_{B,h}$. By Lemma \ref{2:gcp2} (ii), there exists $x_{r+1}\in X_{B}$ which is a major of $u$.   In other words, $u\in U_{A',h,x_{r+1}}$. Since $x_{r+1}\in X_{B}$, we have that $(x_{1},\dots,x_{r+1})$ is admissible. By induction, there exists an admissible tuple $(x_{1},\dots,x_{s+6})\in A^{s+6}$ such that $u\in U_{A',h,x_{i}}$ for all $1\leq i\leq s+6$.
	
	Let $J$ be the set of all admissible tuples of length $s+6$. Recall that $\vert J\vert= O_{\d,D,s}(1)$.
	For any $\x:=(x_{1},\dots,x_{s+6})\in J$, since $\dep(\hat{U}_{A',h,x_{i}})=O_{\d,D,s}(1)$
	by Proposition \ref{2:rhob}, the relation graph of $\hat{U}_{A',h,x_{i}}$ has an $O_{\d,D,s}(1)^{-1}$-dense auxiliary graph $G'_{h,x_{i}}=(\hat{U}_{A',h,x_{i}},Y_{h,x_{i}},E'_{h,x_{i}})$ for all $1\leq i\leq s+6$.
	 By Lemma \ref{2:duplicate}, we may assume without loss of generality that all of $Y_{h,x_{i}}$ have the same cardinality, which we denote by $K$.

	Let
	$\hat{U}_{h}=\{b_{1}\- b_{2}\+ b_{3}\colon (b_{1},b_{2},b_{3})\in U_{A',h}\}$ and
	 $Y_{h}:=\sqcup_{(x_{1},\dots,x_{s+6})\in J} Y_{h,x_{1}}\times\dots\times Y_{h,x_{s+6}}.$
	Consider the bipartite graph $\tilde{G}_{h}=(\hat{U}_{h},Y_{h},\tilde{E}_{h})$ given in the following way. For all $u\in \hat{U}_{h}$   and $y=(y_{1},\dots,y_{s+6})\in Y_{h}$, $(u,y)\in\tilde{E}_{h}$ if and only if there exist $(b_{1},b_{2},b_{3})\in U_{A',h}$ and $\x=(x_{1},\dots,x_{s+6})\in J$  such that
	\begin{itemize}
		\item $u=b_{1}\- b_{2}\+ b_{3}$;
		\item $y\in Y_{h,x_{1}}\times\dots\times Y_{h,x_{s+6}}$;
		\item $(b_{1},b_{2},b_{3})\in U_{A',h,x_{i}}$ (which implies that $u\+ (\bold{0},J^{M}_{\iota(\pi(x_{i}))})\in \hat{U}_{A',h,x_{i}}$) for all $1\leq i\leq s+6$;
		\item $(u\+ (\bold{0},J^{M}_{\iota(\pi(x_{i}))}),y_{i})\in E'_{h,x_{i}}$ for all $1\leq i\leq s+6$.
	\end{itemize}	
		For convenience we say that $b_{1},b_{2},b_{3},\x=(x_{1},\dots,x_{s+6})$ are the \emph{parameters} associated to the edge $(u,y)$ (note that the parameters of an edge is not necessarily unique).

		For any $u\in\hat{U}_{h}$, by definition we may write $u=b_{1}\- b_{2}\+ b_{3}$ for some $(b_{1},b_{2},b_{3})\in U_{A',h}$. By the previous discussions, there exists at least one of $\x\in J$ with $(b_{1},b_{2},b_{3})\in \cap_{i=1}^{s+6}U_{A',h,x_{i}}$. For this $\x$, the number of $y\in Y_{h,x_{1}}\times\dots\times Y_{h,x_{s+6}}$ such that $(u\+ (\bold{0},J^{M}_{\iota(\pi(x_{i}))}),y_{i})\in E'_{h,x_{i}}$ for all $1\leq i\leq s+6$ is $O_{\d,D,s}(1)^{-1}L^{s+6}$. So there are at least $O_{\d,D,s}(1)^{-1}L^{s+6}$ edges in $\tilde{E}_{h}$ having $u$ as a vertex. Since $\vert Y_{h}\vert=\vert J\vert K^{s+6}=O_{\d,D,s}(L^{s+6})$, we have that
		$\tilde{G}_{h}$ is  $O_{\d,D,s}(1)^{-1}$-dense.
		To show that  $\dep(A'\- A'\+ A')=O_{\d,D,s}(1)$, it suffices to show that for all $h\in\Vk$, $\tilde{G}_{h}$ is an auxiliary graph of the relation graph of  $\hat{U}_{h}$.

      Let $u,u'\in \hat{U}_{h}$   and $y=(y_{1},\dots,y_{s+6})\in Y_{h}$ be such that $(u,y),(u',y)$ are edges of $\tilde{E}_{h}$ with parameters $b_{1},b_{2},b_{3},\x=(x_{1},\dots,x_{s+6})$ and $b'_{1},b'_{2},b'_{3},\x'=(x'_{1},\dots,x'_{s+6})$ respectively. Then we must have $\x=\x'$. Since $(u\+ (\bold{0},J^{M}_{\iota(\pi(x_{i}))}),y_{i}), (u'\+ (\bold{0},J^{M}_{\iota(\pi(x_{i}))}),y_{i})\in E'_{h,x_{i}}$ for all $1\leq i\leq s+6$,
      we have that $$u\+ (\bold{0},J^{M}_{\iota(\pi(x_{i}))})=b_{1}\- b_{2}\+ b_{3}\+ (\bold{0},J^{M}_{\iota(\pi(x_{i}))})\sim b'_{1}\- b'_{2}\+ b'_{3}\+ (\bold{0},J^{M}_{\iota(\pi(x_{i}))})=u'\+ (\bold{0},J^{M}_{\iota(\pi(x_{i}))}).$$ Writing $z=(\pi(z),J^{M}_{\iota(\pi(z))}+f_{z})$ for $z=b_{j},b'_{j}, 1\leq j\leq 3$, we have that
      $$f_{b_{1}}-f_{b_{2}}+f_{b_{3}}\equiv f_{b'_{1}}-f_{b'_{2}}+f_{b'_{3}} \mod J^{M}_{\iota(\pi(b_{1})),\iota(\pi(b_{2})),\iota(\pi(b_{3})),\iota(\pi(b'_{1})),\iota(\pi(b'_{2})),\iota(\pi(b'_{3})),\iota(\pi(x_{i}))}$$
      for all $1\leq i\leq s+6$. Since $\x$ is admissible, $\iota(\pi(x_{r}))\notin \sp_{\F_{p}}\{\iota(\pi(x_{1})),\dots,\iota(\pi(x_{r-1}))\}$ for all $1\leq r\leq s+6$. So $\pi(x_{1}),\dots,\pi(x_{s+6})$ are $p$-linearly independent. Since $d\geq 17$ and $$\dim(\sp_{\F_{p}}\{\iota(\pi(b_{i})),\iota(\pi(b'_{i}))\colon 1\leq i\leq 3\})\leq 5$$
      (note that $\pi(b_{1})-\pi(b_{2})+\pi(b_{3})=\pi(u)=h=\pi(u')=\pi(b'_{1})-\pi(b'_{2})+\pi(b'_{3})$),
       by Proposition \ref{2:gri} (setting $m=5$ and $r=1$), we have that
       $$f_{b_{1}}-f_{b_{2}}+f_{b_{3}}\equiv f_{b'_{1}}-f_{b'_{2}}+f_{b'_{3}} \mod J^{M}_{\iota(\pi(b_{1})),\iota(\pi(b_{2})),\iota(\pi(b_{3})),\iota(\pi(b'_{1})),\iota(\pi(b'_{2})),\iota(\pi(b'_{3}))},$$
       or equivalently,
       $u=b_{1}\- b_{2}\+ b_{3}\sim b'_{1}\- b'_{2}\+ b'_{3}=u'$. This means that $\tilde{G}_{h}$ is an auxiliary graph of the relation graph of  $\hat{U}_{h}$ and completes the proof of Proposition \ref{2:ga31}.

\section{Decomposition results for relation graphs}\label{2:s:116}

 As is mentioned in Remark \ref{2:dpcc}, it is more natural to relate the cardinality of additive sets to the clique covering numbers of the relation graphs rather than the density dependence numbers. 
 The purpose of this section is to study the connections between these two numbers.
 The proof of Lemma \ref{2:basicdn} (i) implies that for any graph $G$, we have $\cc(G)\geq \dep(G)$. On the other hand,
 for a general graph $G$, $\cc(G)$ can be much larger than $\dep(G)$.

\begin{ex}\label{2:exex001}
Let $G=(X,E)$ be a graph. The \emph{Mycielskian} of $G$ is the graph $\mu(G)$ whose vertices are $X\sqcup X'\sqcup \{a\}$ (where $X'=\{x'\colon x\in X\}$ is an identity copy of $X$) and the set of edges of $\mu(G)$ is $E\cup\{\{x,y'\}\colon \{x,y\}\in E\}\cup \{\{y',a\}\colon y'\in V'\}$.  Let $M_{2}$ be the graph with two points connected by an edge and $M_{i}:=\mu(M_{i-1})$ for all $i\geq 3$. 
It is known that $M_{i}=(X_{i},E_{i})$ is triangle free, which implies that $\cc(M_{i})\geq \lceil \frac{\vert M_{i}\vert}{2}\rceil=3\cdot 2^{i-3}$ if $i\geq 3$. On the other hand, it is known that the chromatic number of $M_{i}$ is $i$ (i.e. there exists a coloring of the vertices of $M_{i}$ in $i$ colors such that no two adjacent vertices have the same color), which implies that $M_{i}$ admits an auxiliary graph of density $1/i$, meaning that $\dep(M_{i})\leq i$. So we see that $\cc(M_{i})$ is much larger than $\dep(M_{i})$ if $i$ is large.
\end{ex}
 
 Based on the intuition that the structure described in Example \ref{2:exex001} should rarely occur for relation graphs for subsets of $\Gamma^{s}(\Vk,M)$,
 it is natural to make the following conjecture:

 \begin{conj}[Smallness of $\dep(G)$ implies smallness of $\cc(G)$]\label{2:ccsdp}
 Let $d,k,K\in\N_{+}$, $s\in\N$, $p$ be a prime  dividing $K$ and $M\colon\V\to \F_{p}$ be a non-degenerate quadratic form. Let $G$ be the relation graph of a subset of $\Gamma^{s}_{k}(\Vk,M)$. If $d\gg_{k,s} 1$ and $p\gg_{d} 1$, then $\cc(G)=O_{d,\dep(G)}(1)$. 
 \end{conj}
 
  If Conjecture \ref{2:ccsdp} holds, then all the results in Section \ref{2:s:115}  still apply if we replace the density dependence numbers by the clique covering numbers (under a slight modification on relevant constants). 
 Unfortunately, we are unable to prove Conjecture \ref{2:ccsdp}. Instead, we  have the following   weaker version of it.

\begin{prop}[Weak decomposition result for relation graphs]\label{2:gweakcore1}
	Let $d,D,k,K\in\N_{+},s\in\N$ be such that $d\geq 2(k-2)s+6k-1$ if $k\geq 2$, $p\gg_{d,D} 1$ be a prime dividing $K$, and  $M\colon\V\to\F_{p}$ be a non-degenerate quadratic form. There exist $C:=C(d,D), N:=N(d,D)\in \N$  such that for every  subset $X$ of $\Gamma^{s}_{k}(\Vk,M)$ with $\pi(X)=\{\bold{0}\}$ and $\dep(X)\leq D$,
	there exists $Y\subseteq \V$ which is the union of at most $N$ subspaces of $\V$ of dimension at most $(s+2k-2)(k-1)$  such that 
	we can partition $X$ into $X_{b}\cup X_{g}$ such that 
	\begin{itemize}
		\item for all $(\bold{0},J^{M}_{V}+f)\in X_{b}$, we have that $V\cap Y\neq \{\bold{0}\}$;
		\item $\cc(X_{g})\leq C$.\footnote{Note that Proposition \ref{2:gweakcore1} holds trivially if $d\leq (s+2k-1)(k-1)$.} 
	\end{itemize}		
\end{prop}	

Roughly speaking, Proposition \ref{2:gweakcore1} says that
  for every relation graph of a subset $X$ of $\Gamma^{s}_{k}(\Vk,M)$ with $\pi(X)=\{\bold{0}\}$, we may partition $X$ into a 	``good" part $X_{g}$ and a ``bad" part $X_{b}$ such that $\cc(X_{g})$ is controlled by $\dep(X)$, and that the elements of $X_{b}$ has nontrivial intersection with a low dimensional object $Y$ in $\V$.
  In practice, we will use Proposition \ref{2:gweakcore1} for $X$ of the form $kA\- \ell A$ for some $A\subseteq\Gamma^{s}_{1}(\Vk,M)$.
  Although $\vert X_{b}\vert$ is much smaller than $\vert X\vert$, unfortunately it turns out that we can not get rid of the error set $X_{b}$ by passing to a large subset of $A$.
  
  \begin{rem}\label{2:coreimprovement}
  Proposition \ref{2:gweakcore1} is the major reason why we need to impose a rather large lower bound for the dimension $d$ in Theorem \ref{2:aadd}. As we shall see in later sections, the dimension of the error set $X_{b}$ continues to accumulate in the proof of Theorem \ref{2:aadd}, which eventually ends up with a very large dimension. If one can prove Conjecture  \ref{2:ccsdp} (i.e. prove Proposition \ref{2:gweakcore1} with $X_{b}=\emptyset$), then the lower  bound for  $d$ in Theorem \ref{2:aadd} can be improved significantly.
  \end{rem}

\begin{proof}[Proof of Theorem \ref{2:gweakcore1}]
By Lemma \ref{2:spsp1}, $\sim$ is an equivalent relation in $\Gamma_{1}^{s}(\Vk,M)$.
	 So if $k=1$, then the relation graph of $X$ is a finite union of $t$ complete graphs with disjoint vertices for some $t\in\N$. Moreover, it is not hard to see that $t=\cc(X)=\dep(X)\leq D$.
	 So we are done by setting $C=D$ and $N=0$. From now on we assume that $k\geq 2$.	
	
	If $D=1$, then $X$ is an equivalence class by Lemma \ref{2:basicdn} (iv).
	So we can simply take $C=1$, $N=0$ and $X_{1}=X$.
	Suppose that the conclusion holds for $D-1$ for some $D\geq 2$, we show that the conclusion also holds for $D$.

	For convenience denote $L:=s+2k-1$.
	Pick an arbitrary $x_{1}=(\bold{0},J^{M}_{V_{1}}+f_{1})\in X$. Suppose that we have chosen $x_{i}=(\bold{0},J^{M}_{V_{i}}+f_{i})\in X$ for $1\leq i\leq r$ for some $1\leq r\leq L$ such that $V_{1},\dots,V_{r}$ are linearly independent, and that $\{x_{1},\dots,x_{r}\}$ is an equivalence class. Let $B_{1}$ be the set of all $x\in X$ such that $x\sim x_{i}$ for all $1\leq i\leq r$, and set $B_{2}=X\backslash B_{1}$. If for all $x=J^{M}_{V}+f\in B_{1}$, we have that $(V_{1}+\dots+V_{r})\cap V\neq \{\bold{0}\}$. Then we stop the construction. 
	If $r=L$, then we also stop the construction. 
	Otherwise, if $r<L$ and there exists  $x=(\bold{0},J^{M}_{V}+f)\in B_{1}$ such that $(V_{1}+\dots+V_{r})\cap V=\{\bold{0}\}$ then we pick $x_{r+1}=x$. Since $x_{r+1}\in B_{1}$, $\{x_{1},\dots,x_{r+1}\}$ is an equivalence class. Since $(V_{1}+\dots+V_{r})\cap V=\{\bold{0}\}$, $V_{1},\dots,V_{r+1}$ are linearly independent. Thus we can repeat the construction for $x_{1},\dots,x_{r+1}$.
	
	After at most $L$ inductions, we arrive at a sequence of elements $x_{i}=(\bold{0},J^{M}_{V_{i}}+f_{i})\in X$ for $1\leq i\leq r$ for some $1\leq r\leq L$ such that writing $B_{1}$ to  be the set of all $x\in X$ such that $x\sim x_{i}$ for all $1\leq i\leq r$, and $B_{2}=X\backslash B_{1}$, we have that either $r=L$ or  $(V_{1}+\dots+V_{r})\cap V\neq \{\bold{0}\}$ for all $x=(\bold{0},J^{M}_{V}+f)\in B_{1}$.
	
	We may further partition the set $B_{2}$ as the $\cup_{i=1}^{r}B_{2,i}$, where $x_{i}\not\sim x$ for all $x\in B_{2,i}$. Since $\dep(B_{2,i}\cup\{x_{i}\})\leq \dep(X)\leq D$, by Lemma \ref{2:lonely}, $\dep(B_{2,i})\leq D-1$.
	Since $r\leq L$,
	Applying the induction hypothesis to each of $B_{2,i}$, we have that 
	there exist $C':=C'(d,D,L), N':=N'(d,D,L)\in \N$, a set $Y'\subseteq \V$ which is a union of at most $N'$ subspaces of $\V$ of dimension at most $(s+2k-2)(k-1)$ such that
	we can partition $B_{2}=\cup_{i=1}^{r}B_{2,i}$ into $X'_{g}\cup X'_{b}$ such that 
	\begin{itemize}
		\item for all $(\bold{0},J^{M}_{V}+f)\in X'_{b}$, we have that $V\cap Y'\neq \{\bold{0}\}$;
		\item
		$\cc(X'_{g})\leq C'$.
	\end{itemize}
	
	If $r<L$, then $(V_{1}+\dots+V_{r})\cap V\neq \{\bold{0}\}$ for all $x=(\bold{0},J^{M}_{V}+f)\in B_{1}$. 
	Since $\pi(X)=\{\bold{0}\}$, by the definition of $\Gamma^{s}_{k}(\Vk,M)$, $V_{i}$ is of dimension at most $k-1$ and so
	$\dim(V_{1}+\dots+V_{r})\leq r(k-1)\leq (s+2k-2)(k-1)$, the conclusion holds by setting $C=C', N=N'+1$, $Y=Y'\cup (V_{1}+\dots+V_{r})$, $X_{g}=X'_{g}$ and $X_{b}=X'_{b}\cup B_{1}$.
	
	We now consider the case $r=L$. If $B_{1}$ is an equivalence class, then the conclusion holds by setting $C=C'+1, N=N'$, $Y=Y'$, $X_{b}=X'_{b}$ and $X_{g}=X'_{g}\cup B_{1}$. So it suffices to show that $B_{1}$ is an equivalence class. For all $x=(\bold{0},J^{M}_{V}+f),x'=(\bold{0},J^{M}_{V'}+f')\in B_{1}$, by definition, $x\sim x_{i}\sim x'$ for all $1\leq i\leq L$. So
	$f\equiv f_{i}\equiv f' \mod J^{M}_{V+V'+V_{i}}$. Thus 
	$$f\equiv f' \mod \cap_{i=1}^{L}J^{M}_{V+V'+V_{i}}.$$
	Since $d\geq 2(k-2)s+6k-1$, $V_{1},\dots, V_{L}$ are linearly independent and 
	$\dim(V+V')\leq 2(k-1)$,
	by Proposition \ref{2:gri} (setting $m=2(k-1)$ and $r=k-1$),
	 we have that 
	$f\equiv f' \mod J^{M}_{V+V'}$ and so $x\sim x'$.
	So  $B_{1}$ is an equivalence class and we are done.
\end{proof}

We need a variation of Proposition \ref{2:gweakcore1}  for later sections. Instead of requiring $X_{g}$ to be covered by at most $C$ complete subgraphs (or equivalence classes), sometimes we need to further require that each complete subgraph is also a strong equivalence class. For this purpose, we introduce the following definition.

\begin{defn}[$(L,C,D)$-structure-obstacle pair and classification]
	Let $C,D,K,L,s\in\N$, $d\in\N_{+}$ and $p$ be a prime dividing $K$. For
	$\mathcal{C}_{0}\subseteq \st_{\zp,d}(s)$ and $Y\subseteq \V$, we say that $(\mathcal{C}_{0},Y)$  is a \emph{$(L,C,D)$-structure-obstacle pair} (or simply a \emph{structure-obstacle pair}) if $\mathcal{C}_{0}$ is of cardinality at most $L+1$ and contains $0$ with $f\not\equiv g \mod J^{M}$ for all distinct $f,g\in\mathcal{C}_{0}$, and $Y$ is the union of at most $C$ subspaces of $\V$ of dimensions at most $D$.
	
	Let $M\colon\V\to\F_{p}$ be a non-degenerate quadratic form and $X$ be a subset of $\Gamma^{s}(\Vk,M)$ with $\pi(X)=\{\bold{0}\}$. We say that a  $(L,C,D)$-structure-obstacle pair $(\mathcal{C}_{0},Y)$  is a \emph{$(L,C,D)$-classification} (or simply a \emph{classification}) of $X$ if for all $(\bold{0},J^{M}_{V}+f)\in X$, either $V\cap Y\neq \{\bold{0}\}$ or  $f\equiv g \mod J^{M}_{V}$ for some $g\in \mathcal{C}_{0}$.
\end{defn}	

Informally, if $X$ admits a classification $(\mathcal{C}_{0},Y)$, then the set $\mathcal{C}_{0}$ describes the good part of $X$, which has the structure of unions of strong equivalent classes, while the set $Y$ describes the bad part of $X$, which is the obstacle that prevent us from concluding $\cc(X)=O_{\dep(X)}(1)$ in Proposition \ref{2:gweakcore1}.

Combining Propositions \ref{2:gwts} (i) and  \ref{2:gweakcore1},
we have the following variation of   Proposition \ref{2:gweakcore1}:

\begin{prop}[Strong decomposition result for relation graphs]\label{2:gweakcore2}
	Let $d,D,k,K\in\N_{+},s\in\N$ be such that $d\geq 2(k-2)s+6k-1$ 	if $k\geq 2$,  $p\gg_{d,D} 1$ be a prime dividing $K$, and $M\colon\V\to\F_{p}$ be a non-degenerate quadratic form. 	Then every $X\subseteq \Gamma^{s}_{k}(\Vk,M)$ with $\pi(X)=\{\bold{0}\}$ and $\dep(X)\leq D$ admits an $(O_{d,D}(1),O_{d,D}(1),(s+2k-2)(k-1))$-classification.	
\end{prop}	

\begin{rem}
 Using a similar method, one can generalize Propositions \ref{2:gweakcore1} and \ref{2:gweakcore2} for the case when $\pi(X)=\{h\}$ for some $h\in\Vk\backslash\{\bold{0}\}$ (with a slight different lower bound for the dimension $d$  the dimension of the subspace $W$ defined in the above mentioned propositions). However, we do not need them in this paper.
\end{rem}

\section{Properties for  Freiman $M$-homomorphisms of order 16}\label{2:s:117}
In this section, we provide details for Step 3 described in Section \ref{2:s:rdm}.
Let $H$ be a large subset of an abelian group $G$. It is known that the $2H-2H$ has a rich structure in combinatorics. 
To see this, let $S$ be a subset of the Pontryagin dual $\hat{\Vk}$ of $\Vk$ and $\rho>0$. Denote
$$B(S,\rho):=\{h\in \Vk\colon \Vert \xi\cdot h\Vert_{\T}<\rho \text{ for all } \xi\in S\},$$
where $\Vert a\Vert_{\T}$ denotes the distance between $a$ and the closest integer.
For $H\subseteq \Vk$ and $x\in \Vk$, denote
$$R(H,h):=\{(h_{1},h_{2},h_{3},h_{4})\in H\colon h_{1}+h_{2}-h_{3}-h_{4}=h\}.$$
We have the following structure theorem for $2H-2H$.

\begin{prop}\label{2:g324}
	Let $d,K\in\N_{+}, 0<c\leq 1$, $\d>0$, $p$ be a prime dividing $K$, and $H\subseteq \Vk$ be such that $\vert H\vert\geq \d K^{d}$. Then there exist  
	 a set $S\subseteq \hat{\Vk}$ with $\vert S\vert\leq 2\d^{-2}$ and
	a proper and homogeneous generalized arithmetic progression 
		$$P=\{x_{1}v_{1}+\dots+x_{D}v_{D}\colon x_{i}\in \Z\cap(-L_{i},L_{i}), 1\leq i\leq D\}$$
	 in $\Vk$ of rank $D$ for some $D\leq \vert S\vert+d$ and cardinality $\gg_{\d,d} K^{d}$ such that the followings hold:
	\begin{enumerate}[(i)]
		\item $B(S,cD^{-2D}/4)\subseteq P(c)\subseteq P\subseteq B(S,1/4)\subseteq 2H-2H$;
		\item $\vert R(H,h)\vert\geq \d^{4}K^{3d}/2$ for all $h\in P$;
		\item there exists $H'\subseteq H$ with $\vert H'\vert\gg_{c,\d,d} K^{d}$ such that $2H'-2H'\subseteq P(c)$;
		\item the vectors $(\{\alpha\cdot v_{i}\})_{\alpha\in S}\in(\frac{1}{K}\Z)^{\vert S\vert}\subseteq \R^{\vert S\vert}, 1\leq i\leq D$ are linearly independent over $\R$, where $\{x\}$ is the embedding of $x$ from $\R$ to $\T$ in the fundamental domain $[0,1)$.
	\end{enumerate}	
\end{prop}	
\begin{proof}
	For $h\in \Vk$, denote $f\ast g(h):=\E_{h'\in \Vk}f(h')g(h-h')$ for all functions $f,g\colon \Vk\to \R$. It was proved in Lemma 6.3 of \cite{GT08b} (although not stated explicitly) that there exists $S\subseteq \hat{\Vk}$ with $\vert S\vert\leq D_{0}:=2\d^{-2}$ such that for all $h\in B(S,1/4)$,
	we have that 
	$$\bold{1}_{H}\ast\bold{1}_{H}\ast\bold{1}_{-H}\ast\bold{1}_{-H}(h)\geq \d^{4}/2,$$
	This means that $\vert R(H,h)\vert\geq \d^{4}K^{3d}/2$. In particular, $B(S,1/4)\subseteq 2H-2H$.

	By Lemma 10.4 of \cite{GT08b}, there exists a proper generalized arithmetic  progression $P'$ in $\V$ of rank $D'$ for some $0\leq D'\leq \vert S\vert$, 
	a subgroup $Y$ of $\Vk$ such that writing $P:=P'+Y$, we have that $\vert P\vert=\vert P'\vert\vert Y\vert$ (which implies that $P$ is a proper generalized arithmetic  progression in $\V$ of rank at most $D'+d\leq 2\d^{-2}+d$)
	and
	\begin{equation}\label{2:g10.2}
		B(S,cD^{-2D}/4)\subseteq P(c)\subseteq B(S,c/4) \text{ and } 	B(S,D^{-2D}/4)\subseteq P\subseteq B(S,1/4).\footnote{Strictly speaking, Lemma 10.4 of \cite{GT08b} only states that $B(S,D^{-2D}/4)\subseteq P\subseteq B(S,1/4).$ But the conclusion that $B(S,cD^{-2D}/4)\subseteq P(c)\subseteq B(S,c/4)$ follows from the proof of Lemma 10.4 of \cite{GT08b}.}
	\end{equation}	
	Moreover, the vectors $(\{\alpha\cdot v_{i}\})_{\alpha\in S}\in\R^{\vert S\vert}, 1\leq i\leq D$ can be made  linearly independent over $\R$ (and so Property (iv) is satisfied).
	It then follows from Lemma 8.1 of \cite{GT08b} that $\vert P\vert\gg_{\d,d}K^{d}$.
  So  Property (i) is satisfied.
	
	Since $P\subseteq B(S,1/4)$, $\vert R(H,h)\vert\geq \d^{4}K^{3d}/2$ for all $h\in P$. 
	This proves  Property (ii).
	
	We now prove Property (iii). Let $\rho=cD^{-2D}/16$.
Since $\vert H\vert\geq \d K^{d}$, by Lemma 4.1 of \cite{GT08b}, there exists $x_{0}\in \Vk$ such that
	$$\E_{h\in B(S,\rho)}\bold{1}_{H}(x_{0}+h)\geq \E_{h\in \Vk}\bold{1}_{H}(h)\geq \d.$$
	Set 
	$H':=\{x\in H\colon x-x_{0}\in B(S,\rho)\}.$
	We have that $\vert H'\vert\geq \d \vert B(S,\rho)\vert\gg_{c,\d,d}K^{d}$ by Lemma 8.2 of  \cite{GT08b}. Moreover, $2H'-2H'\subseteq   B(S,4\rho)\subseteq P(c)$ by (\ref{2:g10.2}). This finishes the proof.
\end{proof}

 We may derive the following description for  Freiman $M$-homomorphism of order 16 using Proposition \ref{2:g324} (recall the definition in Section \ref{2:s:ffh}):

\begin{prop}\label{2:gef1}
	Let $d,K\in\N_{+},s\in\N$ with $d\geq \max\{4s+19,35\}$, $0<c\leq 1$, $\d>0$, $p\gg_{\d,d} 1$ be a prime dividing $K$,  $M\colon\V\to\F_{p}$ be a non-degenerate quadratic form, and $H$ be a subset of $\Vk$ of cardinality at least $\d K^{d}$. For any Freiman $M$-homomorphism $\xi\colon H\to\st_{\zp,d}(s)$  of order 16,
	 there exist $0<\rho=\rho(\d,d)<1/4$,
	a homogeneous proper generalized arithmetic progression 
		$$P=\{x_{1}v_{1}+\dots+x_{D}v_{D}\colon x_{i}\in \Z\cap(-L_{i},L_{i}), 1\leq i\leq D\}$$
	 in $\Vk$ of rank $O_{\d,d}(1)$ and cardinality $\gg_{\d,d}K^{d}$, a set $S\subseteq \hat{\Vk}$ of cardinality $O_{\d}(1)$, and a map $\omega\colon P\to \st_{\zp,d}(s)$ such that the followings hold:
	\begin{enumerate}[(i)]
		\item $B(S,c\rho)\subseteq P(c)\subseteq P\subseteq B(S,1/4)\subseteq 2H-2H$;
		\item For all $h\in P$, the set $R(H,h):=\{(h_{1},h_{2},h_{3},h_{4})\in H\colon h_{1}+h_{2}-h_{3}-h_{4}=h\}$ is of cardinality $\gg_{\d} K^{3d}$;
			\item there exists $H'\subseteq H$ with $\vert H'\vert\gg_{c,\d,d} K^{d}$ such that $2H'-2H'\subseteq P(c)$;
			\item the vectors $(\{\alpha\cdot v_{i}\})_{\alpha\in S}\in\R^{\vert S\vert}, 1\leq i\leq D$ are linearly independent over $\R$;
		\item for all $u_{1},u_{2},u_{3},u_{4}\in H$ with $u_{1}+u_{2}-u_{3}-u_{4}\in P$, we have
		$$\xi(u_{1})+ \xi(u_{2})- \xi(u_{3})- \xi(u_{4})\equiv\omega(u_{1}+u_{2}-u_{3}-u_{4}) \mod J^{M}_{\iota(u_{1}),\iota(u_{2}),\iota(u_{3}),\iota(u_{4})};$$
		\item for all $u_{1},u_{2},u_{3},u_{4}\in P$ with $u_{1}+u_{2}=u_{3}+u_{4}$, we have
		$$\omega(u_{1})+ \omega(u_{2})\equiv\omega(u_{3})+ \omega(u_{4}) \mod J^{M}_{\iota(u_{1}),\iota(u_{2}),\iota(u_{3})},$$
		i.e. the map $\omega\colon P\to\st_{\zp,d}(s)$ is a Freiman $M$-homomorphism of order 4.
	\end{enumerate}	
\end{prop}	
\begin{proof}
	We always assume that $p\gg_{\d,d} 1$ in the proof.
	By Proposition \ref{2:g324}, there exist $P, S$ and $\rho$ satisfying Properties (i)-(iv). 
	Denote $\tilde{\xi}(h):=(h,J^{M}_{\iota(h)}+\xi(h))$ for $h\in H$.
	Since $\xi$ is a Freiman $M$-homomorphism of order 16,
  $\pi^{-1}(h)\cap (2\tilde{\xi}(H)\- 2\tilde{\xi}(H))$ is an equivalence class for all $h\in P$. 
  
  We claim that $\pi^{-1}(h)\cap (2\tilde{\xi}(H)\- 2\tilde{\xi}(H))$ is actually a strong equivalence class for all $h\in P$. Indeed, fix 
    $h\in P$.  Since $d\geq 4s+19$,
  by Proposition \ref{2:gwts}  (setting $k=4$),
  it suffices to show that for any subspace of $\V$  containing $\iota(h)$ of dimension at most $3s+12$, there exists $(h_{1},h_{2},h_{3},h_{4})\in R(H,h)$ such that $\sp_{\F_{p}}\{\iota(h_{1}),\iota(h_{2}),\iota(h_{3}),\iota(h_{4})\}\cap Y\neq \sp_{\F_{p}}\{\iota(h)\}$. This is clearly true since $d\geq 3s+15$.
  
  The claim implies that for all $h\in H$, there exists $\omega(h)\in \st_{\zp,d}(s)$ such that Property (v) holds.
	Since $\xi$ is a Freiman $M$-homomorphism of order 16,
	it follows from Property (v) that for all $u_{1},u_{2},u_{3},u_{4}\in P$ with $u_{1}+u_{2}=u_{3}+u_{4}$ and for all $(u_{i,1},u_{i,2},u_{i,3},u_{i,4})\in R(H,u_{i})$, $1\leq i\leq 4$, since  $u_{i,4}=u_{i,1}+u_{i,2}-u_{i,3}$, we have	
		$$\omega(u_{1})+ \omega(u_{2})\equiv\omega(u_{3})+ \omega(u_{4}) \mod J^{M}_{\sp_{\F_{p}}\{\iota(u_{j}),\iota(u_{i,j})\colon 1\leq i\leq 4, 1\leq j\leq 3\}}.$$	 
	Since  $\prod_{i=1}^{4}\vert R(H,u_{i})\vert\gg_{\d}K^{12d}$ and $d\geq \max\{15+s,35\}$, by Proposition \ref{2:grm} (setting $\dim(V)=3$, $m=12$, $k=0$, $\dim(U_{0}=0)$) and Lemma \ref{2:iiddpp},  we have that 
	$$\omega(u_{1})+ \omega(u_{2})\equiv\omega(u_{3})+ \omega(u_{4}) \mod J^{M}_{\iota(u_{1}),\iota(u_{2}),\iota(u_{3})}.$$
	So Property (vi) holds.
\end{proof}

To conclude Step 3 described in Section \ref{2:s:rdm}, we need the following result:

\begin{prop}\label{2:g1621}
	Let $C,D,K,L,s\in\N$, $d\in\N_{+}$, $\d>0$, $p$ be a prime dividing $K$, and $M\colon\V\to\F_{p}$ be a non-degenerate quadratic form. Let $H\subseteq \Vk$ with $\vert H\vert\geq \d K^{d}$, $\xi\colon H\to \st_{\zp,d}(s)$ be a map, and denote $\tilde{\xi}(h):=(h,J^{M}_{\iota(h)}+\xi(h))\in\Gamma^{s}_{1}(\Vk,M)$ for all $h\in H$. Suppose that $\pi^{-1}(\bold{0})\cap(8\tilde{\xi}(H)\- 8\tilde{\xi}(H))$ admits a $(L,C,D)$-classification $(\mathcal{C}_{0},Y)$.
	Then for any subspace $V$ of $\V$ with $J^{M}_{V}\cap\mathcal{C}_{0}=\{0\}$  and $V\cap Y=\{\bold{0}\}$, and any subset  $W\subseteq H\cap \iota^{-1}(V)$,\footnote{Here $\iota^{-1}(V)$ is understood as the set of $v\in\Vk$ with $\iota(v)\in V$.} there exists $W'\subseteq W$ with $\vert W'\vert\geq \vert W\vert/80^{L}$ such that  
$\xi$ is   a Freiman $M$-homomorphism of order 16 on $W'$.
\end{prop}

Proposition \ref{2:g1621} should be viewed as the analog of Lemma 9.2 of \cite{GT08b}. Unfortunately, we were unable to prove that Proposition \ref{2:g1621}  holds with $V$ replaced $\V$ (we require the subspace $V$  to ``stay away" from $\mathcal{C}_{0}$ and $Y$). This incurs some additional difficulty in the final step of the proof of Theorem \ref{2:aadd}, as we shall see in Section \ref{2:s:a2}.

\begin{proof}[Proof of Proposition \ref{2:g1621}]
 Since $\mathcal{C}_{0}$ and $\Vk$ are finite sets, there exists $N\in\N$ such that $\mathcal{C}_{0}, \xi(H)\subseteq \st_{\Z/p^{N},d}(s)$. 
 Note that   $\st_{\Z/p^{N},d}(s)/(\st_{\Z/p^{N},d}(s)\cap J^{M}_{V})$ is a finite abelian group, and is thus isomorphic to groups of the form $\Z/p^{a_{1}}\Z\times \dots\times \Z/p^{a_{t}}\Z$ for some $t\in\N_{+}$ and $1\leq a_{1},\dots,a_{m}\leq N$. 
So there exist $1\leq m\leq \binom{d+s-1}{s}$ ($\binom{d+s-1}{s}$ is the number of monomials in $d$ variables of degree $s$), $v_{1},\dots,v_{m}\in \st_{\Z/p^{N},d}(s)$ and some $1\leq a_{1},\dots,a_{m}\leq N$ such that $p^{a_{i}}v_{i}\in \st_{\Z,d}(s)$  for all $1\leq i\leq m$, and that the map $\theta\colon G:=\prod_{i=1}^{m}\Z/p^{a_{i}}\Z\to \st_{\Z/p^{N},d}(s)/(\st_{\Z/p^{N},d}(s)\cap J^{M}_{V})$ induced by 
 $$\theta(x_{1},\dots,x_{m}):=x_{1}v_{1}+\dots+x_{m}v_{m}$$
 is a group isomorphism. 

  Enumerate $\mathcal{C}_{0}$ as $J=\{c_{0}=\bold{0},c_{1},\dots,c_{L}\}$.  We may assume without loss of generality that   $c_{i}\notin J_{V}^{M}$ for all $1\leq i\leq L$.
Then for all $1\leq i\leq L$,  we have that
\begin{equation}\label{2:ewfeff}
\theta^{-1}(c_{i} \mod \st_{\Z/p^{N},d}(s)\cap J^{M}_{V})\neq \bold{0}.
\end{equation}
Let $\phi_{i}\colon G\to\T$ be the linear map given by
$$\phi_{i}(x_{1} \mod p^{a_{1}}\Z,\dots,x_{m} \mod p^{a_{m}}\Z):=\frac{1}{p^{a_{1}}}b_{i,1}x_{1}+\dots+\frac{1}{p^{a_{m}}}b_{i,m}x_{m} \mod \Z$$
for some $b_{i,1},\dots,b_{i,m}\in\Z$. By (\ref{2:ewfeff}), it is 
  not hard to find some appropriate $b_{i,1},\dots,b_{i,m}\in \Z$ such that $$\phi_{i}\circ\theta^{-1}(c_{i} \mod\st_{\Z/p^{N},d}(s)\cap J^{M}_{V})\notin [-1/10,1/10].$$
Let $\Phi\colon G\to \T^{L}$ be the linear map given by 
$$\Phi(x)=(\phi_{1}(x),\dots,\phi_{L}(x))$$
for all $x\in  G$. Then 
$$\Phi\circ\theta^{-1}(c_{i} \mod \st_{\Z/p^{N},d}(s)\cap J^{M}_{V})\notin  [-1/10,1/10]^{L}$$ for all $1\leq i\leq L$.

We may cover $\T^{L}$ by cubes of the form $Q=[\frac{j_{1}}{80},\frac{j_{1}+1}{80})\times \dots\times [\frac{j_{L}}{80},\frac{j_{L}+1}{80})\subseteq \T^{L}$, for some $0\leq j_{1},\dots,j_{L}\leq 79$, and cover $G$ by sets of the form $\Phi^{-1}(Q)$.
	By the Pigeonhole Principle, there exists a cube $Q$ such that the set
	$$W':=\{h\in W\colon \Phi\circ\theta^{-1}(\xi(h) \mod \st_{\Z/p^{N},d}(s)\cap J^{M}_{V})\in Q\}$$
	is of cardinality at least $\vert W\vert/80^{L}$.
	
	It remains to show that $\xi$ is   a Freiman $M$-homomorphism of order 16 on $W'$.
	Let $h_{1},\dots,h_{16}\in W'$ with $h_{1}+\dots+h_{8}=h_{9}+\dots+h_{16}$. Denote
	$$\omega:=(\xi(h_{1})+\dots+ \xi(h_{8}))- (\xi(h_{9})+\dots+ \xi(h_{16})).$$
	It suffices to show that $\omega\in J^{M}_{\iota(h_{1}),\dots,\iota(h_{16})}$.
	Since $h_{1}+\dots+h_{8}=h_{9}+\dots+h_{16}$, $V\cap Y=\{\bold{0}\}$ and $(\mathcal{C}_{0},Y)$ is a classification of $\pi^{-1}(\bold{0})\cap(8\tilde{H}\- 8\tilde{H})$, there exists $f\in\mathcal{C}_{0}$ such that 
	$\eta:=f-\omega \in J^{M}_{\iota(h_{1}),\dots,\iota(h_{16})}\cap\st_{\Z/p^{N},d}(s)\subseteq J^{M}_{V}\cap\st_{\Z/p^{N},d}(s)$. Since 
	$$\Phi\circ\theta^{-1} (\xi(h_{j}) \mod \st_{\Z/p^{N},d}(s)\cap J^{M}_{V})\in Q$$ for all $1\leq j\leq 16$, we have that $$\Phi\circ\theta^{-1}(\omega \mod \st_{\Z/p^{N},d}(s)\cap J^{M}_{V})\in [-1/10,1/10]^{L}.$$ 
	Since 
	$\eta\in J^{M}_{V}\cap\st_{\Z/p^{N},d}(s)$, we have that 
	$$\Phi\circ\theta^{-1}(\eta \mod \st_{\Z/p^{N},d}(s)\cap J^{M}_{V})=\bold{0}.$$ 
	Since $f\in\mathcal{C}_{0}$, we have that 
	$$\Phi\circ\theta^{-1}(f \mod\st_{\Z/p^{N},d}(s)\cap J^{M}_{V})\notin  [-1/10,1/10]^{L}.$$
	unless $f=0$. Since $f-\omega=\eta$, we must have that $f=0$ and thus $\omega=\eta\in J^{M}_{V}$. This completes the proof.
\end{proof}

\section{Structure-obstacle decomposition}\label{2:s:a3}

Note that if one can further require the subspace $V$ appearing in Proposition  \ref{2:g1621} to be $\V$, then Theorem \ref{2:aadd} follows immediately from Theorem \ref{2:gsol}, Propositions \ref{2:gef1} and \ref{2:g1621}. Unfortunately we do not know if such an improvement of  Proposition  \ref{2:g1621} is possible.
To overcome this difficulty, we start by using  Proposition  \ref{2:g1621} to study the Freiman $M$-homomorphism structure on certain subspaces of $\V$.
To this end, we need to introduce the following decomposition of $\V$ which is suitable for our needs.

\begin{defn}[Structure-obstacle decomposition]
Let $d\in\N_{+},d',d'',s\in\N$ with $d=d'+d''$, and $p$ be a prime, and $M\colon \V\to\F_{p}$ be a quadratic form. 
Let $(\mathcal{C}_{0},Y)\subseteq \st_{\zp,d}(s)\times \V$  be a structure-obstacle pair. Let $T,U$ be subspaces of $\V$. We say that $(T,U)$ is an \emph{$(\mathcal{C}_{0},Y,M,d',d'')$-structure-obstacle decomposition} of $\V$ if the following conditions hold:
\begin{itemize}
	\item  $\rank(M\vert_{U})=dim(U)=d''$;
	\item $J_{U}^{M}\cap\mathcal{C}_{0}=\{0\}$ and $U\cap Y=\{\bold{0}\}$;
	\item $T=U^{\perp_{M}}$.
\end{itemize}
%
\end{defn}

Roughly speaking, if $(T,U)$ is a structure-obstacle decomposition, then $U$ is a subspace of $\V$ where Proposition  \ref{2:g1621} applies (and $M\vert_{U}$ is non-degenerate). Treating $\V$ as the unions of $U+v, v\in T$, this allows us 
describe the structure of $\xi\vert_{U+v}$ using the results proved in previous sections.

It follows from Lemma \ref{2:iissoo} that for any structure-obstacle decomposition $(T,U)$, $M\vert_{U}$ and $M\vert_{T}$ are non-degenerate,  $T\cap U=\{\bold{0}\}$, and $T+U=\V$.

The rest of this section is devoted to proving the following proposition:
 
\begin{prop}\label{2:laststand}
	Let $C,d',d'',D,L,s\in\N, d,K\in\N_{+}$ with $d=d'+d''$, $\d>0$, $p\gg_{C,\d,d,L} 1$ be a prime dividing $K$, and $M\colon\V\to\F_{p}$ be a non-degenerate quadratic form. Let $H\subseteq \Vk$ with $\vert H\vert\geq \d K^{d}$, $\xi\colon H\to \st_{\zp,d}(s)$ be a map, and denote $\tilde{\xi}(h):=(h,J^{M}_{\iota(h)}+\xi(h))\in\Gamma^{s}_{1}(\Vk,M)$ for $h\in H$. Suppose that $\pi^{-1}(\bold{0})\cap(8\tilde{\xi}(H)\- 8\tilde{\xi}(H))$ admits a $(L,C,D)$-classification $(\mathcal{C}_{0},Y)$, and let $(T_{0},U_{0})$ be a $(\mathcal{C}_{0},Y,M,d',d'')$-structure-obstacle decomposition of $\V$. If 
	$d'\geq \max\{s+1,D+3,7\}$ and $d''\geq \max\{4s+23,35\}$,	then there exist a subset $H'$ of $H$ of cardinality $\gg_{\d,d,L} K^{d}$ 
	such that for all $h_{1},\dots,h_{16}\in H'$ with $h_{1}+\dots+h_{8}=h_{9}+\dots+h_{16}$, we have that 
		$$\xi(h_{1})+\dots+\xi(h_{8})\equiv\xi(h_{9})+\dots+\xi(h_{16}) \mod J^{M}_{\sp_{\F_{p}}\{\iota(h_{1}),\dots,\iota(h_{15})\}+T_{0}}.$$
%
\end{prop}

We briefly explain the idea of the proof of Proposition \ref{2:laststand}.
Roughly speaking, modulo $J^{M}_{T_{0}}$, Proposition \ref{2:g1621} ensures that $\xi$ is a Freiman $M$-homomorphism of order 16 on many layers of $U_{0}+v, v\in T_{0}$. We then apply Proposition \ref{2:gef1} to compare the structure of  $\xi$ on different layers, and conclude that $\xi$ is a Freiman $M$-homomorphism of order 16 on a large subset of $H$ (again modulo $J^{M}_{T_{0}}$).
To be more precise, the proof of Proposition \ref{2:laststand} is divided into the following steps.

\textbf{Step 0. Preparisions.}
 In the rest of the section we assume that $p\gg_{C,\d,d,L} 1$. Let $A$ be the matrix associated to $M$. We may assume without loss of generality that $M$ is homogenous, i.e. $M(n)=(nA)\cdot n$. 
 Since $H$ is a finite set, we may assume that $\xi(H)\subseteq \st_{\Z/p^{N},d}(s)$ for some $N\in\N_{+}$. Extending $\xi$ periodically to the domain $\Z_{p^{N}K}^{d}$ if necessary, we may assume without loss of generality that $f(n+Km)\equiv f(n) \mod\Z$ for all $f\in\st_{\Z/p^{N},d}(s)$ and $m,n\in\Z^{d}$. We may identify $\xi(h)$ as polynomial map from the ring $\Vk$ to $\Z/p^{N}$.

 We first need to extend the decomposition $\V=T_{0}+U_{0}$ to decompositions of  $\Vk$ and $\Z^{d}$. To achieve this, we need to following lemma.

\begin{lem}\label{2:unist}
    Let $d\in\N_{+}$, $p$ be a prime and $v_{1},\dots,v_{d}$ be a basis of $\F_{p}^{d}$. Then there exist $c_{1},\dots,c_{d}\in \{1,\dots,p-1\}$ and $u_{1},\dots,u_{d}\in\Z^{d}$ with $\iota(u_{i})=c_{i}v_{i}$ for all $1\leq i\leq d$ such that the determinant of the $d\times d$ matrix with $u_{1},\dots,u_{d}$ as row vectors is equal to 1.
\end{lem}
\begin{proof}
    For $u_{1},\dots,u_{d}\in\Z^{d}$, let $B(u_{1},\dots,u_{d})$ denote the $d\times d$ matrix haivng $u_{1},\dots,u_{d}$ as row vectors. Let $m\in\Z$ be the smallest integer in absolute value such that there exist $c_{1},\dots,c_{d}\in \{1,\dots,p-1\}$ and $u_{1},\dots,u_{d}\in\Z$ with $\iota(u_{i})=c_{i}v_{i}$ for all $1\leq i\leq d$ such that $\det(B(u_{1},\dots,u_{d}))=m$. Since $\iota(B(u_{1},\dots,u_{d}))$ is non-degenerate, we have that $m\notin p\Z$ and in particular $m\neq 0$. Replacing $u_{1}$ by $-u_{1}$ if necessary, we may assume without loss of generality $m>0$.
    
    It suffices to show that $m=1$. Suppose on the contrary that $m>1$. Let $b_{i,j}$ denote the $(i,j)$-th entry of $B(u_{1},\dots,u_{d})$, and $B_{i,j}$ be the $(d-1)\times (d-1)$ matrix obtained by removing the $i$-th row and $j$-th column of $B(u_{1},\dots,u_{d})$. Let $t$ be the greatest common divider for all the non-zero numbers among $\det(B_{1,1}),\dots,\det(B_{1,d})$. Then we must have that  $m=kt$ for some $k\in\N_{+}$. Since $m\notin p\Z$, we have that $t\notin p\Z$. If $k\neq 1$, then there exist $x,y\in\Z$ such that $xk+yp=1<k$. Clearly $x\notin p\Z$. Let $v_{1}\in\Z^{d}$ be such that $v_{1}\cdot (\det(B_{1,1}),\dots,\det(B_{1,d}))=t$.
    Then
    $$B(xu_{1}+ypv_{1},\dots,u_{d})=xB(u_{1},\dots,u_{d})+yB(v_{1},u_{2},\dots,u_{d})=xkt+ypt=t<kt=m.$$
    Since $\iota(xu_{1}+ypv_{1})=x\iota(u_{1})=\iota(xc_{1})v_{1}$ and since $x\notin p\Z$, we have a contradiction to the minimality of $m$.
    
    So we must have that $k=1$ and $m=t$. So $\det(B_{1,j})\in m\Z$ for all $1\leq j\leq d$. Using a similar argument for other rows of $B(u_{1},\dots,u_{d})$, we get a similar contradiction unless  $\det(B_{i,j})\in m\Z$ for all $1\leq i,j\leq d$. Since $\det(B(u_{1},\dots,u_{d}))=m$, by the formula of inverse matrices, $B(u_{1},\dots,u_{d})^{-1}$ is a matrix with integer entries. So $m^{-1}=\det(B(u_{1},\dots,u_{d}))^{-1}\in\Z$. This is only possible if $m=\pm 1$, a contradiction. So we must have that $m=1$ and we are done.    
 \end{proof}

Let $v_{1},\dots,v_{d''}$ be a basis of $U_{0}$ and $v_{d''+1},\dots,v_{d}$ be a basis of $T_{0}$. Let $c_{1},\dots,c_{d}$ and $u_{1},\dots,u_{d}$ be given by Lemma \ref{2:unist}. Let $\tilde{U}$ be the set of all linear combinations of $u_{1},\dots,u_{d''}$ over $\Z$, and $U\subseteq\Vk$ be the set of all linear combinations of $u_{1},\dots,u_{d''}$ over $\Z$ modulo $K\Z^{d}$. Define $\tilde{T}$ and $T$ similarly. 
Since $\iota(u_{i})=c_{i}v_{i}$, we have that $\iota(\tilde{T})=\iota(T)=T_{0}$ and  $\iota(\tilde{U})=\iota(U)=U_{0}$. 
For  $1\leq i\leq d''$, let $\phi_{\tilde{U}}\colon\Z^{d''}\to \Z^{d}$ be the linear map given by $\phi_{\tilde{U}}(e_{i}):=u_{i}$, $\phi_{U}\colon\Z_{K}^{d''}\to U$ be the linear map given by $\phi_{U}(e_{i}):=u_{i} \mod \Vk$, and $\phi_{U_{0}}\colon\F_{p}^{d''}\to U_{0}$ to be the linear map given by $\phi_{U_{0}}(e_{i}):=\iota(u_{i})=c_{i}v_{i}$, where $e_{i}$ is the $i$-th standard unit vector. 
Denote $\phi_{\tilde{T}},\phi_{T}$ and $\phi_{T_{0}}$ in similar ways.
It is not hard to see that all of $\tilde{\phi}_{U},\tilde{\phi}_{T},\phi_{U},\phi_{T},\phi_{U_{0}}$ and $\phi_{T_{0}}$ are bijections. Moreover, 
since the determinant of the $d\times d$ matrix with $u_{1},\dots,u_{d}$ as row vectors is equal to 1, 
 we have that $\tilde{T}+\tilde{U}=\Z^{d}$, $\tilde{T}\cap \tilde{U}=\{\bold{0}\}$, $T+U=\Vk$ and $T\cap U=\{\bold{0}\}$. 
Let $\pi_{T_{0}}\colon \V\to T_{0}$, $\pi_{U_{0}}\colon \V\to U_{0}$ be the projection maps with respect to the decomposition $\V=T_{0}+U_{0}$.


In this proof of Proposition \ref{2:laststand}, we need to frequently project vectors in $\Z^{d}$ to the $\tilde{U}$ direction.
%
For $m\in\Z^{d'}$,
Let $\sigma_{m}\colon \st_{\zpn,d}(s)\to \st_{\zpn,d''}(s)$ be the map given by $$\sigma_{m}(f)(n):=f(\phi_{\tilde{U}}(n)+p\phi_{\tilde{T}}(m)).$$ 
Let $M'\colon \F_{p}^{d''}\to\F_{p}$ be the quadratic form given by $M'=M\circ \phi_{U_{0}}$. Since $\rank(M\vert_{U_{0}})=d''$, $M'$ is non-degenerate. 
For $i\in\N_{+}$,
let $\Gamma^{s}_{i}(\Z_{K}^{d''},M')$ be the set of all $(h,I+f)$ where $h\in \Z_{K}^{d''}$, $f\in \st_{\zpn,d''}(s)$ and $I=J^{M'}_{V}$ for some subspace $V$ of $\F_{p}^{d''}$  containing $\iota(h)$ with $\dim(V)\leq i$ if $\iota(h)\neq \bold{0}$ and $\dim(V)\leq i-1$ if $\iota(h)=\bold{0}$. 
Let $\sim'$ be a relation defined on $\Gamma^{s}(\Z_{K}^{d''},M')$ such that
$(h,I+f)\sim' (h',I'+f')$ if and only if $h=h'$ and $f-f'\in I+I'$. 

\begin{prop}\label{2:redf}
Let $f\in\st_{\zpn,d}(s)$ and $V$ be a subspace  of $\V$ of dimension $k$. Suppose that $d\geq 2k+3$.  Then 
\begin{equation}\nonumber
 f\in J^{M}_{V}\Rightarrow \sigma_{m}(f)\in J^{M'}_{\phi_{U_{0}}^{-1}(\pi_{U_{0}}(V))} \text{ for all } m\in\Z^{d'}.
 \end{equation}
 Conversely, there exist $N=O_{d,s}(1)$ 
 such that 
 \begin{equation}\nonumber
  \sigma_{m}(f)\in J^{M'}_{\phi_{U_{0}}^{-1}(\pi_{U_{0}}(V))} \text{ for all } m\in [N]^{d'}\Rightarrow f\in J^{M}_{V+T_{0}}.
 \end{equation}
\end{prop}
\begin{proof} 
Let $v_{1},\dots,v_{k}$ be  a basis of $V$. 
Assume that $\phi_{U_{0}}(n)=nB$ for some $d''\times d$ matrix $B$. Then $BAB^{T}$ is the matrix associated with $M'$. 
Denote $L_{v}(n):=\frac{1}{p}(v\tau(A))\cdot n$ for all $v, n\in\Z^{d}$ and $L'_{v}(n):=\frac{1}{p}(v\tau(BAB^{T}))\cdot n$ for all $v, n\in\Z^{d''}$.


Fix any $m\in\Z^{d'}$ and $n\in\Z^{d''}$. Let $\tilde{M}$ and $\tilde{M}'$ be the regular liftings of $M$ and $M'$ respectively. Then
 \begin{equation}\label{2:redf4}
\sigma_{m}(\tilde{M})(n)\equiv\frac{1}{p}(\phi_{\tilde{U}}(n)\tau(A))\cdot \phi_{\tilde{U}}(n)\equiv \frac{1}{p}(n\tau(BAB^{T}))\cdot n=\tilde{M}'(n)\mod\Z.
 \end{equation}
If $v\in U_{0}$, then assuming that $v=\phi_{0}(w)$ for some $w\in\F_{p}^{d''}$, we have that 
\begin{equation}\nonumber
\begin{split}
 \sigma_{m}(L_{\tau(v)})(n)\equiv\frac{1}{p}(\tau(\phi_{U_{0}}(w))\tau(A))\cdot \phi_{\tilde{U}}(n)\equiv\frac{1}{p}\tau(wBAB^{T})\cdot n\equiv L'_{\tau(w)}(n)=L'_{\tau(\phi_{U_{0}}^{-1}(v))}(n) \mod\Z.
\end{split}
 \end{equation}
On the other hand, if $v\in T_{0}$, then since $T_{0}=U_{0}^{\perp_{M}}$, we have that 
$$\sigma_{m}(L_{\tau(v)})(n)\equiv\frac{1}{p}(\tau(vA)\cdot \phi_{\tilde{U}}(n))\equiv \frac{1}{p}\tau(vA\cdot \iota\circ\phi_{\tilde{U}}(n))=\frac{1}{p}\tau(vA\cdot \phi_{U_{0}}\circ\iota(n))=0 \mod\Z.$$
So
 \begin{equation}\label{2:redf3}
\text{$\sigma_{m}(L_{\tau(v)})(n)\equiv L'_{\tau\circ\phi_{U_{0}}^{-1}(\pi_{U_{0}}(v))}(n) \mod\Z$
for all $v\in\V$, $m\in\Z^{d'}$ and $n\in\Z^{d''}$.}
 \end{equation}

 For all $f\in J^{M}_{V}$, since $d\geq 2k+3$,
 by Proposition  \ref{2:basicpp12}, we may write
 \begin{equation}\nonumber
		f=\sum_{i=(i_{0},\dots,i_{k})\in\N^{k}\colon 2i_{0}+i_{1}+\dots+i_{k}\leq s}R_{i}\tilde{M}^{i_{0}}\prod_{j=1}^{k}L_{\tau(h_{j})}^{i_{j}}
		\end{equation}
	for some  $R_{i}\in \st_{\Z,d}(s-(2i_{0}+i_{1}+\dots+i_{k}))$. By  (\ref{2:redf4}) and (\ref{2:redf3}), we have that 
 \begin{equation}\nonumber
		\sigma_{m}(f)-\sum_{i=(i_{0},\dots,i_{k})\in\N^{k}\colon 2i_{0}+i_{1}+\dots+i_{k}\leq s}\sigma_{m}(R_{i}){\tilde{M'}}^{i_{0}}\prod_{j=1}^{k}{L'}_{\tau\circ\phi_{U_{0}}^{-1}\circ\pi_{U_{0}}(h_{j})}^{i_{j}}
		\end{equation}	
		is an integer valued polynomial with $\Z/p^{N}$-coefficients and thus must
	belong to $\st_{\Z,d}(s)$. So  $\sigma_{m}(f)\in J^{M'}_{\phi_{U_{0}}^{-1}(\pi_{U_{0}}(V))}.$

 



Conversely, suppose that 
$\sigma_{m}(f)\in J^{M'}_{\phi_{U_{0}}^{-1}(\pi_{U_{0}}(V))} \text{ for all } m\in [N]^{d'}$. 
Fix any $n\in\Z^{d''}$ with $\iota(n)\in V(M')\cap (\phi_{U_{0}}^{-1}(\pi_{U_{0}}(V)))^{\perp_{M'}}$.
 We have that 
$\sigma_{m}(f)(n)=f(\phi_{\tilde{U}}(n)+p\phi_{\tilde{T}}(m))\in\Z$
for all $m\in[N]^{d''}$.
%
 %
Since $f(\phi_{\tilde{U}}(n)+p\phi_{\tilde{T}}(\cdot))$ is a polynomial of degree at most $s$, by interpolation, 
 for some appropriate $N=O_{d,s}(1)$, then this implies that $f(\tilde{\phi}_{U}(n)+p\tilde{\phi}_{T}(m))\in\Z$ for all $n\in\Z^{d''}$ with $\iota(n)\in V(M')\cap (\phi_{U_{0}}^{-1}(\pi_{U_{0}}(V)))^{\perp_{M'}}$ and for all $m\in\Z^{d'}$.

 

We claim that for all $x\in\Z^{d''}$ and $y\in\Z^{d'}$ with $\iota(\phi_{\tilde{U}}(x)+\phi_{\tilde{T}}(y))\in V(M)\cap V^{\perp_{M}}$, we have that $\iota(x)\in V(M')\cap (\phi_{U_{0}}^{-1}(\pi_{U_{0}}(V)))^{\perp_{M'}}$ and that $y\in p\Z^{d'}$.
If the claim holds, the for any $n\in\Z^{d}$ with $\iota(h)\in V(M)\cap V^{\perp_{M}}$, we may write $n=\phi_{\tilde{U}}(x)+p\phi_{\tilde{T}}(y)$ for some $x\in\Z^{d''}$ and $y\in\Z^{d'}$. So $f(n)=\sigma_{y}(f)(x)\in\Z$ and we are done.

So it remains to prove the claim.
Projecting everything down to $\V$, it suffices to show that for all $x\in\F_{p}^{d''}$ and $y\in\F_{p}^{d'}$ with $\phi_{U_{0}}(x)+\phi_{T_{0}}(y)$ belonging to $V(M)\cap V^{\perp_{M}}$, we have that $x$ belongs to $V(M')\cap (\phi_{U_{0}}^{-1}(\pi_{U_{0}}(V)))^{\perp_{M'}}$ and that $y=\bold{0}$.
 Indeed, for all $t\in T_{0}$, we have that $0=((\phi_{U_{0}}(x)+\phi_{T_{0}}(y))A)\cdot t=(\phi_{T_{0}}(y)A)\cdot t$. Since $M\vert_{T_{0}}$ is non-degenerate, this implies that $\phi_{T_{0}}(y)=\bold{0}$ and thus $y=\bold{0}$. Therefore, $\phi_{U_{0}}(x)$ belonging to $V(M)\cap V^{\perp_{M}}$, which implies that $x$ belongs to $V(M')\cap (\phi_{U_{0}}^{-1}(\pi_{U_{0}}(V)))^{\perp_{M'}}$.
\end{proof}

\textbf{Step 1. structure of the solutions on different layers.}
We first provide a description of $\xi$ on many layers $U_{0}+v, v\in T_{0}$.

Fix any $m\in\Z^{d}$ and denote $\sigma:=\sigma_{m}$ for convenience. 
For $v\in T$ and $u\in \Z_{K}^{d''}$, let $$\eta_{v}(u):=\sigma\circ\xi(v+\phi_{U}(u))\in\st_{\zpn,d''}(s).$$
Since $\pi^{-1}(\bold{0})\cap(8\tilde{\xi}(H)\- 8\tilde{\xi}(H))$ admits a  classification $(\mathcal{C}_{0},Y)$, for all $v_{1},\dots,v_{16}\in T$ and $w_{1},\dots,w_{16}\in U$ with $v_{i}+w_{i}\in H, 1\leq i\leq 16$  and with $v_{1}+\dots+v_{8}=v_{9}+\dots+v_{16}$, $w_{1}+\dots+w_{8}=w_{9}+\dots+w_{16}$, either $\sp_{\F_{p}}\{\iota(v_{i})+\iota(w_{i})\colon 1\leq i\leq 16\}\cap Y\neq \{\bold{0}\}$ or $$\xi(v_{1}+w_{1})+\dots+\xi(v_{8}+w_{8}))-(\xi(v_{9}+w_{9})+\dots+\xi(v_{16}+w_{16}))\equiv f \mod J^{M}_{\iota(v_{1}+w_{1}),\dots,\iota(v_{15}+w_{15})}$$ for some $f\in \mathcal{C}_{0}$. 
Since $d\geq 33$,
by Proposition \ref{2:redf}, this implies that
\begin{equation}\label{2:basiccondition}
\begin{split}
&\text{for all $v_{1},\dots,v_{16}\in T$ and $u_{1},\dots,u_{16}\in \Z_{K}^{d''}$ with $v_{i}+\phi_{U}(u_{i})\in H, 1\leq i\leq 16$,}  \\&\text{and with $v_{1}+\dots+v_{8}=v_{9}+\dots+v_{16}$ and $u_{1}+\dots+u_{8}=u_{9}+\dots+u_{16}$,}
\\&\text{either $\sp_{\F_{p}}\{\iota(v_{i})+\iota(\phi_{U}(u_{i}))\colon 1\leq i\leq 16\}\cap Y\neq \{\bold{0}\}$} 
\\&\text{or $(\eta_{v_{1}}(u_{1})+\dots+\eta_{v_{8}}(u_{8}))-(\eta_{v_{9}}(u_{9})+\dots+\eta_{v_{16}}(u_{16}))\equiv g \mod J^{M'}_{\iota(u_{1}),\dots,\iota(u_{15})}$}
\\&\text{for some $g\in \sigma(\mathcal{C}_{0})$.}
\end{split}
\end{equation}

Since $\vert H\vert\geq \d K^{d}$, there exists $T_{1}\subseteq T$ with $\vert T_{1}\vert\geq\d K^{d'}/2$ such that $\vert (v+U)\cap H\vert\geq \d K^{d''}/2$ for all $v\in T_{1}$. Let $T_{2}$ be the set of $v\in T_{1}$ such that $J^{M}_{\sp_{\F_{p}}\{\iota(v)\}+U_{0}}\cap \mathcal{C}_{0}=\{0\}$ and $(\sp_{\F_{p}}\{\iota(v)\}+U_{0})\cap Y=\{\bold{0}\}$.

\textbf{Claim 1.} We have that $\vert T_{2}\vert\geq \d K^{d'}/4$.

To see this, let $R_{1}$ be the set of $v\in T$ such that $(\sp_{\F_{p}}\{\iota(v)\}+U_{0})\cap Y\neq\{\bold{0}\}$, and $R_{2}$ be the set of $v\in T$ such that $J^{M}_{\sp_{\F_{p}}\{\iota(v)\}+U_{0}}\cap \mathcal{C}_{0}\neq\{0\}$.
Since $U_{0}\cap Y=\{\bold{0}\}$, if $(\sp_{\F_{p}}\{\iota(v)\}+U_{0})\cap Y\neq\{\bold{0}\}$, then $\pi_{T_{0}}(\iota(v))$ is contained $\pi_{T_{0}}(Y)$. Since $Y$ is the union of at most $C$ subspaces of $\V$ of dimension at most $D$, the set $R_{1}$ is of cardinality at most $O_{C}(K^{D})$. 
On the other hand, we may write $R_{2}=\cup_{f\in \mathcal{C}_{0}\backslash\{\bold{0}\}} R_{2,f}$, where $R_{2,f}$ is the set of all
 $v\in T$, such that
$f\in J^{M}_{\sp_{\F_{p}}\{\iota(v)\}+U_{0}}$. Then $f\in \cap_{v\in R_{2,f}}J^{M}_{\sp_{\F_{p}}\{\iota(v)\}+U_{0}}$.
If $\vert R_{2,f}\vert>K^{s}$, then there exist $v_{1},\dots,v_{s+1}\in R_{2,f}$ which are $p$-linearly independent. Since
$\rank(M\vert_{U_{0}})=d''$, by Lemma \ref{2:iissoo} (iv), we have that $\dim(U_{0}\cap U_{0}^{\pp})=0$. 
By Lemma \ref{2:iissoo} (ii), this implies that
$\rank(M\vert_{U_{0}^{\pp}})=d'$. So
by Proposition \ref{2:gr0} (setting $r=1$), since $d\geq d''+7$ (or equivalently, $d'\geq 7$),
we have that $f\in J^{M}_{U_{0}}$, a contradiction to the fact that $J_{U_{0}}^{M}\cap\mathcal{C}_{0}=\{0\}$. 
Therefore, $\vert R_{2,f}\vert\leq K^{s}$ for all $f\in \mathcal{C}_{0}\backslash\{0\}$ and so $\vert R_{2}\vert\leq LK^{s}$. Since $\vert T_{1}\vert\geq \d K^{d'}/2$ and $d'\geq \max\{D+1,s+1\}$, we have that $\vert T_{2}\vert\geq \d K^{d'}/4$. This proves Claim 1.

\





  Since $\pi^{-1}(\bold{0})\cap(8\tilde{\xi}(H)\- 8\tilde{\xi}(H))$ admits a  classification $(\mathcal{C}_{0},Y)$, 
  by the construction of $T_{2}$ and Proposition \ref{2:g1621}, for all $v\in T_{2}$,
 there exists a subset $W_{v}\subseteq (U+v)\cap H\subseteq \iota^{-1}(\sp_{\F_{p}}\{\iota(v)\}+U_{0})$ with $\vert W_{v}\vert\geq \vert  (U+v)\cap H\vert/80^{L}\gg_{\d,L}K^{d''}$ such that $\xi$ is a Freiman $M$-homomorphism of order 16 on $W_{v}$. Write $U_{v}:=(H-v)\cap U$. Then
  for all $u_{1},\dots,u_{16}\in \phi_{U}^{-1}(U_{v})$ with $u_{1}+\dots+u_{8}=u_{9}+\dots+u_{16}$, we  have that 
  $$\sum_{i=1}^{8}\xi(v+\phi_{U}(u_{i})) \equiv\sum_{i=9}^{16}\xi(v+\phi_{U}(u_{i}))\mod J^{M}_{\sp_{\F_{p}}\{\iota(v+\phi_{U}(u_{i}))\colon 1\leq i\leq 15\}}.$$
 So by Proposition \ref{2:redf},
%
%
%
\begin{equation}\label{2:basiccondition3}
\begin{split}
&\text{for all $v\in T_{2}$ and $u_{1},\dots,u_{16}\in \phi_{U}^{-1}(U_{v})$ with $u_{1}+\dots+u_{8}=u_{9}+\dots+u_{16}$,}
\\&\text{we have $\eta_{v}(u_{1})+\dots+\eta_{v}(u_{8})\equiv\eta_{v}(u_{9})+\dots+\eta_{v}(u_{16}) \mod J^{M'}_{\iota(u_{1}),\dots,\iota(u_{15})}$.}
\end{split}
\end{equation}

\textbf{Claim 2.} There exist a subset $T_{3}$ of $T_{2}$ of cardinality $\gg_{\d,L}K^{d'}$  and some $v_{0}\in T_{3}$ such that $\vert U_{v}\cap U_{v_{0}}\vert\gg_{\d,L}K^{d''}$ for all $v\in T_{3}$.

For all $u\in U$, let $S_{u}$ denote the set of $v\in T_{2}$ such that $u\in U_{v}$. Then $\sum_{u\in U}\vert S_{u}\vert=\sum_{v\in T_{2}}\vert U_{v}\vert\gg_{\d,L} K^{d}$ by Claim 1.
On the other hand, note that 
\begin{equation}\nonumber
\begin{split}
\sum_{v,v'\in T_{2}, v\neq v'}\vert U_{v}\cap U_{v'}\vert=\sum_{u\in U}\binom{\vert S_{u}\vert}{2}\geq \frac{1}{2\vert U\vert}\Bigl(\sum_{u\in U}\vert S_{u}\vert\Bigr)^{2}-\Bigl(\sum_{u\in U}\vert S_{u}\vert\Bigr)\geq O_{\d,L}(1)^{-1}K^{d+d'}-K^{d}.
\end{split}
\end{equation}
By the Pigeonhole Principle, there exist $v_{0}\in T_{2}$ such that 
$$\sum_{v\in T_{2}, v\neq v_{0}}\vert U_{v}\cap U_{v_{0}}\vert\gg_{\d,L} K^{d}.$$
Again by the Pigeonhole Principle, there exist a subset $T'_{3}$ of $T_{2}$ of cardinality $\gg_{\d,L}K^{d'}$ such that $\vert U_{v}\cap U_{v_{0}}\vert\gg_{\d,L}K^{d''}$ for all $v\in T'_{3}$. This proves Claim 2 by setting $T_{3}:=T'_{3}\cup\{v_{0}\}$.

\

Let $T_{3}$ and $v_{0}$ be given by Claim 2. 
By (\ref{2:basiccondition3}), for all $v\in T_{2}$ and $u_{1},\dots,u_{16}\in \phi_{U}^{-1}(U_{v}\cap U_{v_{0}})$, we have that  
\begin{equation}\nonumber
	u_{1}+\dots+u_{8}=u_{9}+\dots+u_{16}\Rightarrow \eta_{v}(u_{1})+\dots+\eta_{v}(u_{8})\equiv\eta_{v}(u_{9})+\dots+\eta_{v}(u_{16}) \mod J^{M'}_{\iota(u_{1}),\dots,\iota(u_{15})}.
\end{equation}
Applying Proposition \ref{2:gef1} to $\phi_{U}^{-1}(U_{v}\cap U_{v_{0}})$, we deduce that if $d''\geq \max\{4s+19,35\}$, then  there exists $0<\rho=\rho(\d,d'')<1/4$ such that for all $v\in T_{3}$, there exist 
	a proper and homogeneous generalized arithmetic progression $P_{v}$ in $\Z_{K}^{d''}$ of rank $O_{\d,d,L}(1)$ and cardinality $\gg_{\d,d,L}K^{d''}$ (note that $\phi_{U}(P_{v})$ is a  generalized arithmetic progression in $U$ of the same rank and cardinality), a set $S_{v}\subseteq \hat{\Z_{K}^{d''}}$ of cardinality $O_{\d,L}(1)$, and a map $\omega_{v}\colon P_{v}\to \st_{\zpn,d''}(s)$ such that the followings hold:
	\begin{enumerate}[(i)]
		\item $B(S_{v},\rho)\subseteq P_{v}\subseteq B(S_{v},1/4)\subseteq 2\phi_{U}^{-1}(U_{v}\cap U_{v_{0}})-2\phi_{U}^{-1}(U_{v}\cap U_{v_{0}})$;
		\item For all $h\in P_{v}$,   the set $R(\phi_{U}^{-1}(U_{v}\cap U_{v_{0}}),h)$ is of cardinality $\gg_{\d,L}K^{3d''}$, where for all $E\subseteq \Z_{K}^{d''}$ and $h\in \Z_{K}^{d''}$, we denote
			$$R(E,h):=\{(h_{1},h_{2},h_{3},h_{4})\in E^{4}\colon h_{1}+h_{2}-h_{3}-h_{4}=h\}.$$
		\item for all $u_{1},u_{2},u_{3},u_{4}\in \phi_{U}^{-1}(U_{v}\cap U_{v_{0}})$ with $u_{1}+u_{2}-u_{3}-u_{4}\in P_{v}$,
		\begin{equation}\label{2:evev1}
		\eta_{v}(u_{1})+ \eta_{v}(u_{2})- \eta_{v}(u_{3})- \eta_{v}(u_{4})\equiv \omega_{v}(u_{1}+u_{2}-u_{3}-u_{4}) \mod J^{M'}_{\iota(u_{1}),\iota(u_{2}),\iota(u_{3}),\iota(u_{4})}.
		\end{equation}
		\item for all $u_{1},u_{2},u_{3},u_{4}\in P_{v}$ with $u_{1}+u_{2}=u_{3}+u_{4}$,
			\begin{equation}\label{2:evev2}
		\omega_{v}(u_{1})+ \omega_{v}(u_{2})\equiv\omega_{v}(u_{3})+ \omega_{v}(u_{4}) \mod J^{M'}_{\iota(u_{1}),\iota(u_{2}),\iota(u_{3})},
			\end{equation}
		i.e. the map $\omega_{v}\colon P_{v}\to\st_{\zpn,d''}(s)$ is a Freiman $M'$-homomorphism of order 4.
	\end{enumerate}

\textbf{Step 2. $\omega_{v}$ coincides with $\omega_{v_{0}}$ on a large subset $W_{v}$ of $P_{v}$ (up to a constant $b_{v}$).} 
Our next step is to compare $\omega_{v}$ with $\omega_{v_{0}}$ and conclude that they only differ by a constant.
Let $T_{4}$ be the set of $v\in T_{3}$ such that  $\sp_{\F_{p}}\{\iota(v_{0}),\iota(v)\}+U_{0}$ intersects $Y$ trivially. Since $U_{0}\cap Y$ and $U\cap T$ are trivial, the set of $v\in T_{3}$ such that  $\sp_{\F_{p}}\{\iota(v_{0}),\iota(v)\}+U_{0}$ intersects $Y$ non-trivially is of cardinality at most 
$Cp^{D+1}(K/p)^{d'}=CK^{d'}/p^{d'-D-1}$.
Since  $d'\geq D+2$, we have
$\vert T_{4}\vert\gg_{\d,L}K^{d'}$.
Pick any $v\in T_{4}$ and $u\in B(S_{v_{0}}\cup S_{v},\rho/4)$. 
Let $G_{v}(u)$ denote the set of all $(u_{1},\dots,u_{16})\in \phi_{U}^{-1}(U_{v_{0}})^{8}\times \phi_{U}^{-1}(U_{v})^{8}$ such that $u_{1}+u_{2}-u_{3}-u_{4}=x$,  $u_{5}+u_{6}-u_{7}-u_{8}=x+u$,  $u_{9}+u_{10}-u_{11}-u_{12}=y$ and $u_{13}+u_{14}-u_{15}-u_{16}=y+u$ for some $x\in B(S_{v_{0}},\rho/4)$ and $y\in B(S_{v},\rho/4)$.
We first estimate the cardinality of $G_{v}(u)$.
For any $x\in B(S_{v_{0}},\rho/4)$ and $y\in B(S_{v},\rho/4)$, since $x+u\in B(S_{v_{0}},\rho)$ and $y+u\in B(S_{v},\rho)$, we have that
$$\vert R(\phi_{U}^{-1}(U_{v_{0}}),x)\vert, \vert R(\phi_{U}^{-1}(U_{v_{0}}),x+u)\vert, \vert R(\phi_{U}^{-1}(U_{v}),y)\vert, \vert R(\phi_{U}^{-1}(U_{v}),y+u)\vert\gg_{\d,L}K^{3d''}.$$ So $$\vert G_{v}(u)\vert\gg_{\d,L}\vert B(S_{v_{0}},\rho/4)\vert\cdot \vert B(S_{v},\rho/4)\vert \cdot K^{12d''}\gg_{\d,d'',L} K^{14d''},$$ where we used the fact that $\vert B(S,\rho)\vert\geq \rho^{\vert S\vert}K^{d''}$ for all $S\subseteq \hat{\Z_{K}^{d''}}$ and $\rho>0$ (see Lemma 8.1 of \cite{GT08b}).

Pick any $(u_{1},\dots,u_{16})\in G_{v}(u)$ and let $x\in B(S_{v_{0}},\rho/4)$ and $y\in B(S_{v},\rho/4)$ be associated to this tuple. Then we have that $x,x+u\in P_{v_{0}}$ and $y,y+u\in P_{v}$. Since $\omega_{v}$ is a Freiman $M'$-homomorphism of order 4 on $P_{v}$, by (\ref{2:basiccondition}) and (\ref{2:evev2}), we have that 
\begin{equation}\nonumber
	\begin{split}
		&\quad \omega_{v_{0}}(u)-\omega_{v}(u)\equiv\omega_{v_{0}}(x+u)- \omega_{v_{0}}(x)- \omega_{v}(y+u)+ \omega_{v}(y)
		\\&\equiv\eta_{v_{0}}(u_{5})+\eta_{v_{0}}(u_{6})-\eta_{v_{0}}(u_{7})-\eta_{v_{0}}(u_{8})-\eta_{v_{0}}(u_{1})-\eta_{v_{0}}(u_{2})+\eta_{v_{0}}(u_{3})+\eta_{v_{0}}(u_{4})
		\\&\quad-\eta_{v}(u_{13})-\eta_{v}(u_{14})+\eta_{v}(u_{15})+\eta_{v}(u_{16})+\eta_{v}(u_{9})+\eta_{v}(u_{10})-\eta_{v}(u_{11})-\eta_{v}(u_{12})
		\\&\equiv f(u_{1},\dots,u_{16}) \mod J^{M'}_{\iota(u_{1}),\dots,\iota(u_{16})}
	\end{split}
\end{equation}
for some $f(u_{1},\dots,u_{16})\in\sigma(\mathcal{C}_{0})$, since $$\sp_{\F_{p}}\{\iota(v_{0})+\iota(\phi_{U}(u_{1})),\dots,\iota(v_{0})+\iota(\phi_{U}(u_{8})),\iota(v)+\iota(\phi_{U}(u_{9})),\dots,\iota(v)+\iota(\phi_{U}(u_{16}))\}$$ as a subspace of $\sp_{\F_{p}}\{\iota(v_{0}),\iota(v)\}+U_{0}$ intersects $Y$ trivially.
By the Pigeonhole Principle, there exist a subset $G'_{v}(u)$ of $G_{v}(u)$ of cardinality $\gg_{\d,L}K^{14d''}$ and some $f_{v}(u)\in \sigma(\mathcal{C}_{0})$ such that 
\begin{equation}\nonumber
	\begin{split}
		\omega_{v_{0}}(u)-\omega_{v}(u)\equiv f_{v}(u) \mod J^{M'}_{\iota(u_{1}),\dots,\iota(u_{16})}=J^{M'}_{\iota(u),\iota(u_{2}),\dots,\iota(u_{15})}
	\end{split}
\end{equation}
for all $(u_{1},\dots,u_{16})\in G'_{v}(u)$ (note that $u_{1}$ and $u_{16}$ are linear combinations of $u,u_{2},\dots,u_{15}$).
So 
\begin{equation}\nonumber
	\begin{split}
		\omega_{v_{0}}(u)-\omega_{v}(u)\equiv f_{v}(u) \mod \bigcap_{(u_{1},\dots,u_{16})\in G'_{v}(u)}J^{M'}_{\iota(u),\iota(u_{2}),\dots,\iota(u_{15})}.
	\end{split}
\end{equation}
Since $d''\geq\max\{s+15,35\}$, $\vert G'_{v}(u)\vert\gg_{\d,L}K^{14d''}$, and since $u_{1},u_{16}$ are uniquely determined by $u,u_{2},\dots,u_{15}$,
by Proposition \ref{2:grm},   we have that
\begin{equation}\label{2:wwfv}
	\begin{split}
		\omega_{v_{0}}(u)-\omega_{v}(u)\equiv f_{v}(u) \mod J^{M'}_{\iota(u)}
	\end{split}
\end{equation}
for all  $v\in T_{4}$ and $u\in B(S_{v_{0}}\cup S_{v},\rho/4)$. 
By the Pigeonhole Principle,  there exists $f_{v}\in \sigma(\mathcal{C}_{0})$ such that the set $$W_{v}:=\{u\in B(S_{v_{0}}\cup S_{v},\rho/4)\colon f_{v}(u)=f_{v}\}$$ is of cardinality at least $$\vert B(S_{v_{0}}\cup S_{v},\rho/4)\vert/\vert\mathcal{C}_{0}\vert\geq (\frac{\rho}{4})^{\vert S_{v_{0}}\cup S_{v}\vert}K^{d''}(L+1)^{-1}\gg_{\d,d'',L}K^{d''},$$ where again we used Lemma 8.1 of \cite{GT08b}.

Write $\omega'_{v}:=\omega_{v}-\omega_{v_{0}}$.
For all $v\in T_{4}$ and $u_{1},u_{2}\in W_{v}$, since $u:=u_{1}-u_{2}\in W_{v}-W_{v}\subseteq B(S_{v_{0}}\cup S_{v},\rho)$,  it follows from (\ref{2:wwfv}) that
$$\omega'_{v}(u_{1})-\omega'_{v}(u_{2})\equiv -f_{v}(u_{1})+f_{v}(u_{2})=f_{v}-f_{v}=0 \mod J^{M'}_{\iota(u_{1}),\iota(u_{2})}.$$
In other words, the set $$X_{v}:=\{(\bold{0},J^{M'}_{\iota(u)}+\omega'_{v}(u))\colon u\in W_{v}\}\subseteq\Gamma^{s}_{2}(\Z_{K}^{d''},M')$$ is an equivalence class with respect to the relation $\sim'$.
Since $\vert W_{v}\vert\gg_{\d,d'',L}K^{d''}$,  and $d''\geq s+2$, it is impossible for all the elements in $\iota(W_{v})$ to be contained in a fixed $s+1$ dimensional subspace of $\F_{p}^{d''}$.
By Proposition \ref{2:gwts} (setting $k=2$), since $d''\geq 9$, we have that $X_{v}$ is a strong equivalence class. So
there exists $b_{v}\in\st_{\zp,d''}(s)$ such that $\omega'_{v}(u)=b_{v} \mod J^{M'}_{\iota(u)}$ for all $u\in W_{v}$.

\textbf{Step 3. $\eta_{v}$ coincides with $\eta_{v_{0}}$ on a large subset $U'_{v}$ of $\phi_{U}^{-1}(U_{v}\cap U_{v_{0}})$ (up to a constant  $g_{v}$).} 
We next compare $\eta_{v}$ with $\eta_{v_{0}}$.
For convenience denote $\eta'_{v}:=\eta_{v}-\eta_{v_{0}}$. Our goal is to use the fact that $\omega'_{v}$ is a constant to conclude that $\eta'_{v}$  is also a constant. In fact, we have

\textbf{Claim 3.} For all $v\in T_{4}$, there exists $g_{v}\in \st_{\zp,d''}(s)$ and a set $U'_{v}\subseteq \phi_{U}^{-1}(U_{v}\cap U_{v_{0}})$ of cardinality $\gg_{\d,d',L} K^{d''}$ such that $\eta'_{v}(u)\equiv g_{v} \mod J^{M'}_{\iota(u)}$ for all $u\in U'_{v}$.

For $v\in T_{4}$, let $B_{v}$ denote the set of $(u_{1},u_{2},u_{3},u_{4})\in \phi_{U}^{-1}(U_{v}\cap U_{v_{0}})^{4}$ such that $u_{1}+u_{2}-u_{3}-u_{4}\in W_{v}$.
Since $u_{1}+u_{2}-u_{3}-u_{4}\in B(S_{v_{0}}\cup S_{v},\rho)\subseteq P_{v}\cap P_{v_{0}}$ for all $(u_{1},u_{2},u_{3},u_{4})\in B_{v}$, by (\ref{2:evev1}), 
\begin{equation}\nonumber
	\begin{split}
		&\quad \eta_{v}(u_{1})+\eta_{v}(u_{2})-\eta_{v}(u_{3})-\eta_{v}(u_{4})
		\\&\equiv\omega_{v}(u_{1}+u_{2}-u_{3}-u_{4})
		\equiv\omega_{v_{0}}(u_{1}+u_{2}-u_{3}-u_{4})+b_{v}
		\\&\equiv\eta_{v_{0}}(u_{1})+\eta_{v_{0}}(u_{2})-\eta_{v_{0}}(u_{3})-\eta_{v_{0}}(u_{4})+b_{v} 
		\mod J^{M'}_{\iota(u_{1}),\iota(u_{2}),\iota(u_{3}),\iota(u_{4})}
	\end{split}	
\end{equation}
In other words,
\begin{equation}\label{2:bvva}
\begin{split}
 \eta'_{v}(u_{1})+\eta'_{v}(u_{2})-\eta'_{v}(u_{3})-\eta'_{v}(u_{4})
\equiv b_{v} 
\mod J^{M'}_{\iota(u_{1}),\iota(u_{2}),\iota(u_{3}),\iota(u_{4})}
\end{split}	
\end{equation}
for all $(u_{1},u_{2},u_{3},u_{4})\in B_{v}$.

Let $B_{v}(u)$ denote the set of $(u_{2},u_{3},u_{4})\in \phi_{U}^{-1}(U_{v}\cap U_{v_{0}})^{3}$ such that $(u,u_{2},u_{3},u_{4})\in B_{v}$.
 Since $\vert R(\phi_{U}^{-1}(U_{v}\cap U_{v_{0}}),u)\vert\gg_{\d,d'',L}K^{3d''}$ for all $u\in W_{v}\subseteq P_{v}$  and since $\vert W_{v}\vert\gg_{\d,L}K^{d''}$, by the Pigeonhole Principle, there exists $U''_{v}\subseteq \phi_{U}^{-1}(U_{v}\cap U_{v_{0}})$ of cardinality $\gg_{\d,d'',L}K^{d''}$ such that for each $u\in U''_{v}$, $B_{v}(u)$ is of cardinality  $\gg_{\d,d'',L}K^{3d''}$.

 We need to use the intersection method.
 Let $Z$ be the set of $$((u_{1,2},u_{1,3},u_{1,4}),\dots,(u_{s+3,2},u_{s+3,3},u_{s+3,4}))\in \phi_{U}^{-1}(U_{v}\cap U_{v_{0}})^{3(s+3)}$$ such that $u_{i,j}, 1\leq i\leq s+3, 2\leq j\leq 4$ are $p$-linearly independent.   Consider the bipartite graph
 $G=(U''_{v},Z,E)$ such that $(u_{1},(u_{1,2},u_{1,3},u_{1,4}),\dots,(u_{4,2},u_{4,3},u_{4,4}))\in E$ if and only if $(u_{i,2},u_{i,3},u_{i,4})\in B_{v}(u_{1})$ for all $1\leq i\leq s+3$.
 Then by Lemma \ref{2:iiddpp}, for each $u_{1}\in U''_{v}$, the number of edges connecting $u_{1}$ in $G$ is at least
 $\gg_{\d,d'',L}K^{(3s+9)d''}$
 since $d''\geq 3s+9$. Since $\vert U''_{v}\vert\gg_{\d,d'',L}K^{d''}$, by the Pigeonhole Principle, there exists $((u_{1,2},u_{1,3},u_{1,4}),\dots,(u_{s+3,2},u_{s+3,3},u_{s+3,4}))\in Z$ such that the set $$U'_{v}:=\{u\in U''_{v}\colon (u,(u_{1,2},u_{1,3},u_{1,4}),\dots,(u_{s+3,2},u_{s+3,3},u_{s+3,4}))\in E\}$$ is of cardinality $\gg_{\d,d'',L}K^{d''}$.

 Now for all $u,u'\in U'_{v}$, by (\ref{2:bvva}), for $1\leq i\leq s+3$,
 \begin{equation}\nonumber
 \begin{split}
 &\quad \eta'_{v}(u)+\eta'_{v}(u_{i,2})-\eta'_{v}(u_{i,3})-\eta'_{v}(u_{i,4})\equiv b_{v}
 \\&\equiv\eta'_{v}(u')+\eta'_{v}(u_{i,2})-\eta'_{v}(u_{i,3})-\eta'_{v}(u_{i,4}) \mod J^{M'}_{\iota(u),\iota(u'),\iota(u_{i,2}),\iota(u_{i,3}),\iota(u_{i,4})},
 \end{split}
 \end{equation}
 and so 
 $$\eta'_{v}(u)\equiv\eta'_{v}(u') \mod \bigcap_{i=1}^{s+3}J^{M'}_{\iota(u),\iota(u'),\iota(u_{i,2}),\iota(u_{i,3}),\iota(u_{i,4})}.$$
 Since $d''\geq 4s+23$ and $u_{i,j}, 1\leq i\leq s+3, 2\leq j\leq s+3$ are $p$-linearly independent, by Proposition \ref{2:gri} (setting $m=2$  and $r=3$),  we have that
 $$\eta'_{v}(u)\equiv\eta'_{v}(u') \mod  J^{M'}_{\iota(u),\iota(u')}.$$
 This means that the set $$X_{v}:=\{(\bold{0},J^{M'}_{\iota(u)}+\eta'_{v}(u))\colon u\in U'_{v}\}\subseteq \Gamma^{s}_{2}(\Z_{K}^{d''},M')$$ is an equivalence class with respect to the relation $\sim'$.  
 Since $\vert U'_{v}\vert\gg_{\d,d'',L}K^{d''}$, and $d''\geq s+2$, it is impossible for all elements in $\iota(W_{v})$ to be contained in a fixed $s+1$ dimensional subspace of $\F_{p}^{d''}$. 
 By Proposition \ref{2:gwts}  (setting $k=2$), since $d''\geq 9$, the set $X_{v}$ is a strong equivalence class. 
 So there exists $g_{v}\in\st_{\zp,d''}(s)$ such that $\eta'_{v}(u)\equiv g_{v} \mod J^{M'}_{\iota(u)}$ for all $u\in U'_{v}$. This proves Claim 3.

\textbf{Step 4. The   map $v\mapsto g_{v}$ is a Freiman homomorphism modulo $J^{M'}$.} 
We have seen that $\eta'_{v}$ equals to a constant $g_{v}$. We now show that  the map $v\mapsto g_{v}$ has a nice Freiman homomorphism structure.
Let $U'_{v}$ be given by Claim 3.  We need the following combinatorial lemma.

\textbf{Claim 4.} The number of $v_{1},v_{2},v_{3},v_{4}\in T_{4}$ and $u_{i}\in U'_{v_{i}}, 1\leq i\leq 4$ with $v_{1}+v_{2}=v_{3}+v_{4}$ and $u_{1}+u_{2}=u_{3}+u_{4}$ is at least 
$\gg_{\d,d'',L}K^{3d}$.

For all $v\in T$, let $Q(v)$ denote the set of $(v_{1},v_{2})\in T_{4}^{2}$ with $v_{1}+v_{2}=v$. Since
$\sum_{v\in T}\vert Q(v)\vert=\vert T_{4}\vert^{2}\gg_{\d,d'',L}K^{2d'}$, there exists a subset $T'$ of $T$ with $\vert T'\vert\gg_{\d,d'',L}K^{d'}$ such that $\vert Q(v)\vert\gg_{\d,d'',L}K^{d'}$ for all $v\in T'$.  
Then for all $v\in T'$,
$$\sum_{u\in\Z_{K}^{d''}}\sum_{(v_{1},v_{2})\in Q(v)}\vert\{(u_{1},u_{2})\in U'_{v_{1}}\times U'_{v_{2}}\colon u_{1}+u_{2}=u\}\vert=\sum_{(v_{1},v_{2})\in Q(v)}\vert U'_{v_{1}}\vert\cdot \vert U'_{v_{2}}\vert\gg_{\d,d'',L}K^{d'+2d''}.$$
By the Pigeonhole Principle, for all $v\in T'$, there exists a subset $U(v)$ of $\Z_{K}^{d''}$ of cardinality $\gg_{\d,d'',L}K^{d''}$ such that for all $u\in U(v)$, we have
$$\sum_{(v_{1},v_{2})\in Q(v)}\vert\{(u_{1},u_{2})\in U'_{v_{1}}\times U'_{v_{2}}\colon u_{1}+u_{2}=u\}\vert\gg_{\d,d'',L}K^{d}.$$
So for any $v\in T'$ and $u\in U(v)$, the number of tuples $(v_{1},v_{2},u_{1},u_{2})$ with $(v_{1},v_{2})\in Q(v), (u_{1},u_{2})\in U'_{v_{1}}\times U'_{v_{2}}, u_{1}+u_{2}=u$ is $\gg_{\d,d'',L}K^{d}$. For two such tuples $(v_{1},v_{2},u_{1},u_{2})$ and $(v_{3},v_{4},u_{3},u_{4})$, it is clear that $v_{1}+v_{2}=v_{3}+v_{4}=v$ and $u_{1}+u_{2}=u_{3}+u_{4}=u$. Since there are $\gg_{\d,d'',L}K^{2d}$ such pairs for any  $v\in T'$ and $u\in U(v)$, and there are $\gg_{\d,d'',L}K^{d}$  pairs of $(v,u)$ with  $v\in T'$ and $u\in U(v)$. So there are in total $\gg_{\d,d'',L}K^{3d}$ many tuples of $(v_{1},\dots,v_{4},u_{1},\dots,u_{4})$ satisfying the requirement of Claim 4.

\

 Denote $U'=\cup_{v\in T_{4}}U'_{v}$. Then $\vert U'\vert\geq \inf_{v\in T_{4}}\vert U'_{v}\vert\cdot\vert T_{4}\vert\gg_{\d,d'',L}K^{d}$. 
By (\ref{2:basiccondition}) and (\ref{2:basiccondition3}), 
for all $v_{i}\in T_{4}$ and $u_{i}\in U'_{v_{i}}\subseteq \phi_{U}^{-1}(U_{v_{0}}), 1\leq i\leq 4$
with $v_{1}+v_{2}=v_{3}+v_{4}$ and $u_{1}+u_{2}=u_{3}+u_{4}$, either 
$\sp_{\F_{p}}\{v_{1}+\phi_{U}^{-1}(u_{1}),\dots,v_{4}+\phi_{U}^{-1}(u_{4})\}\cap Y\neq \{\bold{0}\}$  or
\begin{equation}\nonumber
\begin{split}
&\quad f
\equiv 4(\eta_{v_{1}}(u_{1})+\eta_{v_{2}}(u_{2})-\eta_{v_{3}}(u_{3})-\eta_{v_{4}}(u_{4}))
\\&\equiv 4(\eta_{v_{0}}(u_{1})+\eta_{v_{0}}(u_{2})-\eta_{v_{0}}(u_{3})-\eta_{v_{0}}(u_{4}))+4(g_{v_{1}}+g_{v_{2}}-g_{v_{3}}-g_{v_{4}})
\\&\equiv 4(g_{v_{1}}+g_{v_{2}}-g_{v_{3}}-g_{v_{4}}) \mod J^{M'}_{\iota(u_{1}),\iota(u_{2}),\iota(u_{3})}
\end{split}	
\end{equation}
for some $f\in \sigma(\mathcal{C}_{0})$.
By Claim 4, the number of $v_{1},v_{2},v_{3},v_{4}\in T_{4}$ and $u_{i}\in U'_{v_{i}}, 1\leq i\leq 4$ with $v_{1}+v_{2}=v_{3}+v_{4}$ and $u_{1}+u_{2}=u_{3}+u_{4}$ is at least 
$\gg_{\d,d'',L}K^{3d}$. 
Since $U_{0}\cap Y$ and $U\cap T$ are trivial, among these tuples $(v_{1},v_{2},v_{3},v_{4},u_{1},u_{2},u_{3},u_{4})$, the ones such that   $(\sp_{\F_{p}}\{\iota(v_{i})\colon 1\leq i\leq 4\}+U_{0})\cap Y\neq \{\bold{0}\}$ is of cardinality at most
$$3Cp^{2d'+3d''+(D+2)}(K/p)^{3d}=3CK^{3d}/p^{d'-D-2}.$$
So if   $d'\geq D+3$, then by the Pigeonhole Principle, there exists $E\subseteq T^{4}_{4}$ of cardinality $\gg_{\d,d'',L}K^{3d'}$ such that for all
$(v_{1},v_{2},v_{3},v_{4})\in E$, we have that $v_{1}+v_{2}=v_{3}+v_{4}$, $(\sp_{\F_{p}}\{\iota(v_{i})\colon 1\leq i\leq 4\}+U_{0})\cap Y=\{\bold{0}\}$,
and that the set
$$Z(v_{1},v_{2},v_{3},v_{4}):=\{(u_{1},u_{2},u_{3},u_{4})\in U'_{v_{1}}\times U'_{v_{2}}\times U'_{v_{3}}\times U'_{v_{4}}\colon u_{1}+u_{2}=u_{3}+u_{4}\}$$
is of cardinality $\gg_{\d,d'',L}K^{3d''}$. This means that 
\begin{equation}\nonumber
	\begin{split}
		g_{v_{1}}+g_{v_{2}}-g_{v_{3}}-g_{v_{4}}
		\equiv g(v_{1},\dots,v_{4},u_{1},\dots,u_{4}) \mod J^{M'}_{\iota(u_{1}),\iota(u_{2}),\iota(u_{3})}
	\end{split}	
\end{equation} 
for all  $(v_{1},v_{2},v_{3},v_{4})\in E$ and  $(u_{1},u_{2},u_{3},u_{4})\in Z(v_{1},v_{2},v_{3},v_{4})$ for some $g(v_{1},\dots,v_{4},u_{1}$, $\dots,u_{4})\in 4^{\ast}\sigma(\mathcal{C}_{0})$, where $4^{\ast}$ is any integer with $4^{\ast}4\equiv 1 \mod p^{N}\Z$. By the Pigeonhole Principle, there exist a subset $Z'(v_{1},v_{2},v_{3},v_{4})$ of $Z(v_{1},v_{2},v_{3},v_{4})$ of cardinality $\gg_{\d,d'',L}K^{3d''}$ and some $g(v_{1},v_{2},v_{3},v_{4})\in 4^{\ast}\sigma(\mathcal{C}_{0})$   such that
\begin{equation}\nonumber
	\begin{split}
		g_{v_{1}}+g_{v_{2}}-g_{v_{3}}-g_{v_{4}}
		\equiv g(v_{1},\dots,v_{4}) \mod J^{M'}_{\iota(u_{1}),\iota(u_{2}),\iota(u_{3})}
	\end{split}	
\end{equation} 
for all  $(v_{1},v_{2},v_{3},v_{4})\in E$ and  $(u_{1},u_{2},u_{3},u_{4})\in Z'(v_{1},v_{2},v_{3},v_{4})$. So
\begin{equation}\nonumber
	\begin{split}
		g_{v_{1}}+g_{v_{2}}-g_{v_{3}}-g_{v_{4}}
		\equiv g(v_{1},\dots,v_{4}) \mod \bigcap_{(u_{1},u_{2},u_{3},u_{4})\in Z'(v_{1},v_{2},v_{3},v_{4})}J^{M'}_{\iota(u_{1}),\iota(u_{2}),\iota(u_{3})}.
	\end{split}	
\end{equation} 
Since $d''\geq \max\{s+3,11\}$, $\vert Z'(v_{1},v_{2},v_{3},v_{4})\vert\gg_{\d,d'',L}K^{3d''}$ and $u_{4}=u_{1}+u_{2}-u_{3}$,
by Proposition \ref{2:grm},   we have that 
\begin{equation}\nonumber
	\begin{split}
		g_{v_{1}}+g_{v_{2}}-g_{v_{3}}-g_{v_{4}}
		\equiv g(v_{1},\dots,v_{4}) \mod J^{M'}
	\end{split}	
\end{equation}
for all $(v_{1},v_{2},v_{3},v_{4})\in E$.  

Let $$\Gamma=\{(v,g_{v} \mod J^{M'})\colon v\in T_{4}\}\subseteq T\times (\st_{\zp,d''}(s)\mod J^{M'}),$$ where $T\times (\st_{\zp,d''}(s)\mod J^{M'})$ is a group under the additive operation. 

\textbf{Claim 5.} The set $\Gamma$ has additive energy $\gg_{\d,d'',L}K^{3d'}$. 

Since $\vert E\vert\gg_{\d,d'',L}K^{3d'}$, by repeatedly using the Pigeonhole Principle, there exists $Z\subseteq \Z_{K}^{d'}$ of cardinality $\gg_{\d,d'',L}K^{d'}$ such that for all $u\in Z$, there exists   
$x\in T_{4}$ such that
$x+u\in T_{4}$ and there exists $T_{4}(u)\subseteq T_{4}$ of cardinality $\gg_{\d,d'',L}K^{d'}$ such that for all $y\in T_{4}(u)$, we have $y,y+u\in T_{4}$ and that 
$(x,y+u,y,x+u)\in E$.
Then  $g_{y+u}-g_{y}\equiv(g_{x+u}-g_{x})-g \mod J^{M'}$ for some $g\in 4^{\ast}\sigma(\mathcal{C}_{0})$. 
Again by the Pigeonhole Principle, there exists a subset $T'_{4}(u)$ of $T_{4}(u)$ of cardinality $\gg_{\d,d'',L}K^{d'}$ such that $g_{y+u}-g_{y}\equiv g_{y'+u}-g_{y'} \mod J^{M'}$ for all $y,y'\in T'_{4}(u)$. So the additive energy of $\Gamma$ is at least $\sum_{u\in Z}\vert T'_{4}(u)\vert^{2}\gg O_{\d,d'',L}K^{3d'}$. This proves Claim 5.

\

Since $\Gamma$ is contained in the finite abelian group  $T\times (\st_{\Z/p^{N'},d''}(s)\mod J^{M'})$ for some $N'\in\N_{+}$ and $\Gamma$ has additive energy $\gg_{\d,d'',L}K^{3d'}$ by Claim 5, we may use some standard additive combinatorial machinery   to study the structure of $\Gamma$ (see also \cite{GTZ12}). Using the Balog-Szemer\'edi-Gowers Theorem in \cite{Chang04} followed by the Pl\"unnecke inequality \cite{Plu69,Ruz89}, there exists a subset $\Gamma'$ of $\Gamma$ of cardinality  $\gg_{\d,d'',L}K^{d'}$ such that $\vert 8\Gamma'-8\Gamma'\vert\gg_{\d,d'',L}K^{d'}$. By Lemma 9.2 of \cite{GT08b}, there exists a subset  $T_{5}\subseteq T_{4}$ of cardinality $\gg_{\d,d'',L}K^{d'}$ such that
\begin{equation}\label{2:gvanish}
g_{v_{1}}+\dots+g_{v_{8}}\equiv g_{v_{9}}+\dots+g_{v_{16}} \mod J^{M'}
\end{equation}
for all $v_{1},\dots,v_{16}\in T_{5}$ with $v_{1}+\dots+v_{8}=v_{9}+\dots+v_{16}$.

\textbf{Step 5. Completion of the proof.} We now have enough information for the maps $\eta_{v}$ to complete the proof of  Proposition \ref{2:laststand}.

 Let $H':=\cup_{v\in T_{5}}\cup (\phi_{U}(U'_{v})+v)$. Then $H'\subseteq H$ and $\vert H'\vert\gg_{\d,d'',L}K^{d}$. 
For all $h=v+\phi_{U}(u)\in H'$ with $v\in T_{5}$ and $u\in U'_{v}$,
define $\eta(h):=\eta_{v_{0}}(u)+g_{v}$ (note that $\eta$ is well defined since $T\cap U=\{\bold{0}\}$).
 Then
for all $h_{i}=v_{i}+\phi_{U}(u_{i})\in H'$, $v_{i}\in T_{5}$ and $u_{i}\in U'_{v_{i}}\subseteq \phi_{U}^{-1}(U_{v_{0}}), 1\leq i\leq 16$ with $h_{1}+\dots+h_{8}=h_{9}+\dots+h_{16}$, since $T\cap U=\{\bold{0}\}$, we have that $v_{1}+\dots+v_{8}=v_{9}+\dots+v_{16}$ and $u_{1}+\dots+u_{8}=u_{9}+\dots+u_{16}$. So by (\ref{2:basiccondition3}) and (\ref{2:gvanish}), we have
\begin{equation}\label{2:aaddee}
\begin{split}
&\quad(\eta(h_{1})+\dots+\eta(h_{8}))-(\eta(h_{9})+\dots+\eta(h_{16}))
\\&\equiv(\eta_{v_{1}}(u_{1})+\dots+\eta_{v_{8}}(u_{8}))-(\eta_{v_{9}}(u_{9})+\dots+\eta_{v_{16}}(u_{16}))
\\&\equiv(\eta_{v_{0}}(u_{1})+\dots+\eta_{v_{0}}(u_{8}))-(\eta_{v_{0}}(u_{9})+\dots+\eta_{v_{0}}(u_{16}))
+(g_{v_{1}}+\dots+g_{v_{8}})-(g_{v_{9}}+\dots+g_{v_{16}})
\\&\equiv(\eta_{v_{0}}(u_{1})+\dots+\eta_{v_{0}}(u_{8}))-(\eta_{v_{0}}(u_{9})+\dots+\eta_{v_{0}}(u_{16}))
 \equiv 0 \mod J^{M'}_{\iota(u_{1}),\dots,\iota(u_{15})}.
\end{split}	
\end{equation}
For all $h=v+\phi_{U}(u)\in H'$ with $v\in T_{5}$ and $u\in U'_{v}$, since  $\iota(u)=\phi^{-1}_{U_{0}}\circ\pi_{U_{0}}\circ\iota(h)$, we have that
\begin{equation}\nonumber
\begin{split}
\sigma_{m}(\xi(h))=\eta_{v}(u)=\eta_{v_{0}}(u)+g_{v}=\eta(h) \mod J^{M'}_{\iota(u)}=J^{M'}_{\phi_{U_{0}}^{-1}\circ\pi_{U_{0}}\circ \iota(h)},
\end{split}	
\end{equation}
where we changed the notation of $\sigma$ back to $\sigma_{m}$.
Therefore for all $h_{1},\dots,h_{16}\in H'$ with $h_{1}+\dots+h_{8}=h_{9}+\dots+h_{16}$, it follows from (\ref{2:aaddee}) that 
\begin{equation}\label{2:ssiik2}
\begin{split}
 \sigma_{m}((\xi(h_{1})+\dots+\xi(h_{8}))-(\xi(h_{9})+\dots+\xi(h_{16})))\equiv 0
 \mod J^{M'}_{\phi_{U_{0}}^{-1}\circ\pi_{U_{0}}(\sp_{\F_{p}}\{\iota(h_{1}),\dots,\iota(h_{15})\})}.
\end{split}	
\end{equation}

Since $\pi^{-1}(\bold{0})\cap(8\tilde{\xi}(H')\- 8\tilde{\xi}(H'))$ admits the same   structure-obstacle decomposition as $\pi^{-1}(\bold{0})\cap(8\tilde{\xi}(H)\- 8\tilde{\xi}(H))$, for any $N=O_{d}(1)$, replacing $H'$ by a subset with cardinality $\gg_{\d,d,L}L^{d}$ if necessary, we may use the same argument repeatedly to conclude that   (\ref{2:ssiik2}) holds for all $m\in [N]^{d'}$. 
The conclusion of  Proposition \ref{2:laststand}
 then follows from Proposition \ref{2:redf}.

\section{Proof of the main theorem}\label{2:s:a2}
 
 We are now ready to complete the proof of Theorem \ref{2:aadd}. Our first task is to
  remove the additional modulo over $J^{M}_{T}$ in Proposition \ref{2:laststand}. To this end, we need to use the following lemma to find many structure-obstacle decompositions (which allows us to use the intersection method):

\begin{lem}[Abundance of structure-obstacle decompositions]\label{2:manyd}
	Let $C,d',d'',D,L,s\in\N$, $d,R\in\N_{+}$ with $d=d'+d''$,  $p\gg_{C,d,L,s} 1$ be a prime,   $M\colon \V\to\F_{p}$ be a quadratic form, and $(\mathcal{C}_{0},Y)\subseteq \st_{\zp,d}(s)\times \V$  be a $(L,C,D)$-structure-obstacle pair.
	If $d'\geq \max\{D,3\}, d''\geq (R-1)d'$  and $d\geq 4$, then there exist $(\mathcal{C}_{0},Y,M,d',d'')$-structure-obstacle decompositions $(T_{i},U_{i}), 1\leq i\leq R$ such that $T_{1},\dots,T_{R}$ are linearly independent. 
\end{lem}
\begin{proof}	
	Throughout the proof we assume that $p\gg_{C,d,L,s} 1$.
	Let $A$ be the matrix associated to $M$.
	Suppose that for some $0\leq r\leq R-1$ we have chosen $(\mathcal{C}_{0},Y,M,d',d'')$-structure-obstacle decompositions $(T_{i},U_{i}), 1\leq i\leq r$ such that $T_{1},\dots,T_{r}$ are linearly independent. We continue to construct $(T_{r+1},U_{r+1})$.
	It then suffices to find a subspace $U_{r+1}$ of $\V$ of dimension $d'$ such that ${U_{r+1}}^{\perp_{M}}\cap (T_{1}+\dots+T_{r})=\{\bold{0}\}$, $\rank(M\vert_{U_{r+1}})=d'$, $U_{r+1}\cap Y=\{\bold{0}\}$ and $J^{M}_{U_{r+1}}\cap \mathcal{C}_{0}=\{\bold{0}\}$.

	Assume that $Y=\cup_{i=1}^{C}V_{i}$ and  $\mathcal{C}_{0}=\{{0},f_{1},\dots,f_{L}\}$ for some subspaces $V_{i}$ of dimension at most $D$ and $f_{j}\in \st_{\zp,d}(s)\backslash J^{M}$. 
	Let $\mathcal{B}$ be the set of all $d''$-tuples $\b=(b_{1},\dots,b_{d''})\in(\V)^{d''}$. Let $U(\b)$ denote the span of $b_{1},\dots,b_{d''}$.
	 Let $\mathcal{B}'$ be the set of all linearly dependent $d''$-tuples,  $\mathcal{B}''$ be the set of all $\b\in\mathcal{B}\backslash \mathcal{B}'$ such that $M\vert_{U(\b)}$ is non-degenerate, $\mathcal{B}'''$ be the set of $\b\in \mathcal{B}\backslash \mathcal{B}''$ such that $U(\b)^{\perp_{M}}\cap (T_{1}+\dots+T_{r})\neq\{\bold{0}\}$, $\mathcal{B}_{i}$ be the set of subspaces $\b\in \mathcal{B}\backslash \mathcal{B}''$ such that $U(\b)\cap V_{i}\neq\{\bold{0}\}$, and $\mathcal{B}'_{j}$ be the set of subspaces $\b\in\mathcal{B}\backslash \mathcal{B}''$ such that $f_{j}\in J^{M}_{U(\b)}$. To show the existence of $U_{k+1}$, it suffices to show that $\vert \mathcal{B}\vert>\vert \mathcal{B}'\vert+\vert \mathcal{B}''\vert+\vert \mathcal{B}'''\vert+\sum_{i=1}^{C}\vert \mathcal{B}_{i}\vert+\sum_{j=1}^{L}\vert \mathcal{B}'_{j}\vert$.

	Clearly $\vert\mathcal{B}\vert=p^{dd''}$.
	By Lemma \ref{2:iiddpp}, $\vert\mathcal{B}'\vert=O_{d''}(p^{dd''-1})$. We now compute $\vert \mathcal{B}''\vert$.
	Let $E(\b)$ be the $d''\times d$ matrix with $b_{1},\dots,b_{d''}$ as row vectors. Then for $\b\in \mathcal{B}\backslash \mathcal{B}'$, $U(\b)$ is non-degenerate if and only if $\det(E(\b)AE(\b)^{T})\neq 0$. Not that the map $\b\mapsto \det(E(\b)AE(\b)^{T})$ is a non-trivial polynomial of degree at most $O_{d}(1)$. So it follows from  Lemma \ref{2:ns} that $\vert\mathcal{B}''\vert=O_{d''}(p^{dd''-1})$.	
	
	We next compute  $\vert \mathcal{B}'''\vert$. 
	Let $m:=\dim(T_{1}+\dots+T_{r})\leq (R-1)d'$ and let $v_{1},\dots,v_{m}$ be a basis of $T_{1}+\dots+T_{r}$. Then $\b\in  \mathcal{B}'''$ if and only if there exist $x_{1},\dots,x_{m}\in\F_{p}$ not all equal to zero such that $(b_{i}A)\cdot (x_{1}v_{1}+\dots+x_{m}v_{m})=0$ for all $1\leq i\leq d''$. Since $d''\geq m$, this is true only if the determinant of the matrix $C(\b):=((b_{i}A)\cdot v_{j})_{1\leq i,j\leq m}$ is zero. Note that $\det(C(\b))$ is   a polynomial in $\b$ of degree at most $O_{d}(1)$.
%
 %
	On the other hand, it is not hard to inductively construct for any $1\leq i\leq m\leq d''$ vectors $b_{1},\dots,b_{i}\in\V$ with $b_{i}\notin\sp_{\F_{p}}\{b_{1},\dots,b_{i-1}\}$ and $(b_{i'}A)\cdot v_{j}=\delta_{i',j}$ for all $1\leq i'\leq i$ and $1\leq j\leq m$. 
	So
	$\det(C(\b))\neq 0$ and thus $\det(C(\cdot))$ is not the constant zero polynomial.  So it follows from  Lemma \ref{2:ns} that $\vert\mathcal{B}'''\vert=O_{d''}(p^{dd''-1})$.

	Now fix  $1\leq i\leq C$ and let $v_{1},\dots,v_{m}$ be a basis of $V_{i}$ for some $m\leq D$. Then $U(\b)\cap V_{i}\neq\{\bold{0}\}$ only if there exist $x_{1},\dots,x_{m},y_{1},\dots,y_{d''}\in\F_{p}$ with $y_{1},\dots,y_{d''}$ not all equal to zero such that 
	$$x_{1}v_{1}+\dots+x_{m}v_{m}=y_{1}b_{1}+\dots+y_{d''}b_{d''}.$$
	If $y_{1}\neq 0$, then $b_{1}$ is uniquely determined by the choices of $x_{1},\dots,x_{m},y_{2},\dots,y_{d''}$ and $b_{2},\dots,b_{d''}$. So
	$$\vert \mathcal{B}_{i}\vert\leq d''p^{(d''-1)d+m+d''-1}=O_{d}(p^{d''d-1})$$
	since $d'\geq D$.
	
	 Finally, fix $1\leq j\leq L$ and let
	$P$ be the set of $n\in\V\backslash\{\bold{0}\}$ such that $f_{j}(n)=(nA)\cdot n=0$. Since $f_{j}\notin J^{M}$, by  Lemmas \ref{2:iiddpp}, \ref{2:counting01} and Proposition \ref{2:bzt}, we have
 $\vert P\vert=O_{d,s}(p^{d-2})$ since $d\geq 4$.
	For $\b\in \mathcal{B}\backslash \mathcal{B}'$,
	let $P_{U(\b)}$ be the set of $n\in\V$ such that $f_{j}(n)=(nA)\cdot n=(hA)\cdot n=0$ for all $h\in U(\b)$. In other words,  $P_{U(\b)}$ is the intersection of $P$ and the subspace $$U(\b)^{\perp_{M}}:=\{ n\in\V\colon (hA)\cdot n=0 \text{ for all } h\in U(\b)\}$$ which is of dimension $d-\rank(M\vert_{U(\b)})=d'\geq 3$.
	
	If $f_{j}\in J^{M}_{U(\b)}$ for some $\b\in \mathcal{B}\backslash \mathcal{B}''$, then $P_{U(\b)}$ is the intersection of $V(M)$ and the subspace $U(\b)^{\perp_{M}}$.
	By  Lemma \ref{2:counting01}, $\vert P_{U(\b)}\vert=p^{d-d''-1}(1+O(p^{-1/2}))$.
	Therefore, in order to estimate $\vert \mathcal{B}'_{j}\vert$, it suffices to study the number of  $\b\in \mathcal{B}\backslash \mathcal{B}''$ such that $\vert P_{U(\b)}\vert=p^{d-d''-1}(1+O(p^{-1/2}))$.

Note that for $n\in\V$ and $\b\in \mathcal{B}\backslash \mathcal{B}'$, $n\in U(\b)^{\perp_{M}}$ if and only if $U(\b)\subseteq \sp_{\F_{p}}\{n\}^{\perp_{M}}$, which is a subspace of $\V$ of dimension 1 since $n\neq \bold{0}$. So the number of such $\b$ is at most $p^{(d-1)d''}$.
Then
	\begin{equation}\nonumber
	\begin{split}
	 &\quad\sum_{\b\in\mathcal{B}\backslash \mathcal{B}''}\vert P_{U(\b)}\vert=\sum_{\b\in\mathcal{B}\backslash \mathcal{B}''}\vert P\cap U(\b)^{\perp_{M}}\vert
=\sum_{n\in P}\vert\{\b\in\mathcal{B}\backslash \mathcal{B}''\colon n\in U(\b)^{\perp_{M}}\}\vert
\\&\leq\vert P\vert\cdot p^{(d-1)d''}=O_{d,s}(p^{dd''+d-d''-2}),
	\end{split}
	\end{equation}
	which implies that the number of $\b\in\mathcal{B}\backslash \mathcal{B}''$ with $\vert P_{U(\b)}\vert=p^{d-d''-1}(1+O(p^{-1/2}))$ is at most $O_{d,s}(p^{dd''-1})$. So $\vert \mathcal{B}'_{j}\vert\leq  O_{d,s}(p^{dd''-1})$ for any $1\leq j\leq L$.



	Combining the estimates of the cardinalities of $\mathcal{B}$, $\mathcal{B}'$, $\mathcal{B}_{i}$ and $\mathcal{B}'_{j}$ above,  we have $\vert \mathcal{B}\vert>\vert \mathcal{B}'\vert+\vert \mathcal{B}''\vert+\vert \mathcal{B}'''\vert+\sum_{i=1}^{C}\vert \mathcal{B}_{i}\vert+\sum_{j=1}^{L}\vert \mathcal{B}'_{j}\vert$.  This completes the proof.
\end{proof}

We are now ready to prove  Theorem \ref{2:aadd}.

	\begin{proof}[Proof of Theorem \ref{2:aadd}]
	Throughout the proof we assume that $p\gg_{\d,d,r} 1$. Since $d\geq (2s+16)(15s+453)$, we may write $d=d'+d''$ for some $d'\geq 15s+453$ and 
$d''\geq (2s+15)d'\geq (2s+15)(15s+453)$.
	By assumption, the $M$-energy  of $(\xi_{1},\xi_{2},\xi_{3},\xi_{4})$ is at least $\d K^{3d}$.
	Using the Cauchy-Schwartz inequality (Theorem \ref{2:gcs}) twice and the fact that $(\xi_{1},\xi_{2},\xi_{1},\xi_{2})$ has the same $M$-energy as  $(\xi_{1},\xi_{1},\xi_{2},\xi_{2})$, we have that the $M$-energy of $(\xi_{1},\xi_{1},\xi_{1},\xi_{1})$ is  $\gg_{\d,d}K^{3d}$. From now on we write $\xi:=\xi_{1}$ for convenience, and denote $\tilde{\xi}(h)=(h,J^{M}_{\iota(h)}+\xi(h))$ for all $h\in H$. 
 
	 By Theorem \ref{2:gbig1}, since $d\geq 2s+15$,   there exists a subset $H_{1}$ of $H$ of cardinality $\gg_{\d,d}K^{d}$ such that  $\dep(\tilde{\xi}(H_{1})\+ \tilde{\xi}(H_{1}))=O_{\d,d}(1)$. 
		By Proposition \ref{2:ga3}, since $d\geq \max\{s+32,69\}$,  there exists a subset $H_{2}$ of $H_{1}$ of cardinality $\gg_{\d,d}K^{d}$ such that    $\dep(8\tilde{\xi}(H_{2})\- 8\tilde{\xi}(H_{2}))=O_{\d,d}(1)$. 
		Since $8\tilde{\xi}(H_{2})\- 8\tilde{\xi}(H_{2})\subseteq \Gamma^{s}_{16}(\Vk,M)$ and $d\geq 28s+95$,
		by Proposition \ref{2:gweakcore2},    $\pi^{-1}(\bold{0})\cap(8\tilde{\xi}(H_{2})\- 8\tilde{\xi}(H_{2}))$ admits an $(O_{\d,d}(1),O_{\d,d}(1),15s+450)$-classification for some $(O_{\d,d,s}(1)$, $O_{\d,d,s}(1),15s+450)$-structure-obstacle pair $(\mathcal{C}_{0},Y)$.
	    By Lemma \ref{2:manyd}, since $d'\geq 15s+450, d''\geq (s+15)d'$,   there exist $(\mathcal{C}_{0},Y,M,d',d'')$-structure-obstacle decompositions $(T_{i},U_{i})$, $1\leq i\leq s+16$ such that $T_{1},\dots,T_{s+16}$ are linearly independent.
	   By repeatedly using Proposition \ref{2:laststand}, since $d'\geq 15s+453$ and $d''\geq \max\{4s+23,35\}$,   there exists a subset $H_{3}$ of $H_{2}$ of cardinality $\gg_{\d,d}K^{d}$ 
		such that for all $h_{1},\dots,h_{16}\in H_{3}$ with $h_{1}+\dots+h_{8}=h_{9}+\dots+h_{16}$, and $1\leq i\leq s+16$, we have that 
		\begin{equation}\nonumber
			\begin{split}
				\xi(h_{1})+\dots+\xi(h_{8})
				\equiv\xi(h_{9})+\dots+\xi(h_{16})  \mod J^{M}_{\sp_{\F_{p}}\{\iota(h_{1}),\dots,\iota(h_{15})\}+T_{i}}.
			\end{split}
		\end{equation}
		Since  $T_{1},\dots,T_{s+16}$ are linearly independent  and $d\geq 37+2(s+6)(d'-1)$ (or equivalently $d''\geq (2s+5)d'+2s+25$),
by Proposition \ref{2:gri} (setting $m=15$ and $r=d'$), we have that
\begin{equation}\nonumber
			\begin{split}
				 \xi(h_{1})+\dots+\xi(h_{8})
				\equiv\xi(h_{9})+\dots+\xi(h_{16})  \mod J^{M}_{\iota(h_{1}),\dots,\iota(h_{15})}
			\end{split}
		\end{equation}
		and so $\xi$ is a Freiman $M$-homomorphism of order 16 on $H_{3}$.

		Let $0<c=c(r,s)<1$ be a constant to be chosen later.
	Since $d\geq \max\{4s+19,35\}$, 
	 there exist $0<\rho=\rho(\d,d)<1/4$,
			a homogeneous proper generalized arithmetic progression 
			$$P=\{x_{1}v_{1}+\dots+x_{D}v_{D}\colon x_{i}\in \Z\cap(-L_{i},L_{i}), 1\leq i\leq D\}$$
			 in $\Z_{K}^{d}$ of rank $D=O_{\d,d}(1)$ and cardinality $\gg_{\d,d}K^{d}$, a set $S\subseteq \hat{\Z_{K}^{d}}$ of cardinality $D'=O_{\d,d}(1)$, and a Freiman $M$-homomorphism of order 4 $\omega\colon P\to \st_{\zp,d}(s)$ satisfying the conclusions of   Proposition \ref{2:gef1}:		
				\begin{enumerate}[(i)]
					\item $B(S,c\rho)\subseteq P(c)\subseteq P\subseteq B(S,1/4)\subseteq 2H_{3}-2H_{3}$;
					\item For all $h\in P$, the set $R(H_{3},h):=\{(h_{1},h_{2},h_{3},h_{4})\in H\colon h_{1}+h_{2}-h_{3}-h_{4}=h\}$ is of cardinality $\gg_{\d,d}K^{3d}$; 
						\item there exists $H_{4}\subseteq H_{3}$ with $\vert H_{4}\vert\gg_{\d,d} K^{d}$ such that $2H_{4}-2H_{4}\subseteq P(c)$;
						\item the vectors $(\{\alpha\cdot v_{i}\})_{\alpha\in S}\in\R^{D'}, 1\leq i\leq D$ are linearly independent over $\R$;
					\item for all $h_{1},h_{2},h_{3},h_{4}\in H_{3}$ with $h_{1}+h_{2}-h_{3}-h_{4}\in P$, we have
					\begin{equation}\label{2:ffxxxxw}
					\xi(h_{1})+ \xi(h_{2})- \xi(h_{3})- \xi(h_{4})\equiv \omega(h_{1}+h_{2}-h_{3}-h_{4}) \mod J^{M}_{\iota(h_{1}),\iota(h_{2}),\iota(h_{3}),\iota(h_{4})};
					\end{equation}
					\item for all $h_{1},h_{2},h_{3},h_{4}\in P$ with $h_{1}+h_{2}=h_{3}+h_{4}$, we have
					$$\omega(h_{1})+ \omega(h_{2})\equiv\omega(h_{3})+ \omega(h_{4}) \mod J^{M}_{\iota(h_{1}),\iota(h_{2}),\iota(h_{3})},$$
					i.e. the map $\omega\colon P\to\st_{\zp,d}(s)$ is a Freiman $M$-homomorphism of order 4.
				\end{enumerate}

	By Theorem \ref{2:gsol2}, since $d\geq \max\{s+10,11\}$, for an appropriately chosen $0<c=c(r,s)\leq 1$ depending only on $s$ and $r$,  there exists a   locally linear  map $T\colon P(c)\to \st_{\Z/p^{r},d}(s)$ such that 
	$\omega(u)\equiv T(u) \mod J^{M}_{\iota(u)}$ for all $u\in P(c)$.

		Our next goal is to provide an explicit formula for the locally linear map $T$.
		Assume  that
		$$T(x_{1}v_{1}+\dots+x_{D}v_{D})=u_{0}+x_{1}u_{1}+\dots+x_{D}u_{D}  \text{ for all } x_{i}\in \Z\cap (-cL_{i},cL_{i}), 1\leq i\leq D$$
		for some $u_{0},\dots,u_{D}\in \st_{\Z/p^{r},d}(s)$.
		We may naturally extend $T$ to a locally linear map on $P$ by setting 
		$$T(x_{1}v_{1}+\dots+x_{D}v_{D})=u_{0}+x_{1}u_{1}+\dots+x_{D}u_{D} \text{ for all } x_{i}\in \Z\cap (-L_{i},L_{i}), 1\leq i\leq D.$$
		For convenience we also denote this map by $T$.
		Let
		$\Phi\colon P\to\R^{D'}$ be the map given by
		$$\Phi(x):=(\{\alpha\cdot x\})_{\alpha\in S}.$$
		Since $P\subseteq B(S,\rho)$, $\Phi(P)$ lies inside $[-\rho,\rho]^{D'}$. Since $\rho<1/4$, it is easy to check that
		\begin{equation}\nonumber
		\Phi(x_{1}v_{1}+\dots+x_{D}v_{D})=x_{1}\Phi(v_{1})+\dots+x_{D}\Phi(v_{D}) \text{ for all } x_{i}\in (-L_{i},L_{i}), 1\leq i\leq D.
		\end{equation}

		\textbf{Claim.}  For all $1\leq i\leq D$, there exists $w_{i}=(w_{i,1},\dots,w_{i,D'})=(p_{i,1}/q_{i},\dots,p_{i,D'}/q_{i})\in \mathbb{Q}^{D'}$   for some $p_{i,j},q_{i}\in\Z$ with $\vert q_{i}\vert=O_{\d,d}(1)$ such that $w_{i}\cdot \Phi(v_{i})=1$ and $w_{i}\cdot \Phi(v_{j})=0$ for all $j\neq i$.

		Since $\Phi(v_{i}), 1\leq i\leq D$ are linearly independent over $\R$, it is easy to see the existences of  $w_{i}\in\mathbb{Q}^{D'}$ such that $w_{i}\cdot \Phi(v_{i})=1$ and $w_{i}\cdot \Phi(v_{j})=0$ for all $j\neq i$. The difficulty of the claim is to bound the value of $q_{i}$.

		Since  $\Phi(v_{i}), 1\leq i\leq D$ are linearly independent over $\R$ and $P$ is a proper generalized arithmetic progression,
		 the $D$-dimensional volume of the convex hull of $\Phi(P)$ as a subset of $\R^{D'}$ is at most $1/2^{D}$ (since it lies inside $[-1/4,1/4]^{D'}$). So the  $D$-dimensional volume of $D$-dimensional cube in $\R^{D'}$ generated by $K\Phi(v_{i}), 1\leq i\leq D$
	is at most	 
		  $\frac{K^{D}\vol(\Phi(P))}{\vert P\vert}=O_{\d,d}(1)$. By the Cauchy-Binet formula, there exists a $D\times D$ sub-matrix $B$ of the $D\times D'$ real matrix $(K\Phi(v_{i}))_{1\leq i\leq D}$ (where each $\Phi(v_{i})$ is viewed as a $D'$-dimensional horizontal vector) such that $0<\det(B)=O_{\d,d}(1)$. 
		Since $K\Phi(v_{i})\in \Z^{D'}$, the claim follows from the knowledge of linear algebra.

		\
		
		Let $w_{1},\dots,w_{D}$ be given by the claim.
		Then
		$$\Phi(x_{1}v_{1}+\dots+x_{D}v_{D})\cdot w_{i}=(x_{1}\Phi(v_{1})+\dots+x_{D}\Phi(v_{D}))\cdot w_{i}=x_{i}.$$  Then
	  for all $x=x_{1}v_{1}+\dots+x_{D}v_{D}\in P$,  we have that 
		$$T(x)=u_{0}+x_{1}u_{1}+\dots+x_{D}u_{D}=u_{0}+\sum_{i=1}^{D}(\Phi(x)\cdot w_{i})u_{i}=u_{0}+\sum_{i=1}^{D}(\{\alpha\cdot x\}\cdot w_{i})u_{i}.$$
		Since all the denominators of $w_{i}$ are at most $O_{\d,d}(1)$ in absolute values, it is not hard to see
		that  $T$ is a  $\Z/Qp^{r}$-almost linear  function on  $P(c)$ of complexity $O_{\d,d}(1)$ for some $Q\in\N_{+}$ with $Q=O_{\d,d}(1)$. 
		
		Let $Q^{\ast}$ be any integer with $Q^{\ast}Q\equiv 1 \mod p^{r}\Z$. Then $T':=Q^{\ast}QT$ is a $\Z/p^{r}$-almost linear  function on  $P(c)$ of complexity $O_{\d,d}(1)$. Since $(Q^{\ast}Q-1)T(u)\in \st_{\Z,d}(s)\subseteq J^{M}$ for all $u\in P(c)$, we have that
		$\omega(u)\equiv T'(u)+(1-Q^{\ast}Q)T(u)\equiv T'(u) \mod J^{M}_{\iota(u)}$ for all $u\in P(c)$.

	 We are now ready to find out the expression for $\xi$. 
	 	For all $h_{1},h_{2}\in H_{4}$, 
	 	since $h_{1}-h_{2}\in H_{4}-H_{4}\subseteq P(c)$, setting $h_{2}=h_{3}=h_{4}$ in (\ref{2:ffxxxxw}), we have
	 	$$\xi(h_{1})-\xi(h_{2})\equiv \omega(h_{1}-h_{2})\equiv T'(h_{1}-h_{2}) \mod J^{M}_{\iota(h_{1}),\iota(h_{2})}.$$
	 	Since $\{\alpha\}+\{\beta\}-\{\alpha+\beta\}\in \{0,1\}^{d}$ for all $\alpha,\beta\in\R^{d}$ and  $T'$ is an almost linear  function on  $P(c)$ of complexity $O_{\d,d}(1)$, we may find a set $\Omega\in \st_{\Z/p^{r},d}(s)$ of cardinality $O_{\d,d}(1)$ such that $T'(h_{1}-h_{2})-(T'(h_{1})-T'(h_{2}))\in \Omega$ for all $h_{1},h_{2}\in H_{4}$. So writing $\theta:=\xi-T'$, we have that 
	 	$$\theta(h_{1})-\theta(h_{2})\equiv g(h_{1},h_{2}) \mod J^{M}_{\iota(h_{1}),\iota(h_{2})}$$
	 	for some $g(h_{1},h_{2})\in\Omega$ for all $h_{1},h_{2}\in H_{4}$.
	 	
	 	We need to use the intersection method one more time.
	 	Since $\vert H_{4}\vert\gg_{\d,d} K^{d}$ and $d\geq s+3$, we may find some $p$-linearly independent $h_{1},\dots,h_{s+3}\in H_{4}$ by Lemma \ref{2:iiddpp}. By the Pigeonhole principle, there exists a subset $H_{5}$ of $H_{4}$ with $\vert H_{5}\vert\gg_{\d,d} K^{d}$ such that for all $h\in H_{5}$ and $1\leq i\leq s+3$, $g(h,h_{i})$ takes the same value $g(h_{i})$.
	 	Then for all $h,h'\in H_{5}$,
	 	$$\theta(h)-\theta(h')\equiv (\theta(h)-\theta(h_{i}))-(\theta(h')-\theta(h_{i}))\equiv g(h_{i})-g(h_{i})\equiv 0 \mod J^{M}_{\iota(h),\iota(h'),\iota(h_{i})}$$
	 	for all $1\leq i\leq s+3$. By Proposition \ref{2:gri} (setting $m=2$ and $r=1$), since $d\geq 11$, we have that 
	 	$$\theta(h)\equiv \theta(h') \mod J^{M}_{\iota(h),\iota(h')}$$
	 	for all $h,h'\in H_{5}$.
		 This means that 
		 $$X:=\{(\bold{0},J^{M}_{\iota(h)}+\theta(h))\colon h\in H_{5}\}\subseteq \Gamma_{2}^{s}(\Vk,M)$$ is a weak equivalence class. 
		Since $d\geq s+2$, it is impossible for all the elements of $\iota(H_{5})$ to lie in a common subspace of $\V$ of dimension $s+1$.
		By Proposition \ref{2:gwts} (setting $k=2$), $X$ is a strong equivalence class since $d\geq 9$.
		So there exists $g\in \st_{\zp,d}(s)$ such that $g\equiv\theta(h) \mod J^{M}_{\iota(h)}$ for all $h\in H_{5}$. 
		 Since $\xi(H_{5})\subseteq \st_{\Z/p^{r},d}(s)$, it is not hard to see from the proof of Proposition \ref{2:gwts}  that  we may further require $g\in \st_{\Z/p^{r},d}(s)$.   
			
		In conclusion, we have that 
		$$\xi(h)\equiv T'(h)+g \mod J^{M}_{\iota(h)}$$
		for all $h\in H_{5}$. 
			Since $\{\alpha\}+\{\beta\}-\{\alpha+\beta\}\in \{0,1\}^{d}$ for all $\alpha,\beta\in\R^{d}$  and  $T'$ is an $\Z/p^{r}$-almost linear  function of complexity $O_{\d,d}(1)$, there exits a subset 
		 $H_{6}$ of $H_{5}$ with $\vert H_{6}\vert\gg_{\d,d} K^{d}$ such that 
		   $T$ is an $\Z/p^{r}$-almost linear Freiman homomorphism on $H_{6}$.  We are done.
		   
%
%
%
\end{proof}

\section{Open questions}\label{2:s:op}

We collection of open questions in this section.
 Since the lower bound $d\geq N(s)$ obtained in Theorem \ref{2:aadd} is rather coarse, it is natural to ask:
   
   \begin{ques}
    What is the optimal lower bound for $d$ in Theorem \ref{2:aadd}?
\end{ques}
   
   The major cause of  the coarse lower bound for $d$ in Theorem \ref{2:aadd} is the error term $X_{b}$ appearing in Proposition \ref{2:gweakcore1}, which has nontrivial intersections with some subspaces of $\V$ of rather large dimensions (see also Remark \ref{2:coreimprovement}). Therefore, it is natural to ask:
   
      \begin{ques}\label{2:q:2}
    Can we make the set $X_{b}$ in Proposition \ref{2:gweakcore1} to be empty (or to have a small cardinality)?
\end{ques}
   
   A positive answer the Question \ref{2:q:2} will greatly improve the lower bound for $d$ in Theorem \ref{2:aadd}.
   Recall that this question is connection to Conjecture \ref {2:ccsdp}, which we restate below:

    \begin{ques}\label{2:q:3}
   Let $d,k,K\in\N_{+}$, $s\in\N$, $p$ be a prime dividing $K$ and $M\colon\V\to \F_{p}$ be a non-degenerate quadratic form. Let $G$ be the relation graph of a subset of $\Gamma^{s}_{k}(\Vk,M)$. If $d\gg_{k,s} 1$ and $p\gg_{d} 1$, then is it true that  $\cc(G)=O_{d,\dep(G)}(1)$?
\end{ques}

  Although we have demonstrated in Section \ref{2:s:115} that shifted modules enjoy many additive combinatorial properties similar to the abelian group case, some results in Section \ref{2:s:115}  are not proved in full generality. For example, one can ask:

  \begin{ques}\nonumber
   Does Proposition \ref{2:grt} hold for $B\subseteq \Gamma^{s}_{k}(\Vk,M)$ (instead of $B\subseteq \Gamma^{s}_{1}(\Vk,M)$)? Can we upgrade Proposition \ref{2:grt} from a quasi triangle inequality to a genial triangle inequality? 
\end{ques}

  \begin{ques}\nonumber
   Does Conjecture \ref{2:00ga31} (an improvement of Proposition \ref{2:ga31}) hold? 
\end{ques}

Finally, it is  interesting to ask about the connection between the conventional independence number of a graph and the density dependent number defined in this paper. This is Question \ref{que211} which we restate below:

\begin{ques}
	For any graph, is the independent number equal to $\dep(G)$? If not, then is the independent number bounded below by a constant depending on $\dep(G)$?
\end{ques}

\appendix

\section{Basic properties for quadratic forms}\label{2:AppA}

In this appendix, we collect some results proved in \cite{SunA} on quadratic forms which is used in this paper.

\subsection{The rank of quadratic forms}\label{2:AppA1}

Let $M\colon\V\to\F_{p}$ be a quadratic form associated with the matrix $A$ and $V$ be a subspace of $\V$. Let $V^{\pp}$ denote the set of $\{n\in\V\colon (mA)\cdot n=0 \text{ for all } m\in V\}$.
		A subspace $V$ of $\V$ is \emph{$M$-isotropic} if $V\cap V^{\pp}\neq\{\bold{0}\}$. 
	We say that a tuple $(h_{1},\dots,h_{k})$ of vectors in $\V$ is \emph{$M$-isotropic} 	if the span of $h_{1},\dots,h_{k}$ is an $M$-isotropic subspace.
	We say that a subspace or tuple of vectors is \emph{$M$-non-isotropic} if it is not $M$-isotropic.

\begin{lem}[Proposition 
4.8 of \cite{SunA}]\label{2:iissoo}
	Let $M\colon \V\to\F_{p}$ be a quadratic form and $V$ be a subspace of $\V$ of co-dimension $r$, and $c\in\V$.
	\begin{enumerate}[(i)]
		\item We have $\dim(V\cap V^{\pp})\leq \min\{d-\rank(M)+r,d-r\}$.
		\item The rank of $M\vert_{V+c}$ equals to $d-r-\dim(V\cap V^{\pp})$ (i.e. $\dim(V)-\dim(V\cap V^{\pp})$). 
		\item The rank of $M\vert_{V+c}$ is at most $d-r$ and at least $\rank(M)-2r$.
		\item $M\vert_{V+c}$ is non-degenerate (i.e. $\rank(M\vert_{V+c})=d-r$) if and only if $V$ is not an $M$-isotropic subspace.
	\end{enumerate}	
\end{lem}

  \begin{lem}[Lemma 
  4.9 of \cite{SunA}]\label{2:cbn}
 	Let $M\colon \V\to\F_{p}$ be a non-degenerate quadratic form, $V$ be a subspace of $\V$ of dimension $r$, and $V'$ be a subspace of $\V$ of  dimension $r'$. Suppose that $\rank(M\vert_{V^{\pp}})=d-r$. Then $\rank(M\vert_{(V+V')^{\pp}})\geq d-r-2r'$.
 \end{lem}


   \begin{lem}[Lemma 
  4.11 of \cite{SunA}]\label{2:iiddpp}
	Let $d,k,K\in\N_{+}$, $p$ be a prime dividing $K$, and $M\colon\V\to\F_{p}$ be a non-degenerate quadratic form.
	\begin{enumerate}[(i)]
		\item 
		The number of tuples $(h_{1},\dots,h_{k})\in (\Vk)^{k}$ such that $\iota(h_{1}),\dots,\iota(h_{k})$ are  linearly  dependent is at most 
		$k\frac{K^{dk}}{p^{d-k+1}}$.
		\item 
		The number of tuples $(h_{1},\dots,h_{k})\in (\Vk)^{k}$ such that $\iota(h_{1}),\dots,\iota(h_{k})$ are $M$-isotropic is at most 
		$O_{d,k}(\frac{K^{dk}}{p})$.
		\footnote{Although  Lemma 4.11 of \cite{SunA} was stated for the case $K=p$, the general case can be deduced immidiately.}
	\end{enumerate}	
\end{lem}
 
 \subsection{Some basic counting properties}
 
We refer the readers to Section \ref{2:s:defn} for the notations used in this section.  

\begin{lem}[Lemma 
4.10 of \cite{SunA}]\label{2:ns}
	Let $P\in\poly(\V\to\F_{p})$ be  of degree at most $r$.
	Then $\vert V(P)\vert\leq O_{d,r}(p^{d-1})$ unless $P\equiv 0$.
\end{lem}

    \begin{lem}[Corollary 
    4.14 of \cite{SunA}]\label{2:counting01}
	Let $d,r\in\N, d\geq 1$ and $p$ be a prime number. Let  $M\colon\V\to\F_{p}$ be a   quadratic form   and $V+c$ be an affine subspace of $\V$ of co-dimension $r$.  
	\begin{enumerate}[(i)]
	\item 	If $s:=\rank(M\vert_{V+c})\geq 3$, then
	$$\vert V(M)\cap (V+c)\vert=p^{d-r-1}(1+O(p^{-\frac{s-2}{2}})).$$	
	\item If $\rank(M)-2r\geq 3$, then $$\vert V(M)\cap (V+c)\vert=p^{d-r-1}(1+O(p^{-\frac{1}{2}})).$$		
	\end{enumerate}
\end{lem}

  \begin{lem}[Corollary 
  4.16 of \cite{SunA}]\label{2:counting02}
	Let $d,r\in\N_{+}$ and $p$ be a prime number. Let  $M\colon\V\to\F_{p}$ be a non-degenerate   quadratic form  and $h_{1},\dots,h_{r}$ be linearly independent vectors.  
		If $d-2r\geq 3$, then $$\vert V(M)^{h_{1},\dots,h_{r}}\vert=p^{d-r-1}(1+O(p^{-\frac{1}{2}})).$$	
\end{lem}

\subsection{Irreducible properties for quadratic forms}

\begin{prop}[Lemma 
6.1 and Proposition 
6.3 of \cite{SunA}]\label{2:bzt}
  	Let $d\in\N_{+},s\in\N$  and $p$ be a prime such that $p\gg_{d,s} 1$. Let $M\colon \V\to\F_{p}$ be a  quadratic form of rank at least 3.
  	For any $P\in\poly(\V\to\F_{p})$  of degree at most $s$,  either $\vert V(P)\cap V(M)\vert\leq O_{d,s}(p^{d-2})$ or $V(M)\subseteq V(P)$. 
	
	Moreover, if $V(M)\subseteq V(P)$, then $P=MR$ for some $R\in \poly(\V\to\F_{p})$ of degree at most $s-2$.
  \end{prop}

 \begin{prop}[Corollary 
 6.4 of \cite{SunA}]\label{2:noloop3}
 	Let  $d,k\in\N_{+}$, $s\in\N$, $p\gg_{d,k,s} 1$ be a prime number, 
 	$P\in \poly_{p}(\V\to\F_{p})$ be a polynomial of degree at most $s$, $M\colon\V\to\F_{p}$ be a non-degenerate quadratic form, $c\in\V$, and $V$ be a subspace of $\V$ of dimension $k$ with a basis $h_{1},\dots,h_{k}$. Suppose that $\rank(M\vert_{V^{\pp}})\geq 3$. 
 	 Then either $\vert V(M)\cap (V^{\pp}+c)\cap V(P)\vert\leq O_{d,k,s}(p^{d-k-2})$ or $V(M)\cap (V^{\pp}+c)\subseteq V(P)$.
 	
 	Moreover, if $V(M)\cap (V^{\pp}+c)\subseteq V(P)$, then 
 	\begin{equation}\label{1:pmm}
 	P(n)=M(n)P_{0}(n)+\sum_{i=1}^{k}((h_{i}A)\cdot (n-c))P_{i}(n)
 	\end{equation}
 	for some   polynomials $P_{0},\dots,P_{k}\in \poly_{p}(\V\to\F_{p})$ with $\deg(P_{0})\leq s-2$ and $\deg(P_{i})\leq s-1, 1\leq i\leq k$.
 	
 	In particular, the conclusion of this corollary holds if $d\geq k+\dim(V\cap V^{\pp})+3$, or if $d\geq 2k+3$.
 \end{prop}

	\begin{prop}[Proposition 7.7 of \cite{SunA}]\label{2:basicpp12}
			Let  $d\in\N_{+}$, $k,s\in\N$, $p\gg_{d,k,s} 1$ be a prime number, 
 	 $M\colon\Z^{d}\to\Z/p$ be a non-degenerate quadratic form induced by some $M'\colon\V\to\F_{p}$ associated with the matrix $A$, $c\in\V$, and $V$ be a subspace of $\V$ of dimension $k$ with a basis $\iota(h_{1}),\dots,\iota(h_{k})$ for some $h_{1},\dots,h_{k}\in\Z^{d}$. Denote  
	 $$L_{j}(n):=\frac{1}{p}(h_{j}\tau(A))\cdot (n-\tau(c))$$
		for $1\leq j\leq k$.	Suppose that $\rank(M'\vert_{V^{\perp_{M'}}})\geq 3$. 
 		The followings are equivalent:
		\begin{enumerate}[(i)]
		\item $f$ belongs to $\poly(V_{p}(M)\cap \iota^{-1}(V^{\pp}+c)\to \R\vert\Z)$ and is of degree at most $s$;
		\item we have
		\begin{equation}\nonumber
		Q_{0}f=\sum_{i=(i_{0},\dots,i_{k})\in\N^{k}\colon 2i_{0}+i_{1}+\dots+i_{k}\leq s}R_{i}M^{i_{0}}\prod_{j=1}^{k}L_{j}^{i_{j}}
		\end{equation}
	for some  $Q_{0}\in\N_{+}, Q_{0}\leq O_{d,s}(1)$ and some integer valued polynomials $R_{i}$ of degree at most $\deg(f)-(2i_{0}+i_{1}+\dots+i_{k})$, and that $p^{\deg(f)}f$ is an integer valued polynomial.  
		\end{enumerate}
		Moreover, if all the coefficients of $f$ are in $\Z/p^{r}$ for some $r\in\N$, then we may  write
	\begin{equation}\nonumber
		f=\sum_{i=(i_{0},\dots,i_{k})\in\N^{k}\colon 2i_{0}+i_{1}+\dots+i_{k}\leq s, i_{0}+i_{1}+\dots+i_{k}\leq \min\{r,s\}}R_{i}M^{i_{0}}\prod_{j=1}^{k}L_{j}^{i_{j}}
		\end{equation}
	for some integer coefficient polynomials $R_{i}$ of degree at most $\deg(f)-(2i_{0}+i_{1}+\dots+i_{k})$.
			
		 In particular, the conclusion of this proposition holds if $d\geq k+\dim(V\cap V^{\pp})+3$, or if $d\geq 2k+3$.
	\end{prop}

\bibliographystyle{plain}
\bibliography{swb}
\end{document}